\documentclass[graybox, envcountchap, envcountsect, envcountsame, envcountresetsect]{svmult}

\usepackage{hyperref}
\usepackage{mathptmx}        % selects Times Roman as basic font
\usepackage{helvet}          % selects Helvetica as sans-serif font
\usepackage{courier}         % selects Courier as typewriter font
\usepackage{makeidx}         % allows index generation
\usepackage{graphicx}        % standard LaTeX graphics tool
\usepackage{multicol}        % used for the two-column index
\usepackage[bottom]{footmisc}
\usepackage{lineno}
\font\cyreight=wncyr8
\def\cprime{$'$}
\let\C=\undefined

\usepackage{amsrefs,amsthm,amsmath,amsfonts,amssymb}

\usepackage{mathrsfs,mathtools,bbold,enumitem,stmaryrd,tikz,wrapfig}
\usetikzlibrary{decorations.pathmorphing,calc,intersections,shapes}
\usepackage{people}
\newcommand\C{{\mathbb C}}
\newcommand\N{{\mathbb N}}
\newcommand\Q{{\mathbb Q}}
\newcommand\R{{\mathbb R}}
\newcommand\Z{{\mathbb Z}}
\newcommand\Sym{{\operatorname{Sym}}}
\newcommand\Alt{{\operatorname{Alt}}}
\newcommand\Fol{{\operatorname{F\o l}}}
\newcommand\SO{{\operatorname{SO}}}
\newcommand\PSL{{\operatorname{PSL}}}
\newcommand\GL{{\operatorname{GL}}}

\newcommand\BG{{\mathsf B}}
\newcommand\GG{{\mathsf G}}
\newcommand\IET{{\mathsf{IET}}}
\renewcommand\index[2]{}
\newbox\pairbox
\def\pair<#1>{{\mathsurround=0pt
    \setbox\pairbox\hbox{$\left\langle#1\right\rangle$}
    \left\langle\kern-0.35\ht\pairbox
      \copy\pairbox\kern-0.35\ht\pairbox\right\rangle}}
\newcommand\Alphabet{{\mathcal A}}
\newcommand\restrict{{\downharpoonright}}
\newcommand\supp{{\operatorname{support}}}

\DeclareFontFamily{U}{mathb}{\hyphenchar\font45}
\DeclareFontShape{U}{mathb}{m}{n}{
      <5> <6> <7> <8> <9> <10> gen * mathb
      <10.95> mathb10 <12> <14.4> <17.28> <20.74> <24.88> mathb12
      }{}
\DeclareSymbolFont{mathb}{U}{mathb}{m}{n}
\DeclareMathSymbol{\righttoleftarrow}{3}{mathb}{"FD}

\pgfdeclaredecoration{irregular fractal line}{init}
{
  \state{init}[width=\pgfdecoratedinputsegmentremainingdistance]
  {
    \pgfpathlineto{\pgfpoint{random*\pgfdecoratedinputsegmentremainingdistance}{(random*\pgfdecorationsegmentamplitude-0.02)*\pgfdecoratedinputsegmentremainingdistance}}
    \pgfpathlineto{\pgfpoint{\pgfdecoratedinputsegmentremainingdistance}{0pt}}
  }
}

\pgfdeclaredecoration{Fractal Tree}{init}
{
  \state{init}[width=\pgfdecoratedinputsegmentremainingdistance]
  {
    \pgfpathlineto{\pgfpoint{\pgfdecoratedinputsegmentremainingdistance}{0pt}}
    \pgfpathmoveto{\pgfpoint{.55\pgfdecoratedinputsegmentremainingdistance}{0pt}}
    \pgfpathlineto{\pgfpoint{.55\pgfdecoratedinputsegmentremainingdistance}{.45\pgfdecoratedinputsegmentremainingdistance}}
    \pgfpathmoveto{\pgfpoint{.55\pgfdecoratedinputsegmentremainingdistance}{0pt}}
    \pgfpathlineto{\pgfpoint{.55\pgfdecoratedinputsegmentremainingdistance}{-.45\pgfdecoratedinputsegmentremainingdistance}}
    \pgfpathmoveto{\pgfpoint{.55\pgfdecoratedinputsegmentremainingdistance}{0pt}}
    \pgfpathlineto{\pgfpoint{\pgfdecoratedinputsegmentremainingdistance}{0pt}}
  }
}

\tikzset{
   irregular border/.style={decoration={irregular fractal line, amplitude=0.2},
           decorate,
     },
   ragged border/.style={ decoration={random steps, segment length=4mm, amplitude=1.8mm},
           decorate,
   }
}

\makeatletter
\tikzset{
  use path for main/.code={%
    \tikz@addmode{%
      \expandafter\pgfsyssoftpath@setcurrentpath\csname tikz@intersect@path@name@#1\endcsname
    }%
  },
  use path for actions/.code={%
    \expandafter\def\expandafter\tikz@preactions\expandafter{\tikz@preactions\expandafter\let\expandafter\tikz@actions@path\csname tikz@intersect@path@name@#1\endcsname}%
  },
  use path/.style={%
    use path for main=#1,
    use path for actions=#1,
  }
}
\makeatother

\begin{document}
\iffalse
\setcounter{chapter}{10}
\frontmatter
\tableofcontents
\mainmatter

%%%%%%%%%%%%%%%%%%%%%%%%%%%%%%%%%%%%%%%%%%%%%%%%%%%%%%%%%%%%%%%%

\title{Amenability of groups and $G$-sets}
\author{Laurent Bartholdi}
\authorrunning{L. Bartholdi} 
\institute{Laurent Bartholdi \at  \'Ecole Normale Sup\'erieure, Paris, \email{laurent.bartholdi@gmail.com}}
\thanks{Supported by the ``@raction'' grant ANR-14-ACHN-0018-01}

\maketitle
\else
\markboth{\small L. Bartholdi}{Amenability of groups and $G$-sets}

\noindent\textbf{\LARGE\boldmath Amenability of groups and $G$-sets}\bigskip

\noindent Laurent Bartholdi\footnote{\'Ecole Normale Sup\'erieure, Paris, \email{laurent.bartholdi@gmail.com}\\
Supported by the ``@raction'' grant ANR-14-ACHN-0018-01}\bigskip
\def\abstract*#1{\noindent\textbf{Abstract. }#1}
\fi

\abstract*{This text surveys classical and recent results in the field of
  amenability of groups, from a combinatorial standpoint. It has
  served as the support of courses at the University of G\"ottingen
  and the \'Ecole Normale Sup\'erieure.
  The goals of the text are (1) to be as self-contained as possible, so as
  to serve as a good introduction for newcomers to the field; (2) to stress
  the use of combinatorial tools, in collaboration with functional
  analysis, probability etc., with discrete groups in focus; (3) to
  consider from the beginning the more general notion of amenable
  \emph{actions}; (4) to describe recent classes of examples, and in
  particular groups acting on Cantor sets and topological full groups.}

%%%%%%%%%%%%%%%%%%%%%%%%%%%%%%%%%%%%%%%%%%%%%%%%%%%%%%%%%%%%%%%%
\section{Introduction}
In 1929, \JvNeumann\ introduced in~\cite{vneumann:masses} the notion
of amenability of $G$-spaces.  Fundamentally, it considers the
following property of a group $G$ acting on a set $X$: \emph{The right
  $G$-set $X$ is amenable if there exists a $G$-invariant mean on the
  power set of $X$}, namely a function
$m\colon\{\text{subsets of }X\}\to[0,1]$ satisfying
$m(A\sqcup B)=m(A)+m(B)$ and $m(X)=1$ and $m(A g)=m(A)$ for all
$A,B\subseteq X$ and all $g\in G$.

Amenability may be thought of as a finiteness condition, since
non-empty finite $G$-sets are amenable with $m(A)=\#A/\#X$; it may
also be thought of as a fixed-point property: on a general $G$-set
there exists a $G$-invariant mean; on a compact $G$-set there exists a
$G$-invariant measure, and on a convex compact $G$-set there exists a
$G$-fixed point, see~\S\ref{sec:convex}; on a $G$-measure space there
exists a $G$-invariant measurable family of means on the orbits,
see~\S\ref{ss:equivrel}.

Amenability may be defined for other objects such as graphs and random
walks on sets. If $X$ is a $G$-set and $G$ is finitely generated, then
$X$ naturally has the structure of a graph, with one edge from $x$ to
$x s$ for every $x\in X,s\in S$. Amenability means, in the context of
graphs, that there are finite subsets of $X$ with arbitrarily small
boundary with respect to their size. In terms of random walks, it
means that there are finite subsets with arbitrarily small
connectivity between the set and its complement, and equivalently that
the return probability of the random walk decreases subexponentially
in time; see~\S\ref{sec:kesten}.

The definition may also be modified in another direction: rather than
considering group actions, we may consider equivalence relations, or
more generally groupoids. The case we concentrate on is an equivalence
relation with countable leaves on a standard measure space. The orbits
of a countable group acting measurably naturally give rise to such an
equivalence relation. This point of view is actually very valuable:
quite different groups (e.g.\ one with free subgroups, one without)
may generate the same equivalence relation; see~\S\ref{ss:equivrel}.

One of the virtues of the notion of amenability of $G$-sets is that
there is a wealth of equivalent definitions; depending on context, one
definition may be easier than another to check, and another may be
more useful. In summary, the following will be shown, in the text, to
be equivalent for a $G$-set $X$:
\begin{itemize}
\item $X$ is amenable; i.e.\ there is a $G$-invariant mean on subsets
  of $X$;
\item There is a $G$-invariant normalized positive functional in
  $\ell^\infty(X)^*$, see Corollary~\ref{cor:vneumann};
\item For every bounded functions $h_i$ on $X$ and $g_i\in G$ the
  function $\sum _i(1-g_i)$ is non-negative somewhere on $X$, see
  Theorem~\ref{thm:dixmier};
\item For every finite subset $S\subseteq G$ and every $\epsilon>0$
  there exists a finite subset $F\subseteq X$ with
  $\#(F S\setminus F)<\epsilon\#F$, see Theorem~\ref{thm:folneramen}(5);
\item For every finite subset $S\subseteq G$, every $\epsilon>0$ and
  every $p\in[1,\infty)$ there exists a positive function
  $\phi\in\ell^p(X)$ with $\|\phi s-\phi\|<\epsilon\|\phi\|$ for all
  $s\in S$, see Theorem~\ref{thm:folneramen}(4);

\item There does not exist a ``paradoxical decomposition'' of $X$,
  namely
  $X=Z_1\sqcup\cdots\sqcup Z_m=Z_{m+1}\sqcup\cdots\sqcup Z_{m+n}=Z_1
  g_1\sqcup\cdots\sqcup Z_{m+n}g_{m+n}$ for some $Z_i\subset X$ and
  $g_i\in G$, see Theorem~\ref{thm:pd}(2);
\item There does not exist a map $\phi\colon X\righttoleftarrow$ and a
  finite subset $S\subseteq G$ with $\#\phi^{-1}(x)=2$ and
  $\phi(x)\in x S$ for all $x\in X$, see Theorem~\ref{thm:pd}(4);
\item There does not exist a free action of a non-amenable group $H$
  on $X$ with the property that for every $h\in H$ there is a finite
  subset $S\subseteq G$ with $x h\in x S$ for all $x\in X$, see
  Theorem~\ref{thm:whyte}(3);
\item Every convex compact set equipped with a $G$-equivariant map
  from $X$ admits a fixed point, see Theorem~\ref{thm:fixedpoint};
\item Every compact set equipped with a $G$-equivariant map from $X$
  admits an invariant measure, see Theorem~\ref{thm:fixedmeasure};
\item The isoperimetric constant (Definition~\ref{defn:isoperimetric})
  of every non-degenerate $G$-driven random walk on $X$ vanishes, see
  Theorem~\ref{thm:kesten}(2);
\item The spectral radius (Definition~\ref{defn:isoperimetric}) of
  every non-degenerate $G$-driven random walk on $X$ is equal to $1$,
  see Theorem~\ref{thm:kesten}(3).
\end{itemize}

Amenability has been given particular attention for groups themselves,
seen as $G$-sets under right multiplication; see the next section. We
stress that many results that exclusively concern groups (e.g., the
recent proofs that topological full groups are amenable) are actually
proven using amenable $G$-sets in a fundamental manner. The reason is
that a group is amenable if and only if it acts on an amenable $G$-set
with amenable point stabilizers, see
Proposition~\ref{prop:stabilizers}.

Quotients of amenable $G$-sets are again amenable; but sub-$G$-sets of
amenable $G$-sets need not be amenable. A stronger notion will be
developed in~\S\ref{ss:extensiveamen}, that of \emph{extensively
  amenable} $G$-sets. It has the fundamental property that, if
$\pi\colon X\twoheadrightarrow Y$ is a $G$-equivariant map between
$G$-sets, then $X$ is extensively amenable if and only if both $Y$ and
all $\pi^{-1}(y)$ are extensively amenable, the latter for the action
of the stabilizer $G_y$.

We detail slightly Reiter's characterization of amenability given
above: the space $\ell^1(G)$ of summable functions on $G$ is a Banach
algebra under convolution, and $\ell^1(X)$ is a Banach
$\ell^1(G)$-module. We denote by $\varpi(\ell^1G)$ and
$\varpi(\ell^1X)$ respectively the ideal and submodule of functions
with $0$ sum, and by $\ell^1_+(G)$ and $\ell^1_+(X)$ the cones of
positive elements.

Then $X_G$ is \emph{amenable} if and only if for every $\epsilon>0$
and every $g\in\varpi(\ell^1G)$ there exists $f\in\ell^1_+(X)$ with
$\|f g\|<\epsilon\|f\|$, see Proposition~\ref{prop:Ramenable}.

The quantifiers may be exchanged; we call $X_G$ \emph{laminable} if
for every $\epsilon>0$ and every $f\in\varpi(\ell^1X)$ there exists
$g\in\ell^1_+(G)$ with $\|f g\|<\epsilon\|g\|$, see
Theorem~\ref{thm:Lamenable}. It has the consequence that there
exists a measure $\mu$ on $G$ such that every $\mu$-harmonic function
on $X$ is constant.

In case $X=G_G$, these definitions are equivalent, but for $G$-sets the
properties of being amenable or Liouville are in general position.

\begin{figure}
  \centering\begin{tikzpicture}
    \node[red] (1) at (0,0) {$1$};
    \node[red] (2) at (0.4,0.3) {$\Z/2$};
    \node[anchor=west,red] (Z) at (1.1,0.9) {$\Z$};
    \draw (-0.3,0.9) node [below right] {Finite} rectangle (1.05,-0.2);
    
    \node[anchor=west,red] (Z2) at (1.1,0.3) {$\Z^2$};
    \node[anchor=west,rotate=-20] (abelian) at (2,1) {(virt.) abelian};
    \node[anchor=west,rotate=-20] (nilpotent) at (3,1) {(virt.) nilpotent};
    \node[anchor=west,rotate=-20] (pc) at (4,1) {(virt.) polycyclic};
    \node[anchor=west,rotate=-20] (sol) at (5,1) {(virt.) soluble};
    \draw (-0.36,-0.26) rectangle (10,1.1) node [below left] {Elementary amenable};
    \node[anchor=west] (poly) at (1.7,1.4) {polynomial growth};
    \node[anchor=west] (subexp) at (2.6,1.8) {subexponential growth};
    \node[anchor=west,red] (grig) at (4.5,1.4) {Grigorchuk's group};
    \draw (-0.42,-0.32) -- (-0.42,1.16) -- ++(2.1,0) -- ++(0,1.04) -- (10.06,2.2) node [below left] {Subexponentially amenable} -- (10.06,-0.32) -- cycle;
    \draw[very thick] (-0.48,-0.38) -- (-0.48,1.22) -- ++(2.1,0) -- ++(0,1.4) -- (10.12,7) -- (10.12,-0.38) -- cycle;
    \draw[very thick] (10.12,7) -- (-0.48,7) -- (-0.48,1.22);
    \node[anchor=west,red] (F2) at (1.1,2.0) {$F_2$};
    \node[anchor=east,rotate=-40] at (1.5,2.5) {(virt.) free};
    \node[anchor=east,rotate=-40] at (1.5,3.2) {small cancellation};
    \node[anchor=west,red] (Sg) at (1.1,3.8) {$\pi_1(\Sigma_g)$};
    \node[anchor=east,rotate=-40] at (1.5,4.2) {word-hyperbolic};
    \node[red] at (1.6,6.1) {non-virt.-soluble matrix groups};
    \draw (-0.42,1.28) -- ++(1.98,0) -- ++(0,1.4) -- (4.3,6.94) -- (-0.42,6.94) node [below right] {Groups with free subgroups} -- cycle;
    \node[rotate=20,red] at (3.5,4.0) {Frankenstein's group};
    \node[rotate=20,red] at (4.0,4.7) {$B(n,m)$ for $m\gg1$};
    \node[red] at (2.9,2.6) {Basilica group};
    \node[red] at (7.1,2.6) {Topological full groups of minimal $\Z$-actions};
    \node[anchor=east] at (10,4.0) {Amenable groups};
    \node[rotate=20,anchor=east] at (6.6,6.9) {Non-amenable groups};
    \node[anchor=north east] at (10.12,7.0) [label=below:{``HC SVNT DRACONES''}] {\includegraphics[width=34mm]{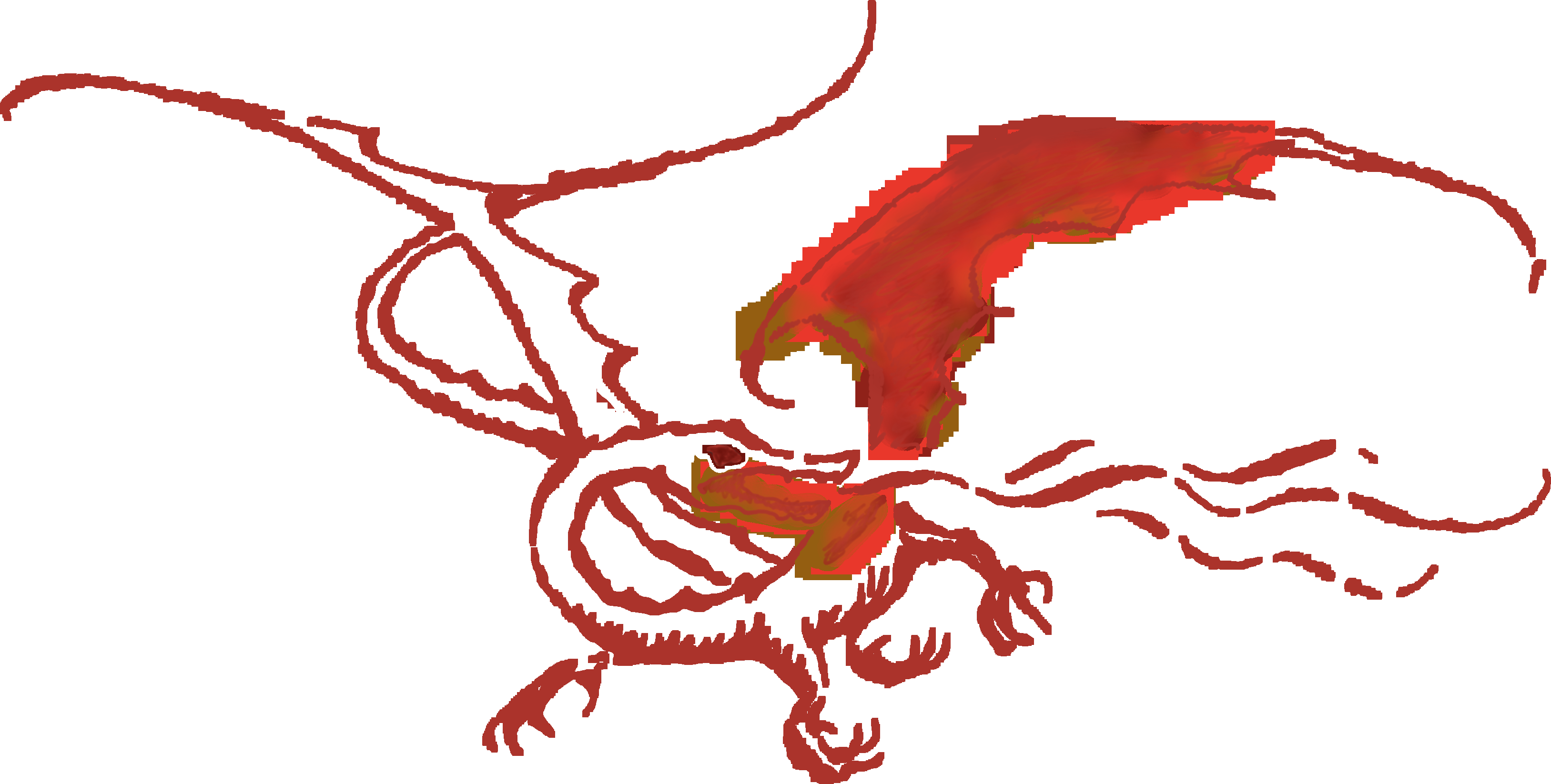}};
  \end{tikzpicture}
  \caption{The universe of groups}\label{fig:universe}
\end{figure}
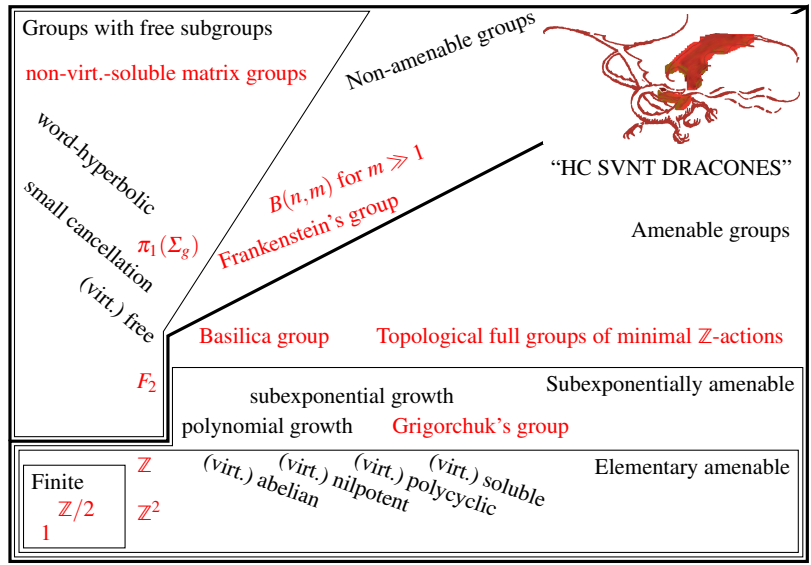

\subsection{Amenability of groups}
\JvNeumann's purpose, in introducing amenability of $G$-spaces, was to
understand better the group-theoretical nature of the
Hausdorff-Banach-Tarski paradox. This paradox, due to Banach and
Tarski~\cite{banach-tarski:pd} and based on Hausdorff's
work~\cite{hausdorff:pd}, states that a solid ball can be decomposed
into five pieces, which when appropriately rotated and translated can
be reassembled in two balls of same size as the original one. It could
have been felt as a death blow to measure theory; it is now resolved
by saying that the pieces are not measurable.

A group is called \emph{amenable} if all non-empty $G$-sets are
amenable; and it suffices to check that the regular $G$-set $G_G$ is
amenable, see Corollary~\ref{cor:regamen}.

Using the ``paradoxical decompositions'' criterion, it is easy to see
that the free group $F_2=\langle a,b\mid\rangle$ is not amenable: we
exhibit a partition
$F_2=G_1\sqcup\cdots\sqcup G_m\sqcup H_1\sqcup\cdots\sqcup H_n$ and
elements $g_1,\dots,g_m,h_1,\dots,h_n$ with
$F_2=G_1g_1\sqcup\cdots\sqcup G_m g_m=H_1h_1\sqcup\cdots\sqcup H_n
h_n$ as follows. Set
\begin{align*}
  G_1 & =\{\text{words whose reduced form ends in }a\}\cup\{1,a^{-1},a^{-2},\dots\},\\
  G_2 &= \{\text{words whose reduced form ends in }a^{-1}\}\setminus\{a^{-1},a^{-2},\dots\},\\
  H_1 &= \{\text{words whose reduced form ends in }b\},\\
  H_2 &= \{\text{words whose reduced form ends in }b^{-1}\};
\end{align*}
then $F_2=G_1\sqcup G_2\sqcup H_1\sqcup H_2=G_1\sqcup G_2 a=H_1\sqcup H_2 b$.

The group of rotations $\SO_3(\R)$ contains a free subgroup $F_2$, and
even one that acts freely on the sphere $S^2$; so its orbits are all
isomorphic to $F_2$. Choose a \emph{transversal}: a subset $T\subset
S^2$ intersecting every $F_2$-orbit in exactly one point. Consider
then the sphere partition $S^2=T G_1\sqcup T G_2\sqcup T H_1\sqcup
T H_2=T G_1\sqcup T G_1 a=T H_1\sqcup T H_2 b$; this is the basis for the
paradoxical Hausdorff-Banach-Tarski decomposition.

\JvNeumann\ also noted that the class of amenable groups is closed
under the following operations $(*)$: subgroups, quotients,
extensions, and directed unions. It contains all finite and abelian
groups. More generally, a criterion due to F\o lner,
Theorem~\ref{thm:folneramen}(5), shows that all groups in which every
finite subset generates a group of subexponential word
growth\footnote{Namely, in which the number of elements expressible as
  a product of at most $n$ generators grows subexponentially in $n$.}
is amenable. One may therefore define the following classes:
\begin{align*}
  EG &= \text{the smallest class containing finite and abelian groups and closed under }(*),\\
  SG &= \text{the smallest class containing groups of subexponential growth and closed under }(*),\\
  AG &= \text{the class of amenable groups},\\
  NF &= \text{the class of groups with no free subgroups};
\end{align*}
and concrete examples show that all inclusions
\[EG\subsetneqq SG\subsetneqq AG\subsetneqq NF
\]
are strict: the ``Grigorchuk group'' $\GG$ for the first inclusion,
see~\S\ref{ss:interm}; the group of ``bounded tree automorphisms'' for
the second inclusion, see~\S\ref{ss:subexpamen}; and the
``Frankenstein group'' for the last one, see~\S\ref{ss:fgfg}.

This text puts a strong emphasis on examples; they are essential to
obtain a (however coarse) picture of the universe of discrete groups,
see Figure~\ref{fig:universe}.  A fairly general framework contains a
large number of important constructions: groups acting on Cantor
sets. On the one hand, if we choose $X=\Alphabet^\N$ as model for the
Cantor set, we have examples of groups defined by automatic
transformations of $X$, namely by actions of invertible
transducers. On the other hand, we may fix a ``manageable'' group $H$
acting on $X$, and consider the group of self-homeomorphisms of $X$
that are piecewise $H$.

Examples of the first kind may be constructed via their
recursively-defined actions on $X$. The Grigorchuk group $\GG$ is the
group acting on $\{0,1\}^\N$ and generated by four elements $a,b,c,d$
defined by
\begin{xalignat*}{2}
  a(x_0x_1\dots)&= (1-x_0)x_1\dots,& b(x_0x_1\dots)&=\begin{cases}x_0\,a(x_1\dots) & \text{ if }x_0=0,\\ x_0\,c(x_1\dots) & \text{ if }x_0=1,\end{cases}\\
  c(x_0x_1\dots)&=\begin{cases}x_0\,a(x_1\dots) & \text{ if }x_0=0,\\ x_0\,d(x_1\dots) & \text{ if }x_0=1,\end{cases}&
  d(x_0x_1\dots)&=\begin{cases}x_0x_1\dots & \text{ if }x_0=0,\\ x_0\,b(x_1\dots) & \text{ if }x_0=1.\end{cases}
\end{xalignat*}
The Grigorchuk group gained prominence in group theory for being a
finitely generated infinite torsion group, and for having intermediate
word-growth between polynomial and exponential,
see~\S\ref{ss:interm}. An amenable group that does not belong to the
class $SG$ is the ``Basilica group'' $\BG$, generated by two elements
$a,b$ acting recursively on $\{0,1\}^\N$ by
\begin{xalignat*}{2}
  a(x_1x_2\dots)&=\begin{cases}1x_2\dots & \text{ if }x_1=0,\\ 0\,b(x_2\dots) & \text{ if }x_1=1,\end{cases}&
  b(x_1x_2\dots)&=\begin{cases}0x_2\dots & \text{ if }x_1=0,\\ 1\,a(x_2\dots) & \text{ if }x_1=1.\end{cases}
\end{xalignat*}
The Basilica group is a subgroup of the group of bounded tree
automorphisms, whose amenability will be proven
in~\S\ref{ss:subexpamen}.

These groups are \emph{residually finite}: the action on $\{0,1\}^\N$
is the limit of actions on the finite sets $\{0,1\}^n$ as
$n\to\infty$, so that the groups may be arbitrarily well approximated
by their finite quotients. More conceptually, the actions of $\GG$ and
$\BG$ on $\{0,1\}^\N$ induce actions on the clopens of
$\{0,1\}^\infty$, and every clopen has a finite orbit, giving rise to
a finite quotient acting by permutation on the orbit.

Examples of the second kind include the ``Frankenstein'' group
mentioned above, which is a non-amenable group acting on the circle by
piecewise projective transformations, and ``topological full groups''
of a minimal action of $H=\Z$ on a Cantor set; for example, let
$\sigma\colon 0\mapsto 01,1\mapsto 0$ be the Fibonacci substitution,
and consider $H=\langle S\rangle$ the two-sided shift on the subset
$X=\overline{S^n(\sigma^\infty(0))}\subset\{0,1\}^\Z$. Let $G$ be the
group of piecewise-$H$ homeomorphisms of $X$. Then $G'$ is an example
of a simple, infinite, finitely generated, amenable group.

These groups' actions on the Cantor set exhibit behaviours at the
exact opposite of $\GG$ and $\BG$: the actions are \emph{expansive}:
the orbit of a clopen may be used to separate points in
$X$. Topological full groups shall be used to produce examples of
finitely generated, infinite, amenable \emph{simple} groups.

Finally, we consider in~\S\ref{ss:linear} the adaptation of
amenability to a linear setting: on the one hand, a natural notion of
amenability of $\mathscr A$-modules for an associative algebra
$\mathscr A$; and, on the other hand, a characterization of
amenability by cellular automata.

\subsection{Why this text?}
After \JvNeumann's initial work in the late 1920's, amenability of
groups has developed at great speed in the 1960's, and then remained
mostly dormant till the late 2000's, when a variety of new techniques
and examples appeared. It seems now to be a good time to reread and
rewrite the fundamentals of the field with these developments in mind.

I have done my best to include all the material I found digestible,
and to express it in the ``best'' generality, namely the maximum
generality that does not come at the price of arcane definitions or
notation. Whenever possible, I included complete proofs of the
results, so that the text may be used for a course as well as for a
reference.

I have also striven to follow von Neumann's use of $G$-sets rather
than groups; it seems to me that clarity is gained by separating the
set $X$ (with a right $G$-action) from the group $G$.

I have also, consciously, avoided any mention of amenability for
topological groups. This notion is well developed for second-countable
locally compact groups, see
e.g.~\cite{bekka-h-v:t,reiter:harmonicanalysis}, so I should justify
its exclusion. I have felt that either the results stated for discrete
groups extend more-or-less obviously to topological groups (and then
there is no point in loading the notation with topology), or they
don't extend, and then the additional effort would be a distraction
from the main topic.

I have also devoted a fairly large part of the text to examples; and,
in particular, to groups defined by their action on a Cantor set, see
the previous section. I have included exercises, with ranking *=just
check the definitions, **=requires some thought, ***=probably very
difficult. Problems are like ***-exercises, but are questions rather
than statements.

I have consulted a large number of sources, and did my best to
attribute to their original authors all results and fragments of proof
that I have used. Apart from articles, these sources include notes
from a course given by Nicolas Monod at EPFL in 2007 and from a course
given by Anna Erschler and myself at ENS in 2016, and books in
preparation by Kate Juschenko and G\'abor Pete. I have also made
abundant use of~\cite{greenleaf:im}, \cite{ceccherini-g-h:amen},
and~\cite{bekka-h-v:t}*{Chapter~5 and Appendix~G}.

I benefited from useful conversations with and remarks from Yves de
Cornulier, Anna Erschler, Vadim Kaimanovich, Peter Kropholler, Yash
Lodha, Nicolas Matte Bon, Nicolas Monod, Volodya Nekrashevych and
Romain Tessera. I thank all of them heartily.

\subsection{Why not this text?}
For lack of space, I have left out much material that I wanted to
include. First and foremost, I have not touched at all at the boundary
initiated by Furstenberg; the ``size'' of its boundary is an
indication of the non-amenability of a $G$-set.

I have also left out much material related to quantitative invariants
--- drift, entropy, on- and off-diagonal probabilities of return of
random walks, and their relation to other invariants such as growth
and best-case distortion of embeddings in convex metric spaces such as
Hilbert space. This topic is evolving rapidly, and I fear that my
rendition would be immediately obsolete.

I would have preferred to write \S\ref{ss:equivrel} in terms of
groupoids, especially since groupoids appear anyways
in~\S\ref{ss:topfull}. In the end, I have opted for directness at the
cost of generality.

Finally, I put as much effort as I could into including applications
and examples in the text; but I omitted the most important ones, e.g.\
Margulis's work on lattices in semisimple Lie groups and percolation
on graphs, feeling they would take us too far adrift.

\subsection{Notation}
We mainly use standard mathematical notation. We try to keep Latin
capitals for sets, Latin lowercase for elements, and Greek for maps. A
subset inclusion $A\subset B$ is strict, while $A\subseteq B$ means
that $A$ could equal $B$. The difference and symmetric difference of
two sets $A,B$ are respectively written $A\setminus B$ and
$A\triangle B$. We denote by $\mathfrak P(X)$ the power set of $X$,
and by $\mathfrak P_f(X)$ the collection of finite subsets of $X$.
Since it appears quite often in the context of amenability, we use
$A\Subset B$ (``compactly contained'') to mean that $A$ is a finite
subset of $B$.

We denote by $A^X$ the set of maps $X\to A$, and by $A^{(X)}$ or by
$\prod'_X A$ the restricted product of $A$, namely the set of
finitely-supported maps $X\to A$. Under the operation of symmetric
difference, $\mathfrak P(X)$ and $\mathfrak P_f(X)$ are respectively
isomorphic to $(\Z/2)^X$ and $(\Z/2)^{(X)}$.

We denote by $\Sym(X)$ the group of finitely-supported permutations of
a set $X$, and abbreviate $\Sym(n)=\Sym(\{1,\dots,n\})$. Groups and
permutations always act on the right, and we denote by
$X\looparrowleft G$ a set $X$ equipped with a right $G$-action.

We denote by $\mathbb 1_A$ the characteristic function of a set $A$,
and also by $\mathbb 1_{\mathcal P}$ the function that takes value $1$
when property $\mathcal P$ holds and $0$ otherwise.

Finally, we write $x\restrict S$ for various kinds of restriction of
the object $x$ to a set $S$.

%%%%%%%%%%%%%%%%%%%%%%%%%%%%%%%%%%%%%%%%%%%%%%%%%%%%%%%%%%%%%%%%
\newpage\section{Means and amenability}
\begin{definition}
  Let $X$ be a set. A \emph{mean} on $X$ is a function\footnote{By
    $\mathfrak P(X)$ we denote the power set of $X$, namely the set of its subsets.}
  $m\colon\mathfrak P(X)\to[0,1]$ satisfying
  \begin{align*}
    m(X)&=1,\\
    m(A\sqcup B)&=m(A)+m(B)\text{ for all disjoint }A,B\subseteq X.
  \end{align*}
  (This last property is often called \emph{finite additivity}, as
  opposed to the $\sigma$-additivity property enjoyed by measures, in
  which countable unions are allowed).
\end{definition}

It easily follows from the definition that $m(\emptyset)=0$; that
$m(A)\le m(B)$ if $A\subseteq B$; and that
$m(A_1\sqcup\dots\sqcup A_k)=m(A_1)+\dots+m(A_k)$ for pairwise
disjoint $A_1,\dots,A_k$.

We denote by $\mathscr M(X)$ the set of means on $X$, with the usual
topology on a set of functions; namely, a sequence
$m_n\in\mathscr M(X)$ converges to $m$ precisely if for every
$\epsilon>0$ and every finite collection $A_1,\dots,A_k\subseteq X$ we
have $|m_n(A_i)-m(A_i)|<\epsilon$ for all $i\in\{1,\dots,k\}$ and
all $n$ large enough.

Observe that $\mathscr M$ is a covariant functor: if $f\colon X\to Y$, then
we have a natural map $f_*\colon\mathscr M(X)\to\mathscr M(Y)$ given by
\[f_*(m)\colon B\mapsto m(f^{-1}(B))\qquad\text{for all }B\subseteq
Y.
\]
In particular, if a group $G$ acts on $X$, then it also acts on $\mathscr
M(X)$. For a right action $\cdot\colon X\times G\to X$, we have a
right action on $\mathscr M(X)$ given by $(m\cdot g)(A)=m(A\cdot
g^{-1})$ for all $A\subseteq X$.

\begin{definition}[von Neumann~\cite{vneumann:masses}]
  Let $G$ be a group and let $X\looparrowleft G$ be a set on which $G$
  acts. The $G$-set $X$ is \emph{amenable} if there is a $G$-fixed
  element in $\mathscr M(X)$.

  A group $G$ is \emph{amenable} if all non-empty right $G$-sets are
  amenable.
\end{definition}
In other words, the $G$-set $X$ is amenable if $\mathscr
M(X)^G\neq\emptyset$, namely if there exists a mean $m$ on $X$ such that
$m(A g)=m(A)$ for all $g\in G$ and all $A\subseteq X$.

\subsection{First examples}
\begin{proposition}
  Every finite, non-empty $G$-set is amenable. More generally, every
  $G$-set with a finite orbit is amenable.
\end{proposition}
Note that, trivially, the empty set is never amenable since a mean
requires $m(\emptyset)=0\neq1=m(X)$.
\begin{proof}
  Let $x G$ be a finite $G$-orbit in the $G$-set $X$. Then $m(A)\coloneqq
  \#(A\cap x G)/\#(x G)$ defines a $G$-invariant mean on $X$.
\end{proof}
In particular, finite groups are amenable.  We shall now see that,
although amenable groups abound, extra logical tools are necessary to
provide more examples.
\begin{proposition}\label{prop:Zamen}
  The infinite cyclic group $\Z$ is amenable.
\end{proposition}
\begin{proof}[False proof]
  Define $m\in\mathscr M(\Z)$ by
  \[m(A)=\lim_{n\to\infty}\frac{\#(A\cap\{1,2,\dots,n\})}{n}.\]
  It is clear that $m(A)$ is contained in $[0,1]$, and the axioms of
  a mean are likewise easy to check. Finally, if $g$ denote the
  positive generator of $\Z$,
  \begin{align*}
    |m(A g)-m(A)|&=\lim_{n\to\infty}\frac{|\#(A g\cap\{1,2,\dots,n\})-\#(A\cap\{1,2,\dots,n\})|}{n}\\
    &=\lim_{n\to\infty}\frac{|\#(A\cap\{0,1,\dots,n-1\})-\#(A\cap\{1,2,\dots,n\})|}{n}\\
    &=\lim_{n\to\infty}\frac{\#(A\cap\{0,n\})}{n}=0.\tag*{\qedhere}
  \end{align*}
\end{proof}
The problem in this proof, of course, is that the limit need not
exist. Consider typically
\[A=\bigcup_{k\ge0}\{2^k+1,2^k+2,\dots,2^k+2^{k-1}\}=\{2,3,5,6,9,10,11,12,17,\dots\}.\]
The arguments of the ``limit'' above oscillate between $2/3$ and $1/2$.  To
correct this proof, we make use of a logical axiom:
\begin{definition}\label{defn:ultrafilter}
  Let $X$ be a set. A \emph{filter} is a family $\mathfrak F$ of
  subsets of $X$, such that
  \begin{enumerate}
  \item $X\in\mathfrak F$ and $\emptyset\not\in\mathfrak F$;
  \item if $A\in\mathfrak F$ and $B\supseteq A$ then $B\in\mathfrak
    F$;
  \item if $A,B\in\mathfrak F$ then $A\cap B\in\mathfrak F$.
  \end{enumerate}

  An \emph{ultrafilter} is a maximal filter (under inclusion). It
  therefore satisfies the extra condition
  \begin{enumerate}\setcounter{enumi}{3}
  \item if $A\subseteq X$, then either $A\in\mathfrak F$ or
    $X\setminus A\in\mathfrak F$.
  \end{enumerate}

  For every $x\in X$, there is a \emph{principal} ultrafilter
  $\mathfrak F_x=\{A\subseteq X\mid x\in A\}$.

  The set of ultrafilters on $X$ is called its \emph{Stone-\v Cech
    compactification} and is written $\beta X$. Its topology is
  defined by declaring open, for every $Y\subseteq X$, the collection
  $\{\mathfrak F\in \beta X\mid Y\in\mathfrak F\}\cong\beta Y$.
\end{definition}

Elements of a filter are thought of as ``large''. As a standard example,
consider the ``cofinite filter'' on $\N$,
\[\mathfrak F_c=\{A\subseteq\N\mid\N\setminus A\text{ is finite}\}.\]
Using this notion, the standard definition of convergence in analysis can
be phrased as follows: ``a sequence $(x_n)$ converges to $x$ if for every
$\epsilon>0$ we have $\{n\in\N\mid \epsilon>|x_n-x|\}\in\mathfrak
F_c$.'' More generally, for a filter $\mathfrak F$ on $\N$ we define
\emph{convergence with respect to $\mathfrak F$} by
\[\lim_{\mathfrak F}x_n=x\quad\text{if and only if}\quad\forall\epsilon>0:\,\{n\in\N\mid \epsilon>|x_n-x|\}\in\mathfrak F.\]

A standard axiom asserts the existence of \emph{non-principal}
ultrafilters on every infinite set. In fact, Zorn's lemma implies that
the cofinite filter $\mathfrak F_c$ is contained in an ultrafilter
$\mathfrak F$. Using this axiom, $\beta X$ is compact, and in fact is
universal in the sense that every map $X\to K$ with $K$ compact
Hausdorff factors uniquely through $\beta X$. We state this universal
property in the following useful form sometimes called ``stone
duality'':
\begin{lemma}\label{lem:stone}
  Let $X$ be a set. The map
  $f\mapsto(\mathfrak F\mapsto\lim_{\mathfrak F}f)$ is an isometry
  between the spaces $\ell^\infty(X)$ of bounded functions on $X$ and
  $\mathcal C(\beta X)$ of continuous functions on $\beta X$ with
  supremum norm.

  In particular, if $\mathfrak F$ is an ultrafilter on $\N$ then every
  bounded sequence converges with respect to $\mathfrak F$.
\end{lemma}
\begin{proof}
  We first prove that if $f\colon X\to\C$ is bounded and $\mathfrak F$
  is an ultrafilter then it has a well-defined limit with respect to
  $\mathfrak F$.  Assume $f(x)\in [L_0,U_0]$ for all $x\in X$. For
  $i=0,1,\dots$ repeat the following.
\begin{enumerate}
  \item Set $M_i=(L_i+U_i)/2$.
  \item Define $A_i=\{x\in X\mid f(x)\in[L_i,M_i]\}$ and
    $B_i=\{x\in X\mid f(x)\in[M_i,U_i]\}$.
  \item By induction, $A_i\cup B_i\in\mathfrak F$; so either
    $A_i\in\mathfrak F$ or $B_i\in\mathfrak F$. In the former case,
    set $(L_{i+1},U_{i+1})=(L_i,M_i)$ while in the latter case set
    $(L_{i+1},U_{i+1})=(M_i,U_i)$.
  \end{enumerate}
  Then $(L_i)$ is an increasing sequence, $(U_i)$ is a decreasing
  sequence, and they both have the same limit; call that limit
  $f(\mathfrak F)$.

  We have extended $f$ to $\beta X$. Let us show that this extension
  is continuous at every $\mathfrak F$: keeping the notation from the
  previous paragraph, for every $\epsilon>0$ there is some $i$ with
  $U_i-L_i<\epsilon$; so
  $\{x\in X\mid\epsilon>|f(x)-f(\mathfrak F)|\}\supseteq A_i\cup
  B_i\in\mathfrak F$ and therefore $f(x)\to f(\mathfrak F)$ when
  $x\to\mathfrak F$.

  Finally the inverse map $\mathcal C(\beta X)\to\ell^\infty(X)$ is simply
  given by restriction to the discrete subspace $X\subseteq\beta X$.
\end{proof}

\begin{exercise}[*]
  Prove that the Stone-\v Cech compactification $\beta X$ is
  homeomorphic to the set of continuous \emph{algebra} homomorphisms
  $\ell^\infty(X)\to\C$, with the induced topology of
  $\ell^\infty(X)^*$.
\end{exercise}

\begin{exercise}[*]
  Let $\mathfrak F$ be an ultrafilter on $\N$. Prove
  $\lim_{\mathfrak F}(x_n+y_n)=\lim_{\mathfrak F}x_n+\lim_{\mathfrak
    F}y_n$ when these last two limits exist.
\end{exercise}

Using a non-principal ultrafilter $\mathfrak F$ on $\N$, we may
correct the ``proof'' that $\Z$ is amenable, by replacing `$\lim$' by
`$\lim_{\mathfrak F}$'; but in some sense we have done nothing except
shuffling axioms around. Indeed, an ultrafilter $\mathfrak F$ on $X$
is precisely the same thing as a $\{0,1\}$-valued mean on $X$: given
an ultrafilter $\mathfrak F$, we define a mean $m$ on $X$ by
\[m(A)=\begin{cases}
  0 & \text{ if }A\not\in \mathfrak F,\\
  1 & \text{ if }A\in \mathfrak F,\end{cases}
\]
and given a mean $m$ taking $\{0,1\}$ values we define a filter
$\mathfrak F=\{A\subseteq X\mid m(A)=1\}$; so the construction of
complicated means is as hard as the construction of complicated
filters.

\begin{proposition}\label{prop:fnnotamen}
  The free group $F_k$ is not amenable if $k\ge2$.
\end{proposition}
\begin{proof}
  We reason by contradiction, assuming that the regular right
  $F_k$-set $F_k\looparrowright F_k$ is amenable. Assume that there
  were an invariant mean $m\colon\mathfrak P(F_k)\to[0,1]$. In
  $F_k=\langle x_1,\dots,x_k\mid\rangle$, let $A$ denote those
  elements whose reduced form ends by a non-trivial (positive or
  negative) power of $x_1$. Then clearly $F_k=A\cup A x_1$, so
  \[1=m(F_k)\le m(A)+m(A x_1)=2m(A).\]
  On the other hand, $F_k\supseteq A x_2^{-1}\sqcup A\sqcup A x_2$, so
  \[1=m(F_k)\ge m(A x_2^{-1})+m(A)+m(A x_2)=3m(A).\]
  These statements imply $1/2\le m(A)\le1/3$, a contradiction.
\end{proof}

\subsection{Elementary properties}
\begin{proposition}\label{prop:quotientX}
  Let $G,H$ be groups, let $X\looparrowleft G$ and $Y\looparrowleft H$
  be respectively a $G$-set and an $H$-set, let
  $\phi\colon G\twoheadrightarrow H$ be a surjective homomorphism, and
  let $f\colon X\to Y$ be an equivariant map, namely satisfying
  $f(x g)=f(x)\phi(g)$ for all $x\in X,g\in G$. If $X$ is amenable,
  then $Y$ is amenable.
\end{proposition}
\begin{proof}
  If $\mathscr M(X)^G\neq\emptyset$, then
  $f_*(\mathscr M(X)^G)=f_*(\mathscr M(X))^{\phi(G)}\subseteq\mathscr
  M(Y)^H$ so $\mathscr M(Y)^H\neq\emptyset$.
\end{proof}

\begin{corollary}[\cite{greenleaf:amenableactions}*{Corollary~3.2}]\label{cor:regamen}
  Let $G$ be a group. Then $G$ is amenable if and only if the right
  $G$-set $G_G$ is amenable.
\end{corollary}
\begin{proof}
  Assume the right $G$-set $G\looparrowleft G$ is amenable. For every
  non-empty $G$-set $X$, choose $x\in X$; then $g\mapsto x g$ is a
  $G$-equivariant map $G\to X$, so $X$ is amenable by
  Proposition~\ref{prop:quotientX}. The converse is obvious.
\end{proof}

Thus amenability of a group is equivalent to amenability of the
right-regular action, and also to amenability of all actions. We give
another characterization:
\begin{proposition}\label{prop:freeaction}
  Let $G$ be a group. Then the following are equivalent:
  \begin{enumerate}
  \item $G$ is amenable;
  \item every non-empty $G$-set is amenable;
  \item $G$ admits an amenable free action.
  \end{enumerate}
\end{proposition}
\begin{proof}
  In view of the previous corollary, it suffices to prove
  $(3)\Rightarrow(1)$. Let $X$ be a free $G$-set, and choose a
  $G$-isomorphism $X\cong T\times G$. Let
  $m\colon\mathfrak P(X)\to[0,1]$ be a $G$-invariant mean. Define a
  mean $m'$ on $G$ by $m'(A)=m(T\times A)$, and check that $m'$ is
  $G$-invariant.
\end{proof}

\begin{exercise}[*]
  Let $X,Y$ be $G$-sets. Then
  \begin{enumerate}
  \item $X\sqcup Y$ is amenable if and only if $X$ or $Y$ is amenable;
  \item $X\times Y$ is amenable if and only if $X$ and $Y$ are amenable.
  \end{enumerate}
\end{exercise}

Proposition~\ref{prop:quotientX} says that quotients of amenable
$G$-sets are amenable. Note however that subsets of amenable $G$-sets
need not be amenable; the empty set being the extreme example.
See~\S\ref{ss:extensiveamen} for a notion of amenability better suited
to subsets and extensions of $G$-sets.

\begin{definition}[Wreath product]
  We introduce a construction of groups that serve as important
  examples. Let $A,G$ be groups and let $X$ be a $G$-set. Their
  \emph{(restricted) wreath product} is
  \begin{equation}\label{eq:wreath}
    A\wr_X G\coloneqq A^{(X)}\rtimes G,
  \end{equation}
  the semidirect product of the group of finitely-supported maps
  $X\to A$ with $G$, under the action of $G$ at the source. Elements
  of $A\wr_X G$ may be written as $(f,g)$ with $f\colon X\to A$ and
  $g\in G$; they multiply by
  $(f,g)\cdot(f',g')=(f'\cdot(f' g^{-1}),g g')$ with
  $(f' g^{-1})(x)=f'(x g)$.
\end{definition}
In case $G$ acts faithfully on $X$, elements of $A\wr_X G$ may be
thought of as ``decorated permutations'': permutations, say $\sigma$
represented by a diagram with vertex set $X$ and an arrow from $x$ to
$\sigma(x)$, and with a label in $A$ on each arrow in such a manner
that only finitely many labels are non-trivial. Decorated permutations
are composed by concatenating their arrows and multiplying their
labels.

The wreath product is associative, in the sense that if $A,G,H$ are
groups, $X$ is a $G$-set and $Y$ is an $H$-set, then $G\wr_Y H$
naturally acts on $X\times Y$ and
$A\wr_{X\times Y}(G\wr_Y H)=(A\wr_X G)\wr_Y H$.

On the other hand, for groups $A,G$ we write `$A\wr G$' for the wreath
product $A\wr_G G$ with regular right action of $G$ on itself, and
that operation is \emph{not} associative.

\begin{definition}[Tree automorphisms]\label{defn:treeaut}
  For a finite set $\Alphabet$, consider the set
  $X\coloneqq\Alphabet^*$ of words over $\Alphabet$. This set is
  naturally the vertex set of a rooted tree $\mathcal T$; the root is
  the empty word, and there is an edge between $x_1\cdots x_n$ and
  $x_1\cdots x_n x_{n+1}$ for all $x_i\in\Alphabet$. The space
  $\Alphabet^\N$ corresponds to infinite paths in $\mathcal T$, and
  thus naturally describes the boundary of $\mathcal T$.

  Let $G$ be the group of graph automorphisms of $\mathcal T$: maps
  $\Alphabet^*\righttoleftarrow$ that preserve the edge set. Then
  there is a natural map $\pi\colon G\to\Sym(\Alphabet)$ defined by
  restricting the action of $G$ to the neighbours of the root; and
  $\ker(\pi)$ acts on the $\#\Alphabet$ disjoint trees hanging from
  the root, so is isomorphic to $G^\Alphabet$. We therefore have a
  natural isomorphism
  \begin{equation}\label{eq:treerec}
    \Phi\colon G\longrightarrow G\wr_\Alphabet\Sym(\Alphabet).
  \end{equation}

  A subgroup $H\le G$ is called \emph{self-similar} if the
  isomorphism~\eqref{eq:treerec} restricts to a homomorphism
  $\Phi\colon H\to H\wr_\Alphabet\Sym(\Alphabet)$. In that case,
  elements of $H$ may be defined recursively in terms of their image
  under $\Phi$, and conversely such a recursive description defines
  uniquely an action on $\mathcal T$.
\end{definition}

The Grigorchuk group $\GG$ (see~\S\ref{ss:interm} or the Introduction)
acts faithfully on the binary rooted tree $\mathcal T_2$, and as such
is a subgroup of the automorphism group of $\mathcal T_2$. It is
self-similar, and the generators $\{a,b,c,d\}$ of $\GG$ may be written
using decorated permutations as follows:
\[a\mapsto\tikz[baseline=-2ex]{\draw[->] (0,0) -- (0.5,-0.5); \draw[->](0.5,0) -- (0,-0.5);},\qquad
b\mapsto\tikz[baseline=-2ex]{\draw[->] (0,0) -- node[right=-2pt] {\small$a$} (0,-0.5); \draw[->](0.5,0) -- node[right=-2pt] {\small$c$} (0.5,-0.5);},\qquad
c\mapsto\tikz[baseline=-2ex]{\draw[->] (0,0) -- node[right=-2pt] {\small$a$} (0,-0.5); \draw[->](0.5,0) -- node[right=-2pt] {\small$d$} (0.5,-0.5);},\qquad
d\mapsto\tikz[baseline=-2ex]{\draw[->] (0,0) -- (0,-0.5); \draw[->](0.5,0) -- node[right=-2pt] {\small$b$} (0.5,-0.5);}.
\]

\begin{example}[The ``lamplighter group'']\label{ex:lamplighter}
  Consider $G=\Z$ acting on itself by translation, and $A=\Z/2$. The
  wreath product $W=A\wr G$ is called the ``lamplighter group''. The
  terminology is justified as follows: consider an bi-infinite street
  with a lamp at each integer location. The group $G$ consists of
  invertible instructions for a person, the ``lamplighter'': either
  move up or down the street, or toggle the state of a lamp before
  him/her.

  If we denote by $a$ the operation of toggling the lamp at position
  $0$ and by $t$ the movement of the lamplighter one step up the
  street, then $G$ is generated by $\{a,t\}$; and it admits as
  presentation
  \begin{equation}
    G=\langle a,t\mid [a,a^{t^k}]\text{ for all }k\in\N\rangle.
  \end{equation}
\end{example}

\begin{exercise}[**]\label{ex:perfect}
  Let $A$ be a simple group and let $H$ be perfect. Let
  $G\coloneqq H\wr_X A$ be their wreath product. Then $G$ is perfect,
  and all normal subgroups of $G$ are $G$ or of the form $N^X$ for a
  normal subgroup $N\triangleleft H$.
\end{exercise}

\begin{example}[Monod-Popa~\cite{monod-popa:coamenability}]\label{ex:monodpopa}
  There are groups $K\triangleleft H\triangleleft G$ such that the
  $G$-sets $K\backslash G$ and $H\backslash G$ are amenable but the
  $H$-set $K\backslash H$ is not.

  Choose indeed any non-amenable group $Q$, and set
  $G\coloneqq Q\wr\Z$ and $H=\prod'_\Z Q$ and $K=\prod'_\N Q$.

  The $G$-set $H\backslash G$ is clearly amenable, since the action of
  $G$ factors through an action of $\Z$. To prove that $K\backslash G$
  is amenable, it therefore suffices to find an $H$-invariant mean on
  $\ell^\infty(K\backslash G)$, and then apply
  Proposition~\ref{prop:stabilizers}. Let $t$ denote the positive
  generator of $\Z$. For every $k\in\N$, define a mean $m_k$ by
  $m_k(f)=f(K t^k)$ for $f\in\ell^\infty(K\backslash G)$. This mean is
  invariant by the group $K^{t^k}$. Since $H=\bigcup_{k\in\N}K^{t^k}$,
  any weak limit of the $m_k$ is an $H$-invariant mean.

  On the other hand, $K\backslash H$ is just a restricted direct
  product of $Q$'s, so is not amenable by Proposition~\ref{prop:freeaction}.
\end{example}

\begin{exercise}[**]
  Give an amenable $G$-set such that none of its orbits are amenable.

  \emph{Hint:} Consider the ``lamplighter group''
  $G=\langle a,t\rangle$, see Example~\ref{ex:lamplighter}, and the
  groups
  $G_n=\langle a,t\mid [a,a^{t^k}]\text{ for all
  }k=1,\dots,n\rangle$. Consider the natural action of
  $F_2=\langle a,t\mid\rangle$ on $X=\bigsqcup_{n\ge0}G_n$, and show
  that (i) each $G_n$ is non-amenable (ii) the group $G$ is amenable
  (iii) the action on $X$ approximates arbitrarily well the action on
  $G$.
\end{exercise}

We return to the definition of means we started with; we shall see
more criteria for amenability. Recall that $\mathscr M(X)$ denotes the
set of means on $X$.

\begin{lemma}\label{lem:mcompact}
  $\mathscr M(X)$ is compact.
\end{lemma}
\begin{proof}
  Since $\mathscr M(X)$ is a subset of $[0,1]^{\mathfrak P(X)}$ which
  is compact by Tychonoff's theorem\footnote{We are using here, and
    throughout this chapter, the Axiom of Choice;
    see~\cite{kelley:tychonoff}.}, it suffices to show that $\mathscr
  M(X)$ is closed.

  Now each of the conditions defining a mean, namely $m(X)-1=0$ and
  $m(A\cup B)-m(A)-m(B)=0$, defines a closed subspace of
  $[0,1]^{\mathfrak P(X)}$ because it is the zero set of a continuous map. The
  intersection of these closed subspaces is $\mathscr M(X)$ which is
  therefore closed.
\end{proof}

\noindent Here are simple examples of means. For $x\in X$, define
$\delta_x\in\mathscr M(X)$ by
\[\delta_x(A)=\begin{cases} 0 & \text{ if }x\not\in A,\\
  1 & \text{ if }x\in A.\end{cases}\] It is easy to see that the
axioms of a mean are satisfied. We have thus obtained a map
$\delta\colon X\to\mathscr M(X)$, which is clearly injective.

\begin{lemma}
  $\delta(X)$ is discrete\footnote{Recall that $D$ is discrete in a
    topological space $X$ if for every $x\in D$ there is an open set
    $\mathcal U\ni x$ with $D\cap\mathcal U=\{x\}$.} in $\mathscr
  M(X)$.
\end{lemma}
\begin{proof}
  Given $x\in X$, set
  \[\mathcal U=\{m\in\mathscr M(X)\mid m(\{x\})>0\}.\qedhere\]
\end{proof}

\begin{corollary}
  If $X$ is infinite, then $\delta(X)$ is not closed.
\end{corollary}
\begin{proof}
  Indeed, if $\delta(X)$ is closed in $\mathscr M(X)$, then it is
  compact; being furthermore discrete, it is finite; $\delta$ being
  injective, $X$ itself is finite.
\end{proof}

Recall that a subset $K$ of a topological vector space is
\emph{convex} if for all $x,y\in K$ the segment
$\{(1-t)x+t y\mid t\in[0,1]\}$ is contained in $K$;
see~\S\ref{sec:convex} for more on convex sets. The \emph{convex hull}
of a subset $S$ of a topological vector space is the intersection
$\widehat S$ of all the closed convex subspaces containing $S$.

\begin{lemma}\label{lem:mconvex}
  $\mathscr M(X)$ is convex.
\end{lemma}
\begin{proof}
  Consider means $m_i$ and positive numbers $t_i$ such that $\sum
  t_i=1$. Then $\sum t_i m_i$ clearly satisfies the axioms of a mean.
\end{proof}

% Now given $m_1\neq m_2\in\mathscr M(X)$, there must exist $A\subseteq X$
% with $m_1(A)\neq m_2(A)$. Consider then the affine map
% $\phi\colon\mathscr M(X)\to\R$ defined by $\phi(m)=m(A)$; clearly $\phi$
% distinguishes $m_1$ from $m_2$.

For a set $X$ and $p\in[1,\infty)$ we denote by $\ell^p(X)$ the Banach
space of functions $\phi\colon X\to\R$ satisfying
$\|\phi\|^p\coloneqq\sum|\phi(x)|^p<\infty$, and by $\ell^\infty(X)$
the space of bounded functions with supremum norm. For
$p\in[1,\infty]$ the space $\ell^p(X)$ carries a natural isometric
$G$-action by $(\phi g)(x)=\phi(x g^{-1})$.  Of particular interest is
the space $\ell^1(X)$, and its subset
\begin{equation}\label{eq:proba}
  \mathscr P(X)=\big\{\mu\in\ell^1(X)\mid\mu\ge0,\sum_{x\in X}\mu(x)=1\big\},
\end{equation}
the space of \emph{probability measures} on $X$. It is a convex
subspace of $\ell^1(X)$, compact for the weak*-topology, and (for
infinite $X$ strictly) contained in $\mathscr M(X)$:
\begin{proposition}\label{prop:linfty}
  For a set $X$, consider the following subset of $\ell^\infty(X)^*$:
  \[\mathscr B(X)\coloneqq\{m\in\ell^\infty(X)^*\mid m(f)\ge 0\text{
    whenever }f\ge0,\,m(\mathbb 1)=1\}.\]
  Then the map $\int\colon \mathscr B(X)\to\mathscr M(X)$ defined by
  \[({\textstyle\int}m)(A)\coloneqq m(\mathbb 1_A)\text{ with }\mathbb 1_A\text{ the characteristic function of }A\]
  is a homeomorphism, functorial in $X$.

  The subspace $\ell^1(X)\cap\mathscr B(X)\subset\ell^\infty(X)^*$
  corresponds via $\int$ to the convex hull $\widehat{\delta(X)}$ of
  $\delta(X)$.
\end{proposition}

We recall that there is a natural non-degenerate pairing
$\ell^1(X)\times\ell^\infty(X)\to\C$, given by
$(f,g)\mapsto\sum f(x)g(x)$. For that pairing,
$(\ell^1X)^*=\ell^\infty(X)$; but $(\ell^\infty X)^*$ is much
bigger than $\ell^1(X)$, as is clear from the proposition. In fact,
$\ell^\infty(X)$ is in isometric bijection with the space of
continuous functions on the Stone-\v Cech compactification $\beta X$
of $X$, see Lemma~\ref{lem:stone}, so
\begin{equation}\label{eq:stonerep}
  (\ell^\infty(X))^*=L^1(\beta X)\text{ the set of Borel measures on }\beta X.
\end{equation}

\begin{proof}[Proof of Proposition~\ref{prop:linfty}]
  Let $\mathscr S$ be the set of \emph{simple} functions on $X$, namely the
  functions that take only finitely many values.  Consider first
  $m\in\ell^\infty(X)^*$ with $m(\mathbb 1_A)=0$ for all $A\subseteq
  X$. Then $m$ vanishes on $\mathscr S$ by linearity; and $\mathscr S$ is
  dense in $\ell^\infty(X)$, so $m=0$. This proves that $\int$ is
  injective.

  On the other hand, let $m\colon\mathfrak P(X)\to[0,1]$ be a mean. For
  $f\in\mathscr S$, we have
  \[({\textstyle\int}m)(f)=\sum_{v\in f(X)}v m(f^{-1}(v)).
  \]
  We check that $\int m$ is a continuous function $\mathscr S\to\C$ for the
  $\ell^\infty$ norm on $\mathscr S$; indeed, for $f,g$ simple
  functions on $X$,
  \begin{align*}
    |({\textstyle\int m})(f)-({\textstyle\int m})(g)| &=
    \left|\sum_{v\in f(X),w\in g(X)}(v-w)\mu(f^{-1}(v)\cap g^{-1}(w))\right|\\
    &\le \sum_{v\in f(X),w\in g(X)}|v-w|\mu(f^{-1}(v)\cap g^{-1}(w))\\
    &\le\|f-g\|_\infty\sum_{v\in f(X),w\in g(X)}\mu(f^{-1}(v)\cap g^{-1}(w))\\
    &\le \|f-g\|_\infty.
  \end{align*}
  Therefore, $\int m$ extends to a continuous function
  $\ell^\infty(X)\to\C$, which clearly belongs to $\mathscr B(X)$. Since
  $\mathscr S$ is dense, this extension is unique.

  Finally recall that $\ell^1(X)$ embeds in $\ell^\infty(X)^*$ by
  $f\mapsto (f'\mapsto \sum_x f(x)f'(x))$. The element $f\in\ell^1(X)$
  therefore corresponds to the affine combination $\sum f(x)\delta_x$
  of Dirac means.
\end{proof}
From now on, we will use interchangeably the notations
$m\in\mathscr M(X)$ and $m\in(\ell^\infty(X))^*$; they correspond to
each other via the proposition.

\begin{corollary}\label{cor:vneumann}
  Let $X$ be a $G$-set. Then $X$ is amenable if and only if there
  exists a $G$-invariant positive functional in $\ell^\infty(X)^*$.\qed
\end{corollary}

The fact that the ``Dirac'' means $\widehat{\delta(X)}$ constitute a small
subset of $\mathscr M(X)$ may be confirmed as follows.  Every mean $m\in
\widehat{\delta(X)}$ enjoys an additional property, namely
$\sigma$-additivity: for disjoint $A_1,A_2,\dots$ we have
\[m(\bigcup A_i)=\sum m(A_i).\]
Consider now an invariant mean $m$ on $\Z$, as given by
Proposition~\ref{prop:Zamen}. Assume for contradiction that $m$ were
$\sigma$-additive. Then either $m(\{0\})=0$, so $m(\{n\})=0$ for all
$n\in\Z$ by $\Z$-invariance and $m(\Z)=0$ by $\sigma$-additivity; or
$m(\{0\})=\epsilon>0$ and $m(\{0,1,\dots,n\})>1$ as soon as
$n>1/\epsilon$. In all cases we have reached a contradiction.

\begin{proposition}\label{prop:stabilizers}
  Let $X$ be an amenable $G$-set such that all point stabilizers $G_x$
  are amenable. Then $G$ itself is amenable.
\end{proposition}
\begin{proof}
  Thanks to Proposition~\ref{prop:linfty}, for all $Y$ we view
  $\mathscr M(Y)$ as the set of normalized positive functionals
  $m\colon\ell^\infty(Y)\to\R$.  Let us first define a map
  $\Phi\colon X\to\mathscr M(G)$.

  Since every $G_x$ is amenable, there exists for all $x\in X$ an
  invariant mean $m_x\in\mathscr M(G_x)^{G_x}$, which we extend via
  the inclusion $G_x\hookrightarrow G$ to mean still written
  $m_x\in\mathscr M(G)^{G_x}$. Choose for every $G$-orbit in $X$ a
  point $x$, and set $\Phi(x g)=m_x g$ on that orbit. This is
  well-defined: if $x g=x h$, then $h g^{-1}\in G_x$ so
  $m_x h=m_x h g^{-1}g=m_x g$. It follows automatically that $\Phi$ is
  $G$-equivariant.

  By functoriality, $\Phi$ induces a $G$-equivariant map
  $\Phi_*\colon\mathscr M(X)\to\mathscr M(\mathscr M(G))$.

  Now there is, for all $Y$, a functorial map
  $\beta\colon\mathscr M(\mathscr M(Y))\to\mathscr M(Y)$ called the
  \emph{barycentre}: it is given by
  \begin{equation}\label{eq:barycentre}
    \Upsilon(m)(f)=m(n\mapsto n(f))\text{ for }m\in\mathscr M(\mathscr M(Y)),f\in\ell^\infty(Y),n\in\mathscr M(Y).
  \end{equation}
  Composing, we get a map
  $\Upsilon\circ\Phi_*\colon\mathscr M(X)\to\mathscr M(G)$, which is
  still $G$-equivariant. Now since $X$ is amenable $\mathscr M(X)^G$
  is non-empty, so $\mathscr M(G)^G$ is also non-empty.
\end{proof}

\begin{corollary}\label{cor:extensions}
  Let
  $1\longrightarrow N\longrightarrow G\longrightarrow
  Q\longrightarrow1$ be an exact sequence of groups. Then $G$ is
  amenable if and only if both $N$ and $Q$ are amenable.
\end{corollary}
\begin{proof}
  If $G$ is amenable, then its quotient $Q$ is amenable by
  Proposition~\ref{prop:quotientX}, and its subgroup $H$ is amenable
  by Proposition~\ref{prop:freeaction}, since it acts freely on the
  amenable $G$-set $G$.

  Conversely, if $N$ and $Q$ are amenable, then the natural action of
  $G$ on $Q$ satisfies the hypotheses of
  Proposition~\ref{prop:stabilizers}.
\end{proof}

\begin{exercise}[*]\label{ex:leftamen}
  Let $G$ be a group. We might have called $G$ \emph{left-amenable} if
  there exists a \emph{left-invariant} mean on $G$, namely a mean
  $m\in\mathscr M(G)$ with $m(g A)=m(A)$ for all $g\in G,A\subseteq G$; and
  have called $G$ \emph{bi-amenable} if there exists a mean $m\in\mathscr
  M(G)$ with $m(g A h)=m(A)$ for all $g,h\in G,A\subseteq G$.

  Prove that in fact $G$ is amenable if and only if it is left-amenable, if
  and only if it is bi-amenable.
\end{exercise}

\noindent We conclude with yet another criterion,
attributed\footnote{Erroneously!} to Dixmier:
\begin{theorem}[F\o lner~\cite{folner:bogoliouboff}*{Theorem~4}, Dixmier~\cite{dixmier:moyennes}*{Th\'eor\`eme~1}; see~\cite{gournay:percolation}*{Theorem~4.2}]\label{thm:dixmier}
  Let $X$ be a $G$-set. Then $X$ is amenable if and only if for any
  $h_1,\dots,h_n\in\ell^\infty(X)$ and any
  $g_1,\dots,g_n\in G$ the function
  \[H\coloneqq\sum_{i=1}^n(h_i-h_i g_i)\qquad\text{satisfies }\sup_{x\in X}H(x)\ge0.\]
\end{theorem}
\begin{proof}
  If $X$ is amenable then there is an invariant positive mean
  $m\in\ell^\infty(X)^*$; then for every function $H$ as above
  $m(H)=0$ by invariance while $m(H)\le\sup H$ by positivity.

  On the other hand, if $\sup H\ge0$ for all $H$ as above, then an
  invariant mean may be constructed as follows: set
  \[\tilde m(f)=\inf_{H\text{ as above}}\sup_X(f+H).
  \]
  Clearly $\tilde m$ satisfies
  $\tilde m(\lambda f)=\lambda\tilde m(f)$ for $\lambda\ge0$ and
  $\tilde m(f g)=\tilde m(f)$ for $g\in G$ and $\tilde m(\mathbb1)=1$
  and $\tilde m(f)\ge0$ if $f\ge0$; and
  $\tilde m(f+g)\le\tilde m(f)+\tilde m(g)$ because if
  $\tilde m(f)\ge\sup_X(f+H)-\epsilon$ and
  $\tilde m(g)\ge\sup_X(g+K)-\epsilon$ then
  $\tilde m(f+g)\le\sup_X(f+g+H+K)\le\sup_X(f+H)+\sup_X(g+K)\le\tilde
  m(f)+\tilde m(g)-2\epsilon$. The Hahn-Banach theorem~(see
  e.g.~\cite{rudin:fa}*{Theorem~3.12}) implies the existence of a
  linear functional $m$ with the same properties.
\end{proof}

%%%%%%%%%%%%%%%%%%%%%%%%%%%%%%%%%%%%%%%%%%%%%%%%%%%%%%%%%%%%%%%% 
\newpage\section{F\o lner and Reiter's criteria}\label{ss:folner}
The following combinatorial criterion will be shown equivalent to
amenability; it is sometimes the easiest path to prove a group's
amenability. It was introduced by \EFolner~\cite{folner:banach},
though the idea of averaging over larger and larger finite sets to
construct invariant means can be traced back at least
to Ahlfors~\cite{ahlfors:coverings}*{Chapter~III.25}.

\begin{definition}
  Let $X$ be a $G$-set. We say that $X$ satisfies \emph{F\o
    lner's condition} if for all finite $S\Subset G$ and all
  $\epsilon>0$, there is a finite subset $F\Subset X$ with
  \[\#(F S\setminus F)<\epsilon\#F.\]

  When we say that a group $G$ satisfies F\o lner's condition, we mean
  it for the right $G$-set $X=G\looparrowleft G$.
\end{definition}

For example, $\Z$ satisfies F\o lner's condition: given $\epsilon>0$
and $S\subset\Z$ finite, find $k$ such that
$S\subseteq\{-k,\dots,k\}$. Let $\ell\in\N$ be such that
$\ell>2k/\epsilon$, and set $F=\{1,2,\dots,\ell\}$. Then
$F S\setminus F\subseteq\{1-k,\dots,0,\ell+1,\dots,\ell+k\}$ has size
at most $2k$, so $\#(F S\setminus F)<\epsilon\#F$.

Actually, the definition makes sense in a much more general context,
that of \emph{graphs}:
\begin{definition}\label{defn:graphs}
  A directed graph (digraph) is a pair of sets $\mathscr G=(V,E)$
  called \emph{vertices} and \emph{edges}, with maps
  $\pm\colon E\to V$ giving for each edge $e\in E$ its \emph{head}
  $e^+\in V$ and \emph{tail} $e^-\in V$.

  A graph $\mathscr G=(V,E)$ has \emph{bounded valency} if there is a
  bound $K\in\N$ such that at every vertex $v\in V$ there are at most
  $K$ incoming and outgoing edges, namely if
  $\#\{e\in E\mid v=e^+\}\le K$ and $\#\{e\in E\mid v=e^-\}\le K$.
\end{definition}

Consider a $G$-set $X$ and a finite set $S\subset G$. The
\emph{Schreier graph} of $X$ with respect to $S$ is the graph with
vertex set $V=X$ and edge set $E=X\times S$, with $(x,s)^-=x$ and
$(x,s)^+=x s$.  In other words, there is an edge from $x$ to $x s$ for
all $x\in X,s\in S$. If $X=G\looparrowleft G$, then the Schreier graph
is usually called the \emph{Cayley graph} of $G$.

Let $(V,E)$ be a graph. For a subset $F\subseteq V$, its
\emph{boundary} is the set of edges connecting $F$ to its complement,
in formul\ae
\[\partial F=\{e\in E\mid e^-\in F,e^+\not\in F\}.\]

\begin{definition}\label{defn:folnergraph}
  A graph $\mathscr G=(V,E)$ satisfies \emph{F\o lner's condition} if
  for all $\epsilon>0$ there is a finite subset $F\Subset V$ with
  $\#\partial F<\epsilon\# F$.
\end{definition}

Thus F\o lner's criterion asks for the existence of subgraphs of $X$
with an arbitrarily small relative outer boundary.  It is clear that a
$G$-set $X$ satisfies F\o lner's condition if and only if its Schreier
graphs satisfy it for all choices of $S\Subset G$.

\begin{lemma}\label{lem:folner1}
  Let $X$ be a $G$-set. F\o lner's condition is equivalent to: for all
  finite subsets $S\Subset G$ and all $\epsilon>0$, there is a finite
  subset $F\Subset X$ with
  \[\#(F s\setminus F)<\epsilon\#F\text{ for all }s\in S.\]
\end{lemma}
\begin{proof}
  If $\#(F S\setminus F)<\epsilon\#F$, then in particular
  $\#(F s\setminus F)<\epsilon\#F$ for all $s\in S$. Conversely, if
  $\#(F s\setminus F)<\epsilon\#F/\#S$ for all $s\in S$ then
  $\#(F S\setminus F)<\epsilon\#F$.
\end{proof}

Recall that a \emph{directed set} is a partially ordered set
$(\mathscr N,\le)$ with finite upper bounds, i.e. for every
$m,n\in\mathscr N$ there exists an element $\max\{m,n\}\in\mathscr N$
with $m,n\le\max\{m,n\}$. A \emph{net} is a sequence indexed by a
directed set. For $(x_n)_{n\in\mathscr N}$ a real-valued net, we write
\begin{equation}\label{eq:convnet}
  \lim_{n\to\infty}x_n=x\quad\text{to mean}\quad\forall\epsilon>0:\exists n_0\in\mathscr N:\forall n\ge n_0:|x_n-x|<\epsilon,
\end{equation}
as in usual calculus.
\begin{exercise}[*]
  Let $\mathscr N$ be a non-empty net. Then
  $\{F\subseteq\mathscr N\mid \exists n_0\in\mathscr N:n\ge
  n_0\Rightarrow n\in F\}$ is a filter on $\mathscr N$, and the
  notions of convergence in~\eqref{eq:convnet} and in the filter
  coincide.
\end{exercise}

\noindent We have the following alternative definition of F\o lner's condition:
\begin{lemma}\label{lem:folnerlim}
  Let $G$ be a group and let $X$ be a $G$-set. Then $X$ satisfies F\o
  lner's condition if and only if there exists a net
  $(F_n)_{n\in\mathscr N}$ of finite subsets of $X$ with
  \begin{equation}\label{eq:net:folner}
    \lim_{n\to\infty}\frac{\#(F_n g\setminus F_n)}{\#F_n}=0\text{ for all
  }g\in G.
\end{equation}
\end{lemma}
\begin{proof}
  Assume~\eqref{eq:net:folner}, and let $S\Subset G,\epsilon>0$ be
  given. For each $s\in S$, let $n(s)\in\mathscr N$ be such that
  $\#(F_n s\setminus F_n)<\epsilon\#F_n/\#S$ for all $n\ge n(s)$, and
  set $F=F_{\max\{n(s)\}}$; then
  $\#(F S\setminus F)\le\sum_{s\in S}\#(F s\setminus F)<\epsilon\#F$,
  so F\o lner's condition is satisfied.

  Conversely, define
  $\mathscr N=\{(S,\epsilon)\mid S\Subset G\text{ finite
  },\epsilon>0\}$, ordered as follows: $(S,\epsilon)\le(T,\delta)$ if
  $S\subseteq T$ and $\epsilon>\delta$; so
  $\max\{(S,\epsilon),(T,\delta)\}=(S\cup
  T,\min\{\epsilon,\delta\})$. For each $n=(S,\epsilon)\in\mathscr N$,
  choose a finite set $F_n\Subset X$ with
  $\#(F S\setminus F)<\epsilon\#F$. These
  satisfy~\eqref{eq:net:folner}.
\end{proof}

\noindent In case $G$ is finitely generated, we also have the following
alternative definition:
\begin{lemma}\label{lem:folnerfg}
  Let $G$ be finitely generated, say by a finite set $S$ containing $1$,
  and let $X$ be a $G$-set. Then $X$ satisfies F\o lner's condition if and
  only if for all $\epsilon>0$ there is a finite subset $F\Subset X$ with
  \[\#(F S\setminus F)<\epsilon\#F.\]
\end{lemma}
\begin{proof}
  One direction is obvious. In the other direction, let $S'\Subset G$
  and $\epsilon'>0$ be given. Since $S$ generates $G$, there exists
  $k\in\N$ with $S'\subseteq S^k$. Set $\epsilon=\epsilon'/k$, and let
  $F\Subset X$ satisfy $\#(F s\setminus F)<\epsilon\#F$ for all
  $s\in S$.

  Consider $g\in S'$, and write it as $g=s_1\dots s_k$ with
  $s_1,\dots,s_k\in S$. Then
  \begin{align*}
    F g\setminus F &=\bigsqcup_{j=1}^k F s_j\cdots s_k\setminus F
    s_{j+1}\cdots s_k,\\
    \intertext{so}
    \#(F g\setminus F)&=\sum\#(F s_j\cdots s_k\setminus F s_{j+1}\cdots s_k)\\
    &=\sum\#(F s_j\setminus F)s_{j+1}\cdots s_k< k\epsilon\#F=\epsilon'\#F.
  \end{align*}
  We are done by Lemma~\ref{lem:folner1}.
\end{proof}

We shall see in Theorem~\ref{thm:folneramen} that a $G$-space $X$
satisfies F\o lner's criterion if and only if it is amenable. This can
be used to prove (non-)amenability in numerous cases; for example,
\begin{proposition}\label{prop:restrict}
  A $G$-set $X\looparrowleft G$ is amenable if and only if for every finitely
  generated subgroup $H\le G$ the $H$-set $X\looparrowleft H$ is amenable.
\end{proposition}
\begin{proof}
  ($\Leftarrow$) Given $S\Subset G$ and $\epsilon>0$, consider
  $H=\langle S\rangle$ and apply F\o lner's criterion.

  ($\Rightarrow$) Every $G$-invariant mean is also $H$-invariant.
\end{proof}
Thus for instance the action of $\Q$ on $\Q/\Z$ is amenable, because every
finitely generated subgroup of $\Q$ has a finite orbit on $\Q/\Z$. (We
shall later see that all actions of $\Q$ are amenable.)

\begin{example}
  The group of permutations $\Sym(\N)$ of $\N$ with finite support is
  amenable; indeed every finite subset generates a finite group.
\end{example}

\begin{example}\label{ex:W(Z)}
  The group of ``bounded-displacement permutations of $\Z$''
  \[G=W(\Z)=\{\tau\colon\Z\righttoleftarrow\mid\sup_{n\in\Z}|\tau(n)-n|<\infty\}\]
  acts amenably on $\Z$. Indeed given $S\subset G$ finite and
  $\epsilon>0$, the maximum displacement of elements of $S$ is
  bounded, say $\le k$; and then $\Z\looparrowleft G$ satisfies F\o lner's
  condition with $F=\{0,\dots,\lceil k/\epsilon\rceil\}$.
\end{example}

\begin{example}\label{ex:folnerll}
  The ``lamplighter group'' $G$ from Example~\ref{ex:lamplighter} is
  amenable. Indeed elements of $G$ may be written as pairs $(f,m)$
  with $f\colon\Z\to\Z/2$ and $m\in\Z$, and one may consider as F\o
  lner sets
  \[F_n=\{(f,m):\supp(f)\subseteq[-n,n]\text{ and
  }m\in[-n,n]\}.\]
\end{example}

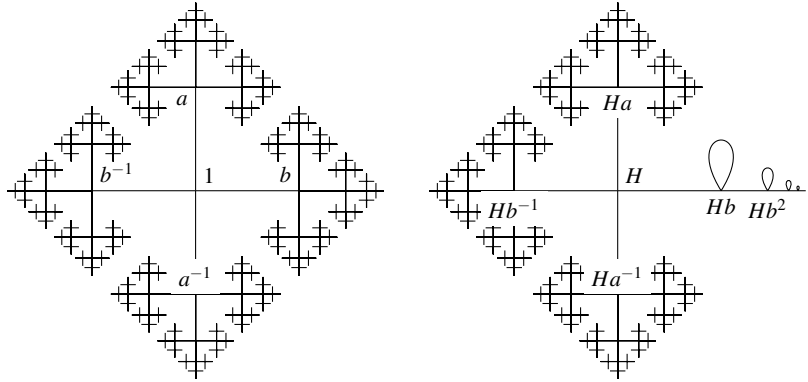
\begin{figure}
  \centerline{\begin{tikzpicture}[decoration=Fractal Tree,thin,xscale=2.5,yscale=2.5]
      \draw decorate{ decorate{ decorate { decorate { (0,0) -- (1,0) } } } };
      \draw decorate{ decorate{ decorate{ decorate { (0,0) -- (-1,0) } } } };
      \draw decorate{ decorate{ decorate{ decorate { (0,0) -- (0,1) } } } };
      \draw decorate{ decorate{ decorate{ decorate { (0,0) -- (0,-1) } } } };
      \node at (0,0) [above right] {$1$};
      \node at (1-0.45,0) [above left] {$b$};
      \node at (0,1-0.45) [below left] {$a$};
      \node at (-1+0.45,0) [above right] {$b^{-1}$};
      \node at (0,-1+0.45) [above,fill=white] {$a^{-1}$};
    \end{tikzpicture}\qquad
    \begin{tikzpicture}[decoration=Fractal Tree,thin,xscale=2.5,yscale=2.5]
      \draw (0,0) -- (1,0);
      \foreach\x in {0.45,0.45*0.45,0.45*0.45*0.45,0.45*0.45*0.45*0.45} {\draw (1-\x,0) .. controls +(0.5*\x,0.8*\x) and +(-0.5*\x,0.8*\x) .. (1-\x,0);}
      \draw decorate{ decorate{ decorate { decorate { (0,0) -- (-1,0) } } } };
      \draw decorate{ decorate{ decorate{ decorate { (0,0) -- (0,1) } } } };
      \draw decorate{ decorate{ decorate{ decorate { (0,0) -- (0,-1) } } } };
      \node at (0,0) [above right] {$H$};
      \node at (1-0.45,0) [below] {$H b$};
      \node at (1-0.45*0.45,0) [below] {$H b^2$};
      \node at (0,1-0.45) [below,fill=white] {$H a$};
      \node at (-1+0.45,0) [below,fill=white] {$H b^{-1}$};
      \node at (0,-1+0.45) [above,fill=white] {$H a^{-1}$};
    \end{tikzpicture}}
  \caption{The Cayley graph of the free group $F_2$, and the coset space of $H$ (see Example~\ref{ex:freeamen})}\label{fig:cayleyF2}
\end{figure}

\begin{example}\label{ex:freefolner}
  The free group $F_k=\langle x_1,\dots,x_k\mid\rangle$ is amenable if
  and only if $k\le1$, see Proposition~\ref{prop:fnnotamen}. Indeed if
  $k\le1$ then $F_k$ is $\{1\}$ or $\Z$; while in general, choose
  $S=\{x_1^{\pm1},\dots,x_k^{\pm1}\}$ and consider $F\Subset X$. In
  the Cayley graph of $F_k$, which is a $2k$-regular tree (see
  Figure~\ref{fig:cayleyF2} left), consider the subgraph spanned by
  $F$. It suffices to consider connected components of the graph once
  at a time; each connected component is a tree, with say $v$ vertices
  and therefore $v-1$ edges. The sum of the vertex degrees within that
  tree is therefore $2v-2$, so the total number of edges pointing out
  of the component is at least $2k v-(2v-2)\ge(2k-2)v$; these edges
  point to distinct elements in $SF\setminus F$. Therefore, F\o lner's
  criterion is not satisfied as soon as $\epsilon<2n-2$.
\end{example}

There are plenty of non-amenable groups with amenable actions, and
even faithful amenable actions; here is one.
\begin{example}\label{ex:freeamen}
  Consider $F_2=\langle a,b\mid\rangle$ and its subgroup
  $H=\langle a^{b^n}:n\le0\rangle$. Then $F_2$ acts naturally on the
  coset space $X\coloneqq H\backslash F_2$, see
  Figure~\ref{fig:cayleyF2} right, and this action is amenable. Indeed
  with $S=\{a^{\pm1},b^{\pm1}\}$ and $\epsilon>0$ given, consider the
  set $F=\{H,H b,\dots,H b^n\}$ for $n>\epsilon^{-1}$. It
  satisfies F\o lner's criterion. Note that the action of $F_2$ on $X$
  is not free, but it is nevertheless faithful.
\end{example}

\subsection{Growth of sets}\label{ss:growthsets}
Let $X\looparrowleft G$ be a $G$-set, and consider $S\Subset G$ and
$x_0\in X$. The \emph{orbit growth} of $X$ is the function
$v_{X,x_0,S}\colon\N\to\N$ given by
\[v_{X,x_0,S}(n)=\#\{x\in X\mid x=x_0 s_1\cdots s_m\text{ for some }s_i\in S,m\le n\}.
\]
If $G$ is finitely generated, then the orbit growth depends only
mildly on the choice of $S$ as soon as it generates $G$: if $S'$ be
another generating set of $G$, then there exists a constant $C>0$ with
$v_{X,x_0,S}(n)\le v_{X,x_0,S'}(C n)$ and
$v_{X,x_0,S'}(n)\le v_{X,x_0,S}(C n)$. Similarly, if $x_0,x'_0\in X$
belong to the the same $G$-orbit, then there exists a constant
$C\in\N$ with $v_{X,x_0,S}(n)\le v_{X,x_0',S}(n+C)$ and
$v_{X,x_0',S}(n)\le v_{X,x_0,S}(n+C)$. Therefore, the equivalence
class of $v_{X,x_0,S}$ under linear transformations of its argument is
independent of the choice of $S$ if $S$ generates $G$, and of $x_0$ if
$X$ is transitive; it is denoted simply $v_{X,x_0}$, $v_X$ and
$v_{x_0}$ respectively.

As usual, we consider $G$ as a $G$-set under right translation, and
denote by $v_{G,S}$ and $v_G$ its growth function. We also write
$B_{G,S}(n)$ for the ball of radius $n$ in $G$, and more generally
$B_{X,x_0,S}(n)$ for the ball of radius $n$ in $X$ around $x_0$.

\begin{proposition}\label{prop:subexp=>amenable}
  Let $X$ be a $G$-set and let $x_0\in X$ be such that $v_{X,x_0,S}$
  grows subexponentially for all $S\Subset G$. Then $X$ satisfies
  F\o lner's condition.
\end{proposition}
\begin{proof}
  Let a finite subset $S\Subset G$ and $\epsilon>0$ be given. Since
  $X$ has subexponential growth, we have
  $\lim\sqrt[n]{v_{X,x_0,S}(n)}=1$; therefore
  $\liminf\frac{v_{X,x_0,S}(n+1)}{v_{X,x_0,S}(n)}=1$, so for some $n$
  we have $v_{X,x_0,S}(n+1)<(1+\epsilon)v_{X,x_0,S}(n)$. Set
  $F=B_{X,x_0,S}(n)$. We have $\#(F S)<(1+\epsilon)\#F$, so $X$
  satisfies F\o lner's condition by Lemma~\ref{lem:folnerfg}.
\end{proof}

Note that it is unknown whether in every finitely generated group $G$
of subexponential growth we have
$\lim\frac{v_{G,S}(n+1)}{v_{G,S}(n)}=1$; only the `$\liminf$' is known
to equal $1$.

One may study more quantitatively the F\o lner condition as follows:
let $X$ be a $G$-set and let $S$ be a generating set for $G$. Define
$\Fol\colon\N\to\N\cup\{\infty\}$ by
\begin{equation}\label{eq:fol}
  \Fol(n)=\inf\{\#F\mid F\Subset X, \#(F\triangle F s)<\#F/n\text{ for all }s\in S\}.
\end{equation}
Then $\Fol(n)<\infty$ for all $n$ precisely if $X$ is amenable. A
similar definition may be given for graphs, which we leave to the
reader. Groups admit the following lower bound on $\Fol$:
\begin{proposition}[Coulhon-(Saloff-Coste)~\cite{coulhon-sc:isoperimetry}]\label{prop:csc}
  Let $G=\langle S\rangle$ be a finitely generated group, with growth
  function $v_{G,S}(n)$. Then
  \[\Fol(n)\ge\frac12 v_{G,S}(n)\text{ for all }n\in\N.\]
\end{proposition}
\begin{proof}
  We shall prove the following equivalent form: given $F\Subset G$,
  choose $n\in\N$ such that $v_{G,S}(n)\ge2\#F$. We are required to
  find $s\in S$ with $\#(F\triangle F s)\ge\#F/n$.

  First, for all $x\in F$ we have
  $\#(x B_{G,S}(n)\setminus F)\ge\#F\ge\#(x B_{G,S}(n)\cap F)$, so
  \begin{gather*}
    \sum_{g\in B_{G,S}(n)}\mathbb 1_{x g\not\in F}\ge v(n)/2\ge\sum_{g\in B_{G,S}(n)}\mathbb 1_{x g\in F},\\
    \sum_{g\in B_{G,S}(n)}\sum_{x\in F}\mathbb 1_{x g\not\in F}\ge v(n)\#F/2,
  \end{gather*}
  so for some $g\in B_{G,S}(n)$ we have
  $\sum_{x\in F}\mathbb 1_{x g\not\in F}\ge\#F/2$, namely
  $\#(F\triangle F g)\ge\#F$. Write now $g=s_1\cdots s_n$; then
  $F\triangle F g=(F\triangle F s_n)\triangle\cdots\triangle(F
  s_2\cdots s_n\triangle F g)\subseteq(F\triangle F s_n)\cup\cdots\cup
  (F\triangle F s_1)s_2\cdots s_n$. It follows that there exists some
  $k\in\{1,\dots,n\}$ with $\#(F\triangle F s_k)\ge\#F/n$.
\end{proof}

On the other hand, we have an upper bound on $\Fol$
coming from balls: with $F=B_{G,S}(n)$ we have
$F\triangle F s\subseteq B_{G,S}(n+1)$ so
$\#(F\triangle F s)\le v(n+1)-v(n)$, and therefore
$\Fol(v(n)/(v(n+1)-v(n)))\le v(n)$. Assuming that $v$
is the restriction to $\N$ of a differentiable function, we may seek a
function $f$ satisfying $f(1/\log(v)')=v$ to obtain an upper bound
$\Fol(n)\le f(n)$. For example, if $v(n)\propto n^d$
then $f(n)\propto n^d$ and therefore the estimate given by
Proposition~\ref{prop:csc} is at worst a constant off. The ``$1/2$''
in Proposition~\ref{prop:csc} cannot easily be eliminated: in a finite
group, we shouldn't expect any good estimates for sets larger than
half of the group.

Note also that we have $\Fol(n)>n$ as soon as $X$ is infinite, since
then $\#(F\triangle F s)\ge1$ for all $F\Subset X$. No analogue of
Proposition~\ref{prop:csc} may hold for $G$-sets in general:
\begin{exercise}[**]
  Let $X$ be a $G$-set for a finitely generated group $G$. Prove that
  $\Fol(n)$ is linear (i.e.\ $\Fol(n)\le C n$ for some constant $C$)
  if and only if the Schreier graph of $X$ has bounded cutsets, namely
  there is a bound $C'$ such that every finite set of vertices can be
  separated by removing at most $C'$ vertices.
\end{exercise}

\begin{exercise}[**]
  We saw in Exercise~\ref{ex:leftamen} that a group is
  ``left-amenable'' if and only if it is amenable. First prove
  directly that if a group admits sets that are almost invariant under
  right translation, then it admits sets that are almost invariant
  under left translation.

  Next, prove that the infinite dihedral group
  $D_\infty=\langle a,b\mid a^2,b^2\rangle$ admits finite subsets that
  are almost right-invariant but far from left-invariant, namely
  subsets $F_n\Subset D_\infty$ with
  $\#(F_n\triangle F_n g)/\#F_n\to 0$ for all $g\in D_\infty$ but
  $\#(F_n\triangle g F_n)/\#F_n\not\to 0$.

  Give on the other hand a family of sets $F_n\Subset D_\infty$ with
  $\#(F_n\triangle g F_n h)/\#F_n\to 0$ for all $g,h\in D_\infty$.
\end{exercise}

We return to Definition~\ref{defn:folnergraph}. A connected graph
$\mathscr G=(V,E)$ endows its set of vertices $V$ with the structure
of a metric space still written $\mathscr G$: the distance between two
vertices is the minimal length of a path connecting them. Given two
metric spaces (e.g.\ connected graphs) $X,Y$, a map $f\colon X\to Y$
is \emph{quasi-Lipschitz} if there is a constant $C$ with
\[d(f(x),f(y))\le C d(x,y)+C,
\]
and $f$ is a \emph{quasi-isometry} if there is a quasi-Lipschitz map
$g\colon Y\to X$ with $\sup_{x\in X}d(x,g(f(x)))<\infty$ and
$\sup_{y\in Y}d(y,f(g(y)))<\infty$.
\begin{exercise}[*]
  Let $\mathscr G=(V,E)$ be a graph, and let $\mathscr G'=(V',E')$ be
  its barycentric subdivision: $V'=V\sqcup E$ and $E'=E\times\{+,-\}$
  with $(e,\pm)^\pm=e^\pm$ and $(e,\pm)^{\mp}=e$. Prove that
  $\mathscr G$ and $\mathscr G'$ are quasi-isometric.
\end{exercise}

\begin{exercise}[*]
  Let $G$ be a finitely generated group. Prove that all Cayley graphs
  of $G$ with respect to finite generating sets are quasi-isometric;
  that all finite-index subgroups of $G$ are have quasi-isometric
  Cayley graphs; and that all quotients of $G$ by finite subgroups
  have quasi-isometric Schreier graphs.
\end{exercise}

\begin{proposition}
  Let $\mathscr G=(V,E)$ and $\mathscr G'=(V',E')$ be bounded-degree
  graphs, and let $f\colon\mathscr G\to\mathscr G'$ be quasi-Lipschitz
  with $\sup_{y\in V'}d(y,f(V))<\infty$. If $\mathscr G$ is amenable
  then $\mathscr G'$ is amenable.

  In particular, if $\mathscr G,\mathscr G'$ are quasi-isometric then
  $\mathscr G$ is amenable if and only if $\mathscr G'$ is amenable.
\end{proposition}
\begin{proof}
  Let $\mathscr G''=(V'',E'')$ be a graph. For $F\Subset V''$ and
  $k\in\N$, define
  \[\partial^k(F)=\{(e_1,\dots,e_k)\mid e_i\in E'',e_i^+=e_{i+1}^-,e_1^-\in F,e_k^+\not\in F\}.\]
  Recall that $\mathscr G$ is amenable if and only if
  $\inf_{F\Subset V}\#\partial F/\#F=0$. Equivalently,
  $\inf_{F\Subset V}\#\{e^+\mid e\in\partial F\}/\#F=0$.  There exists
  a constant $D$ such that, for every $F\Subset V$, we have
  $\{f(e^+)\mid e\in\partial F\}\subseteq\{e_D^+\mid
  (e_1,\dots,e_D)\in\partial^D(f(F))\}$. Therefore,
  $\inf_{F\Subset V}\#\partial^D(f(F))/\#f(F)=0$, and therefore
  $\inf_{F'\Subset V'}\#\partial(F')/\#F'=0$.
\end{proof}

\begin{exercise}[*]
  Prove that if $\mathscr G,\mathscr G'$ are quasi-isometric graphs
  then their F\o lner functions~\eqref{eq:fol} are equivalent in the
  sense that $\Fol_{\mathscr G}(n)\le C\Fol_{\mathscr G'}(C n+C)+C$
  and conversely, for some constant $C$.
\end{exercise}

There are quasi-invariant groups with quite distinct algebraic
properties; e.g., $A\wr\Z$ and $B\wr\Z$ are quasi-isometric for all
finite groups $A,B$ of same cardinality. If $A$ is Abelian but $B$ is
simple, then $A\wr\Z$ is metabelian and residually finite but $B\wr\Z$
is neither. However, these groups are quasi-isometric (and both
amenable).

\subsection{Reiter's criterion}
F\o lner sets --- finite subsets $F\Subset X$ that are almost
invariant under translation --- may be thought of as almost-invariant
characteristic functions.

\begin{definition}[see~\cite{reiter:harmonicanalysis}*{page~168}]
  Let $X$ be a $G$-set. It satisfies \emph{Reiter's condition} for
  $p\ge1$ if for every finite subset $S\Subset G$ and every
  $\epsilon>0$ there exists a positive function $\phi\in\ell^p(X)$
  with $\|\phi s-\phi\|<\epsilon\|\phi\|$ for all $s\in S$.
\end{definition}

\begin{theorem}\label{thm:folneramen}
  Let $X$ be a $G$-set. The following are equivalent:
  \begin{enumerate}
  \item $X$ is amenable;
  \item $X$ satisfies Reiter's condition for $p=1$;
  \item $X$ satisfies Reiter's condition for some $p\in[1,\infty)$;
  \item $X$ satisfies Reiter's condition for all $p\in[1,\infty)$;
  \item $X$ satisfies F\o lner's condition.
  \end{enumerate}
\end{theorem}
\begin{proof}
  $(1)\Rightarrow(2)$ Given $S\Subset G$ and $\epsilon>0$, consider
  the subset
  \[K=\big\{\bigoplus_{s\in S}(\mu s-\mu)\mid \mu\in\mathscr
    P(X)\}\subset\ell^1(X)^S.
  \]
  Since $X$ is amenable, there exists a $G$-invariant functional
  $m\in\ell^\infty(X)^*$ by Corollary~\ref{cor:vneumann}. Since
  $\ell^1(X)$ is weak*-dense in $\ell^\infty(X)^*$, there exists a net
  $(\mu_n)_{n\in\mathscr N}$ in $\ell^1(X)$ with $\mu_n\to\mu$ in the
  weak*-topology, so $\bigoplus_{s\in S}(\mu_n s-\mu_n)\in K$
  converges to $0$ in the weak*-topology on $\ell^1(X)^S$, so
  $\overline K^{\text{weak*}}\ni0$. Since $K$ is convex, its norm
  closure $\overline K$ also contains $0$, by the Hahn-Banach theorem
  (see e.g.~\cite{rudin:fa}*{Theorem~3.12}); so there exists
  $\mu\in\mathscr P(X)$ with $\|\mu s-\mu\|<\epsilon$ for all
  $s\in S$.

  $(2)\Rightarrow(4)$ Let $\psi\in\ell^1(X)$ satisfy
  $\|\psi s-\psi\|<\epsilon\|\psi\|$ for all $s\in S$. Define
  $\phi(x)\coloneqq\psi(x)^{1/p}$; then $\phi\in\ell^p(X)$ with
  $\|\phi\|_p=\|\psi\|^{1/p}$, and
  \begin{align*}
    \|\phi s-\phi\|_p^p&=\sum_{x\in X}|\phi(x s^{-1})-\phi(x)|^p=\sum_{x\in X}|\psi(x s^{-1})^{1/p}-\psi(x)^{1/p}|^p\\
                       &\le\sum_{x\in X}|\psi(x s^{-1})-\psi(x)|\text{ because }|A^{1/p}-B^{1/p}|\le|A-B|^{1/p}\text{ for all }A,B\\
                       &=\|\psi s-\psi\|<\epsilon\|\psi\|=\epsilon\|\phi\|_p^p.
  \end{align*}

  $(4)\Rightarrow(3)$ is obvious, and so is $(2)\Rightarrow(3)$.

  $(3)\Rightarrow(2)$ Let $\psi\in\ell^p(X)$ satisfy
  $\|\psi s-\psi\|<\epsilon\|\phi\|$ for all $s\in S$. Define
  $\phi(x)\coloneqq\psi(x)^p$; then $\phi\in\ell^1(X)$ with
  $\|\phi\|_1=\|\psi\|^p$, and
  \begin{align*}
    \|\phi s-\phi\|&=\sum_{x\in X}|\phi(x s^{-1})-\phi(x)|=\sum_{x\in X}|\psi(x s^{-1})^p-\psi(x)^p|\\
                   &\le\sum_{x\in X}p|\psi(x s^{-1})-\psi(x)|\max\{\psi(x
                     s^{-1}),\psi(x)\}^{p-1}\text{ because }|X^p-Y^p|\le p|X-Y|\max\{X,Y\}^{p-1}\\
                   &\le p\Big(\sum_{x\in X}|\psi(x s^{-1})-\psi(x)|^p\Big)^{1/p}\Big(\sum_{x\in X}|\psi(x s^{-1})+\psi(x)|^p\Big)^{1-1/p}\text{ by H\"older's inequality}\\
                   &= p \|\psi s-\psi\|_p \|\psi s+\psi\|_p^{p-1}<p\epsilon\|\psi\|_p2^{p-1}\|\psi\|_p^{p-1}=p2^{p-1}\epsilon\|\phi\|.
  \end{align*}

  $(2)\Rightarrow(5)$ Given $S\Subset G$ and $\epsilon>0$, let
  $\phi\in\ell^1(X)$ be positive and satisfy
  $\|\phi s-\phi\|<\epsilon\|\phi\|$ for all $s\in S$. For all
  $r\in\R_+$, consider the set $F_r=\{x\in X\mid \phi(x)\ge r\}$. Then
  $\phi=\int\mathbb1_{F_r}d r$ and $\phi s=\int\mathbb1_{F_r s}d r$, so
  \[\int(\#F_r s\triangle F_r)d r=\|\phi s-\phi\|<\epsilon\|\phi\|=\epsilon\int\#F_r d r;
  \]
  therefore, there exists $r\in\R_+$ with
  $\#(F_r s\triangle F_r)<\epsilon\#F_r$, and $X$ satisfies F\o lner's
  criterion by Lemma~\ref{lem:folner1}.

  $(5)\Rightarrow(1)$ By Lemma~\ref{lem:folnerlim}, there exists a net
  $(F_n)_{n\in\mathscr N}$ with
  $\lim_{n\to\infty}\#(F_n g\setminus F_n)/\#F_n\to0$ for all
  $g\in G$.

  For each $n\in\mathscr N$, consider the ``discrete'' mean
  $\mu_n\in\mathscr M(X)$ defined by
  \[\mu_n(A)=\frac{\#(A\cap F_n)}{\#F_n}.\]

  Since $\mathscr M(X)$ is compact, the net $(\mu_n)_{n\in\mathscr N}$
  has an accumulation point, say $\mu$. We will show that $\mu$ is a
  $G$-invariant mean by a standard ``$\delta/3$'' argument.

  Given $g\in G$ and $A\subseteq X$, we show
  $|\mu(A)-\mu(A g)|<\delta$ for any $\delta>0$.  There is
  $n\in\mathscr N$ with
  \[n>(\{g,g^{-1}\},\delta/3),\quad |\mu_n(A)-\mu(A)|<\delta/3,\quad
    |\mu_n(A g)-\mu(A g)|<\delta/3,
  \]
  because the $\mu_n$ converge pointwise to $\mu$. Then
  \begin{align*}
    |\#(A\cap F_n) - \#(A g\cap F_n)| &= |\#(A\cap F_n)-\#(A\cap F_n g^{-1})|\\
      &\le\max\{\#(F_n\setminus F_n g^{-1}),\#(F_n g^{-1}\setminus F_n)\}\\
      &\le\#(F_n\{g,g^{-1}\}\setminus F_n)<\epsilon\#F_n<\delta/3\#F_n,
  \end{align*}
  so $|\mu_n(A)-\mu_n(A g)|<\delta/3$ and
  \[|\mu(A)-\mu(A g)|\le|\mu_n(A)-\mu(A)|+|\mu_n(A)-\mu_n(A g)|+|\mu_n(A g)-\mu(A g)|<\delta.
  \]
  Since this holds for all $\delta>0$, we get
  $\mu(A)=\mu(A g)$.
\end{proof}

In fact, the `$\#(F s\setminus F)/\#F\to0$' in F\o lner's condition
can be substantially weakened:
\begin{proposition}[Gournay]
  Let $X$ be a $G$-set. Then $X$ is amenable if and only if there is a
  constant $c<1$ with the following property: for every finite subset
  $S\Subset X$ there is a finite subset $F\Subset X$ with
  $\#(F s\setminus F)\le c\#F$ for all $s\in S$.
\end{proposition}
\begin{proof}
  $(\Rightarrow)$ is obvious, by Lemma~\ref{lem:folner1} and
  Theorem~\ref{thm:folneramen}.

  $(\Leftarrow)$ by the condition of the proposition, there exists a
  net $(F_n)_{n\in\mathscr N}$ of finite subsets of $X$ (say indexed
  by $\mathfrak P_f(G)$) with
  $\limsup_{n\to\infty}\#(F_n g\setminus F_n)/\#F_n\le c$ for all
  $g\in G$. Set
  $\xi_n\coloneqq\mathbb 1_{F_n}/\sqrt{\#F_n}\in\ell^2(X)$ be the
  normalized characteristic function of $F_n$. We have
  \[2-2\langle\xi_n,\xi_n g\rangle=\|\xi_n g-\xi_n\|_2^2=\|\mathbb 1_{F_n g}-\mathbb 1_{F_n}\|_1/\#F_n=2\#(F_n g\setminus F_n)/\#F_n,
  \]
  so $\langle\xi_n,\xi_n g\rangle\ge1-c$ for all $n\gg1$.

  Choose now a non-principal ultrafilter $\mathfrak F$ on
  $\mathscr N$, and consider the ultraproduct space
  $\mathscr H\coloneqq\ell^2(X)^{\mathfrak F}$: it is a Hilbert space,
  whose elements are equivalence classes of sequences
  $(\eta_n)_{n\in\mathscr N}$ with $\eta_n\in\ell^2(X)$ for all $n$
  and $\sum_{n\in\mathscr N}\|\eta_n\|^2<\infty$, under the relation
  $(\eta_n)\sim(\eta'_n)$ if $\lim_{\mathfrak
    F}\|\eta_n-\eta'_n\|=0$.

  Write $\xi=(\xi_n)\in\mathscr H$, and let $K$ denote the convex hull
  in $\mathscr H$ of $\{\xi g\mid g\in G\}$. We have
  $\langle\xi g,\xi\rangle\ge1-c$ for all $g\in G$, so
  $\langle\xi,\eta\rangle\ge1-c>0$ for all $\eta\in K$, and in
  particular $0\not\in K$. Let $\zeta'$ be the element of $K$ of
  minimal norm, and set $\zeta=\zeta'/\|\zeta'\|$, represented by a
  sequence $(\zeta_n)_{n\in\mathscr N}$ with $\zeta_n\in\ell^2(X)$ of
  norm $1$. Since $\zeta g=\zeta$ by unicity of the element of minimal
  norm in $K$, we have $\|\zeta_n-\zeta_n g\|\to0$ for all $g\in G$,
  so $X$ is amenable by Theorem~\ref{thm:folneramen}(3).
\end{proof}

We finally present a result that puts as much symmetry between $X$ and
$G$ as possible, with an eye towards the corresponding notion with the
roles of $G$ and $X$ interchanged, see
Theorem~\ref{thm:Lamenable}:
\begin{proposition}\label{prop:Ramenable}
  Let $X$ be a $G$-set. Then $X$ is amenable if and only if for every
  $\epsilon>0$ and every $g\in\varpi(\ell^1G)$ there exists a positive
  function $f\in\ell^1(X)$ with $\|f g\|<\epsilon\|f\|$.
\end{proposition}
\begin{proof}
  ($\Rightarrow$) Given $\epsilon>0$ and
  $g=\sum_{x\in G}g_x x\in\ell^1G$ with $\sum g_x=0$, let $S\Subset G$
  be such that $g'\coloneqq g-\sum_{s\in S}g_s(s-1)$ satisfies
  $\|g'\|<\epsilon/2$. Since $X$ is amenable, there exists
  $F\Subset X$ with $\#(F s\setminus F)<\epsilon/4\|g\|\#F$ for all
  $s\in S$. Set $f\coloneqq\mathbb 1_F$. Then
  \[\|f g\|\le\|f g'\|+\sum_{s\in S}\|g_s(f s-f)\|<\#F\|g'\|+2\#(F S\setminus F)\|g\|<\epsilon\#F=\epsilon\|f\|.\]

  ($\Leftarrow$) Given $\epsilon>0$ and $S\Subset G$, set
  $g=\sum_{s\in S}s-1$, and let $f\in\ell^1(G)$ be a positive function
  satisfying $\|f g\|<\epsilon\|f\|/2$. Then
  \[\epsilon\|f\|>2\|f g\|\ge2\Big\|\sum_{s\in S}\max(f s-f,0)\Big\|=\sum_{s\in S}2\|\max(f s-f,0)\|\ge\sum_{s\in S}\|f s-f\|.\qedhere\]
\end{proof}

\subsection{Non-amenability}
It may be interesting to consider weaker versions of amenability for
groups; for instance, to consider groups admitting faithful amenable
actions.

Osin considers in~\cite{osin:weaklyamen} a class of ``weakly amenable
groups'', which in our context are groups $G$ admitting an amenable
action $X$ such that, for every finite $F\subset G$, there exists
$x\in X$ with $\#(x F)=\#F$; namely, the orbit map $f\mapsto x\cdot f$
is injective on $F$. An example of a weakly amenable, non-amenable
group is the Baumslag-Solitar group
$\langle a,t\mid a^mt=ta^n\rangle$, for $m>n\ge2$.

If a group $G$ is not amenable, but all its proper subgroups are
amenable, then $G$ does not have any ``interesting'' amenable actions:
by Proposition~\ref{prop:stabilizers}, every amenable action of $G$
has a fixed point. This applies in particular to Tarski
monsters~\cite{olshansky:monsters}, which are non-amenable torsion
groups in which every proper subgroup is cyclic.

\begin{definition}[Kazhdan, see~\cite{kazhdan:T} or~\cite{bekka-h-v:t}]
  A group $G$ has \emph{property (T)} if every unitary representation
  $G\to U(\mathscr H)$ in a Hilbert space $\mathscr H$ with almost
  invariant vectors (in the sense that for every $\epsilon>0$ and
  every finite $S\Subset G$ there exists non-trivial $x\in\mathscr H$
  with $\|x-x s\|<\epsilon$ for all $s\in S$) has a non-trivial fixed
  vector.
\end{definition}

If $G$ is infinite, then Kazhdan's property (T) restricted to the
unitary representation on $\ell^2(G)$ is thus precisely the negation
of amenability: there are invariant vectors in $\ell^2(G)$ if and only
if $G$ is finite, and the existence of almost-invariant vectors is
Reiter's condition for $p=2$.

Thus an amenable group with property (T) is finite,\footnote{This was
  exploited in a fundamental manner by Margulis
  in~\cite{margulis:subgroups} to prove his ``normal subgroup
  theorem''.} and more generally a group with property (T) does not
have any ``interesting'' amenable actions: every amenable action has a
finite orbit.

Glasner and Monod consider in~\cite{glasner-monod:amenableactions}
another group property, which they call \emph{property (F)}: ``every
amenable action has a fixed point''. They show that a free product of
groups always has a faithful, transitive, amenable action unless one
factor is (F) and the other is virtually (F). Thus for example $G*\Z$
is not amenable if $G\neq1$, yet admits a faithful, transitive,
amenable action.

% \cite{juschenko-dlsalle:wobbling}: if $X$ grows subexponentially,
% then $W(X)$ does not contain an infinite countable group with (T)

%%%%%%%%%%%%%%%%%%%%%%%%%%%%%%%%%%%%%%%%%%%%%%%%%%%%%%%%%%%%%%%%
\newpage\section{Growth of groups}\label{ss:growth}
We cover here some classical material on asymptotic growth of
groups. Recall from~\S\ref{ss:growthsets} that $v_{G,S}(n)$ denotes
the number of elements in a group $G$ that are expressible as products
of at most $n$ elements of $S$. The group $G$ has \emph{exponential
  growth} if $v_{G,S}(n)\ge B^n$ for some $B>1$, and
\emph{subexponential growth} otherwise; it has \emph{polynomial
  growth} if $v_{G,S}(n)\le p(n)$ for some polynomial $p$; and it has
\emph{intermediate growth} if its growth is neither polynomial nor
exponential. These properties are easily seen to be independent of the
choice of generating set $S$.

\subsection{Groups of polynomial growth}
Groups of polynomial growth admit an elegant algebraic
characterization. The ``if'' part is due to \HBass~\cite{bass:nilpotent}
and independently \YGuivarch~\cite{guivarch:poly1}, with an explicit
computation of the growth degree of $G$, which is always an integer; the
harder, ``only if'' part is due to \MGromov.

We recall some basic group theoretical terminology. For $\mathscr P$ a
property of groups (abelian, \dots), a group group $G$ is called
\emph{virtually $\mathscr P$} if $G$ admits a finite-index subgroup
satisfying $\mathscr P$.\footnote{Much to the annoyance of finite group
  theorists, some people call finite groups ``virtually trivial''.}

A group $G$ is \emph{nilpotent} if there exists a constant $c$ such
that every $(c+1)$-fold iterated commutator
$[g_0,[g_1,\dots,[g_{c-1},g_c]\cdots]]$ vanishes in $G$; the minimal
$c$ is called the \emph{nilpotency class} of $G$. A group $G$ is
\emph{polycyclic} if it admits a sequence of subgroups
$G=G_0\triangleright G_1\triangleright\cdots\triangleright G_n=1$ with
$G_k/G_{k+1}$ cyclic for all $k$. Finitely generated nilpotent groups
are polycyclic.

\begin{theorem}[Gromov~\cite{gromov:nilpotent}]\label{thm:polygrowth}
  Let $G$ be a finitely generated group. Then $G$ has polynomial
  growth if and only if $G$ is virtually nilpotent, namely $G$ has a
  finite-index nilpotent subgroup.
\end{theorem}
\begin{proof}[Proof of Theorem~\ref{thm:polygrowth}, ``if'' direction]
  Let $G_0$ be a finite-index nilpotent subgroup of $G$. It suffices
  to prove that $G_0$ has polynomial growth, since then $G$ will have
  polynomial growth of same degree as $G_0$. Denote by $c$ the
  nilpotency class of $G_0$, so all $(c+1)$-fold iterated commutators
  vanish.

  Let $(G_k)_{0\le k\le\ell}$ be a composition series for $G$, namely
  a series of subgroups such that $G_k/G_{k+1}$ is cyclic for all $k$;
  and for each $k$ let $x_k\in G_k$ be a lift of a generator of
  $G_k/G_{k+1}$ so that $G_0=\langle x_0,x_1,\dots x_{\ell-1}\rangle$.

  We reason by induction on $\ell$. If $\ell=0$, or if $G/G_1$ is
  finite, we are done. Assume then $G_0/G_1\cong\Z$, and by induction
  that the growth of $G_1$ is bounded by a polynomial, say of
  degree $d$.

  Consider $x\in G_0$, of the form
  $x=x_{i_1}^{\pm1}\cdots x_{i_n}^{\pm1}$. Write it in the form
  $x_0^e z$, with $e\in\Z$ and $z\in G_1$. This requires us to
  exchange past each other some letters $x_0$ and $x_{i_j}$, producing
  subexpressions $[x_{i_j},x_0,\dots,x_0]$ along the process: indeed
  one has $W x_0=x_0W[W,x_0]$ for any expression $W$.

  There are at most $n$ letters $x_0$ in $x$; each of them must be
  brought past at most $n$ other letters, producing at most $n^2$
  expressions $[x_i,x_0]$; each of these produces in turn at most
  $n^3$ expressions $[x_i,x_0,x_0]$; etc. We take as generating set
  $S$ for $G_1$ all expressions of the form $[x_i,x_0,\dots,x_0]$
  with $i\ge1$ and length $\in\{1,\dots,c\}$. We have then expressed
  $x$ by an integer $e\in\{-n,\dots,n\}$ and a word $z$ of length at
  most $n+\dots+n^c$ in these generators; so
  \[v_{G_0,S\cup\{x_0\}}(n)\le(2n+1)v_{G_1,S}(n+\dots+n^c)\]
  is bounded by a polynomial of degree $\le c d+1$.
\end{proof}

We shall give at the end of~\S\ref{ss:harmonic} a sketch of the ``only
if'' direction, via slowly growing harmonic functions.

\begin{corollary}\label{cor:abelian}
  Let $G$ be a virtually nilpotent group. Then $G$ is amenable.
\end{corollary}
\begin{proof}
  If $G$ is virtually nilpotent, then every finitely generated
  subgroup of $G$ is also virtually nilpotent, so by
  Theorem~\ref{thm:polygrowth} has polynomial growth, so is amenable
  by Proposition~\ref{prop:subexp=>amenable}.
\end{proof}

\subsection{Groups of exponential growth}
At the other end of the growth spectrum, we find groups of
\emph{exponential growth}. In fact, as soon as a group has a
non-abelian free subgroup, it has exponential growth; so a large class
of groups, including all non-elementary hyperbolic
groups~\cite{ghys-h:gromov}, have exponential growth.

In the class of soluble groups, the growth of a group is either
polynomial or exponential, as we shall see below. Recall that the
\emph{derived series} of a group $G$ is the series of normal subgroups
defined by $G^{(0)}=G$ and $G^{(i+1)}=[G^{(i)},G^{(i)}]$, and that $G$
is \emph{soluble} if $G^{(n)}=1$ for some $n$. The minimal such $n$ is
called the \emph{derived length} of $G$.

\begin{proposition}\label{prop:G'fg}
  Let $G$ be a finitely generated group of subexponential growth, and let
  $N\triangleleft G$ be a normal subgroup with $G/N\cong\Z$. Then $N$
  is also finitely generated.
\end{proposition}
\begin{proof}
  Let $S=\{x_1,\dots,x_d\}$ generate $G$, and let $x\in G$
  generate $G/N$. Write each $x_i=x^{e_i}y_i$, with $y_i\in N$; so
  $G=\langle x,y_1,\dots,y_d\rangle$, and $N=\langle
  y_1,\dots,y_d\rangle^G$.

  Consider further $N_i=\langle y_i^{x^n}\mid n\in\Z\rangle$, so that
  $N=\langle N_1,\dots,N_d\rangle$. It is sufficient to show that each
  $N_i$ is finitely generated.

  Write then $y=y_i$, and consider all expressions
  $x^{-1}y^{e_1}x^{-1}y^{e_2}\dots x^{-1}y^{e_n}x^n$, with all
  $e_j\in\{0,1\}$. There are $2^n$ such expressions, and their length
  is linear in $n$, so two must be equal in $G$ because $G$ has
  subexponential growth. Let
  \begin{equation}\label{eq:G'fg:1}
    y^{e_1x}\cdots y^{e_m x^m}=y^{f_1x}\cdots y^{f_m x^m}
  \end{equation}
  be such an equality in $G$, without loss of generality with
  $1=e_m\neq f_m=0$.  It follows that $y^{x^m}$ is in the group
  generated by $\{y^x,\dots,y^{x^{m-1}}\}$, so that $N_i=\langle
  y_i^{x^n}\mid n<m\rangle$. Now a similar argument, replacing $x$ by
  $x^{-1}$ in~\eqref{eq:G'fg:1}, shows that $N_i$ is finitely
  generated.
\end{proof}

\begin{corollary}\label{cor:Nfg}
  Let $G$ be a finitely generated group of subexponential growth, and
  let $N\triangleleft G$ be a normal subgroup such that $G/N$ is
  virtually polycyclic. Then $N$ is finitely generated.\qed
\end{corollary}

\begin{corollary}[Milnor]
  Let $G$ be a finitely generated soluble group of subexponential
  growth. Then $G$ is polycyclic.
\end{corollary}
\begin{proof}  
  Consider the derived series $G^{(i)}$ of $G$; by assumption,
  $G^{(s+1)}=1$ for some minimal $s\in\N$. Set $A=G^{(s)}$. We may
  assume, by induction, that $G/A$ is polycyclic. By
  Corollary~\ref{cor:Nfg}, the subgroup $A$ is finitely generated and
  abelian, so is polycyclic too. It follows that $G$ is polycyclic.
\end{proof}

\begin{lemma}\label{lem:vnil}
  Let $G$ be a finitely generated group that is an extension $N.Q$ of
  finitely generated virtually nilpotent groups. Then $G$ is virtually
  soluble.
\end{lemma}
\begin{proof}
  Assume first that $N$ is finite; we then claim that $G$ is virtually
  nilpotent. Indeed the centralizer $Z_G(n)$ has finite index in $G$,
  so $Z=\bigcap_{n\ge0}Z_G(n)$ has finite index in $G$. Then $Z$ is a
  central extension of $Z\cap N$ by $Z/(Z\cap N)$, so is virtually
  nilpotent; and then so is $G$.

  We turn to the general case. Let $N_0$ be a nilpotent subgroup of
  finite index in $N$. Up to replacing $N_0$ by
  $\bigcap_{[N:M]=[N:N_0]}M$, we may assume $N_0$ is characteristic in
  $N$, and therefore normal in $G$. By the first paragraph, $G/N_0$ is
  virtually nilpotent, so $G$ is virtually soluble.
\end{proof}

We recall that a group is \emph{noetherian} if all its subgroups are
finitely generated; in other words, if every chain $H_1<H_2<\cdots$ of
subgroups of $G$ is finite.

\begin{lemma}\label{lem:polycyclic=soluble+noetherian}
  A group $G$ is polycyclic if and only if it is both soluble and
  noetherian.
\end{lemma}
\begin{proof}
  Note first that an abelian group is noetherian if and only if it is
  finitely generated: if finitely generated, it is of the form
  $\Z^d\times F$ for a finite abelian group $F$, and is clearly
  noetherian.

  If $G$ is soluble and noetherian, then all quotients
  $G^{(i)}/G^{(i+1)}$ along its derived series are also noetherian, so
  finitely generated; the derived series may then be refined into a
  polycyclic series.

  Conversely, an extension of noetherian groups is noetherian, so if
  $G$ is polycyclic, then it is noetherian by induction.
\end{proof}

This reduction to polycyclic groups brings us closer to groups of
polynomial growth; the next step is the
\begin{theorem}[Wolf]
  Let $G$ be a polycyclic group of subexponential growth. Then $G$ is
  virtually nilpotent.
\end{theorem}
\begin{proof}
  Let $G=G_0>G_1>\cdots$ be a polycyclic series of minimal length. If
  $[G:G_1]<\infty$, proceed inductively with $G_1$. Assume therefore
  that $G/G_1\cong\Z=\langle x\rangle$. By induction, there is a
  nilpotent subgroup $N\le G_1$ of finite index. Furthermore, since
  $G_1$ is finitely generated by Proposition~\ref{prop:G'fg}, we may
  suppose that $N$ is characteristic in $G_1$, at the cost of
  intersecting it with its finitely many images under automorphisms of
  $G_1$; so we may assume $N\triangleleft G$. We have $N\langle
  x\rangle\le G$ of finite index, and we replace $G$ by $N\langle
  x\rangle$, to simplify notation.

  We now seek a central series $(N_k)$ in $N$, i.e.\ a series with
  $N_0=N$, all $N_k$ normal in $G$, and $N_k/N_{k+1}\le Z(N/N_{k+1})$;
  and we require that some non-zero power $x^n$ centralizes
  $N_k/N_{k+1}$ for all $k$. Then $\langle N,x^n\rangle$ will be the
  finite-index nilpotent subgroup of $G$ we are after.

  Among central series, choose one maximizing the number of $k$ such
  that $N_k/N_{k+1}$ is infinite; it exists because the number of
  factors is bounded by the Hirsch length of $G$. The torsion subgroup
  of $N_k/N_{k+1}$ is characteristic, so insert it in the series
  between $N_k$ and $N_{k+1}$. The resulting series is such that each
  quotient $N_k/N_{k+1}$ is either finite or free abelian; and, in the
  latter case, if $M\triangleleft G$ and $N_{k+1}\le M\le N_k$, then
  either $N_{k+1}=M$ or $N_k/M$ is finite.

  If $N_k/N_{k+1}$ is finite, then certainly some non-zero power of
  $x$ will act trivially on it. We therefore consider
  $N_k/N_{k+1}\cong\Z^m$, and we study the $\Q[x]$-module
  $V:=N_k/N_{k+1}\otimes\Q\cong\Q^m$.

  The module $V$ is irreducible; indeed, otherwise there would exist a
  proper, non-trivial invariant subspace $W<V$; then $M:=\{x\in
  N_k\mid x N_{k+1}\in W\}$ is a normal subgroup of $G$, of infinite
  index in $N_k$, contradicting the maximality of the number of
  infinite factors in $(N_k)$. We then use the
  \begin{lemma}[Schur]\label{lem:schur}
    Let $V$ be an irreducible module. Then $\operatorname{End}(V)$ is a division ring.
  \end{lemma}
  \begin{proof}
    Let $\alpha\neq0\in\operatorname{End}(V)$ be an endomorphism; then $\ker(\alpha)$ and
    $\alpha(V)$ are invariant subspaces, so $\ker(\alpha)=0$ and
    $\alpha(V)=V$; so $\alpha$ is invertible.
  \end{proof}

  We see $x\in G$ as an endomorphism of $V$; by Lemma~\ref{lem:schur},
  the ring $\operatorname{End}(V)$ does not contain nilpotent elements, so $x$
  generates a field $\Q(x)$ within $\operatorname{End}(V)$. Since $\operatorname{End}(V)$ is
  finite-dimensional, $x$ is algebraic. Since $x$ preserves the
  lattice $N_k/N_{k+1}\subset V$, it is an algebraic integer. We now
  recall the classical
  \begin{lemma}[Kronecker]\label{lem:kronecker}
    Let $\tau$ be an algebraic number, all of whose conjugates have
    norm $1$. Then $\tau$ is a root of unity.
  \end{lemma}
  \begin{proof}
    Let $\tau$ be algebraic of degree $n$, and consider some power
    $\sigma=\tau^N$. Then $\sigma\in\Q(\tau)$, and all conjugates of
    $\sigma$ have norm $1$, so the coefficients of the minimal
    polynomial of $\sigma$, which are symmetric functions of the
    conjugates of $\sigma$, have norm at most $2^n$. It follows that
    there are finitely many such minimal polynomials, so
    $\sigma^N=\sigma^M$ for some $M>N$.
  \end{proof}

  We are now ready to finish the proof. Either all conjugates of $x$
  (seen now as an algebraic number) have norm $\le1$; and then $x$ is
  a root of unity by Lemma~\ref{lem:kronecker}, so $x^n$ acts
  trivially for some $n>0$; or there exists an embedding of $\Q(x)$ in
  $\C$ such that $|x|>1$.

  In that last case, we may replace $x$ by a power of itself so that
  $|x|>2$. Choose $y\in N_k\setminus N_{k+1}$, seen as a vector
  $v\neq0\in V$. Consider as in the proof of Proposition~\ref{prop:G'fg}
  all expressions $x^{-1}y^{e_1}x^{-1}y^{e_2}\dots x^{-1}y^{e_n}x^n$,
  with all $e_j\in\{0,1\}$. There are $2^n$ such expressions, and
  their length is linear in $n$, so two must be equal in $G$ because
  $G$ has subexponential growth. This leads in $V$ to the relation
  \[(e_1-f_1)x(v)+\cdots+(e_{n-1}-f_{n-1})x^{n-1}(v)+x^n(v)=0,\]
  so $(e_1-f_1)x+\cdots+(e_{n-1}-f_{n-1})x^{n-1}+x^n=0$, because only
  $0$ is non-invertible in $\operatorname{End}(V)$. Now taking norms we get
  \[|x|^n\le(e_1-f_1)|x|+\cdots+(e_{n-1}+f_{n-1})|x|^{n-1}\le|x|\frac{|x|^{n-1}-1}{|x|-1}\le|x|^n\]
  using $|x|>2$, a contradiction.
\end{proof}

\begin{corollary}\label{cor:milnorwolf}
  Let $G$ be a virtually soluble finitely generated group. Then $G$
  has either polynomial of exponential growth, and has polynomial
  growth precisely when it is virtually nilpotent.\qed
\end{corollary}

\subsection{Groups of intermediate growth}\label{ss:interm}
The previous sections were aimed at showing that ``most'' groups have
polynomial or exponential growth; \JMilnor\ asked in 1968 whether
there existed any groups of intermediate
growth~\cite{milnor:5603}. There can be no such examples among
virtually soluble groups, as we saw above; nor among linear groups
(subgroups of matrix groups over fields), by Tits'
alternative~\cite{tits:linear}.

Milnor's question has, however, a positive answer, which was given in
the early 1980's by \RGrigorchuk. We give here his example.

Set $\Alphabet=\{0,1\}$, and consider the following group $\GG$ acting
recursively on the set $X\coloneqq\Alphabet^\N$ of infinite sequences
over $\Alphabet$. It is generated by four elements $a,b,c,d$ defined
by
\begin{align*}
  (x_0x_1\cdots)a &=(1-x_0)x_1\cdots,\\
  (x_0x_1\cdots)b &= \begin{cases}x_0\cdots(1-x_n)x_{n+1}\cdots&\text{ if }x_0=\cdots=x_{n-2}=0\neq x_{n-1},n\not\equiv0\pmod3\\
    x_0x_1\cdots & \text{ else},
  \end{cases}\\
  (x_0x_1\cdots)c &= \begin{cases}x_0\cdots(1-x_n)x_{n+1}\cdots&\text{ if }x_0=\cdots=x_{n-2}=0\neq x_{n-1},n\not\equiv2\pmod3\\
    x_0x_1\cdots & \text{ else},
  \end{cases}\\
  (x_0x_1\cdots)d &= \begin{cases}x_0\cdots(1-x_n)x_{n+1}\cdots&\text{ if }x_0=\cdots=x_{n-2}=0\neq x_{n-1},n\not\equiv1\pmod3\\
    x_0x_1\cdots & \text{ else}.
  \end{cases}
\end{align*}
This action is the limit of an action on finite sequences
$\Alphabet^*$, which is the vertex set of the binary rooted tree, and
$\GG$ is self-similar, see Definition~\ref{defn:treeaut}.

\begin{theorem}[Grigorchuk]\label{thm:grig}
  The group $\GG$ has intermediate growth. More precisely, let
  $\eta\approx0.811$ be the positive root of $X^3+X^2+X-2=0$; then
  \[\exp(n^{1/2})\precsim v_{G,S}(n)\precsim\exp(n^{\log(2)/(\log(2)-\log(\eta))}).\]
\end{theorem}

We begin by a series of exercises deriving useful properties of
$\GG$. Details may be found e.g.\ in~\cite{harpe:ggt}*{Chapter~8}. The
self-similar structure of $\GG$ is at the heart of all arguments; let
us describe it again, starting from the action above.

There is an injective group homomorphism
$\Phi\colon\GG\to(\GG\times\GG)\rtimes C_2$, written
$g\mapsto\pair<g_0,g_1>\pi_g$, and defined as follows. If $g$ permutes
$0X$ and $1X$ then $\pi_g=\varepsilon\neq1$ while if $g$ preserves
them setwise then $\pi_g=1$. Then $g\pi_g^{-1}$ preserves $0X$ and
$1X$, and for $i=0,1$ define a permutation $g_i$ of $X$ by
$(x_0x_1\dots)g=(x_0\pi_g)\;(x_1\dots)g_{x_0}$.  To see that the $g_i$
belong to $\GG$, note that $\Phi$ is given on the generators by
\[\Phi\colon\begin{cases}a&\mapsto\pair<1,1>\varepsilon,\\
    b&\mapsto\pair<a,c>,\\
    c&\mapsto\pair<a,d>,\\
    d&\mapsto\pair<1,b>.\\
  \end{cases}
\]

\begin{exercise}[*]\label{ex:grigrel}
  Check in $\GG$ the relations $a^2=b^2=c^2=d^2=b c d=(ad)^4=1$.
\end{exercise}

We fix once and for all the generating set $S=\{a,b,c,d\}$ of
$\GG$. It follows from the exercise that every element of $\GG$ may be
written as a word of minimal length in the form
$s_0a s_1\cdots s_{n-1}a s_n$ for some $s_0,s_n\in\{1,b,c,d\}$ and
other $s_i\in\{b,c,d\}$.

We let $\eta\approx0.811$ be the real root of $X^3+X^2+X-2=0$, and
define a metric on $\GG$ by setting
\[\|a\|=1-\eta^3,\quad
  \|b\|=\eta^3,\quad\|c\|=1-\eta^2,\quad\|d\|=1-\eta\] and extending
the metric to $\GG$ by the triangle inequality:
$\|g\|=\min\{\|s_1\|+\cdots+\|s_n\|\mid g=s_1\cdots s_n\}$.

\begin{lemma}\label{lem:grigineq}
  If $\Phi(g)=\pair<g_0,g_1>\pi$, then $\|g_0\|+\|g_1\|\le\eta(\|g\|+\|a\|)$.
\end{lemma}
\begin{proof}
  Consider $g\in\GG$. Since $\|c\|+\|d\|\ge\|b\|$ etc., $g$ may be
  written as a word of minimal norm in the form
  $s_0a s_1\cdots s_{n-1}a s_n$ for some $s_0,s_n\in\{1,b,c,d\}$ and
  other $s_i\in\{b,c,d\}$, using Exercise~\ref{ex:grigrel}. Now among
  the $s_i$, each `$b$', taken with the `$a$' after it, contributes
  $\|b\|+\|a\|=1$ to $\|g\|$, and contributes at most
  $\|a\|+\|c\|=\eta$ to $\|g_0\|+\|g_1\|$ because
  $\Phi(b)=\pair<a,c>$. Similarly, each `$c$'+`$a$' contributes $\eta$
  to $\|g\|$ and at most $\eta^2$ to $\|g_0\|+\|g_1\|$, and each
  `$d$'+`$a$' contributes $\eta^2$ to $\|g\|$ and at most $\eta^3$ to
  $\|g_0\|+\|g_1\|$. Only the last $s_n$ may not have an `$a$' after
  it. Summing all these inequalities proves the lemma.
\end{proof}

\begin{exercise}[**]
  Define $\sigma\colon\GG\to\GG$ by
  \[\sigma\colon a\mapsto c^a,\quad b\mapsto d,\quad d\mapsto c,\quad c\mapsto b,
  \]
  extended multiplicatively. Prove
  $\Phi(\sigma(g))=\pair<\theta(g),g>$ for all $g\in\GG$, where
  $\theta(a)=d,\theta(b)=1,\theta(c)=\theta(d)=a$ is a homomorphism to
  the finite group $\langle a,d\rangle\cong D_4$. Deduce that $\sigma$
  is well-defined, and is an injective endomorphism of $\GG$. For the
  usual word metric, prove that $|\sigma(g)|\le2|g|+1$ for all
  $g\in\GG$.
\end{exercise}

\begin{proof}[Proof of Theorem~\ref{thm:grig}, see~\cite{bartholdi:upperbd}]
  For the lower bound, consider the map (not homomorphism!)
  \[F\colon\GG\times\GG\to\GG,\qquad(g_0,g_1)\mapsto
    \sigma(g_0)^a\sigma(g_1).
  \]
  By the exercise, we have
  $\Phi(F(g_0,g_1))=\pair<g_0\theta(g_1),\theta(g_0)g_1>$. Since
  $\#\theta(\GG)=8$ and $\Phi$ is injective, we have $\#\Phi^{-1}(g)=8$ for
  all $g\in\GG$. Also, $|\sigma(g)|\le2|g|+1$ for the usual word metric, so
  $|F(g_0,g_1)|\le2|g_0|+2|g_1|+4$. Denoting by $B(n)$ the ball of radius
  $n$ in $\GG$ for the word metric, we have $F(B(n)\times B(n))\subseteq
  B(4n+4)$, so the growth function $v(n)$ of $\GG$ satisfies $8v(n-2)^2\le
  v(4(n-2)+4)\le v(4n-2)$. Iterating, we have $v(4^t n-2)\ge 8^{2^t-1}
  v(n-2)^{2^t}$, so $v(n)\ge 8^{\sqrt{n/8}-1}$.

  For the upper bound, we make use of the norm $\|\cdot\|$, and
  represent every $g\in\GG$ by a finite rooted tree $R(g)$. Fix any
  constant $K>\|a\|/(\eta-1)$. Given $g\in\GG$, construct $R(g)$ as
  follows. If $\|g\|\le K$, let $R(g)$ be the one-vertex tree with
  label $g$ written at the root, which is also a leaf of the tree.

  If $\|g\|>K$, compute $\Phi(g)=\pair<g_0,g_1>\pi$. Note
  $\|g_0\|,\|g_1\|<\|g\|$, and construct $R(g_0),R(g_1)$
  recursively. Let then $R(g)$ be the tree with a root labeled $\pi$
  connected by two edges leading to the roots of $R(g_0)$ and $R(g_1)$
  respectively. Since $\Phi$ is injective, the map $R$ is injective,
  and it remains to count the number of trees of given size.

  Up to replacing $\|g\|$ by $\max\{1,\|g\|-K\}$, we may assume that,
  in Lemma~\ref{lem:grigineq}, we have $\|g_0\|+\|g_1\|\le\eta\|g\|$
  as soon as $\|g\|$ is large enough.

  Let us denote by $\#R(g)$ the number of leaves of $R(g)$, and set
  $\alpha=\log2/(\log2-\log\eta)$. We claim that there is a constant
  $D$ such that $\#R(g)\le D\|g\|^\alpha$ for all $g\in\GG$. This is
  certainly true if $\|g\|$ is small enough. For $\|g\|>K$, we proceed
  by induction:
  \begin{align*}
    \#R(g) &= \#R(g_0)+\#R(g_1)\le D(\|g_0\|^\alpha+\|g_1\|^\alpha)\\
    &\le 2D\bigg(\frac{\|g_0\|+\|g_1\|}{2}\bigg)^\alpha\text{ by convexity of $X^\alpha$}\\
    &\le2D\|g\|^\alpha\big(\frac\eta2\big)^\alpha=D\|g\|^\alpha.
  \end{align*}
  We finally count the number of trees with $n$ leaves. There are
  $\operatorname{Catalan}(n)$ such tree shapes; each of the $n-1$
  non-leaf vertices has a label in $\{1,\varepsilon\}$, and each of
  the $n$ leaf vertices has a label in $B(K)$. It follows that there
  are $\operatorname{Catalan}(n)2^{n-1}B(K)^n\le E^n$ trees with at
  most $n$ leaves, for some constant $E$; and then
  $v(n)\le E^{n^\alpha}$.
\end{proof}

\begin{exercise}[**]\label{ex:torsion}
  Prove that $\GG$ is a torsion group.

  \emph{Hint:} Use Exercise~\ref{ex:grigrel}, Lemma~\ref{lem:grigineq}
  and induction.
\end{exercise}
  
%%%%%%%%%%%%%%%%%%%%%%%%%%%%%%%%%%%%%%%%%%%%%%%%%%%%%%%%%%%%%%%% 
\newpage\section{Paradoxical decompositions}\label{sec:pd}
We consider again the general case of a group $G$ acting on a set $X$, and
shall derive other characterizations of amenability, based on finite
partitions of $X$.
\begin{definition}\label{def:paradoxical}
  A $G$-set $X$ is \emph{paradoxical} if there are partitions
  \[X=Y_1\sqcup\dots\sqcup Y_m=Z_1\sqcup\dots\sqcup Z_n,\]
  and $g_1,\dots,g_m,h_1,\dots,h_n\in G$, such that
  \[X=Y_1g_1\sqcup\dots\sqcup Y_m g_m\sqcup Z_1h_1\sqcup\dots\sqcup
  Z_n h_n.\qedhere\]
\end{definition}

As a naive example, relax the condition that $G$ be a group, and consider
the monoid of affine transformations of $\N$. Then $\N=\N g_1\sqcup\N h_1$
for $g_1(n)=2n$ and $h_1(n)=2n+1$ defines a paradoxical
decomposition\footnote{This should not come as a surprise, since
  $\{g_1,h_1\}$ generate a free monoid.}.

\begin{example}\label{ex:paraf2}
  We return to Proposition~\ref{prop:fnnotamen}. More precisely, now,
  consider $X=G=\langle x_1,x_2\mid\rangle$ a free group of rank $2$; and
  \begin{xalignat*}{2}
    Y_1 &= \{\text{reduced words ending in }x_1\}, & Y_2 &= G\setminus Y_1,\\
    Z_1 &= \{\text{reduced words ending in }x_2\}\cup\{1,x_2^{-1},x_2^{-2},\dots\}, & Z_2 &= G\setminus Z_1;
  \end{xalignat*}
  then $G=Y_1\sqcup Y_2=Z_1\sqcup Z_2=Y_1\sqcup Y_2 x_1^{-1}\sqcup
  Z_1\sqcup Z_2x_2^{-1}$.
\end{example}

\subsection{Hausdorff's Paradox}\label{ss:hbt}
\JvNeumann\ had noted already in~\cite{vneumann:masses} that
non-amenability of $F_2$ was at the heart of the
Hausdorff-Banach-Tarski paradox. We first show:
\begin{proposition}\label{prop:f2so3}
  The group $\SO_3(\R)$ of rotations of the sphere contains a
  non-abelian free subgroup.
\end{proposition}
\begin{proof}
  There are many classical proofs of this fact. Consider for example
  the matrices
  \[U=\begin{pmatrix}0&1&0\\1&0&0\\0&0&-1\end{pmatrix},\qquad
    V=\begin{pmatrix}1&0&0\\0&-\frac12&\frac{\sqrt3}2\\0&-\frac{\sqrt3}2&-\frac12\end{pmatrix}
  \]
  in $\SO_3(\R)$. They satisfy the relations $U^2=V^3=1$, but no
  other, since in a product
  $W=U^{\varepsilon_1}V^{\pm1}U\cdots V^{\pm1}U^{\varepsilon_2}$ with
  $\varepsilon_1,\varepsilon_2\in\{0,1\}$ and $n$ letters $V^{\pm1}$
  we have
  \[W=\frac1{2^n}\begin{pmatrix}a_{1,1}&a_{1,2}&\sqrt3 a_{1,3}\\
      a_{2,1}&a_{2,2}&\sqrt3 a_{2,3}\\\sqrt3 a_{3,1}&\sqrt3 a_{3,2}&a_{3,3}
    \end{pmatrix}
  \]
  with $a_{i,j}\in\Z$ and $a_{3,3}$ odd, as can be seen from computing
  $2^n W\bmod 2$; so $W\neq1$. Then $\langle [U,V],[U,V^{-1}]\rangle$
  is a free group of rank $2$.

  Here is another proof: $\SO_3(\R)$ is the group of quaternions of norm
  $1$. Let $p$ be a prime $\equiv1\pmod4$, and set
  \[S=\{(a+bi+c j+d k)/\sqrt p\mid
  a\in2\N+1,b,c,d\in2\Z,a^2+b^2+c^2+d^2=p\}.\] It follows from Lagrange's
  Theorem on sums of four squares that $\#S=p+1$, and from the unique
  factorization of quaternions that $S$ generates a free group of rank
  $(p+1)/2$. See~\cite{hurwitz:quaternions} for proofs of these facts.
\end{proof}

\noindent The following paradox follows:
\begin{theorem}[Hausdorff~\cite{hausdorff:pd}]
  There exists a partition of the sphere $S^2$, or of the ball $B^3$,
  in two pieces; and a further partition of each of these into
  respectively two and three pieces, in such a manner that these be
  reassembled, using only isometries of $\R^3$, into two spheres or
  balls respectively.
\end{theorem}
\begin{proof}
  We first show the following: there is a countable subset $D\subset
  S^2$ such that one can decompose $S^2\setminus D=P\sqcup Q$, and
  further decompose $P=P_1\sqcup\dots\sqcup P_m$ and
  $Q=Q_1\sqcup\dots\sqcup Q_n$, so that $S^2\setminus
  D=P_1g_1\sqcup\dots\sqcup P_m g_m= Q_1h_1\sqcup\dots\sqcup Q_n h_n$.

  Indeed, by Proposition~\ref{prop:f2so3}, there is a free subgroup
  $G$ of $\SO_3(\R)$, acting on the sphere. Every non-trivial element
  of $G$ acts as a rotation, and therefore has two fixed points. Let
  $D$ denote the collection of all fixed points of all non-trivial
  elements of $G$; clearly $D$ is countable. The group $G$ acts freely
  on $S^2\setminus D$; let $T$ be a choice of one point per
  orbit\footnote{The Axiom of Choice is required here.}. Let
  $(Y_i,Z_j,g_i,h_j)$ be a paradoxical decomposition of $G$ as in
  Definition~\ref{def:paradoxical}. Set then $P_i=T Y_i g_i^{-1}$ and
  $Q_j=T Z_j h_j^{-1}$ for $i=1,\dots,m$ and $j=1,\dots,n$.

  Keeping the same notation, we now show that $S^2$ can be cut as
  $S^2=U\sqcup V$, such that for an appropriate rotation $\rho$ we
  have $\rho(U)\sqcup V=S^2\setminus D$. Since $D$ is countable, there
  is a direction $\R v\subset\R^3$ that does not intersect $D$. There
  are continuously many rotations $\rho$ with axis $\R v$, and only
  countably many that satisfy $D\cap\rho^n(D)\neq\emptyset$ for some
  $n\neq0$; let $\rho$ be any other rotation. Set
  $U=\bigcup_{n\ge0}\rho^n(D)$ and $V=S^2\setminus U$; then
  $\rho(U)=U\setminus D$ and we are done.

  These paradoxical decompositions can be combined (see
  Corollary~\ref{cor:equidectrans} below for details), proving the
  statement for $S^2$.

  The same argument works for all concentric spheres simultaneously,
  and therefore for $B^3\setminus\{0\}$. It remains to show that $B^3$
  and $B^3\setminus\{0\}$ can respectively be cut into isometric
  pieces. Let $\rho$ be a rotation about $(\frac12,0,\R)$ with angle
  $1$ (in radians), and set $W=\{\rho^n(0)\mid n\in\N\}$. Then
  $\rho(W)=W\setminus\{0\}$, so $B^3=W\sqcup(B^3\setminus W)$ and
  $B^3\setminus\{0\}=\rho(W)\sqcup(B^3\setminus W)$.
\end{proof}

\subsection{Doubling conditions}\label{ss:doubling}
Let us restate paradoxical decompositions in a more sophisticated way.
\begin{definition}
  Let a group $G$ act on a set $X$. A \emph{$G$-wobble} is a map
  $\phi\colon Y\to Z$ for two subsets $Y,Z\subseteq X$, such that there
  exists a finite decomposition $Y=Y_1\sqcup\dots\sqcup Y_n$ and elements
  $g_1,\dots,g_n\in G$ with $\phi(y)=y g_i$ whenever $y\in Y_i$.

  We define a preorder\footnote{I.e.\ a transitive, reflexive
    relation.} on subsets of $X$ by $Y\precsim Z$ if there exists an
  injective $G$-wobble $Y\to Z$; and an equivalence relation $Y\sim Z$
  if there exists a bijective $G$-wobble $Y\to Z$; in that case, we
  say that $Y$ and $Z$ are \emph{equidecomposable}.
\end{definition}
Using that terminology, the $G$-set $X$ is paradoxical if one may
decompose $X=Y\sqcup Z$ with $Y\sim X\sim Z$.

\begin{lemma}\label{lem:equidecset}
  The map $\phi\colon Y\to Z$ is a $G$-wobble if and only if there exists a
  finite subset $S\Subset G$ such that $\phi(y)\in y S$ for all $y\in
  Y$.
\end{lemma}
\begin{proof}
  If $\phi$ is a $G$-wobble, set $S=\{g_1,\dots,g_n\}$, and note
  $\phi(y)\in y S$ for all $y\in Y$.

  Conversely, if $\phi(y)\in y S$ for all $y\in Y$, write
  $S=\{g_1,\dots,g_n\}$, and set
  \[Y_n=\{y\in Y\mid \phi(y)=y g_n\text{ and }\phi(y)\neq y g_m\text{
    for all }m<n\}.\qedhere\]
\end{proof}

\begin{corollary}
  The composition of $G$-wobbles is again a $G$-wobble, and the
  inverse of a bijective $G$-wobble is also a $G$-wobble.\qed
\end{corollary}

It follows that the set of invertible $G$-wobbles is actually a
group. If the space $X$ is assumed compact and the pieces in the
decomposition are open, then this group is known as the ``topological
full group'' of $G$, see~\S\ref{ss:topfull}.

\begin{corollary}\label{cor:equidectrans}
  The relation $\precsim$ is a preorder, and $\sim$ is an
  equivalence relation.
\end{corollary}
\begin{proof}
  Consider injective $G$-wobbles $\phi\colon Y\to Z$ and
  $\psi\colon W\to Y$. By Lemma~\ref{lem:equidecset}, there are
  $S,T\Subset G$ such that $\phi(y)\in y S$ and $\phi(w)\in w T$ for
  all $y\in Y,w\in W$. Then $\phi\psi(w)\in w T S$ for all $w\in W$,
  so $\phi\psi\colon W\to Z$ is an injective $G$-wobble, again by
  Lemma~\ref{lem:equidecset}.
\end{proof}

\begin{theorem}[Cantor-Schr\"oder-Bernstein~\cite{cantor:mitteilungen}]\label{thm:cantorbernstein}
  Let $Y,Z$ be sets. If there exists an injection $\alpha\colon Y\to Z$ and
  an injection $\beta\colon Z\to Y$, then there exists a bijection
  $\gamma\colon Y\to Z$.

  Furthermore, $\gamma$ may be chosen so that
  $\gamma(y)\in\{\alpha(y),\beta^{-1}(y)\}$ for all $y\in Y$.
\end{theorem}
\begin{proof}
  Let $\alpha\colon Y\to Z$ and $\beta\colon Z\to Y$ be injective maps. Set
  $Y_0=Y$ and $Z_0=Z$; and, for $n\ge 1$, set $Y_n=\beta(Z_{n-1})$ and
  $Z_n=\alpha(Y_{n-1})$. Partition $Y$ as follows:
  \[U=\bigsqcup_{n\in\N}Y_{2n}\setminus Y_{2n+1},\qquad
  V=\bigsqcup_{n\in\N}Y_{2n+1}\setminus Y_{2n+2},\qquad
  W=\bigcap_{n\in\N}Y_n.\] Define then $\gamma\colon Y\to Z$ as follows:
  \[\gamma(y)=\begin{cases}
    \alpha(y) & \text{ if }y\in U;\\
    \beta^{-1}(y) & \text{ if }y\in V\cup W.
  \end{cases}\] Therefore $\gamma$ sends $Y_{2n}\setminus Y_{2n-1}$ to
  $Z_{2n+1}\setminus Z_{2n}$ and $Y_{2n+1}\setminus Y_{2n+2}$ to
  $Z_{2n}\setminus Z_{2n-1}$; while sending $\bigcap Y_n$ to $\bigcap
  Z_n$. It follows that $\gamma$ is a bijection.
\end{proof}

\begin{corollary}
  If $Y\precsim Z$ and $Z\precsim Y$, then $Y\sim Z$.
\end{corollary}
\begin{proof}
  Consider injective $G$-wobbles $\alpha\colon Y\to Z$ and
  $\beta\colon Z\to Y$. By Lemma~\ref{lem:equidecset}, there are
  finite sets $S,T\Subset G$ such that $\alpha(y)\in y S$ and
  $\beta(z)\in z T$ for all $y\in Y,z\in Z$. Let $\gamma\colon Y\to Z$
  be the bijection given by Theorem~\ref{thm:cantorbernstein}, with
  $\gamma(y)\in y(S\cup T^{-1})$. Then $\gamma$ is a bijective
  $G$-wobble, again by Lemma~\ref{lem:equidecset}.
\end{proof}

\noindent We also need a little more terminology, coming from graph
theory and following Definition~\ref{defn:graphs}:
\begin{definition}
  A digraph $(V,E)$ is \emph{bipartite} if there is a decomposition
  $V=V^+\sqcup V^-$ such that $e^+\in V^+$ and $e^-\in V^-$ for every
  edge.

  If $V^+$ and $V^-$ are $G$-sets and are identified, the graph
  $(V,E)$ is \emph{bounded} if there exists a finite subset
  $S\Subset G$ with $e^+\in e^- S$ for all $e\in E$.

  An \emph{$m:n$ matching} in $(V,E)$ is a subgraph $(V,\mathcal M)$
  with $\mathcal M\subset E$, such that for each $v\in V^+$ there are
  precisely $n$ edges $e\in\mathcal M$ with $e^+=v$, and for each
  $v\in V^-$ there are precisely $m$ edges $e\in\mathcal M$ with
  $e^-=v$. We define similarly $m:(\le n)$ and $m:(\ge n)$ matchings.

  If $X$ is a $G$-set, a \emph{bounded matching on $X$} is a matching
  in a bounded graph with vertex set $X\sqcup X$.
\end{definition}
In particular, a $1:1$ matching is nothing but a bijection $V^-\to
V^+$; and a bounded $1:1$ matching is a bijective $G$-wobble. A
$1:(\le1)$ matching is an injective map, and a $1:(\ge0)$ matching is
just a map.

\begin{theorem}[Hall~\cite{hall:subsets}-Hall-Rado~\cite{rado:transfinite}]\label{thm:hallrado}
  Let $V,W$ be sets, and for each $v\in V$, let $E_v\subset W$ be a
  finite set. Assume that, for every finite subset $F\Subset V$,
  \begin{equation}\label{eq:hallrado}
    \text{the set }E_F\coloneqq \bigcup_{v\in F}E_v\text{ contains at least $\#F$ elements}.
  \end{equation}
  Then there exists an injection $e\colon V\to W$ with $e(v)\in E_v$ for all
  $v\in V$.
\end{theorem}
\begin{proof}
  Assume first that $(E_v)$ satisfies~\eqref{eq:hallrado}, and that
  $\#E_v\ge2$ for some $v\in V$. We show that we may replace $E_v$ by
  $E_v\setminus\{w\}$ for some $w\in E_v$ and still
  satisfy~\eqref{eq:hallrado}.

  Indeed, consider $w_0\neq w_1\in E_v$, and assume that neither $w_0$
  nor $w_1$ may be removed from $E_v$. Then there are
  $F_0,F_1\Subset V$ and
  $N_i=E_{F_i}\cup(E_v\setminus\{w_i\})\subseteq W$ for $i=0,1$ such
  that $\#N_i<\#(F_i\cup\{v\})$; i.e.\ $\#N_i\le\#F_i$. Then
  \begin{align*}
    \#F_0+\#F_1 &\ge\#N_0+\#N_1=\#(N_0\cup N_1)+\#(N_0\cap N_1)\\
    &\ge\#(E_{F_0\cup F_1}\cup E_v)+\#(E_{F_0\cap F_1})\\
    &\ge\#(F_0\cup F_1)+1+\#(F_0\cap F_1)=\#F_0+\#F_1+1,
  \end{align*}
  a contradiction. Then, inductively, we may suppose $\#E_v=1$ for any
  given $v\in V$.

  If $V$ is finite, we are done by repeatedly replacing each $E_v$ by
  a singleton; the injection is $v\mapsto w$ for the unique $w\in E_v$.

  If $V$ is countable, we may write $V=\{v_1,v_2,\dots\}$ and define
  recursively $E_v^0=E_v$ for all $v\in V$, and, for $i,j>0$,
  \[E_{v_i}^j=\begin{cases}
    E_{v_i}^{j-1} & \text{ if }j\neq i,\\
    \text{the singleton coming from the above operation} & \text{ if }j=i;
  \end{cases}\] then the required injection is $v_i\mapsto w$ for the
  unique $w\in E_{v_i}^i$.

  For general $V$, we need the help of an axiom. Order all systems
  $(E'_v)$ satisfying~\eqref{eq:hallrado} by $(E'_v)\le(E''_v)$ if
  $E'_v\subseteq E''_v$ for all $v\in V$. By Zorn's lemma,
  $\{(E'_v)\le(E_v)\}$ admits a minimal element $(E'_v)$. If
  $\#E'_v\ge2$ for some $v\in V$, then by the above it could be made
  strictly smaller; therefore $\#E'_v=1$ for all $v\in V$ and we again
  have an injection $V\to W$.
\end{proof}
Note that, if one drops the assumption that $E_v$ is finite for all
$v$, then there are counterexamples to the theorem, e.g.\ $V=W=\N$,
$E_0=\N$ and $E_{n+1}=\{n\}$ for all $n\in\N$. For more details
see~\cite{mirsky:transversal}.

\begin{corollary}
  Let $(V,E)$ be a bipartite graph, and assume that for all
  $\varepsilon\in\{\pm1\}$ and all finite subsets $F\subset
  V^\varepsilon$ the set
  \[\{v\in V^{-\varepsilon}\mid e^{-\varepsilon}=v,e^\varepsilon\in
    F\text{ for some }e\in E\}
  \]
  is finite and contains at least $\#F$ elements. Then there exists a
  $1:1$ matching in $(V,E)$.
\end{corollary}
\begin{proof}
  By Theorem~\ref{thm:hallrado}, there exists a subgraph of $(V,E)$
  defining an injection $V^-\to V^+$; and symmetrically there exists a
  subgraph of $(V,E)$ defining an injection $V^+\to V^-$. Applying
  Theorem~\ref{thm:cantorbernstein}, there exists a subgraph of
  $(V,E)$ defining a bijection $V^-\to V^+$.
\end{proof}

\noindent We are ready to prove the equivalence of our new notions:
\begin{theorem}\label{thm:pd}
  Let $X$ be a $G$-set. The following are equivalent:
  \begin{enumerate}
  \item $X$ is paradoxical;
  \item $X$ is not amenable;
  \item For any $m>n>0$ there exists a bounded $m:n$ matching on $X$;
  \item There exists a $G$-wobble $\phi\colon X\to X$ with
    $\#\phi^{-1}\{x\}=2$ for all $x\in X$;
  \item There exists a $G$-wobble $\phi\colon X\to X$ with
    $\#\phi^{-1}\{x\}\ge2$ for all $x\in X$.
  \end{enumerate}
\end{theorem}
\begin{proof}
  $(1)\Rightarrow(2)$ Assume that there exists a $G$-invariant mean
  $\mu\in\mathscr M(X)$. Then
  \[1=\mu(X)=\sum_{i=1}^m\mu(Y_i g_i)+\sum_{j=1}^n\mu(Z_j h_j)=
    \sum_{i=1}^m\mu(Y_i)+\sum_{j=1}^n\mu(Z_j)= \mu(X)+\mu(X)=2,\] a
  contradiction.

  $(2)\Rightarrow(3)$ Assume that $X$ does not satisfy F\o lner's
  condition, so there are $S\Subset G$ and $\epsilon>0$ with
  $\#(F S)\ge(1+\epsilon)\#F$. Given $m>n>0$, let $k\in\N$ be such
  that $(1+\epsilon)^k\ge m/n$.

  Construct now the following bipartite graph: its vertex set is
  $V=X\times\{1,\dots,m\}\sqcup X\times\{\overline1,\dots,\overline
  n\}$. There is an edge from $(x,i)$ to $(x g,\overline j)$ for all
  $g\in S^k$ and all $i\in\{1,\dots,m\},j\in\{1,\dots,n\}$. Consider
  first a finite subset $F\Subset V^-$, and project it to
  $F'\subseteq X$. Then
  \[\#(F' S^k\times\{\overline1,\dots,\overline n\})=n\#(F' S^k)\ge
    m\#F'\ge\#F,
  \]
  and all these vertices are reached from $F$ by edges in
  $(V,E)$. Conversely, fix $g\in S^k$, and consider a finite subset
  $F\Subset V^+$. Because $m>n$, every $(x,\overline i)\in F$ is
  connected by an edge to $(x g^{-1},i)\in V^-$. Therefore, every
  finite $F\subset V^{\pm}$ has at least $\#F$ neighbours in
  $V^{\mp}$.

  We now invoke the Hall-Rado theorem~\ref{thm:hallrado} to obtain a
  $1:1$ matching $(V,\mathcal M)$; which we project to a bounded $m:n$
  matching $(X\sqcup X,\mathcal M)$ by setting $e^\pm=x$ whenever we
  had $e^\pm=(x,*)$ in $(V,\mathcal M)$.

  $(3)\Rightarrow(4)$ Let $\mathcal M$ be a bounded $2:1$ matching on
  $X$. Given $x\in X$, there is a unique $e\in\mathcal M$ with
  $e^-=x$; set $\phi(x)=e^+$. This defines a $G$-wobble
  $\phi\colon X\to X$ with $\#\phi^{-1}(y)=2$.

  $(4)\Rightarrow(5)$ is obvious.

  $(4)\Rightarrow(1)$ For each $x\in X$ choose $y_x\in X$ with
  $\phi(y_x)=y$; this is possible using the Axiom of Choice. Set
  $Y=\{y_x\mid x\in X\}$, and $Z=X\setminus Y$. We have $X=Y\sqcup Z$,
  and $\phi$ restricts to bijective $G$-wobbles $Y\to X$ and $Z\to X$,
  so $Y\sim X\sim Z$.

  $(5)\Rightarrow(2)$ Let $S\Subset G$ satisfy $\phi(x)S\ni x$ for all
  $x\in X$. Then, for any finite $F\Subset X$, we have
  $\phi^{-1}(F)\subseteq F S$ so $\#(F S)\ge 2\#F$.
\end{proof}

If a group $G$ contains a non-abelian free subgroup, then $G$ is not
amenable. The converse is not true, as we shall see
in~\S\ref{ss:fgfg}. However, the following weaker form of the converse
holds:
\begin{theorem}[see~\cite{whyte:amenability}]\label{thm:whyte}
  Let $X$ be a $G$-set. The following are equivalent:
  \begin{enumerate}
  \item $X$ is not amenable;
  \item There is a free action of the free group $F_2$ on $X$ by
    bijective $G$-wobbles;
  \item There is a free action of a non-amenable group on $X$ by
    bijective $G$-wobbles.
  \end{enumerate}
\end{theorem}
\begin{proof}
  $(1)\Rightarrow(2)$ Assume that $X$ is non-amenable, so by
  Theorem~\ref{thm:pd} there exists a $G$-wobble $\phi\colon X\to X$
  with $\#\phi^{-1}\{x\}=2$ for all $x\in X$. Let $S\Subset G$
  satisfy $\phi(x)\in x S$ for all $x\in X$.

  View $X$ as a directed graph $T$, with an edge from $x$ to $\phi(x)$
  for all $x$; and let $U$ be the corresponding undirected
  graph. These graphs are $3$-regular: in $T$ every vertex has one
  outgoing and two incoming edges. Assume that there is a cycle in
  $U$. This cycle is necessarily oriented, for otherwise there would
  be two outgoing edges at a vertex. Furthermore, there cannot be two
  cycles in the same connected component of $U$: if there were two
  such cycles, consider a minimal path $p$ joining them. At least one
  of $p$'s extremities would be oriented away from its end, and again
  there would be two outgoing edges at a vertex.

  It follows that all connected components of $U$ are either
  $3$-regular trees, or cycles with $3$-regular trees attached to
  them. Remove an edge from each cycle, creating in this manner either
  two vertices of degree $2$ or one of degree $1$. In all cases, at
  each vertex $v$ of degree $<3$ choose a ray $\rho_v$ going to
  infinity consisting entirely of degree-$3$ vertices, and shift the
  edges attached to $\rho_v$ towards $v$ along $\rho_v$ so as to
  increase the degree of $v$. In this manner, we obtain a $3$-regular
  forest $U$ with vertex set $X$, with the following property: there
  exists a finite subset $S'\Subset G$ such that every edge of $U$,
  joining say $x$ to $y$, satisfies $y\in x S'$. In fact,
  $S'=S\cup S^{-1}S^2$ will do.

  Now label all edges of $U$ with $\{a,b,c\}$ in such a manner that
  at every vertex all three colours appear exactly once on the
  incident edges. This is easy to do: on each connected component
  label arbitrarily an edge; then at each extremity label the two
  other incident edges by the two remaining symbols, and continue.

  In this manner, every connected component of $U$ becomes the Cayley graph
  of $H\coloneqq\langle a,b,c\mid a^2,b^2,c^2\rangle$.  In effect, we have
  defined an action of $H$ on $X$ by $G$-wobbles: the image of $x$ under
  $a,b,c$ respectively is the other extremity of the edge starting at $x$
  and labeled $a,b,c$ respectively. The group $H$ contains a free subgroup
  of rank $2$, namely $\langle ab,b c\rangle$.

  $(2)\Rightarrow(3)$ is obvious.

  $(3)\Rightarrow(1)$ Assume that $X$ admits a free action of a
  non-amenable group $H$ by bijective $G$-wobbles; without loss of
  generality, $H$ is finitely generated, say by a set $T$. Since $H$
  is not amenable and acts freely, it does not satisfy F\o lner's
  condition by Proposition~\ref{prop:freeaction}, so there exists
  $\delta>0$ such that $\#(F T)\ge\delta\#F$ for all
  $F\Subset X$.

  Let $S\Subset G$ satisfy $x T\subset x S$ for all $x\in X$.
  In particular, $\#(F S)\ge\delta\#F$, so $X$ does not satisfy F\o
  lner's condition.
\end{proof}
Note that the proof becomes trivial in case $X=G\looparrowleft G$ and
$G$ contains a non-abelian free subgroup; indeed the action of $G$
itself is by $G$-wobbles.

It is possible to modify slightly this construction to make $F_2$ act
transitively by $G$-wobbles, see~\cite{seward:burnside}.

%%%%%%%%%%%%%%%%%%%%%%%%%%%%%%%%%%%%%%%%%%%%%%%%%%%%%%%%%%%%%%%%
\newpage\section{Convex sets and fixed points}\label{sec:convex}
We consider an abstract version of convex sets, introduced by Stone
in~\cite{stone:barycentric} as sets with barycentric coordinates:
\begin{definition}
  A \emph{convex space} is a set $K$ with an operation $[0,1]\times K\times
  K\to K$ of taking convex combinations, written $(t,x,y)\mapsto t(x,y)$,
  satisfying the axioms
  \begin{align*}
    0(x,y)&=x=t(x,x),\\
    t(x,y)&=(1-t)(y,x),\qquad\text{ for all $x,y,z\in K$ and $0\le u\le t\le 1$}.\\
    \textstyle t(x,\frac u t(y,z))&=\frac{t-u}{1-u}(x,u(y,z)),
  \end{align*}
  It is called \emph{cancellative} if it furthermore satisfies the axiom
  \[t(x,y)=t(x,z), t>0\Rightarrow y=z.
  \]
  An \emph{affine map} is a map $f\colon K\to L$ between convex spaces
  satisfying $t(f(x),f(y))=t(x,y)$ for all $t\in[0,1]$ and all
  $x,y\in K$.
\end{definition}

Usual convex subsets of vector spaces are typical examples; if $K\subseteq
V$ is convex, then $t(x,y)\coloneqq (1-t)x+t y$ gives $K$ the structure of
a convex space. There are other examples: for any set $X$ one may take
$K=\mathfrak P(X)$ with $t(x,y)=x\cup y$ whenever $t\in(0,1)$.

As another example, trees (and more generally $\R$-trees: geodesic metric
spaces in which every triangle is isometric to a tripod) are convex spaces:
for $x,y$ in a tree, there is a unique geodesic from $x$ to $y$, and
$t(x,y)$ is defined as the point at distance $t\,d(x,y)$ from $x$ along
this geodesic. Unless the tree is a line segment, this convex space is not
cancellative.

The set of closed balls in an ultrametric space\footnote{Namely, a
  metric space in which the ultratriangle inequality
  $d(x,z)\le\max\{d(x,y),d(y,z)\}$ holds.}, with Hausdorff distance, is
also an example of a convex space; it is actually isomorphic to the
convex space associated with an $\R$-tree, see~\cite{hughes:trees}.

It turns out~\cite{stone:barycentric}*{Theorem~2} that those convex spaces
that are embeddable in real vector spaces as convex subsets are precisely
the cancellative ones.

A \emph{topological convex space} is a convex space $K$ with the
structure of a topological space, such that the structure map
$[0,1]\times K\times K\to K$ is continuous. A \emph{convex $G$-space}
is a convex space on which a group $G$ acts by affine maps. The
\emph{convex hull} of a subset $X\subseteq K$ of a convex space is the
intersection $\widehat X$ of all convex subspaces of $K$ containing
$X$.

\begin{exercise}[*]
  Convex spaces form a variety. Prove that the free convex space on
  $n+1$ generators is isomorphic to the standard $n$-simplex
  $\{(x_0,\dots,x_n)\in\R^{n+1}\mid x_i\ge0,\sum x_i=1\}$, and also to
  the convex hull of the basis vectors in $\R^{n+1}$. In particular,
  it is cancellative.
\end{exercise}

\begin{definition}
  Let $X,Y$ be $G$-sets. We say that $Y$ is \emph{$X$-markable} if
  there exists an equivariant $G$-map $X\to Y$.
\end{definition}

\begin{theorem}\label{thm:fixedpoint}
  Let $X$ be a $G$-set. The following are equivalent:
  \begin{enumerate}
  \item $X$ is amenable;
  \item Every compact $X$-markable convex space admits a fixed point;
  \item Every compact $X$-markable convex subset of a locally compact
    topological vector space admits a fixed point.
  \end{enumerate}
\end{theorem}
\begin{proof}
  $(1)\Rightarrow(2)$ By Lemma~\ref{lem:folnerlim}, there exists a
  net $(F_n)_{n\in\mathscr N}$ of F\o lner sets in $X$. Let $K$ be a
  compact $X$-markable convex space, and let $\pi\colon X\to K$ be a
  $G$-equivariant map. For each $n\in\mathscr N$, set
  \[k_n\coloneqq\sum_{x\in F_n}\frac1{\#F_n}\pi(x)\in K.
  \]
  Then $(k_n)$ is a net in $K$, so by compactness admits a cluster
  point, say $k$. The $k_n g$ have the same limit, so $k$ is a fixed
  point.

  $(2)\Rightarrow(3)$ is obvious.

  $(3)\Rightarrow(1)$ Take $K=\mathscr M(X)$; it is compact by
  Lemma~\ref{lem:mcompact}, $X$-marked by $\delta$, convex by
  Lemma~\ref{lem:mconvex}, and contained in the topological vector
  space $\ell^1(X)^*$ which is locally compact by the Banach-Alaoglu
  theorem~\cite{rudin:fa}*{Theorem~3.15}. A fixed point is an
  invariant mean on $X$.
\end{proof}

In particular, a group $G$ is amenable if and only if every compact
non-empty convex $G$-space admits a fixed point.  We may thus show
that amenability of $G$-sets is stable under amenable extensions:
\begin{proposition}
  Let $X$ be a $G$-set, and let $N\triangleleft G$ be a normal
  subgroup with $G/N$ amenable. Then $X\looparrowleft G$ is amenable
  if and only if $X\looparrowleft N$ is amenable.
\end{proposition}
\begin{proof}
  Let $K$ be an $X$-markable convex compact space.  The ``if''
  direction is obvious, since every $G$-fixed point in $K$ is
  $N$-fixed. Conversely, if $K^N\neq\emptyset$, then $K^N$ is a
  non-empty convex compact space on which $G/N$ acts, and has a fixed
  point because $G/N$ is amenable. Clearly $(K^N)^{G/N}=K^G$, so
  $X\looparrowleft G$ is amenable.
\end{proof}

\subsection{Measures}\label{ss:measures}
Consider a topological space $X$. We recall that $\mathcal C(X)$ denotes the
space of continuous functions $X\to\R$, and that probability measures
on $X$ are identified with functionals $\lambda\in \mathcal C(X)^*$ such that
$\lambda(\mathbb 1)=1$ and $\lambda(\phi)\ge0$ if $\phi\ge0$. One
sometimes writes $\lambda(\phi)=\int\phi d\lambda$.

An important property of measures on subsets of vector spaces is that
they have \emph{barycentres}:
\begin{lemma}\label{lem:barycentre}
  Let $K$ be a non-empty convex compact subset of a locally compact
  topological vector space, and let $\mu\in \mathcal C(K)^*$ be a probability
  measure. Then there exists a unique $b\in K$ such that
  $\mu(\phi)=\phi(b)$ for all affine maps $\phi\in \mathcal C(K)$. We write
  $b=\int t d\mu(t)$ and call it the \emph{barycentre} of $\mu$.
\end{lemma}
\begin{proof}
  For any affine function $\phi\colon K\to\R$, set
  \[K_\phi\coloneqq\{x\in K\mid \mu(\phi)=\phi(x)\}.\]
  It is clear that $K_\phi$ is convex and compact. Furthermore, it is
  non-empty; more generally, we will show that $K_{\phi_1}\cap\dots\cap
  K_{\phi_n}\neq\emptyset$ for all affine
  $\phi_1,\dots,\phi_n\colon K\to\R$.

  Write $\phi=(\phi_1,\dots,\phi_n)\colon K\to\R^n$. Define
  $L=\{\phi(x)\colon x\in K\}$; this is a convex compact in
  $\R^n$. Define $p\in\R^n$ by $p_i=\mu(phi_i)=\int_K\phi_i d\mu$. We
  claim that $p$ belongs to $L$; once this is shown, every $x\in K$
  with $\phi(x)=p$ belongs to $K_{\phi_1}\cap\dots\cap K_{\phi_n}$, so
  the intersection is not empty.

  We now show that, for any $q\not\in L$, we have $p\neq q$. There
  exists then a hyperplane that separates $q$ from $L$, namely the
  nullspace of any affine map $\tau\colon\R^n\to\R$ with $\tau(q)<0$
  and $\tau(L)>0$. In particular $\tau(\phi(x))>0$ for all $x\in K$,
  so by integrating $\tau(p)>0$, and therefore $p\neq q$.

  Set now $B=\bigcap_{\phi\text{ affine}}K_\phi$. It is non-empty
  by compactness of $K$, because any finite sub-intersection is
  non-empty.

  Affine functions separate points\footnote{Note that we use here the
    Hahn-Banach theorem, which requires certain logical axioms.} in
  $K$, so $B$ contains a single point $b$.
\end{proof}

\begin{theorem}\label{thm:fixedmeasure}
  Let $X$ be a $G$-set. The following are equivalent:
  \begin{enumerate}
  \item $X$ is amenable;
  \item Every compact $X$-markable set admits an invariant probability
    measure.
  \end{enumerate}
\end{theorem}
\begin{proof}
  $(1)\Rightarrow(2)$ Let $K$ be a compact $G$-set and let
  $\pi\colon X\to K$ be a $G$-equivariant map. Let
  $m\in\ell^\infty(X)^*$ be a $G$-invariant positive functional; then
  $m\circ\pi^*\colon\ell^\infty(K)\to\ell^\infty(X)\to\R$ is a
  $G$-invariant, positive functional on $K$, and its restriction to
  $\mathcal C(K)$ is an invariant probability measure on $K$.

  $(2)\Rightarrow(1)$ Let $K$ be a compact $X$-markable convex subset
  of a locally compact topological vector space, and let $\lambda$ be
  an invariant probability measure on $K$. Then $\lambda$'s
  barycentre, which exists by Lemma~\ref{lem:barycentre}, is a fixed
  point in $K$, so $X$ is amenable by
  Theorem~\ref{thm:fixedpoint}$(3)\Rightarrow(1)$.
\end{proof}

\begin{exercise}[*]
  Reprove that the free group $F_2$ is not amenable as follows: write
  $F_2=\langle a,b\mid\rangle$, and make it act on the circle
  $X=[0,1]/(0\sim1)$ by $x a=x^2$ and $x b=(x+1/2)\mod 1$ for all
  $x\in[0,1]$. Show that the only $a$-invariant measure on $X$ is
  $\delta_0$, and that it is not $b$-invariant.
\end{exercise}

We proved in Corollary~\ref{cor:abelian} that abelian groups are
amenable. We may reprove it as follows:
\begin{proposition}[Kakutani~\cite{kakutani:fp}-Markov~\cite{markov:abeliens}]\label{prop:abelian}
  Let $G$ be an abelian group. Then $G$ is amenable.
\end{proposition}
\begin{proof}
  Let $G$ act affinely on a convex compact $K$. For every $g\in G$ and
  every $n\ge1$ define a continuous transformation $A_{n,g}\colon K\to K$ by
  \[A_{n,g}(x) = \frac1n\sum_{i=0}^{n-1}x g^i.\] Let $\mathcal S$
  denote the monoid generated by $\{A_{n,g}\mid g\in G,n\ge
  1\}$. We show that $\bigcap_{s\in\mathcal S}s(K)$ is not
  empty. Since $K$ is compact, it suffices to show that every finite
  intersection $s_1(K)\cap\dots\cap s_k(K)$ is non empty. To that end, set
  $t=s_1\dots s_k$. We have
  \[s_i(K)\subseteq s_is_1\dots\widehat{s_i}\dots s_k(K)=t(K),\]
  because $\mathcal S$ is commutative. Therefore $s_1(K)\cap\dots
  s_k(K)$ contains $t(K)$ so is not empty.

  Pick now $x\in\bigcap_{s\in\mathcal S}s(K)$. To show that $x$ is
  $G$-fixed, choose any affine function $\phi\colon K\to\R$, and any $g\in
  G$. For all $n$, write $x=A_{n,g}(y)$, and compute
  \[\phi(x)-\phi(x g)=\frac1n\big(\phi(y)-\phi(y g^n)\big)\le\frac2n\|\phi\|_\infty;\]
  Since $\phi,g$ are fixed and $n$ is arbitrary, we have
  $\phi(x)=\phi(x g)$ for all affine $\phi\colon K\to\R$, from which
  $x=x g$.
\end{proof}

Furstenberg studied in~\cite{furstenberg:boundary} a condition at the
exact opposite of amenability: a \emph{boundary} for a group $G$ is a
compact $G$-space $K$ which is minimal and such that every probability
measure on $K$ admits point measures in the closure of its
$G$-orbit. By Theorem~\ref{thm:fixedmeasure}, if $G$ is amenable then
its only boundary is the point. See~\S\ref{ss:boundary} for more details.

\subsection{Amenability of equivalence relations}\label{ss:equivrel}
In the previous section, we gave conditions on a compact $G$-set to
admit an invariant measure. Here, we assume that we are given a
measure space on which a group acts measurably.

In the abstract setting, we are given a set $X$, a $\sigma$-algebra
$\mathfrak M$ of subsets of $X$, and a map
$\lambda\colon\mathfrak M\to\R$.

To simplify the presentation, and focus on the interesting cases, we
assume that $(X,\lambda)$ is \emph{$\sigma$-finite}, namely $X$ is the
countable union of subsets of finite measure. In this case, it costs
nothing to assume that $\lambda$ is a \emph{probability} measure,
namely $\lambda(X)=1$. (Indeed, if $X=\bigsqcup_{n\in\N}X_n$ with
$\lambda(X_n)<\infty$, define a new measure
$\lambda'(A)=\sum_{n\in\N}2^{-n}\lambda(A\cap X_n)/\lambda(X_n)$.) We
will even assume that $(X,\lambda)$ is a \emph{standard probability
  space}~\cite{vneumann:measures}, such as $([0,1],\text{Lebesgue})$
or $(\{0,1\}^\N,\text{Bernoulli})$; these spaces are isomorphic as
measure spaces.

Let $G$ be a group, and assume that $G$ acts measurably on
$(X,\lambda)$. Recall that this means that $G$ acts on $\lambda$-null
sets: if $A\subset X$ satisfies $\lambda(A)=0$, then
$(\lambda g)(A)=\lambda(A g^{-1})=0$ for all $g\in G$. In other words,
the measures $\lambda$ and $\lambda g$ are absolutely continuous with
respect to each other, and the Radon-Nikodym
theorem~\cite{nikodym:radon} implies that there is an essentially
unique measurable function
$\partial(\lambda g)/\partial\lambda\colon X\to\R$ satisfying
\[\int_X f(x g)d\lambda(x)=\int_X f(x)\frac{\partial(\lambda g)}{\partial\lambda}d\lambda(x)\text{ for all }f\in L^1(X,\lambda).\]
If $(X,\lambda)=([0,1],\text{Lebesgue})$ and $g\colon X\to X$ is
differentiable, then $\partial(\lambda g)/\partial\lambda=d g/d x$, the
usual derivative. The chain rule gives a ``cocycle'' identity
\[\frac{\partial(\lambda g h)}{\partial\lambda}=\frac{\partial(\lambda g)}{\partial\lambda}\cdot\bigg(\frac{\partial(\lambda h)}{\partial\lambda}g\bigg).\]

In the extreme case (which is not the typical case we are interested
in), the measure $\lambda$ might be $G$-invariant:
$\lambda(A)=\lambda(A g)$ for all $A\subseteq X,g\in G$, and then the
Radon-Nikodym derivative is constant $\equiv1$.

To simplify the presentation and concentrate on the useful cases, we
also restrict ourselves to a countable group $G$.  Recall that an
action is \emph{essentially free} if $\lambda$-almost every point has
a trivial stabilizer, namely $\lambda(\{x\in X\mid
G_x\neq1\})=0$. More generally, everything is considered ``up to
measure $0$'': a group action, isomorphisms between measured actions
etc.\ only need to be defined on sets of full measure.
\newpage
It will be useful to forget much about the group action, and only
remember its orbits. This is captured in the following definitions:
\begin{definition}
  A \emph{countable (respectively finite) measurable equivalence
    relation} on $(X,\lambda)$ is an equivalence relation
  $R\subseteq X\times X$ that is measurable qua subset of $X\times X$,
  such that for every $x\in X$ the equivalence class
  $x R\coloneqq\{y\in X\mid (x,y)\in R\}$ is countable (respectively
  finite) and such that for every measurable $A\subseteq X$ with
  $\lambda(A)=0$ one has $\lambda(A R)=0$.

  The set $R$ itself is treated as a measure space, with the counting
  measure on each equivalence class: $d\mu(x,y)=d\lambda(x)$.
\end{definition}

A fundamental example is given by a measurable action of a countable
group $G$, as above: one sets
$R_G=\{(x,y)\in X^2\mid \exists g\in G\text{ with }x g =y\}$.

\begin{definition}
  A countable measurable equivalence relation $R$ on $(X,\lambda)$ is
  \emph{amenable} if there is a measurable invariant mean
  $m\colon X\to\mathscr M(R)$, written $x\mapsto m_x$, with
  $m_x\in\mathscr M(x R)$ for all $x\in X$. Here ``measurable'' means
  that for every $F\in L^\infty(X,\lambda)$ the map $x\mapsto m_x(F)$
  is measurable, and ``invariant'' means that $m_x=m_y$ almost
  whenever $(x,y)\in R$.
\end{definition}

By~\cite{connes-feldman-weiss:amenable}, a countable measurable
equivalence $R$ relation is amenable if and only if it is
\emph{hyperfinite}: $R$ is the increasing union of countably many
finite measurable equivalence relations, if and only if it is given by
an action of $\Z$.

The following lemma rephrases amenability of equivalence relations as
an analogue of Reiter's criterion; we omit the proof which essentially
follows that of Theorem~\ref{thm:folneramen};
see~\cite{kaimanovich:isop}:
\begin{lemma}
  The equivalence relation $R$ on $(X,\lambda)$ is amenable if and
  only if there exists a system $(\phi_{x,n})_{x\in X,n\in\N}$ of
  measures, with $\phi_{x,n}\in\ell^1(x R)$, which is
  \begin{description}
  \item[---] \emph{measurable}: for all $n\in\N$ the function
    $(x,y)\mapsto\phi_{x,n}(y)$ is measurable on $R$,
  \item[---] \emph{asymptotically invariant}:
    $\|\phi_{x,n}-\phi_{y,n}\|\to0$ for almost all $(x,y)\in R$.\qed
  \end{description}
\end{lemma}

\begin{proposition}
  If $G$ is amenable and acts measurably on $(X,\lambda)$, then $G$
  generates an amenable equivalence relation.
\end{proposition}
\begin{proof}
  Since $G$ is amenable, there exists a sequence of almost invariant
  measures $\phi_n\in\ell^1(G)$, in the sense that
  $\|\phi_n-\phi_n g\|\to0$ for all $g\in G$.  Let $R_G$ be the
  equivalence relation generated by $G$ on $X$. For $x\in X$, set
  $\phi_{x,n}\coloneqq x\cdot\phi_n$, the push-forward of $\phi_n$
  along the orbit of $x$. Clearly $(\phi_{x,n})_{x\in X,n\in\N}$ is an
  asymptotically invariant system, and it is measurable since for all
  $n\in\N$ the level sets $\{(x,y)\in R\mid \phi_{x,n}(y)>a\}$ are the
  unions of the graphs of finitely many elements of $G$.
\end{proof}
Note that the proposition does not admit a converse: for instance, if
$G$ is a discrete subgroup of a Lie group $L$ and $P\le L$ is
soluble, then the action of $G$ on $P\backslash L$ is
amenable. Indeed the action of $G$ on $L$ is amenable: letting $T$ be
a measurable transversal of $G$ in $L$, choose arbitrarily a
measurable assignment $m\colon T\to\mathscr M(R_G)$ on the
transversal, and extend it to $L$ by translation. The map $m$ may
easily be required to be $P$-invariant, so passes to the quotient
$P\backslash L$.

\begin{proposition}
  If $G$ acts essentially freely by measure-preserving transformations
  on the probability space $(X,\lambda)$, and the generated
  equivalence relation $R$ is amenable, then $G$ is amenable.
\end{proposition}
\begin{proof}
  Given $f\in\ell^\infty(G)$, set
  \[m(f)=\int_X m_x(x g\mapsto f(x))d\lambda(x).\qedhere\]
\end{proof}

It is possible for a non-amenable group to act essentially freely on a
probability space:
\begin{example}
  Let $F_k=\langle x_1,\dots,x_k\mid\rangle$ be a free group of rank
  $k$, and consider its \emph{boundary} $\partial F_k$: it is the
  space of infinite reduced words over the generators of $F_k$,
  \[\partial F_k=\{a_0a_1\dots\in\{x_1^\pm,\dots,x_k^\pm\}^\N\mid a_i a_{i+1}\neq1\text{ for all }i\in\N\}.\]
  The measure is equidistributed on cylinders:
  $\lambda(a_0a_1\dots
  a_n\{x_1^\pm,\dots,x_k^\pm\}^\N)=(2k)^{-1}(2k-1)^{1-n}$.  The action
  of $F_k$ on $\partial F_k$ is by pre-catenation:
  \[(a_0a_1\dots)\cdot x_i=\begin{cases}
      x_i a_0a_1\dots & \text{ if }x_i a_0\neq 1,\\
      a_1\dots & \text{ if }x_i a_0=1.
    \end{cases}
  \]
  
  Then the action of $F_k$ on $\partial F_k$ is essentially free and
  amenable, although $F_k$ is not amenable.
\end{example}
\begin{proof}
  For $1\neq g=a_1\dots a_n\in F_k$, its only fixed points in
  $\partial F_k$ are $g^\infty$ and $g^{-\infty}$; since $F_k$ is
  countable and $\partial F_k$ has the cardinality of the continuum,
  the action of $F_k$ is free almost everywhere in $\partial F_k$.

  For all $x=a_0a_1\dots\in\partial F_k$, define probability measures
  $\mu_{x,n}$ on the orbit of $x$ by
  \[\mu_{x,n}=\frac1n\big(\delta_{x}+\delta_{x a_0}+\cdots+\delta_{x a_0\cdots a_{n-1}}\big).\]
  These measures converge weakly to a mean $m_x$ on the orbit of $x$,
  and clearly $m_x$ and $m_{x g}$ have the same limit, since the sums
  defining $\mu_{x,n}$ and $\mu_{x g,n}$ agree on all but at most
  $|g|$ terms. Therefore, $m\colon X\to R_{F_k}$ is invariant, so
  $R_{F_k}$ is amenable.
\end{proof}

Consider a non-amenable group acting on $(X,\lambda)$. So as to
guarantee that the equivalence relation $R_G$ be non-amenable, we may
relax somewhat the condition that $G$ preserve $\lambda$.  We also
assume that $X$ is a compact topological space on which $G$ acts by
homeomorphisms. In fact, this is not a strong restriction: given a
measurable action of $G$ on $(X,\lambda)$, we may always construct a
compact topological $G$-space $Y$, with a measure $\mu$ on its Borel
subsets, such that $(X,\lambda)$ and $(Y,\mu)$ are isomorphic as
$G$-measure spaces; see~\cite{becker-kechris:polish}*{Theorem~5.2.1}.

We will call the action of $G$ \emph{indiscrete} if for every
$\epsilon>0$ and every neighbourhood $\mathcal U$ of the diagonal in
$X\times X$ there exists $g\neq1\in G$ with
$\{(x,x g)\mid x\in X\}\subseteq\mathcal U$ and
$\partial(\lambda g)/\partial\lambda\in(1-\epsilon,1+\epsilon)$ almost
everywhere.

The measurable action of $G$ on $X$ induces an action of $G$ by
isometries on the Banach space $L^1(X,\lambda)$ of integrable
functions on $X$, by
\[(f g)(x)=\Big(\frac{\partial(\lambda g)}{\partial\lambda}f\Big)(x g^{-1})\text{ for }f\in L^1(X,\lambda).
\]
\begin{lemma}\label{lem:autocont}
  If we give $G$ the topology of uniform convergence in its action on
  $X$, then the action of $G$ on $L^1(X,\lambda)$ is continuous.
\end{lemma}
\begin{proof}
  Consider $f\in L^1(X,\lambda)$; we wish to show $f g\to f$ whenever
  $g\to1$.

  The closure of $G$ in the homeomorphism group of $X$ is
  second-countable locally compact; it therefore admits a Haar measure
  $\eta$. Let $K\subseteq\overline G$ be a compact with $\eta(K)=1$,
  and let $V$ be a compact neighbourhood of $1$ in $\overline
  G$. Since the Haar measure is invariant, we have
  \[\|f g-f\|=\int_K\|f g h-f h\|d\eta.
  \]
  since $f$ is measurable, there is for all $\epsilon>0$ a continuous
  function $f'\colon V K\to\C$ with
  $\int_{V K}\|f h-f' h\|d\eta<\epsilon$, and there is also a
  neighbourhood $W$ of $1$ in $V$ such that
  $\|f' g h- f' h\|<\epsilon$ for all $h\in K,g\in W$. Then
  $\|f g-f\|<3\epsilon$ as soon as $g\in W$ by a standard `$3\delta$'
  argument.
\end{proof}

\begin{proposition}[Monod]
  Let $G$ contain an indiscrete non-abelian free group acting
  essentially freely on a measure space $(X,\lambda)$. Then $G$
  generates a non-amenable equivalence relation.
\end{proposition}
\begin{proof}
  It suffices to prove the claim with $G=\langle a,b\mid\rangle$
  itself free.  Let $A\subset G$ denote those elements whose reduced
  form starts with a non-trivial power of $a$, and define similarly
  $B$ using $b$; so $G=A\sqcup B\sqcup\{1\}$.

  Assume for contradiction that $R_G$ is amenable, and let
  $m\colon X\to\mathscr M(R_G)$ be an invariant mean. Define
  measurable maps $u,v\colon X\to[0,1]$ by
  \[u(x)=m_x(x A),\qquad v(x)=m_x(x B).
  \]
  Then $u+v=1$ almost everywhere, and $0\le\sum_{n\in\Z}u(x b^n)\le 1$
  and $0\le\sum_{n\in\Z}v(x a^n)\le 1$ almost everywhere, because the
  sets $b^n A$ are all disjoint. In particular, if $v(x)>\frac12$ then
  $v(x a^n)<\frac12$ for all $n\neq0$, so if $u(x)<\frac12$ then
  $u(x a^n)>\frac12$ for all $n\neq0$. Define
  \[P=\{x\in X\mid u(x)<\tfrac12\},\qquad Q=\{x\in X\mid u(x)>\tfrac12\}.
  \]
  Denote furthermore by $A'\subset A$ those elements of $G$ that start
  and end with a non-trivial power of $a$, and by $B'\subset B$ those
  elements of $G$ that start and end with a non-trivial power of
  $b$. Then $P A'\subseteq Q$, and $Q B'\subseteq P$.

  Since $G$ is indiscrete, there exist $g_n\in G\setminus\{1\}$ with
  $g_n\to1$ and $\partial(\lambda g_n)/\partial\lambda\to1$
  uniformly. Up to taking a subsequence, we may assume all $g_n$ have
  the same first letter and the same last letter, and have increasing
  lengths. Up to switching the roles of $a$ and $b$, we may assume
  they all start with $a^{\pm1}$. Up to replacing $g_n$ by
  $g_n g_{n-1}g_n^{-1}$, we may assume they all belong to $A'$.

  Since $P$ is measurable, its characteristic function $\mathbb1_P$ is
  measurable and $\lambda(P)=\int_X\mathbb1_P d\lambda$. Then
  $\lambda(P\triangle P g_n)=\int_X|\mathbb1_P - \mathbb1_{P
    g_n}|d\lambda$; now
  $\mathbb1_{P g_n}=\partial(\lambda g_n)/\partial\lambda\mathbb1_P
  g_n$ with $\partial(\lambda g_n)/\partial\lambda\to1$, and by
  Lemma~\ref{lem:autocont} $\mathbb1_P g_n\to\mathbb1_P$, so
  $\lambda(P\triangle P g_n)\to0$ as $n\to\infty$.

  However, $P g_n\subseteq Q\subseteq X\setminus P$ so
  $\lambda(P\triangle P g_n)=2\lambda(P)$; so $\lambda(P)=0$. Next
  $\lambda(Q B')\le\lambda(P)=0$ so $\lambda(Q)=0$. It follows that
  $u=\frac12$ almost everywhere, but this contradicts
  $0\le \sum_{n\in\Z}u(x b^n)\le 1$.
\end{proof}

\begin{example}\label{example:ghys}
  Let $G$ be a countable indiscrete, non-soluble subgroup of
  $\PSL_2(\R)$. Then $G$ contains a non-discrete free group acting
  essentially freely on $X=\mathbb P^1(\R)$. It follows that $G$
  generates a non-amenable equivalence relation on $X$.

  Indeed, $G$ contains an elliptic element of infinite order, namely an
  element with $|\operatorname{trace}(g)|\in[-2,2]\setminus2\cos(\pi\Q)$,
  see~\cite{jorgensen:subgroups}. The group generated by some power of $g$
  and of a hyperbolic element not fixing $g$'s fixed points is a
  non-discrete Schottky group.
\end{example}

Note that groups and equivalence relations are two special cases of
\emph{groupoids}, see Definition~\ref{defn:groupoid}. There is a
well-developed theory of amenability for groupoids with a measure on
their space of units, see~\cite{kaimanovich:isop}, and
\cite{anantharaman-renault:ag} for a full treatise.

%%%%%%%%%%%%%%%%%%%%%%%%%%%%%%%%%%%%%%%%%%%%%%%%%%%%%%%%%%%%%%%%
\newpage\section{Elementary operations}
We turn to a more systematic study of the class $AG$ of amenable
groups.  \JvNeumann~ already noted in~\cite{vneumann:masses} that $AG$
is closed under the following operations:
\begin{proposition}\label{prop:elemamen}
  Let $G$ be a group.
  \begin{enumerate}
  \item Let $N\triangleleft G$ be a normal subgroup. If $G$ is
    amenable, then $G/N$ is amenable.
  \item Let $H<G$ be a subgroup. If $G$ is amenable, then $H$ is
    amenable.\label{prop:elemamen:2}
  \item Let $N\triangleleft G$ be a normal subgroup. If $N$ and $G/N$
    are amenable, then $G$ is amenable.
  \item Let $(G_n)_{n\in\mathscr N}$ be \emph{directed} family of
    groups: $\mathscr N$ is a directed set, and for all $m<n$ there is
    a homomorphism $f_{m n}\colon G_m\to G_n$, with
    $f_{m n}f_{n p}=f_{mp}$ whenever $m<n<p$.  If $G_n$ is amenable for
    all $n$, then $\varinjlim G_n$ is amenable.

    In particular, if the $G_n$ form a nested sequence of amenable
    groups, i.e.\ $G_m\le G_n$ for $m<n$, then
    $\bigcup_{n\in\mathscr N} G_n$ is amenable.\label{prop:elemamen:4}
  \end{enumerate}
\end{proposition}
It is an amusing exercise to prove the proposition using a specific
definition of amenability. Below we prove it using the fixed point
property of convex compact $G$-sets, and give references to previous
statements where other proofs were given.
\begin{proof}
  \begin{enumerate}
  \item Proposition~\ref{prop:quotientX}.

    For another proof, let $G/N$ act on a non-empty convex compact
    $K$. Then in particular $G$ acts on $K$, and since $G$ is amenable
    we have $K^G\neq\emptyset$ by Theorem~\ref{thm:fixedpoint}. Then
    $K^{G/N}\neq\emptyset$ so $G/N$ is amenable.
  \item Proposition~\ref{prop:freeaction}.

    For another proof, let $H$ act on a non-empty convex compact $K$,
    and define
    \[K^{G/H}=\{f\colon G\to K\mid f(x h)=f(x) h\text{ for all }x\in G,h\in H\}.
    \]
    Then $K^{G/H}$ is a convex compact $G$-set under the action
    $(f\cdot g)(x)=f(g x)$, so admits a fixed point. This fixed point
    is a constant function, whose value is an $H$-fixed point in $K$.
  \item Proposition~\ref{prop:stabilizers}.

    For another proof, let $G$ act on a non-empty convex compact
    $K$. Since $N$ is amenable, $K^N\neq\emptyset$. Since $N$ is
    normal, $G/N$ acts on $K^N$, and since $G/N$ is amenable,
    $(K^N)^{G/N}\neq\emptyset$. But this last set is nothing but
    $K^G$.
  \item Proposition~\ref{prop:restrict}.

    For another proof, write $G=\varinjlim G_n$, with natural homomorphisms
    $f_n\colon G_n\to G$ such that $f_m=f_{m n}f_n$ for all $m<n$. Let
    $G$ act on a non-empty convex compact $K$. Then each $G_n$ acts on
    $K$ via $f_n$, and $K^{G_n}$ is non-empty because $G_n$ is
    amenable. Furthermore the $K^{G_n}$ form a directed sequence of
    closed subsets of $K$: given $I\Subset\mathscr N$ finite, there is
    $n\in\mathscr N$ greater than $I$, so
    $\bigcap_{i\in I}K^{G_i}\supseteq K^{G_n}$ is not empty. By
    compactness,
    $\bigcap_{n\in\mathscr N}K^{G_n}=K^G\neq\emptyset$.\qedhere
  \end{enumerate}
\end{proof}

\noindent We deduce immediately
\begin{corollary}
  A group $G$ is amenable if and only if all its finitely generated
  subgroups are amenable.
\end{corollary}
Indeed one direction follows from~\eqref{prop:elemamen:2}, the other
from~\eqref{prop:elemamen:4} with $\mathscr N$ the family of finite
subsets of $G$, ordered by inclusion, and $G_n=\langle n\rangle$.

\subsection{Elementary amenable groups}\label{ss:elemamen}
Finite groups are amenable; we saw in Corollary~\ref{cor:abelian} and
Proposition~\ref{prop:abelian} that abelian groups are amenable; and
saw in Proposition~\ref{prop:elemamen} that the class of amenable
groups is closed under extensions and
colimits. Following~\MDay~\cite{day:amen}, let us define the class of
\emph{elementary amenable groups}, $EG$. This is the smallest class of
groups that contains finite and abelian groups, and is closed under
the four operations of Proposition~\ref{prop:elemamen}: quotients,
subgroups, extensions, and directed unions.

\begin{example}
  Virtually soluble groups are in $EG$.
\end{example}
Indeed, they are obtained by a finite number of extensions using
finite and abelian groups.

\begin{example}\label{ex:finitary:eg}
  For a set $X$, the group $\Sym(X)$ of finitely-supported
  permutations is in $EG$.
\end{example}
Indeed, $X$ is the union of its finite subsets, so $\Sym(X)$ is the
directed limit of finite symmetric groups.

\begin{example}\label{ex:limsolvable}
  Consider
  \[G=\langle \dots,x_{-1},x_0,x_1,\dots\mid \langle
    x_i,\dots,x_{i+k}\rangle^{(k)}\text{ for all }i\in\Z,k\in\N\rangle,
  \]
  where $F^{(k)}$ denotes the $k$th term of the derived series of
  $F$. Then $G$ is in $EG$.

  Obviously the map $x_i\mapsto x_{i+1}$ extends to an automorphism of
  $G$; let $\widehat G$ denote the extension $G\rtimes\Z$ using this
  automorphism. Then $\widehat G$ also is in $EG$.
\end{example}
Indeed, $G=\bigcup_{k\in\N}\langle x_{-k},\dots,x_k\rangle$, where each
term is soluble. However, $G$ itself is not soluble.

\begin{example}\label{ex:limnilpotent}
  This example is similar to~\ref{ex:limsolvable}, but more
  concrete. Consider formal symbols $e_{m n}$ for all $m<n\in\Z$. The
  group $M$ is the set of formal expressions
  $1+\sum_{m<n}\alpha_{m n}e_{m n}$, with $\alpha_{m n}\in\Z$ and almost
  all $0$; multiplication is defined by the formulas
  $e_{m n}e_{n p}=e_{mp}$, all other products being $0$.
  Then $M$ is locally nilpotent, so is in $EG$.

  Extend then $M$ by the automorphism
  $\sigma\colon e_{m n}\mapsto e_{m+1,n+1}$; the resulting group
  $\widehat M=M\rtimes\Z$ is again in $EG$, and is finitely generated,
  by $1+e_{12}$ and $\sigma$.
\end{example}

The class $EG$ may be refined using transfinite induction. Let $EG_0$
denote the class of finite or abelian groups. For an ordinal $\alpha$,
let $EG_{\alpha+1}$ denote the class of extensions or directed unions
of groups in $EG_\alpha$; and for a limit ordinal $\alpha$ set
$EG_\alpha=\bigcup_{\beta<\alpha}EG_\beta$.
\begin{lemma}
  A group is elementary amenable if and only if it belongs to
  $EG_\alpha$ for some ordinal $\alpha$.
\end{lemma}
\begin{proof}
  It suffices to see that the classes $EG_\alpha$ are closed under
  subgroups and quotients. This is clear for $EG_0$. If $\alpha$ is a
  successor, consider a subgroup $H\le G\in EG_\alpha$. Either $G=N.Q$
  is an extension of groups in $EG_{\alpha-1}$; and then $H=(N\cap
  H).(H/N\cap H)$ with $H/N\cap H\le Q$; or $G=\bigcup G_i$, in which
  case $H=\bigcup(H\cap G_i)$; in both cases, $H\in EG_\alpha$ by
  induction. Consider next a quotient $\pi:G\twoheadrightarrow
  H$. Either $G=N.Q$, and $H=\pi(N).(H/\pi(N))$ with
  $Q\twoheadrightarrow H/\pi(N)$, or $G=\bigcup G_i$, in which case
  $H=\bigcup\pi(G_i)$; in both cases, $H\in EG_\alpha$ by induction.

  If $\alpha$ is a limit ordinal, then each $G\in EG_\alpha$ actually
  belongs to $EG_\beta$ for some $\beta<\alpha$ and there is nothing
  to do.
\end{proof}

\begin{example}
  Continuing Example~\ref{ex:finitary:eg}, consider
  $H=\Sym(\Z)\rtimes\Z$, with $\Z$ acting on functions in $\Sym(\Z)$
  by shifting: $(n\cdot p)(x)=p(x-n)$. Then $H$ is $2$-generated, for
  example by $(1,2)\in\Sym(\Z)$ and a generator of $\Z$.

  Since $\Sym(\Z)$ is a union of finite groups but is neither finite
  nor abelian, $\Sym(\Z)\in EG_1\setminus EG_0$. Likewise,
  $H\in EG_2\setminus EG_1$.
\end{example}

Example~\ref{ex:limsolvable} is a bit more
complicated. $F_k/F_k^{(k)}$ is soluble of class precisely $k$; so it
belongs to $EG_{k-1}\setminus EG_{k-2}$. Therefore, $G\in EG_\omega$,
but $G\not\in EG_n$ for finite $n$. Similarly, $\widehat G\in
EG_{\omega+1}$. The same holds for $M$ and $\widehat M$ from
Example~\ref{ex:limnilpotent}.

Note also in Example~\ref{ex:finitary:eg} that the group of \emph{all}
permutations of $\Z$ is not amenable. Indeed it contains every countable
group (seen as acting on itself); so if it were amenable then by
Proposition~\ref{prop:elemamen} every countable group would be
amenable.

Recall that $AG$ denotes the class of amenable
groups. In~\cite{day:amen}, \MDay~asks whether the inclusion
$EG\subseteq AG$ is strict; in other words, is there an amenable group
that may not be obtained by repeated application of
Proposition~\ref{prop:elemamen} starting with finite or abelian
groups?
\begin{theorem}[Chou~\cite{chou:elementary}*{Theorems~2.3 and~3.2}]
  Finitely generated torsion groups in $EG$ are finite.

  No finitely generated group in $EG$ has intermediate word-growth.
\end{theorem}
The inequality $EG\neq AG$ follows, since there exist finitely
generated infinite torsion groups (see~\cite{golod:nil} or
Exercise~\ref{ex:torsion}) and groups of intermediate word growth, see
Theorem~\ref{thm:grig}.
\begin{proof}
  The two statements are proven in the same manner, by transfinite
  induction. We only prove the second, and leave the (easier) first
  one as an exercise. Let us show that, if $G\in EG$ has
  subexponential word-growth, then $G$ is virtually nilpotent. Groups
  in $EG_0$ have polynomial growth, and are therefore virtually
  nilpotent by Theorem~\ref{thm:polygrowth}. Consider next $\alpha$ a
  limit ordinal, and $G\in EG_\alpha$ a finitely generated group. We
  may assume that $\alpha$ is minimal, so in particular $\alpha$ is
  not a limit ordinal. Since $G$ is finitely generated, we have
  $G=N.Q$ for $N,Q\in EG_{\alpha-1}$. By induction $Q$ is virtually
  nilpotent, so in particular is virtually polycyclic. By
  Corollary~\ref{cor:Nfg} the subgroup $N$ is finitely generated, so
  is virtually nilpotent by induction. By Lemma~\ref{lem:vnil}, the
  group $G$ is virtually soluble, and by
  Corollary~\ref{cor:milnorwolf} it has either polynomial or
  exponential growth.
\end{proof}

\subsection{Subexponentially amenable groups}\label{ss:subexpamen}
In~\cite{ceccherini-g-h:amen}*{\S14}, \TCeccherini, \PdlHarpe\ and
\RGrigorchuk\ consider the class $SG$ of \emph{subexponentially
  amenable groups} as the smallest class containing groups of
subexponential growth and closed under taking subgroups, quotients,
extensions, and direct limits. We then have
$EG\subsetneqq SG\subseteq AG$, and we shall see promptly that the
last inclusion is also strict.

We introduce a general construction of groups: let $H$ be a
permutation group on a set $\Alphabet$. We assume that the action is
transitive, and choose a point $0\in\Alphabet$. Let us construct a
self-similar group $\mathcal M(H)$ acting on the rooted tree
$\Alphabet^*$, see Definition~\ref{defn:treeaut}. The group
$\mathcal M(H)$ is generated by two subgroups, written $H$ and $K$ and
isomorphic respectively to $H$ and to
$H\wr H_0=H^{(\Alphabet\setminus\{0\})}\rtimes H_0$. We first define
the actions of $H$ and $K$ on the boundary $\Alphabet^\N$ of the
tree. The action of $h\in H$ is on the first letter:
\[(a_0a_1\dots)h=(a_0h)a_1\dots.
\]
The action of $(f,h)\in K$, with
$f\colon\Alphabet\setminus\{a\}\to H$ finitely supported, fixes
$a^\N$ and is as follows on its complement:
\[(a_0a_1\dots)(f,h)=0\dots0(a_n h)(a_{n+1}f(a_n))a_{n+2}\dots\text{ with $n$ minimal such that }a_n\neq0.\]

The self-similarity of $\mathcal M(H)$ is encoded by an injective
homomorphism
$\Phi\colon\mathcal M(H)\to\mathcal M(H)\wr_\Alphabet H=\mathcal
M(H)^\Alphabet\rtimes H$, written
$g\mapsto\pair<g_a\mid a\in\Alphabet>\pi$ and defined as
follows. Given $g\in\mathcal M(H)$, its image $\pi$ in $H$ is the
natural action of $g$ on
$\{a\Alphabet^\N\mid a\in\Alphabet\}\cong\Alphabet$. The permutation
$g_a$ of $\Alphabet^\N$ is the composition
$\Alphabet^\N\to a\Alphabet^\N\to(a\pi)\Alphabet^\N\to\Alphabet^\N$ of
the maps $(w\mapsto a w)$, $g$ and $((a\pi)w\mapsto w)$ respectively.
On the generators of $\mathscr M(H)$, we have
\[\Phi(h)=\pair<1\mid a\in\Alphabet>h,\qquad\Phi((f,h))=\pair<f(a)\mid a\in\Alphabet>h.\]

\begin{proposition}
  If $H$ is perfect and $2$-transitive, then $\Phi$ is an isomorphism.
\end{proposition}
\begin{proof}
  First, if $H$ is $2$-transitive, then $\mathscr M(H)$ is generated
  by three subgroups $H,H_0,\overline H$. Fix a letter
  $1\in\Alphabet$; then $H_0$ consists of those $(1,h)\in K$,
  and $\overline H$ consists of those $(f,1)$ where $f(a)=1$ for all
  $a\neq1$. To avoid confusions between these subgroups, we write
  $h,h_0,\overline h$ for respective elements of $H,H_0,\overline H$.

  To prove that $\Phi$ is an isomorphism, it suffices to prove that
  $\pair<h,1,\dots,1>$, $\pair<h_0,1,\dots,1>$ and
  $\pair<\overline h,1,\dots,1>$ belong to $\Phi(\mathscr M(H))$ for
  all $h\in H,h_0\in H_0,\overline h\in\overline H$.

  First, choose $k\in H_0$ with $1k\neq1$. For all
  $\overline h,\overline h'\in\overline H$ we have
  $\Phi([\overline h,(\overline h')^k])=\pair<[\overline h,\overline
  h'],1,\dots,1>$; and since $\overline H\cong H$ is perfect we get
  that the image of $\Phi$ contains $\overline
  H\times1\dots\times1$. Consider next $h_0\in H_0$; then
  $\Phi(h_0 h^{-1})=\pair<h_0,1,\dots,1>$.  Finally,
  $\Phi(\overline h)=\pair<\overline h,h,1,\dots,1>$ and
  $\pair<\overline h,1,\dots,1>$ belongs to the image of $\Phi$, so
  $\pair<1,h,1,\dots,1>$ also belongs to its image. Conjugating by an
  appropriate element of $H$, we see that $\pair<h,1,\dots,1>$ belongs
  to the image of $\Phi$.
\end{proof}

\begin{theorem}[\cite{bartholdi-k-n-v:ba}; see~\cite{brieussel:folner} for the proof]\label{thm:motheramen}
  If $H$ is finite, then the group $\mathcal M(H)$ is amenable.
\end{theorem}
\begin{proof}
  If $H\le \widehat H$ as permutation groups then
  $\mathcal M(H)\le\mathcal M(\widehat H)$. It therefore does not
  reduce generality, in proving that $\mathcal M(H)$ is amenable, to
  consider $H$ perfect and $2$-transitive.

  We consider $S=H\cup K$ as generating set for $\mathcal M(H)$. Let
  us define finite subsets $I_k\subseteq L_k$ of $\mathcal M(H)$
  inductively as follows:
  \begin{align*}
    I_0 &= K,\qquad L_0 = I_0H,\\
    I_k &= H\cdot\Phi^{-1}(I_{k-1}\times L_{k-1}^{\Alphabet\setminus\{0\}}),\\
    L_k &= H\cdot\Phi^{-1}(L_{k-1}^\Alphabet\setminus (L_{k-1}\setminus I_{k-1})^\Alphabet).
  \end{align*}

  \begin{lemma}\label{lem:1}
    For all $k\in\N$ we have $I_k K=I_k$ and $I_k H=L_k H=L_k$;
    therefore, $I_k S=L_k$.
  \end{lemma}
  \begin{proof}
    The claims are clear for $k=0$. Also, $L_k H=L_k$ for all
    $k$. Consider $g\in I_k$ and $f\in K$, and write them
    $g=h\pair<g_a\mid a\in\Alphabet>$ and
    $f=\pair<f_a\mid a\in\Alphabet>h'$. Note
    $g f=a\pair<g_a f_a\mid a\in\Alphabet>h'$. We have $f_a\in H$ for
    all $a\neq0$, so $g_a f_a\in L_{k-1}$ for all $a\neq0$; and
    $f_0\in K$ so $g_0 f_0\in I_{k-1}$.
  \end{proof}

  \begin{lemma}
    Setting $\rho_k=\#I_k/\#L_k$, we have
    \[\rho_k=\frac{\rho_{k-1}}{1-(1-\rho_{k-1})^{\#\Alphabet}}.\]
  \end{lemma}
  \begin{proof}
    Set $d=\#\Alphabet$. From the definition, we get
    $\#L_k=\#L_{k-1}^d\#H(1-(1-\rho_{k-1})^d)$ and
    $\#I_k=\#I_{k-1}\#L_{k-1}^{d-1}\#H$, so
    \[\rho_k=\frac{\#I_k}{\#L_k}=\frac{\#I_{k-1}}{\#L_{k-1}(1-(1-\rho_{k-1})^d)}.\qedhere\]
  \end{proof}

  We are ready to prove that the sequence $(I_k)$ is a F\o lner
  sequence. In view of Lemma~\ref{lem:1}, it suffices to prove
  $\rho_k\to1$. Note $0<\rho_{k-1}<\rho_k<1$, so the sequence
  $(\rho_k)$ has a limit, $\rho$. Then $\rho$ satisfies
  $\rho=\rho/(1-(1-\rho)^d)$, so $\rho=1$.
\end{proof}

\noindent To prove that $\mathcal M(H)$ has exponential growth, we use
a straightforward criterion:
\begin{proposition}\label{prop:freemonoid}
  Let a left-cancellative monoid $G=\langle S\rangle_+$ act on a
  set $X$; let there be a point $x\in X$ and disjoint subsets
  $Y_s\subseteq X\setminus\{x\}$ satisfying $x s\in Y_s$ and
  $Y_s S\subseteq Y_s$ for all $s\in S$. Then $G$ is free on $S$,
  namely $G\cong S^*$.
\end{proposition}
\begin{proof}
  Consider distinct words $u=u_1\dots u_m,v=v_1\dots v_n\in S^*$; we
  are to prove that they have distinct images in $G$. Since $G$ is
  left-cancellative, we may assume either $m=0$ or $u_1\neq v_1$. In
  the first case $x u=x\neq x v\in Y_{v_1}$, and in the second case
  $Y_{u_1}\ni x u\neq x v\in Y_{v_1}$.
\end{proof}

The proposition implies that $\mathcal M(H)$ has exponential growth
for almost all $H$; it seems difficult to formulate a general result,
so we content ourselves with an example:
\begin{example}\label{ex:M3exp}
  The group $\mathcal M(S_3)$ has exponential growth.
\end{example}
\begin{proof}
  Write $\Alphabet=\{0,1,2\}$ and $S_3=\langle(0,1),(0,2)\rangle$. In
  our notation, consider the elements $s=\overline{(0,1)}(0,1)$ and
  $t=\overline{(0,2)}^{(1,2)_0}(0,2)$.  A quick calculation gives
  \[\Phi(s)=\pair<s(0,1),(0,1),1>(0,1),\qquad\Phi(t)=\pair<t(0,2),1,(0,2)>(0,2),\]
  Proposition~\ref{prop:freemonoid} applies with
  $G=\langle s,t\rangle_+$ and $X=\Alphabet^\N$ and $x=0^\N$ and
  $Y_s=\Alphabet^*10^\N$ and $Y_t=\Alphabet^*20^\N$.
\end{proof}

The first construction of an amenable, not subexponentially amenable
group appears in~\cite{bartholdi-v:amenability}, with an explicit
subgroup of (what was later defined to be) $\mathcal M(D_4)$.
\begin{example}\label{ex:AG-SG}
  The group $\mathcal M(A_5)$ belongs to $AG\setminus SG$.
\end{example}
\begin{proof}
  The group $G\coloneqq\mathcal M(A_5)$ is amenable by
  Theorem~\ref{thm:motheramen}. It contains $\mathcal M(S_3)$, e.g.\
  because the permutations $(0,1)(3,4)$ and $(0,2)(3,4)$ generate a
  copy of $S_3$ in $A_5$, so $G$ has exponential growth
  by Example~\ref{ex:M3exp}.

  It remains to prove that $G$ does not belong to $SG$, and we do this
  by transfinite induction, defining (just as we did for $EG$) the
  class $SG_0$ of groups of subexponential growth and for an ordinal
  $\alpha$ by letting $SG_\alpha$ denote those extensions and directed
  unions of groups in $SG_\beta$ for $\beta<\alpha$.

  By way of contradiction, let $\alpha$ be the minimal ordinal such
  that $G$ belongs to $SG_\alpha$. Since $G$ is finitely generated, it
  is an extension of groups in $SG_\beta$ for some $\beta<\alpha$. Now
  the only normal subgroups of $G$ are $1$ and the groups $G_n$ in the
  series defined by $G_0=G$ and
  $G_{n+1}=\Phi^{-1}(G_n^\Alphabet\times H)$; the argument is similar
  to that used to show that $\GG$ is not in $EG$, see
  Exercise~\ref{ex:perfect}.  In particular, every non-trivial normal
  subgroup of $G$ maps onto $G$, so cannot belong to $SG_\beta$ for
  some $\beta<\alpha$.
\end{proof}

\subsection{Free group free groups}\label{ss:fgfg}
For levity, in this section by ``free group'' we always mean ``non-abelian
free group''.  It follows from Proposition~\ref{prop:elemamen} that every
group containing a free subgroup is itself not amenable; this covers
surface groups, or more generally word-hyperbolic groups; free products of
a group of size at least $2$ with a group of size at least $3$; and
$\SO_3(\R)$; that last example is important in relation to the
Banach-Tarski paradox, see~\S\ref{ss:hbt}.

Let us denote by $NF$ the class of groups with no free
subgroup. In~\cite{day:amen}, \MDay\ asks whether the inclusion
$AG\subseteq NF$ is an equality; in other words, does every
non-amenable group contains a free subgroup?

This was made into a conjecture
by~\FGreenleaf~\cite{greenleaf:im}*{Page~9},
attributed\footnote{Infelicitously!} to von
Neumann. \CChou~\cite{chou:elementary} proved $EG\neq NF$; while
\AOlshanskii~\cite{olshansky:invmean} proved $AG\neq NF$, see also
\SAdyan~\cite{adyan:rw}. Indeed, they proved the much stronger result
that the free Burnside groups
\begin{equation}\label{eq:burnside}
  B(n,m)=\langle x_1,\dots,x_n\mid w^m\text{ for all words $w$ in }x_1^{\pm1},\dots,x_n^{\pm1}\rangle
\end{equation}
are non-amenable as soon as $n\ge2$ and $m\ge665$ is odd. These
groups, of course, do not contain any non-\emph{trivial} free
subgroup.

The following examples of groups are called ``Frankenstein groups'',
since (as their namesake) they have rather different properties than
the groups they are built of:
\begin{theorem}[Monod~\cite{monod:fgfg}]\label{thm:monod}
  Let $\mathbb A$ be a countable subring of $\R$ properly containing
  $\Z$; let $P_{\mathbb A}\subseteq\mathbb P^1(\R)$ be the set of
  fixed points of hyperbolic elements in $\PSL_2(\mathbb A)$, and let
  $H(\mathbb A)$ be the group of self-homeomorphisms of
  $\mathbb P^1(\R)$ that fix $\infty$ and are piecewise elements of
  $\PSL_2(\mathbb A)$ with breakpoints in $P_{\mathbb A}$. Then
  $H(\mathbb A)$ is a nonamenable free group free group.
\end{theorem}
\begin{proof}
  Since $\mathbb A$ properly contains $\Z$, it is dense in $\R$, so
  $\PSL_2(\mathbb A)$ is a countable dense subgroup of
  $\PSL_2(\R)$. It therefore generates a non-amenable equivalence
  relation on $\mathbb P^1(\R)$, by Example~\ref{example:ghys}.

  \begin{lemma}[\cite{monod:fgfg}*{Proposition~9}]
    For all $p\in\mathbb P^1(\R)\setminus\{\infty\}$ we have
    \[p\cdot\PSL_2(\mathbb A)\subseteq\{\infty\}\cup p\cdot H(\mathbb A).\]
  \end{lemma}
  \begin{proof}
    Given $g\in\PSL_2(\mathbb A)$ with $p g\neq\infty$, we seek
    $h\in H(\mathbb A)$ with $p h=p g$. It will be made of two pieces,
    $g$ near $p$ and $z\mapsto z+r$ near $\infty$ for a suitable
    choice of $r\in\mathbb A$. Consider the quotient
    $q\coloneqq g\cdot(z\mapsto z-r)\in\PSL_2(\mathbb A)$; if $q$ is
    hyperbolic, say with fixed points $\xi_\pm$, and $\{\xi_\pm\}$
    separates $p$ from $\infty$, then we may define $h$ as $g$ on the
    component of $\mathbb P^1(\R)\setminus\{\xi_\pm\}$ containing $p$ and
    as $z\mapsto z+r$ on its complement. Now an easy calculation shows
    that $q$ is hyperbolic for all $|r|$ large enough, and as
    $|r|\to\pm\infty$ one of the fixed points of $q$ approaches
    $\infty$ and the other approaches $\infty g$, and as the sign of
    $r$ changes the approach to $\infty$ is from opposite sides; so in
    all cases it is easy to find a suitable $r$.
  \end{proof}

  Therefore, the equivalence relation generated by $H(\mathbb A)$ is
  non-amenable, so $H(\mathbb A)$ is itself non-amenable.

  On the other hand, consider $f,g\in H(\mathbb A)$. We claim that
  they do not generate a free group, and more precisely either that
  $\langle f,g\rangle$ either is metabelian or contains a subgroup
  isomorphic to $\Z^2$.

  Let $1\neq h\in\langle f,g\rangle''$ belong to the second derived
  subgroup, and intersect as few connected components of
  $\supp(f)\cup\supp(g)$ as possible --- if no such $h$ exists, we are
  already done. For every endpoint
  $p\in\partial(\supp(f)\cup\supp(g))$, the element $h$ acts trivially
  in a neighbourhood of $p$, because both $f$ and $g$ act as affine
  maps in a neighbourhood of $p$; so the support of $h$ is strictly
  contained in $\supp(f)\cup\supp(g)$. Since the dynamics of
  $\langle f,g\rangle$ has attracting elements in the neighbourhood of
  $p$, there exists $k\in\langle f,g\rangle$ such that $\supp(h)$ and
  $\supp(h)k$ are disjoint; then $\langle h,h^k\rangle\cong\Z^2$.
\iffalse
  On the other hand, $H(\mathbb A)$ acts faithfully on
  $\mathbb P^1(\R)$, has soluble groups of germs, and does not have
  free orbits of free subgroups: given $a,b\in H(\mathbb A)$ and
  $p\in\R\subset\mathbb P^1(\R)$, set $m=\inf(p\langle
  a,b\rangle)$. Then $m$ (possibly $=\infty$) is a fixed point of $a$
  and $b$, so in a neighbourhood of $m$ both $a$ and $b$ act as affine
  transformations and therefore do not generate a free action of a
  free group. Proposition~\ref{prop:nonfree} applies.
\fi
\end{proof}

\begin{exercise}[***]
  Since $H(\mathbb A)$ is not amenable, there is a free action of
  $\PSL_2(\mathbb A)$ on $\R$ by $H(\mathbb A)$-wobbles. Construct
  explicitly such an action.

  \emph{Hint:} This is essentially what~\cite{lodha:tarski} does in
  computing the minimal number of pieces in a paradoxical
  decomposition of $H(\mathbb A)$, but it's still highly non-explicit.
\end{exercise}

Thus we have $EG\subsetneqq SG\subsetneqq AG\subsetneqq NF$. The last
inequality also holds for finitely generated groups --- any finitely
generated nonamenable subgroup of $H(\mathbb A)$ will do. Lodha and
Moore construct finitely \emph{presented} examples
in~\cite{lodha-moore:fp}.

\begin{problem}
  Is the group $H(\Z)$ amenable?
\end{problem}

The group $H(\Z)$ is related to a famous group acting on the real
line, consider Thompson's group $F$ (see
Problem~\ref{problem:thompson}), which we describe here.
\begin{example}\label{ex:F}
  Let $F$ be the group of self-homeomorphisms of $[0,1]$ that are
  piecewise affine with slopes in $2^\Z$ and breakpoints in
  $\Z[\tfrac12]$.

  Conjugating $F$ by Minkowski's ``?'' map, defined by
  $?(x)=\sum_{n\ge0}(-1)^n2^{-a_0-\cdots-a_n}$ if $x$'s continued
  fraction expansion is $[a_0,a_1,\dots]$, one obtains a group of
  piecewise-$\PSL_2(\Z)$ homeomorphisms of the real line with rational
  breakpoints; it is easy to see that having rational breakpoints is
  equivalent to the maps being \emph{diffeo}morphisms.

  The same argument as that given in the proof of
  Theorem~\ref{thm:monod} shows that $F$ is a free group free group.

  The difference with $H(\Z)$ is that breakpoints of maps in $H(\Z)$
  are in $\mathbb P_\Z$, which is disjoint from $\Q$. There are
  embeddings of $F$ in $H(\Z)$, so amenability of $H(\Z)$ would imply
  that of $F$.

  Yet another description of $F$ is by an action on the Cantor
  set. For this, break the interval $[0,1]$ open at every dyadic
  rational; one obtains in this manner a Cantor set, modeled on
  $\{0,1\}^\N$ by the usual binary expansion of real numbers, except
  that one does not identify $a_1\dots a_n01^\infty$ with
  $a_1\dots a_n10^\infty$. The action of $F$ is then by
  lexicographical order-preserving maps that are piecewise of the form
  $a_1\dots a_n v\mapsto b_1\dots b_k v$ for a collection of words
  $(a_1\dots a_n,b_1\dots b_k)$ and every $v\in\{0,1\}^\N$. The group
  $F$ is finitely generated, by the elements
  $x_0\colon00v\mapsto0v,01v\mapsto10v,1v\mapsto11v$ and
  $x_1\colon0v\mapsto0v,1v\mapsto 1x_0(v)$, and is even finitely
  presented. See~\cite{cannon-f-p:thompson} for a detailed survey of
  $F$.
\end{example}

%%%%%%%%%%%%%%%%%%%%%%%%%%%%%%%%%%%%%%%%%%%%%%%%%%%%%%%%%%%%%%%%
\newpage\section{Random walks}\label{sec:kesten}
We now turn to other criteria for amenability, expressed in terms of
random walks. For a thorough treatment of random walks consult the
book~\cite{woess:rw}; we content ourselves with the subset most
relevant to amenability. One is given a space $X$, and a random walker
$W$ moving at random in $X$. There is thus a random process
$W\in X\rightsquigarrow S(W)\in X$, describing a \emph{single step} of
the random walk. One asks for the distribution $W_n$ of the random
walker after a large number $n$ of iterations of $S$.

More formally, we are given \emph{one-step} transition probabilities
$p_1(x,y)=\mathbb P(W_n=x|W_{n-1}=y)$ of moving to $x$ for a particle
lying at $y$; they satisfy $p_1(x,y)\ge0$ and
$\sum_{x\in X}p_1(x,y)=1$ for all $y\in X$. We define iteratively
$p_n(x,y)=\sum_{z\in X}p_{n-1}(x,z)p_1(z,y)$, and then ask for
asymptotic properties of $p_n$.

Here are two fundamental examples. First, if $X$ is a graph with
finite degree, set $p_1(x,y)=1/\deg(y)$ if $x,y$ are neighbours, and
$p_1(x,y)=0$ otherwise. This is called the \emph{simple random walk} (SRW)
on the graph $X$.

Another fundamental example is given by a group $G$, a right $G$-set
$X$, and a probability measure $\mu$ on $G$, namely a map
$\mu\colon G\to[0,1]$ with $\sum_{g\in G}\mu(g)=1$ as
in~\eqref{eq:proba}. The random walk is then defined by
\begin{equation}
  p_1(x,y)=\sum_{g\in G, x=y g}\mu(g).\label{eq:driven}
\end{equation}
It is called the \emph{random walk driven by $\mu$}. The measure $\mu$
is called \emph{symmetric} if $\mu(g)=\mu(g^{-1})$ for all $g\in G$,
and is called \emph{non-degenerate} if its support generates $G$ qua
semigroup.

These two examples coincide in case $G=\langle S\rangle$ is finitely
generated and the driving measure $\mu$ is equidistributed on $S$; one
considers then SRW on the Schreier graph of the action of $G$ on $X$.

A random walk $p$ on a set $X$ is \emph{reversible} if there exists a
function $s\colon X\to(0,\infty)$ satisfying
$s(x)p_1(x,y)=s(y)p_1(y,x)$ for all $x,y\in X$. SRW is reversible on
undirected graphs, with $s(x)=\deg(x)$, and if $\mu$ is symmetric then
the random walk driven by $\mu$ is reversible with $s(x)\equiv1$. We
shall always assume that the random walks we consider are reversible,
and to lighten notation actually assume that they are symmetric:
$s(x)\equiv1$ so $p_1(x,y)=p_1(y,x)$.

\subsection{Spectral radius}
We shall prove a criterion, due to \HKesten, relating the spectral
radius of the linear operator associated with $p$ to amenability. It
first appeared in~\cite{kesten:rwalks}.  Let $p$ be a reversible
random walk on a set $X$, assumed symmetric for simplicity. Set
$E=\{(x,y)\in X^2\mid p_1(x,y)>0\}$. We introduce two Hilbert spaces:
\begin{align*}
  \ell^2_0 &= \{f\colon X\to\C\mid \langle f,f\rangle<\infty\},\\
  \ell^2_1 &= \{g\colon E\to\C\mid g(x,y)=-g(y,x),\langle g,g\rangle<\infty\}
\end{align*}
with scalar products $\langle f,f'\rangle=\sum_{x\in X}\overline{f(x)}f'(x)$ and
$\langle g,g'\rangle=\frac12\sum_{x,y\in
  X}p_1(x,y)\overline{g(x,y)}g'(x,y)$. Elements of $\ell^2_1$ are
naturally extended to functions on $X^2$ which vanish on
$X^2\setminus E$.

One step of the random walk $p$ induces a linear operator $T$ on
$\ell^2_0$ given by
\[(T f)(x) = \sum_{y\in X}p_1(x,y)f(y).
\]
Writing $\delta_x$ for the function taking value $1$ at $x\in X$ and
$0$ elsewhere, we then have $p_n(x,y)=(T^n\delta_y)(x)$. We also
define operators $d,d^*$ between $\ell^2_0$ and $\ell^2_1$ by
\begin{xalignat*}{2}
  d&\colon\ell^2_0\to\ell^2_1, & (d f)(x,y) &= f(x)-f(y),\\
  d^*&\colon\ell^2_1\to\ell^2_0, & (d^*g)(x) &= \sum_{y\in X}p_1(x,y)g(x,y).
\end{xalignat*}

\begin{lemma}
  $T$ is a self-adjoint operator on $\ell^2_0$ of norm at most
  $1$. The operator $d^*$ is the adjoint of $d$, and $T=1-d^*d$.
\end{lemma}
\begin{proof}
  The first claim follows from the second. For $f\in\ell^2_0$ and
  $g\in\ell^2_1$, we compute
  \begin{align*}
    \langle d f,g\rangle &= \frac12\sum_{(x,y)\in E}p_1(x,y)(\overline{f(x)}-\overline{f(y)})g(x,y)\\
                        &= \frac12\sum_{x\in X}\overline{f(x)}\sum_{y\in X}p_1(x,y)(g(x,y)-g(y,x))\\
    &= \sum_{x\in X}\overline{f(x)}(d^*g)(x) = \langle f,d^*g\rangle,
    \intertext{and}
    (1-d^*d)f(x) &= f(x) - \sum_{y\in X}p_1(x,y)(f(x)-f(y)) = \sum_{y\in X}p_1(x,y)f(y).\qedhere
  \end{align*}
\end{proof}

The following definitions are more commonly given in the context of
graphs; our more general setting coincides with it if $p$ is the
simple random walk:
\begin{definition}\label{defn:isoperimetric}
  Let $p$ be a random walk on a set $X$. The \emph{isoperimetric
    constant} of $p$ is
  \[\iota(p)=\inf_{F\Subset X}\frac{p_1(F,X\setminus F)}{\#F}=\inf_{F\Subset X}\frac{\sum_{x\in F,y\in X\setminus F}p_1(x,y)}{\#F}.\]
  The \emph{spectral radius} of $p$ is the spectral radius --- or,
  equivalently, the norm --- of the operator $T$.
\end{definition}

The following inequalities relating spectral radius and isoperimetric
constant appear, with different notation and normalization,
in~\cite{biggs-m-st:spectralradius}:
\begin{proposition}\label{prop:rho<>i}
  Let $p$ be a reversible random walk on a set $X$. Then the
  isoperimetric constant $\iota$ and spectral radius $\rho$ of $p$ are
  related by
  \[\iota^2+\rho^2\le1\le\iota+\rho.\]
\end{proposition}
\begin{proof}
  We begin by the second inequality. For $\epsilon>0$, let
  $F\Subset X$ satisfy
  $p_1(F,X\setminus F)/\#F<\iota+\epsilon$. Let
  $\phi\in\ell^2_0$ denote the characteristic function of $F$. Then
  $\|\phi\|^2=\#F$, and
  \[\|d\phi\|^2 = \frac12\sum_{(x,y)\in E}p_1(x,y)(\phi(x)-\phi(y))^2 = \sum_{x\in F,y\in X\setminus F}p_1(x,y)<(\iota+\epsilon)\|\phi\|^2;
  \]
  then
  $(\rho+\iota+\epsilon)\|\phi\|^2>\langle\phi,T\phi\rangle +
  \|d\phi\|^2 =
  \langle\phi,(1-d^*d)\phi\rangle+\|d\phi\|^2=\|\phi\|^2$. The
  conclusion $\rho+\iota\ge1$ follows under $\epsilon\to0$.

  In the other direction, consider for finite $F\Subset X$ the
  projection $\pi_F\colon\ell^2_0\righttoleftarrow$ defined by
  $(\pi_F f)(x)=f(x)$ if $x\in F$ and $0$ otherwise, and set
  $T_F\coloneqq\pi_F T\pi_F$. The operator $T_F$ is self-adjoint, and
  converges strongly to $T$ as $F$ increases, so the spectral radius
  of $T_F$ converges to $\rho$. For $\epsilon>0$, let $F$ be such that
  the spectral radius $\rho_F$ of $T_F$ is larger than
  $\rho-\epsilon$. Since $T_F$ has non-negative entries, its
  eigenvalue $\rho_F$ is simple and has a non-negative eigenvector
  $\phi$, by the Perron-Frobenius theorem. We extend $\phi$ by $0$
  into an element of $\ell^2_0$, and normalize it so that
  $\|\phi\|=1$. Set then
  \[A\coloneqq \frac12\sum_{(x,y)\in E}p_1(x,y)|\phi(x)^2-\phi(y)^2|,\]
  and compute
  \begin{align*}
    A^2 &= \bigg(\frac12\sum_{(x,y)\in
        E}p_1(x,y)|\phi(x)+\phi(y)|\cdot|\phi(x)-\phi(y)|\bigg)^2\\
    &\le \frac12\sum_{(x,y)\in E}p_1(x,y)(\phi(x)+\phi(y))^2\cdot
    \frac12\sum_{(x,y)\in E}p_1(x,y)(\phi(x)-\phi(y))^2\\
    &=(\|\phi\|^2+\langle\phi,T_F\phi\rangle)(\|\phi\|^2-\langle
      \phi,T_F\phi\rangle)=(1+\rho_F)(1-\rho_F),
  \end{align*}
  because
  $\sum_{(x,y)\in E}p_1(x,y)\phi(x)\phi(y)=\sum_{x\in
    F}\phi(x)\sum_{y\in X}p_1(x,y)\phi(y)=\langle\phi,T_F\phi\rangle$.

  On the other hand, let $0<s_1<s_2<\dots<s_n$ denote the finitely
  many values that $\phi$ takes, and define, for $k=1,\dots,n$,
  \[F_k = \{x\in X\mid \phi(x)\ge s_k\},
  \]
  with the additional conventions $s_0=0$ and
  $F_{k+1}=\emptyset$. Then
  \begin{align*}
    A &= \frac12\sum_{(x,y)\in E}p_1(x,y)|\phi(x)^2-\phi(y)^2|
        = \sum_{k=1}^n\sum_{x\in F_k,y\not\in F_k}p_1(x,y)(s_k^2-s_{k-1}^2)\\
      &\ge\sum_{k=1}^n\iota\#F_k(s_k^2-s_{k-1}^2)=\iota\sum_{k=1}^n(\#F_k-\#F_{k+1})s_k^2=\iota\|\phi\|=\iota.
  \end{align*}
  Combining, we get
  \[(1-(\rho-\epsilon)^2)\ge 1-\rho_F^2\ge A^2\ge\iota^2;\]
  and the conclusion $\rho^2+\iota^2\le1$ follows under $\epsilon\to0$.
\end{proof}

\noindent This section's main result is the following characterization
of amenable $G$-sets:
\begin{theorem}\label{thm:kesten}
  Let $\mu$ be a symmetric, non-degenerate probability measure on a
  group $G$, let $X$ be a $G$-set, and let $p$ be the random walk on
  $X$ driven by $\mu$. Then the following are equivalent:
  \begin{enumerate}
  \item $X$ is amenable;
  \item $\iota(p)=0$;
  \item $\rho(p)=1$.
  \end{enumerate}
\end{theorem}
\begin{proof}
  $(1)\Rightarrow(2)$ Assume first that $X$ is amenable, and let
  $\epsilon>0$ be given. Let $S\Subset G$ satisfy
  $\mu(S)>1-\epsilon/2$.  Let $F\Subset X$ satisfy
  $\#(F S\setminus F)<\epsilon\#F/2$. Then
  \[\sum_{x\in F,y\not\in F}p_1(x,y)\le\sum_{x\in F,y\in F S\setminus F}\mu(\{s\in S\mid y=x s\})+\sum_{x\in F,g\in G\setminus S}\mu(g)\le\epsilon\#F,\]
  so $\iota(p)\le\epsilon$ for all $\epsilon>0$. (Note that we have
  not used the assumption that $\mu$ is non-degenerate here.)

  $(2)\Rightarrow(1)$ Let $\epsilon>0$ and a finite subset $S$ of $G$
  be given. By assumption, there exists $n\in\N$ and $\delta>0$ such
  that $\mu^n(s)\ge\delta$ for all $s\in S$. Let $F$ be a finite
  subset of $X$ such that
  $\sum_{x\in F,y\not\in F}p_n(x,y)<\delta\epsilon\#F$. Then
  $\#(F S\setminus F)<\epsilon\#F$, so $X$ is amenable by F\o lner's
  criterion, Theorem~\ref{thm:folneramen}$(5)\Rightarrow(1)$.

  The equivalence $(2)\Leftrightarrow(3)$ is given by
  Proposition~\ref{prop:rho<>i}.
\end{proof}

The spectral radius of the random walk has a direct interpretation in
terms of probabilities of return of the random walk, at least when we
restrict to \emph{transitive} random walks: random walks with the
property that, for any two $x,y\in X$ there exists $n\in\N$ such that
$p_n(x,y)>0$ (not to be confused with random walks invariant under a
transitive group action!). Let us make the following temporary
\begin{definition}
  The \emph{spectral radius} of the random walk $p$ based at $x$ is
  \[\rho(p,x)\coloneqq \limsup_{n\to\infty}\sqrt[n]{p_n(x,x)}.\]
\end{definition}

\begin{lemma}[Fekete]\label{lem:fekete}
  Let $N\in\N$ be given, and let $\alpha\colon \{N,N+1,\dots\}\to\R$ be a
  subadditive function, i.e.\ a function satisfying
  $\alpha(m+n)\le\alpha(m)+\alpha(n)$. Then
  \[\lim_{n\to\infty}\frac{\alpha(n)}{n}=\inf_{n>0}\frac{\alpha(n)}{n};\]
  in particular $\alpha(n)/n$ either converges, or diverges to $-\infty$.
\end{lemma}
\begin{proof}
  Consider any $a\ge N$, and write every $k\ge N$ as $k=q a+r$ with
  $q\in\N$ and $r\in\{N,N+1,\dots,N+a-1\}$. Then, for $k\ge N$,
  \[\frac{\alpha(k)}{k}\le\frac{q\alpha(a)+\alpha(r)}{q a+r}\le\frac{\alpha(a)}{a}+\frac{\alpha(r)}{k};\]
  letting $k\to\infty$, we get
  $\limsup_{n\to\infty}\alpha(k)/k\le\alpha(a)/a$ for every $a\ge
  N$; so $\limsup_{n\to\infty}\alpha(k)/k=\inf_{a\in\N}\alpha(a)/a$
  converges or diverges to $-\infty$.
\end{proof}

The ``limsup'' in the definition of the spectral radius is in fact a
limit, and is independent of the starting and endpoints:
\begin{proposition}\label{prop:problim}
  Assume $p$ is transitive. Then
  \[\rho(p,z)=\limsup_{n\to\infty}\sqrt[n]{p_n(x,y)}=\lim_{n\to\infty}\sqrt[2n]{p_{2n}(x,x)}\text{ for all }x,y,z\in X.\]
\end{proposition}
\begin{proof}
  For the first claim, consider more generally $w,x,y,z\in X$. There
  are $\ell\in\N$ such that $p_\ell(x,w)>0$; and $m\in\N$ such that
  $p_m(z,y)>0$. Since
  $p_{n+\ell+m}(x,y)\ge p_\ell(x,w)p_n(w,z)p_m(z,y)$ for all $n\in\N$,
  we have
  \[\limsup_{n\to\infty}\sqrt[n]{p_n(x,y)}\ge\limsup_{n\to\infty}\sqrt[n-\ell-m]{p_\ell(x,w)p_m(z,y)}\sqrt[n-\ell-m]{p_n(w,z)}=\limsup_{n\to\infty}\sqrt[n]{p_n(w,z)}.\]
  Applying it to $(w,x,y,z)=(z,x,y,z)$ and $(x,z,z,y)$ respectively
  gives the claim.

  It is then clear that
  $\limsup_{n\to\infty}\sqrt[2n]{p_{2n}(x,x)}\le\rho(p,x)$; but
  conversely $p_{2n}(x,x)\ge p_n(x,x)^2$, so
  $\limsup_{n\to\infty}\sqrt[2n]{p_{2n}(x,x)}\ge\limsup_{n\to\infty}\sqrt[n]{p_n(x,x)}=\rho(p,x)$.

  Now $p_{2r+2s}(x,x)\ge p_{2r}(x,x)p_{2s}(x,x)$ for all
  $r,s\in\N$. Setting $\alpha(r)=-\log p_{2r}(x,x)$, we get
  $\alpha(r+s)\le\alpha(r)+\alpha(s)$; furthermore, because $p$ is
  transitive, $\alpha(r)$ is defined for all $r$ large enough, and
  $\alpha(r)\ge0$ because $p_{2r}(x,x)\le 1$. By
  Lemma~\ref{lem:fekete}, $\alpha(r)/r$ converges, whence
  $\sqrt[2n]{p_{2n}(x,x)}$ converges.
\end{proof}

\begin{proposition}\label{prop:sr=norm}
  Let $p$ be symmetric and transitive. Then the spectral radius of $p$
  is equal to the norm of $T$ acting on $\ell^2(X)$.
\end{proposition}
\begin{proof}
  Let us write $\|T\|$ the operator norm of $T$ on $\ell^2(X)$. First, by
  Proposition~\ref{prop:problim},
  \begin{align*}
    \rho(p,x)&=\lim_{n\to\infty}\sqrt[2n]{p_{2n}(x,x)}
    =\lim_{n\to\infty}\sqrt[2n]{\langle T^{2n}\delta_x,\delta_x\rangle}
    \le\sqrt[2n]{\|T^{2n}\|}\le\|T\|.
  \end{align*}
  Next, consider $f\in\C X$. By Cauchy-Schwartz's inequality, for all
  $m\in\N$ we have
  \[\langle T^{m+1}f,T^{m+1}f\rangle=\langle
  T^m f,T^{m+2}f\rangle\le\|T^m f\|\cdot\|T^{m+2}f\|;\]
  so $\|T^{m+1}f\|/\|T^m f\|$ is increasing, with limit
  $\lim_{n\to\infty}\sqrt[n]{\|T^n f\|}$. Now
  \begin{align*}
    \lim_{n\to\infty}\sqrt[n]{\|T^n f\|}&=\lim_{n\to\infty}\sqrt[2n]{\langle
      T^n f,T^n f\rangle}=\lim_{n\to\infty}\sqrt[2n]{\langle
      T^{2n}f,f\rangle}\\
    &=\lim_{n\to\infty}\sqrt[2n]{\sum_{x,y\in\supp(f)}p_{2n}(x,y)f(x)f(y)}=\rho(p,x),
  \end{align*}
  because the sum is finite. Taking $m=0$, we obtain
  $\|T f\|/\|f\|\le\rho(p,x)$ for all $f\in\C X$; and since $\C X$ is
  dense in $\ell^2(X)$ we have $\|T\|\le\rho(p,x)$.
\end{proof}

The probabilities of return, in the case of SRW, have a
straightforward interpretation in terms of paths: say we consider a
$k$-regular graph $X$ with basepoint $*$. Then there are $k^n$ paths
of length $n$ starting at $*$, and among these asymptotically
$\rho(p)^n$ will end at $*$. Therefore, non-amenable graphs are
characterized as those graphs in which exponentially few paths are
closed.

\begin{example}\label{ex:rwZ}
  Let us look first at an amenable example: $X=\Z$ and
  $p_1(x,x\pm1)=\frac12$; this is SRW on the line. We write $p_n(x,y)$
  for the probability that a particle starting at $x$ reaches $y$ at
  time $n$; that is, the probability that $T^n(x)=y$. The simple
  formula
  \[p_n(x,y)=\begin{cases} \frac1{2^n}\binom n{\frac{n+x-y}2} & \text{
        if }n+x-y\equiv0\pmod2,\\ 0 & \text{ else}\end{cases}
  \]
  is easily justified as follows: at each step, one chooses $+1$ or
  $-1$ with equal probabilities; at time $n$ we then made $2^n$
  choices. If $(n+x-y)/2$ of these are $+1$ and $(n-x+y)/2$ are $-1$,
  then we end up at $x+(n+x-y)/2-(n-x+y)/2=y$.

  In particular, if $n$ is even, we have
  $p_n(x,x)=2^{-n}\binom n{n/2}$, so by Stirling's formula
  $n!\propto\sqrt{2\pi n}(n/e)^n$ we get
  \[p_n(x,x)\propto \sqrt{\frac{2}{\pi n}}.
  \]
\end{example}

\begin{example}
  Consider the free group $F_d$, whose Cayley graph is a $2d$-regular
  tree $\mathcal T$. Fix an edge of this tree, e.g.\ between $1$ and
  $x_1$, let $A_n$ denote the number of closed paths in $\mathcal T$
  based at $1$, and by $B_n$ the number of closed paths in
  $\mathcal T$, based at $1$, that do not cross the fixed
  edge. Consider the generating series $A(z)=\sum A_n z^n$ and
  $B(z)=\sum B_n z^n$. Then $A(z)=1/(1-2d z^2 B(z))$, because every
  closed path factors uniquely as a product of closed paths that reach
  $1$ only at their endpoints; and $B(z)=1/(1-(2d-1)z^2 B(z))$ for the
  same reason; so
  \[A(z)=\frac{1-d+d\sqrt{1-4(2d-1)z^2}}{1-4d^2z^2}
  \]
  and $A_n\propto(8d-4)^{n/2}$ and
  $p_n(1,1)\propto(8d-4)^{n/2}/(2d)^n$. Therefore, SRW on $F_d$ has
  spectral radius $\rho(p)=\sqrt{2d-1}/d$.

  The isoperimetric constant of SRW may also easily be computed. A
  connected, finite subset $F$ of $\mathcal T$ has $\#F$ vertices and
  is connected to $2n\#F$ edges, of which $2(\#F-1)$ point back to
  $F$, so $\sum_{x\in F,y\not\in F}p_1(x,y)=((2n-2)\#F+2)/2n$. The
  isoperimetric constant is therefore $\iota(p)=1-1/n$.
\end{example}

\begin{exercise}[**]
  Compute the isoperimetric constant of SRW on the surface group
  $\Sigma_g=\langle a_1,b_1,\dots,a_g,b_g\mid[a_1,b_1]\cdots[a_g,b_g]=1\rangle$.

  \emph{Hint:} its Cayley graph is a tiling of hyperbolic plane by
  $4g$-gons, meeting $4g$ per vertex. Use Euler characteristic.

  Note that is is substantially harder to compute the spectral radius
  of SRW; only estimates are known, proportional to $\sqrt g$;
  see~\cite{gouezel:sprad} for the best bounds.
\end{exercise}

It is sometimes easier to count \emph{reduced} paths in graphs, rather
than general paths. Formally, this may be expressed as follows: let
$G=\langle S\cup S^{-1}\rangle$ be a finitely generated group, and
write $S^\pm=S\cup S^{-1}$. There is a natural map
$\pi\colon F_S\to G$ induced by the inclusion $S\hookrightarrow
G$. The spectral radius of SRW on $G$ is
\[\rho=\lim_{n\to\infty}\frac{\sqrt[n]{\#\{w\in(S^\pm)^n\mid w=_G1\}}}{\#S^\pm}\in[0,1].\]
The \emph{cogrowth} of $G$ is
\[\gamma=\lim_{n\to\infty}\sqrt[n]{\#{w\in F_S\mid \pi(w)=1}}{\#(S\cup S^{-1})}\in[1,\#S^\pm-1].\]
\begin{theorem}[\cite{grigorchuk:rw}; see also~\cite{cohen:cogrowth,szwarc:cogrowth,woess:cogrowth,bartholdi:cogrowth}]
  The parameters $\gamma,\rho$ are related by the equation
  \[\rho=\frac{\gamma+(\#S^\pm-1)/\gamma}{\#S^\pm}\text{ if }\gamma>1.
  \]
  In particular, $G$ is amenable if and only if $\gamma=\#S^\pm-1$. 
\end{theorem}
\begin{proof}
  The most direct proof is combinatorial. Define formal matrices $B,C$
  indexed by $G$ with power series co\"efficients by
  \[B(z)_{g,h}=\sum_{w\in F_S:g\pi(w)=h}z^{|w|},\qquad C(z)_{g,h}=\sum_{w\in(S^\pm)^*: g w=_G h}z^{|w|}.
  \]
  Set for convenience $q\coloneqq\#S^\pm-1$. We shall prove the formal
  relationship
  \begin{equation}\label{eq:cogrowthfe}
    \frac{B(z)}{1-z^2}=\frac{C(z/(1+q z^2))}{1+q z^2},
  \end{equation}
  from which the claim of the theorem follows. Define the adjacency matrix
  \[A_{g,h}=\sum_{s\in S^\pm: g s=h}1;
  \]
  then $C(z)=1/(1-z A)$. If for all $s\in S^\pm$ we define
  $B_s(z)_{g,h}=\sum_{w\in F_S\setminus\{1\}:w_1=s,g\pi(w)=h}z^{|w|}$
  then
  \[B(z)=1+\sum_{s\in S^\pm}B_s(z),\qquad B_s(z)=s z(B(z)-B_{s^{-1}}(z))
  \] 
  which solve to $B_s(z) = (1-z^2)^{-1}(s z-z^2)B(z)$ and therefore to
  \[\frac{1+q z^2}{1-z^2}B(z)=1 + \sum_{s\in S^\pm}\frac{z}{1-z^2}s B(z)=1+\frac{z}{1-z^2} A B(z);\]
  so $(1+q z^2)/(1-z^2)\cdot B(z)=1/(1-z/(1+q z^2)A)$, which is
  equivalent to~\eqref{eq:cogrowthfe}.
\end{proof}
It is also known that $\rho\ge\sqrt{\#S^\pm-1}$, with equality if and
only if $G\cong F_S$, see~\cite{paschke:norm}.

\subsection{Harmonic functions}\label{ss:harmonic}
We shall obtain, in this subsection, yet another characterization of
amenability in terms of bounded harmonic functions.
\begin{definition}
  Let $p$ be a random walk on a set $X$. A \emph{harmonic function} is
  a function $f\colon X\to\R$ satisfying
  \[f(x)=\sum_{y\in X}p_1(y,x)f(y).
  \]
  In other words, $f$ is a \emph{martingale}: along a trajectory
  $(W_n)$ of a random walk, the expectation of $f(W_n)$ given
  $W_0,\dots,W_{n-1}$ is $f(W_{n-1})$.

  A random walk is called \emph{Liouville} if the only bounded
  harmonic functions are the constants.

  If $X$ is a $G$-set and $p$ is the random walk driven by a measure
  $\mu$ on the group $G$, we say that $(X,\mu)$ is Liouville when the
  corresponding random walk is Liouville.
\end{definition}
Bounded harmonic functions are fundamental in understanding long-term
behaviour of random walks. The space of trajectories of a random walk
on $X$ is $(X^\N,\nu)$, in which the trajectory $(W_0,W_1,\dots)$ has
probability
$\nu(W_0,W_1,\dots)=\prod_{n\ge0}\mu(\{g\in G\mid W_n
g=W_{n+1}\})$. An \emph{asymptotic event} on $(X^\N,\nu)$ is a
measurable subset of $X^\N$ that is invariant under the shift map of
$X^\N$. Given a asymptotic event $E$, we define a bounded function
$f(x)=\nu(E\cap\{W_0=x\})$ and check that it is harmonic by
conditioning on the first step of the random walk; conversely, given a
bounded harmonic function $f$ the limit $f(W_n)$ almost surely exists
along trajectories, by Doob's martingale convergence theorem, so
$E_{[a,b]}=\{(W_0,W_1,\dots)\mid \lim f(W_n)\in[a,b]\}$ is a asymptotic
event. In summary, a random walk is Liouville if and only if there are
no non-trivial asymptotic events.

Let us continue with the example of SRW on $\Z$: a harmonic function
satisfies $f(x-1)+f(x+1)=2f(x)$, so $f$ is affine. In particular, SRW
on $\Z$ is Liouville.

Let us consider next the example of SRW on the Cayley graph of
$F_2=\langle a,b\mid\rangle$, which is a tree. The random walk $(W_n)$
escapes at speed $1/2$ towards the boundary of the tree, since at
every position except the origin it has three ways of moving one step
farther and one way of moving one step closer; so in particular almost
surely $W_n\neq1$ for all $n$ large enough. Let $A\subset F_2$ denote
those elements whose reduced form starts with $a$, and define
\[f(g)=\mathbb P(W_n\in g^{-1}A\text{ for all $n$ large enough}).
\]
In words, $f(g)$ is the probability that a random walk started at $g$
escapes to the boundary of the tree within $A$. It is clear that $f$
is bounded, and it is seen to be harmonic by conditioning on the first
step of the random walk. More succinctly, ``the random walk eventually
escapes in $A$'' is a non-trivial asymptotic event. Therefore, SRW on a
regular tree is not Liouville.

\begin{exercise}[*]
  Let $(X,\mu)$ and $(Y,\nu)$ be Liouville random walks. Prove that
  $(X\times Y,\mu\times\nu)$ is Liouville.
\end{exercise}

Let us recal some properties of measures and random walks. The set
$\ell^1(G)$ of summable functions on $G$ is a Banach *-algebra, for
the \emph{convolution product}
\begin{equation}\label{eq:convolution}
  (\mu\nu)(g)=\sum_{g=h k}\mu(h)\nu(k)\text{ for }\mu,\nu\in\ell^1(G).
\end{equation}
We denote by $\check\mu$ the adjoint of $\mu$, defined by
$\check\mu(g)=\mu(g^{-1})$. If $X$ is a $G$-set, then $\ell^p(X)$ is an
$\ell^1(G)$-module for all $p\in[1,\infty]$, under
\[(f\mu)(x)=\sum_{x=y h}f(y)\mu(h)\text{ for }f\in\ell^p(X),\mu\in\ell^1(G).
\]
If a random walk is driven by a measure $\mu$, then
from~\eqref{eq:driven} we get $T f=f\mu$.  With our notation, a
function $f\in\ell^\infty(X)$ is harmonic for the random walk driven
by a measure $\mu$ if and only if $f\check\mu=f$.

The Liouville property is fundamentally associated with a measure, or
a random walk. It has a counterpart which solely depends on the space,
and is a variant of amenability with switched quantifiers (see
Proposition~\ref{prop:Ramenable}):
\begin{definition}\label{defn:laminable}
  A $G$-set $X$ is called \emph{laminable}\footnote{This is a
    contraction of ``Liouville'' and ``amenable''.} if for every
  $\epsilon>0$ and every $f\in\varpi(\ell^1X)$ there exists a positive
  function $g\in\ell^1(G)$ with $\|f g\|<\epsilon\|g\|$.
\end{definition}

\begin{proposition}\label{prop:amenlamin}
  Let $G$ be a group, viewed as a right $G$-set $G_G$. Then $G_G$ is
  amenable if and only if $G_G$ is laminable.

  Let $G$ be an amenable group, and let $X\looparrowleft G$ be a
  $G$-set. Then $X$ is laminable if and only if it is transitive or
  empty.
\end{proposition}
\begin{proof}
  By Proposition~\ref{prop:Ramenable}, $G_G$ is amenable if for every
  $\epsilon>0$ and every $0\neq g\in\varpi(\ell^1G)$ there exists a
  positive function $f\in\ell^1(X)$ with
  $\|f g\|<\epsilon\|f\|\,\|g\|$; equivalently,
  $\|\check g\check f\|<\epsilon\|g\|\,\|f\|$, which is the definition
  of laminability of $G_G$.

  For the second statement: if there is more than one $G$-orbit on
  $X$, choose $x,y$ in different orbits; then
  $\|(\delta_x-\delta_y)g\|=2\|g\|$ for all positive
  $g\in\ell^1(G)$. Conversely, given $\epsilon>0$ and
  $f\in\varpi(\ell^1X)$, choose $x\in X$ and $h\in\varpi(\ell^1G)$
  with $f=\delta_x h$. Since $G$ is amenable, there is a positive
  function $g\in\ell^1(G)$ with $\|g\check h\|<\epsilon\|g\|$; so
  $\|f\check g\|=\|x h\check g\|=\|g\check h\|<\epsilon\|g\|$, and $X$
  is laminable.
\end{proof}

\noindent The following easy proposition is an analogue to
Proposition~\ref{prop:quotientX}:
\begin{proposition}\label{prop:quotientL}
  Let $G,H$ be groups, let $X\looparrowleft G$ and $Y\looparrowleft H$
  be respectively a $G$-set and an $H$-set, let
  $\phi\colon G\twoheadrightarrow H$ be a homomorphism, and let
  $f\colon X\to Y$ be a surjective equivariant map, namely satisfying
  $f(x g)=f(x)\phi(g)$ for all $x\in X,g\in G$. If $X$ is laminable,
  then $Y$ is laminable.
\end{proposition}
\begin{proof}
  Given $\epsilon>0$ and $e\in\varpi(\ell^1Y)$, there is
  $e'\in\varpi(\ell^1X)$ with $e'=e\circ f$, because $f$ is
  surjective; then there is a positive function $g\in\ell^1(G)$ with
  $\|e' g\|<\epsilon\|g\|$, because $X$ is laminable; then
  $\|e\phi(g)\|\le\|e' g\|<\epsilon\|g\|=\epsilon\phi(g)$, so $Y$ is
  laminable.
\end{proof}

\begin{corollary}\label{cor:amentrans}
  Let $X\looparrowleft G$ be a $G$-set and let $H\le G$ be a
  subgroup. If $H$ is amenable and transitive, then $X$ is
  laminable.\qed
\end{corollary}

\begin{lemma}\label{lem:sx-laminable}
  If $X\looparrowleft G$ is laminable, then for every $x\in X$, every
  finite subset $S\Subset X$ and every $\epsilon>0$ there exists a
  positive function $g\in\ell^1(G)$ with
  \[\|\delta_s g-\delta_x g\|<\epsilon\|g\|\text{ for all }s\in S.\]
  Furthermore $g$ may be supposed to be of finite support.
\end{lemma}
\begin{proof}
  Consider $f=\sum_{s\in S}\delta_s-\#S\delta_x$. Since $X$ is
  laminable, there is for every $\epsilon>0$ a positive function
  $g\in\ell^1(G)$ with $\|f g\|<\epsilon\|g\|/2$. Then
  \begin{align*}
    \epsilon\|g\|&>2\|f g\|\ge2\Big\|\sum_{s\in S}\max(\delta_s g-\delta_x g,0)\Big\|=\sum_{s\in S}2\|\max(\delta_s g-\delta_x g,0)\|\\
    &\ge\sum_{s\in S}\|\delta_s g-\delta_x g\|.
  \end{align*}
  Using density of finitely-supported functions in $\ell^1(G)$ gives
  the last claim.
\end{proof}

\noindent The main result of this section is:
\begin{theorem}\label{thm:Lamenable}
  Let $X$ be a $G$-set.  The following are equivalent:
  \begin{enumerate}
  \item $X$ is laminable;
  \item There exists a symmetric measure $\mu$ with support equal to
    $G$ such that $(X,\mu)$ is Liouville;
  \item There exists a measure $\mu$ on $G$ such that $(X,\mu)$ is
    Liouville.
  \end{enumerate}
\end{theorem}

\begin{corollary}[Kaimanovich-Vershik~\cite{kaimanovich-v:entropy}]\label{thm:kv}
  Let $G$ be a group. Then $G$ is amenable if and only if there exists
  a measure $\mu$ (\emph{ad lib.} symmetric, with full support) such
  that $(G,\mu)$ is Liouville.\qed
\end{corollary}

Note that there exist amenable non-laminable $G$-sets, such as
Example~\ref{ex:freeamen}, and non-amenable graphs for which SRW is
Liouville, see~\cite{benjamini-kozma:naliouville}
or~\cite{benjamini:cgr}*{Chapter~13}. At the extreme, note that the
empty set is laminable but not amenable, and the disjoint union of two
points is amenable but not laminable. Here is a slightly less
contrived example:
\begin{example}[Kaimanovich]
  Consider the binary rooted tree with vertex set $\{0,1\}^*$ and an
  edge between $a_1\dots a_n$ and $a_1\dots a_{n+1}$ for all
  $a_i\in\{0,1\}$. Fix a function $f\colon\N\to\N$ satisfying $f(n)<n$
  for all $n\in\N$, and put also an edge between $a_1\dots a_n$ and
  $a_1\dots \widehat{a_{f(n)}}\dots a_n$ for all
  $a_i\in\{0,1\}$. Finally add some loops at the root so as to make
  the graph $6$-regular; we have constructed a graph $\mathscr G$,
  with a natural action of $F_6$ once the edges are appropriately
  labeled. We consider SRW on $\mathscr G$.

  On the one hand, $\mathscr G$ is not amenable; for example, because
  SRW drifts away from the root at speed $(4-2)/6=1/3$, or because the
  isoperimetric inequality in $\mathscr G$ is at least as bad as in a
  binary tree.

  On the other hand, if $f$ grows slowly enough then SRW on
  $\mathscr G$ is Liouville; indeed SRW converges to the boundary of
  the binary tree, represented by binary sequences $\{0,1\}^\N$, and
  it suffices to show that there are no asymptotic events on this
  boundary. If $f$ is such that $f^{-1}(n)$ is infinite for all
  $n\in\N$, then each c\"oordinate in $\{0,1\}^\N$ is randomized
  infinitely often by the walk when it follows the $f$-edges, so there
  is no non-constant measurable function on the space of trajectories.
\end{example}

For the remainder of the section, we assume the hypotheses of the
theorem: a countable group $G$ and a transitive $G$-set $X$ are
fixed. We also assume that all measures $\mu$ under consideration
satisfy $\mu(1)>0$, and call such $\mu$ \emph{aperiodic}. This is
harmless: a function $f$ is harmonic for $\mu$ if and only if it is
harmonic for $q\mu+(1-q)\delta_1$ whenever $q\in(0,1]$.
\begin{lemma}
  Let $\mu$ be an aperiodic measure on $G$. Then there exists a
  sequence $(\epsilon_n)\to0$, depending only on $\mu(1)$, such that,
  for all $f\in\ell^\infty(X)$,
  \[\|f\mu^n-f\mu^{n+1}\|\le\epsilon_n\|f\|.
  \]
\end{lemma}
\begin{proof}
  Since
  $\|f\mu^n-f\mu^{n+1}\|_\infty\le\|f\|_\infty\cdot\|\mu^n-\mu^{n+1}\|_1$, it
  suffices to prove $\|\mu^n-\mu^{n+1}\|_1\to0$. Set $q=\mu(1)$; we
  assume $q\in(0,1)$. Define a measure $\lambda$ on $\N$ by
  $\lambda(0)=q,\lambda(1)=p=1-q$, and let $\nu$ be the probability
  measure on $G$ such that
  $\mu=q\delta_1+p\nu$. Then
  \begin{align*}
    \mu^n(g)&=(q\delta_1+p\nu)^n(g)=\sum_{i=0}^n\lambda^n(i)\nu^i(g),\\
    \text{so }\|\mu^n-\mu^{n+1}\|&=\bigg\|\sum_{i=0}^{n+1}(\lambda^n(i)-\lambda^{n+1}(i))\nu^i\bigg\|.
  \end{align*}
  Since
  $\lambda^n(i)-\lambda^{n+1}(i)=\lambda^n(i)-q\lambda^n(i)-p\lambda^n(i-1)=p(\lambda^n(i)-\lambda^n(i-1))$,
  it suffices to prove
  $\sum_{i=0}^{n+1}|\lambda^n(i)-\lambda^n(i-1)|\to0$. Remembering
  $\lambda^n(i)=\binom n i p^i q^{n-i}$, the argument of the absolute
  value is positive for $i<p n$ and negative for $i> p n$, so
  $\sum_{i=0}^{n+1}|\lambda^n(i)-\lambda^n(i-1)|\le\lambda^n(\lfloor p
  n\rfloor)+\lambda^n(\lceil p n\rceil)\to0$.
\end{proof}

\begin{corollary}\label{cor:makeharmonic}
  For every bounded sequence of functions $(F_n)$ in $\ell^\infty(X)$,
  every pointwise accumulation point of the sequence $(F_n\check\mu^n)$ is
  harmonic.\qed
\end{corollary}

\begin{proposition}\label{prop:liouville}
  Let $\mu$ be aperiodic and non-degenerate. Then $(X,\mu)$ is
  Liouville if and only if
\[\text{ for all }x,y\in X:\quad\|\delta_x\mu^n-\delta_y\mu^n\|_1\to0\text{ as }n\to\infty.\]
\end{proposition}
\begin{proof}
  Let first $f\in\ell^\infty(X)$ be harmonic. Then for all $n\in\N$
  \begin{align*}
    |f(x)-f(y)|&=\bigg|\sum_{z\in X}f(z)\Big(\sum_{z=x h}\mu^n(h)-\sum_{z=y h}\mu^n(h)\Big)\bigg|\\
    &\le\|f\|_\infty\cdot\sum_{z\in X}\big|(\delta_x\mu^n)(z)-(\delta_y\mu^n)(z)\big|=\|f\|_\infty\cdot\|\delta_x\mu^n-\delta_y\mu^n\|\to0,
  \end{align*}
  so $f$ is constant.

  Conversely, assume that there exist $x,y\in X$ and a sequence
  $(n_i)$ such that
  $\|\delta_x\mu^{n_i}-\delta_y\mu^{n_i}\|\ge4c>0$ for all
  $i\in\N$. Set $V_i\coloneqq\{z\in X\mid(\delta_x\mu^{n_i})(z)>(\delta_y\mu^{n_i})(z)\}$; then
  \[\sum_{z\in V_i}(\delta_x\mu^{n_i})(z)-(\delta_y\mu^{n_i})(z)\ge2c.\]
  Set then
  $W_i\coloneqq\{z\in X\mid
  (\delta_x\mu^{n_i})(z)\ge(1+c)(\delta_y\mu^{n_i})(z)\}$;
  then $(\delta_x\mu^{n_i})(W_i)\ge c$, for otherwise one would
  have
  $\sum_{z\in
    V_i}(\delta_x\mu^{n_i})(z)-(\delta_y\mu^{n_i})(z)=\sum_{z\in
    W_i}(\dots)+\sum_{z\in V_i\setminus W_i}(\dots)<c+c=2c$. Set
  finally $f_i=\mathbb1_{W_i}\check\mu^{n_i}$. Note then
  \[f_i(y)=\sum_{z=y h}\mathbb1_{W_i}(z)\mu^{n_i}(h)=\sum_{z\in W_i}(\delta_y\mu^{n_i})(z)\le\sum_{z\in W_i}(\delta_x\mu^{n_i})(z)/(1+c)=f_i(x)/(1+c),
  \]
  and similarly $f_i(x)=(\delta_x\mu^{n_i})(W_i)\ge c$, so any
  accumulation point of the $f_i$ is harmonic by
  Corollary~\ref{cor:makeharmonic}, bounded and non-constant.
\end{proof}

\begin{proof}[Proof of Theorem~\ref{thm:Lamenable}]
  $(1)\Rightarrow(2)$ We assume throughout that $X$ is transitive and
  therefore countable. Fix a basepoint $x_0\in X$, and let
  $\{x_0,x_1,\dots\}$ be an enumeration of $X$. Choose two sequences
  $(t_i)_{i\in\N}$ and $(\epsilon_i)_{i\in\N}$ of positive real
  numbers with $\sum t_i=1$ and $\lim\epsilon_i=0$. Let $(n_i)$ be a
  sequence of integers with $(t_1+\cdots+t_{i-1})^{n_i}<\epsilon_i$
  for all $i$. Since $X$ is laminable, by Lemma~\ref{lem:sx-laminable}
  there exists for every $i$ a positive function
  $\alpha_i\in\ell^1(G)$, normalized by $\|\alpha_i\|=1$ and supported
  on a finite set (say $F_i$), with
  \[\|(\delta_s-\delta_{x_0})\alpha_i\|<\epsilon_i\text{ for all }s\in
    \{x_0,\dots,x_i\}\cdot(\{1\}\cup F_1\cup\cdots\cup F_{i-1})^{n_i}.
  \]
  Let us set $\mu=\sum_{i\in\N} t_i\alpha_i$.  To prove that $(G,\mu)$
  is Liouville, it suffices, by Proposition~\ref{prop:liouville}, to
  prove that $\|\delta_x\mu^n-\delta_{x_0}\mu^n\|\to0$ for all
  $x\in X$. Say $x=x_\ell$; we claim that
  $\|\delta_x\mu^{n_\ell}-\delta_{x_0}\mu^{n_\ell}\|<4\epsilon_\ell$, and
  this is sufficient to conclude the proof.  For convenience let us
  write $n_\ell=n$, and expand
  \begin{equation}\label{eq:kvsum}
    \mu^n={\sum\sum}_{k_1,\dots,k_n}t_{k_1}\cdots t_{k_n}\alpha_{k_1}\cdots\alpha_{k_n}.
  \end{equation}
  We subdivide the sum~\eqref{eq:kvsum} into two summands, $\nu_1$ on
  which all $k_i<\ell$ and $\nu_2=\mu^n-\nu_1$. First,
  $\|\nu_1\|=\sum_{k_i<\ell}t_{k_1}\cdots
  t_{k_n}=(t_1+\cdots+t_{\ell-1})^{n_\ell}<\epsilon_\ell$, so
  $\|\delta_x\nu_1-\delta_{x_0}\nu_1\|<2\epsilon_\ell$. Secondly,
  consider a summand $\theta=\alpha_{k_1}\cdots\alpha_{k_n}$ appearing
  in $\nu_2$; by hypothesis $k_i\ge\ell$ for some $i$, which we choose
  minimal. The summand then has the form
  $\theta_1\alpha_{k_i}\theta_2$. The supports of $\delta_x\theta_1$
  and of $\delta_{x_0}\theta_1$ are by hypothesis contained in
  $\{x,x_0\}\cdot(\{1\}\cup F_1\cup\cdots\cup F_{\ell-1})^{n_\ell}$,
  so
  $\|\delta_x\theta_1\alpha_{k_i}-\delta_{x_0}\alpha_{k_i}\|<\epsilon_\ell$
  and
  $\|\delta_{x_0}\theta_1\alpha_{k_i}-\delta_{x_0}\alpha_{k_i}\|<\epsilon_\ell$. Consequently,
  $\|\delta_x\theta-\delta_{x_0}\theta\|<2\epsilon_\ell$, so
  $\|\delta_x\nu_2-\delta_{x_0}\nu_2\|<2\epsilon_\ell$ and finally
  $\|\delta_x\mu^n-\delta_{x_0}\mu^n\|<4\epsilon_\ell$ as required.

  $(2)\Rightarrow(3)$ is obvious.

  $(3)\Rightarrow(1)$ By Proposition~\ref{prop:liouville}, the
  sequence $(\mu^n)_{n\in\N}$ is asymptotically invariant. Consider
  $\epsilon>0$ and $f\in\varpi(\ell^1X)$. There is then a subset
  $S\Subset X\times X$ such that
  $f=f'+\sum_{(x,y)\in S}\delta_x-\delta_y$ with $\|f'\|<\epsilon/2$;
  we have
  $\|f\mu^n\|\le\sum_{(x,y)\in
    S}\|\delta_x\mu^n-\delta_y\mu^n\|+\|f'\mu^n\|$, and for $n$ large
  enough each $\delta_x\mu^n-\delta_y\mu^n$ has norm at most
  $\epsilon/2\#S$, from which $\|f\mu^n\|<\epsilon=\epsilon\|\mu^n\|$,
  and $X$ is laminable.
\end{proof}

\begin{exercise}[*]
  Let $G$ be a group and let $\mu$ be a probability measure on $G$.
  Prove that $(G,\mu)$ is Liouville if and only if $(G,\check\mu)$ is
  Liouville.
\end{exercise}

%%%%%%%%%%%%%%%%%%%%%%%%%%%%%%%%%%%%%%%%%%%%%%%%%%%%%%%%%%%%%%%% 
\newpage\section{Extensive amenability}\label{ss:extensiveamen}
We introduce now a property stronger than amenability for $G$-sets, a
property that behaves better with respect to extensions of $G$-sets
(whence the name). This section is based
on~\cite{juschenko-mattebon-monod-delasalle:extensive}.

\begin{definition}\label{defn:extamen}
  Let $X$ be a set; recall that $\mathfrak P_f(X)$ denotes the
  collection of finite subsets of $X$.  An \emph{ideal} in
  $\mathfrak P_f(X)$ is a subset, for some $x\in X$, of the form
  $\{E\Subset X\mid x\in E\}$.\footnote{It is really the ideal
    generated by $\{x\}$ in the semigroup $(\mathfrak P_f(X),\cup)$.}

  Let $X$ be a $G$-set. It is \emph{extensively amenable} if there
  exists a $G$-invariant mean $m$ on $\mathfrak P_f(X)$ giving weight
  $1$ to every ideal.
\end{definition}
It follows immediately from the definition that $m(\{\emptyset\})=0$
if $X\neq\emptyset$, and that for every $E\Subset X$ we have
$m(\{F\Subset X\mid E\subseteq F\})=1$.

Recall that $\mathfrak P_f(X)$ is an abelian group under symmetric
difference $\triangle$, and is naturally isomorphic to $\prod'_X\Z/2$
under the map $E\mapsto\mathbb 1_E$. Recall also
from~\eqref{eq:wreath} that the \emph{wreath product} $\Z/2\wr_X G$ is
the semidirect product $G\ltimes(\prod'_X\Z/2)$, with $G$ acting on
$\prod'_X\Z/2$ by permuting its factors.

\begin{lemma}\label{lem:extamen0}
  If $G$ is amenable, then all $G$-sets are extensively
  amenable. Every extensively amenable non-empty $G$-set is amenable.
\end{lemma}
\begin{proof}
  Let $G$ be amenable and let $X$ be a $G$-set. Consider the set $K$
  of means on $\mathfrak P_f(X)$ giving full weight to every
  ideal. Clearly $K$ is a convex compact subset of
  $\ell^\infty(\mathfrak P_f(X))^*$, and is non-empty because it
  contains any cluster point of $(\delta_E)_{E\Subset X}$. Since $G$
  is amenable, there exists a fixed point in $K$, so $X$ is
  extensively amenable.

  Let next $X\looparrowleft G$ be extensively amenable, and let $m$ be an invariant
  mean in $\ell^\infty(\mathfrak
  P_f(X)\setminus\{\emptyset\})^*$. Define a mean on $X$ by
  \[\ell^\infty(X)\ni f\mapsto m\Big(E\mapsto\frac1{\#E}\sum_{x\in E}f(x)\Big),\]
  and note that it is $G$-invariant because $m$ is.
\end{proof}

\begin{lemma}\label{lem:extamen}
  Let  $X$ be a $G$-set. Then the following are
  equivalent:
  \begin{enumerate}
  \item $X$ is extensively amenable;
  \item For every finitely generated subgroup $H$ of $G$ and every
    $H$-orbit $Y\subseteq X$, the $H$-set $Y$ is extensively amenable;
  \item For every finitely generated subgroup $H$ of $G$ and every
    $x_0\in X$, there is an $H$-invariant mean on
    $\mathfrak P_f(x_0 H)$ that gives non-zero weight to
    $\{E\Subset x_0 H\mid x_0\in E\}$;
  \item There is a $G$-invariant mean on $\mathfrak P_f(X)$ that gives
    non-zero weight to $\{E\Subset X\mid x_0\in E\}$ for all
    $x_0\in X$.
  \end{enumerate}
\end{lemma}
\begin{proof}
  $(1)\Rightarrow(4)$ by definition.

  $(4)\Rightarrow(3)$ There is a natural map
  $\ell^\infty(\mathfrak P_f(x_0 H))\to\ell^\infty(\mathfrak P_f(X))$
  given by $f\mapsto f({-}\cap x_0 H)$, inducing an $H$-equivariant
  map
  $\mathscr M(\mathfrak P_f(X))\to\mathscr M(\mathfrak P_f(x_0 H))$.

  $(3)\Rightarrow(2)$ Let $Y=x_0 H$ be an $H$-orbit, and let $m_0$ be
  an $H$-invariant mean on $\mathfrak P_f(Y)$ that gives positive
  weight to $\{A\Subset Y\mid x_0\in A\}$. As in
  Theorem~\ref{thm:folneramen}, the mean $m_0$ may be approximated by
  a net $p_n$ of probability measures on $\mathfrak P_f(Y)$: these are
  maps $\mathfrak P_f(Y)\to[0,1]$ with total mass $1$. Define now for
  every $k\in\N$ new probability measures on $\mathfrak P_f(Y)$ by
  \[p_{n,k}(E)=\sum_{E_1\cup\cdots\cup E_k=E}p_n(E_1)\cdots p_n(E_k).
  \]
  Let $m$ be an cluster point of the $p_{n,k}$ as $n,k\to\infty$; then
  $m$ is an $H$-invariant mean on $\mathfrak P_f(Y)$, and we check
  that it gives mass $1$ to the ideal
  $S\coloneqq\{E\Subset Y\mid x_0\in E\}$, and therefore also to every
  ideal because $H$ acts transitively on $Y$ and $m$ is
  $H$-invariant: since $m_0(S)>0$, there exists $\delta<1$ such that
  $p_n(S)>1-\delta$ for all $n$ large enough, and then
  $p_{n,k}(S)>1-\delta^k$ so at the cluster point $m(S)=1$.

  $(2)\Rightarrow(1)$ For every finitely generated subgroup $H$ of $G$
  and every finite union $Y=Y_1\cup\cdots\cup Y_n$ of $H$-orbits,
  choose for $i=1,\dots,n$ an $H$-invariant mean $m_i$ on
  $\mathfrak P_f(Y_i)$, and construct a mean $m_{H,Y}$ on
  $\mathfrak P_f(X)$ by
  $m_{H,Y}(S) = m_1(\{E\cap Y_1\mid E\in S\})\cdots m_n(\{E\cap
  Y_n\mid E\in S\})$.  Clearly $m_{H,Y}$ is $H$-invariant and gives
  full weight to ideals in $\mathfrak P_f(Y)$. Order the pairs $(H,Y)$
  by inclusion, and consider a cluster point of the net
  $(m_{H,Y})$. It is $G$-invariant, and gives full weight to ideals in
  $\mathfrak P_f(X)$.
\end{proof}

Note that Lemma~\ref{lem:extamen}(2) implies in particular that
extensively amenable sets are \emph{hereditarily amenable}: every
subgroup acting on every orbit is amenable. We obtain in this manner
an abundance of amenable actions that are not extensively
amenable. For instance, consider Example~\ref{ex:freeamen} of an
amenable action of $F_2=\langle a,b\mid\rangle$, and the subgroup
$K=\langle a^{b^{-1}},a^{b^{-2}}\rangle$. Then $K$ is a free group of
rank $2$, and the $K$-orbit $Y$ of $1$ in $X$ is free, so $Y$ is not
an amenable $K$-set, and therefore $X$ is not extensively amenable. We
shall see in Example~\ref{ex:herednotext} a hereditarily amenable
$G$-set that is not extensively amenable.

We come to the justification of the terminology ``extensive
amenability'': the analogue of Corollary~\ref{cor:extensions} for
$G$-sets.
\begin{proposition}
  Let $G$ be a group acting on two sets $X,Y$, and let
  $q\colon X\to Y$ be $G$-equivariant. If $Y$ is extensively amenable
  and if for every $y\in Y$ the $G_y$-set $q^{-1}(y)$ is an
  extensively amenable, then $X$ is extensively amenable. The converse
  holds if $q$ is onto.
\end{proposition}
\begin{proof}
  The proof follows closely that of
  Proposition~\ref{prop:stabilizers};
  see~\cite{juschenko-mattebon-monod-delasalle:extensive}*{Proposition~2.4}
  for details. Assume that $q^{-1}(y)$ is extensively amenable for all
  $y\in Y$, and let $m_y$ be a $G_y$-invariant mean giving full weight
  to ideals. By making one choice per $G$-orbit, we may also assume
  that $m_{y'}$ is the push-forward by $g$ of $m_y$ whenever $y'=y
  g$. Extend every $m_y$ to a mean on $\mathfrak P_f(X)$; then $(m_y)$
  is a $G$-equivariant map $Y\to\mathscr M(\mathfrak P_f(X))$.

  For every $F=\{y_1,\dots,y_n\}\Subset Y$, we set
  $m_F(S)=m_{y_1}(\{E\cap q^{-1}(y_1)\mid E\in S\})\cdots
  m_{y_n}(E\cap q^{-1}(y_n))$, and note that $m_F$ gives full weight
  to every ideal of the form $\{E\Subset X\mid x\in E\}$ for some
  $x\in q^{-1}(F)$. The map $F\mapsto m_F$ defines a $G$-equivariant
  map $\mathfrak P_f(Y)\to\mathscr M(\mathfrak P_f(X))$. Composing with the
  barycentre $\Upsilon$ as in~\eqref{eq:barycentre}, we obtain a
  $G$-equivariant map
  $m_*\colon\mathscr M(\mathfrak P_f(Y))\to\mathscr M(\mathfrak
  P_f(X))$.

  Assume now that $Y$ is extensively amenable, and let $n$ be a
  $G$-invariant mean on $\mathfrak P_f(Y)$ giving full weight to
  ideals. Set $m\coloneqq m_*(n)$; then $m$ is a $G$-invariant mean on
  $\mathfrak P_f(X)$ giving full weight to ideals, so $X$ is
  extensively amenable.

  Assume finally that $q$ is onto and that $X$ is extensively
  amenable. By Lemma~\ref{lem:extamen}, the $G_y$-subset $q^{-1}(y)$
  of $X$ is extensively amenable for all $y\in Y$. Let $m$ be a mean
  on $\mathfrak P_f(X)$ giving full weight to ideals, and
  define a mean $n$ on $\mathfrak P_f(Y)$ by
  $n(S)=m(\{E\Subset X\mid q(E)\in S\})$. Given $y\in Y$, choose
  $x\in q^{-1}(y)$, and note
  \[n(\{F\Subset Y\mid y\in F\})=m(\{E\Subset X\mid y\in q(E)\})\ge m(\{E\Subset X\mid x\in E\})=1.\qedhere\]
\end{proof}

In particular, let $K\le H\le G$ be groups. Then $K\backslash G$ is an
extensively amenable $G$-set if and only if both $K\backslash H$ and
$H\backslash G$ are extensively amenable. This is in contrast with
Example~\ref{ex:monodpopa}, where the corresponding property is shown
\emph{not} to hold for amenability of sets.

The following proposition relates Definition~\ref{defn:extamen} to the
original definition; we begin by introducing some vocabulary. Let
$\mathbf A$ denote the category of group actions: its objects are
pairs $X\looparrowleft G$ of a set $X$ and an action of $G$ on $X$,
and a morphism $(X\looparrowleft G)\to(Y\looparrowleft H)$ is a pair
of maps $(f\colon X\to Y,\phi\colon G\to H)$ intertwining the actions
on $X$ and $Y$, namely satisfying $f(x)\phi(g)=f(x g)$ for all
$x\in X,g\in G$. We denote by $\mathbf{AA}$ and $\mathbf{EA}$ the
subcategories of amenable, respectively extensively amenable actions.

We are interested in functors
$F\colon\{\text{finite sets, injections}\}\to\mathbf{AA}$, written
$F(X)=F_0(X)\looparrowleft F_1(X)$ for a group $F_1(X)$ and an
$F_1(X)$-set $F_0(X)$. Since amenable actions are closed under
directed unions, and every set is the directed union of its finite
subsets, we get by continuity a functor still written
$F\colon\{\text{sets, injections}\}\to\mathbf{AA}$, called an
\emph{amenable functor}. If furthermore $F$ takes values in
$\mathbf{EA}$ then we call it an \emph{extensively amenable functor}.
We call the functor $F$ \emph{tight} if the map
$F_0(X\setminus\{x\})\to F_0(X)$ is never onto.

We already saw some examples of tight functors: for any amenable group
$A$, the functor $X\mapsto A^{(X)}\looparrowleft A^{(X)}$ since
$A^{(X)}$ is the directed union of its amenable subgroups $A^E$ over
all $E\Subset X$; the functor $X\mapsto \Sym(X)\looparrowleft\Sym(X)$,
by the same reasoning (see Example~\ref{ex:finitary:eg}); and the
functor $X\mapsto X\looparrowleft\Sym(X)$. Note that if $X$ is a
$G$-set then $F_0(X)$ and $F_1(X)$ inherit $G$-actions by
functoriality.

\begin{proposition}[\cite{juschenko-mattebon-monod-delasalle:extensive}*{Theorem~3.14}]\label{prop:F(X)amenable}
  Let $F$ be a functor as above, and let $X$ be a $G$-set. If $X$ is
  extensively amenable and $F$ is amenable then
  $F_0(X)\looparrowleft(G\ltimes F_1(X))$ is amenable, and if
  furthermore $F$ is extensively amenable then
  $F_0(X)\looparrowleft(G\ltimes F_1(X))$ is extensively amenable.

  Conversely, if $F$ is tight and
  $F_0(X)\looparrowleft(G\ltimes F_1(X))$ is amenable then $X$ is
  extensively amenable.
\end{proposition}
\begin{proof}
  Assume first that $F$ is amenable.  For every $E\Subset X$ let
  $m_E\in\mathscr M(F_0(E))^{F_1(E)}$ be an invariant mean, and extend
  it functorially to a mean still written
  $m_E\in\mathscr M(F_0(X))^{F_1(E)}$. By choosing once $m_E$ per
  cardinality class of subsets of $X$, we may ensure that we have
  $f_*(m_E)=m_{E'}$ for every bijection $f\colon E\to E'$. We obtain
  in this manner a $G$-equivariant map
  $\mathfrak P_f(X)\to\mathscr M(F_0(X))$, and therefore, composing with
  the barycentre $\Upsilon$ as in~\eqref{eq:barycentre}, a map
  $\mathscr M(\mathfrak P_f(X))^G\to\mathscr M(F_0(X))^G$.

  By assumption, there exists $m_0\in\mathscr M(\mathfrak P_f(X))^G$
  giving full mass to ideals; let $m$ be the image of $m_0$ under the
  above map. Clearly $m$ is a $G$-invariant mean on $F_0(X)$. It is
  also $F(A)$-invariant for every $A\Subset X$: one may restrict $m_0$
  to $\{E\Subset X\mid A\subseteq E\}$ and still obtain a mean. Every
  $m_E$ is $F_1(E)$-invariant, so is in particular $F(A)$-invariant,
  and therefore $m$ is also $F(A)$-invariant. In summary, $m$ is
  $G\ltimes F_1(X)$-invariant, so $F_0(X)$ is an amenable
  $G\ltimes F_1(X)$-set.

  For the converse, define a $G$-equivariant map
  $\supp\colon F_0(X)\to\mathfrak P_f(X)$ by
  \[\supp(x)=\bigcap\{E\Subset X\mid x\in\operatorname{image}(F_0(E)\to F_0(X))\}.\]
  Assume that $F_0(X)$ is an amenable $G\ltimes F_1(X)$-set, and let
  $m_0$ be a $G$-invariant mean on $F_0(X)$. Let $m$ the push-forward
  of $m_0$ via $\supp$; it is a $G$-invariant mean on
  $\mathfrak P_f(X)$. Choose $x_0\in X$. By definition,
  $m(\{E\Subset X\mid x_0\in E\})=m_0(S)$ for the ideal
  \begin{align*}
    S &= \{x\in F_0(X)\mid x_0\in\supp(x)\}\\
      &= \bigcap_{x_0\not\in E\Subset X}(F_0(X)\setminus\operatorname{image}(F_0(E)\to F_0(X)))\\
      &= F_0(X)\setminus F_0(X\setminus\{x_0\}).
  \end{align*}
  Since $F$ is tight, $S\neq\emptyset$. Furthermore, $m_0$ is
  $F_1(X)$-invariant, so $m_0(S)>0$. We conclude by
  Lemma~\ref{lem:extamen} that $X$ is extensively amenable.

  Finally, to prove that $F_0(X)$ is an extensively amenable
  $G\ltimes F_1(X)$-set whenever $F$ is an extensively amenable
  functor, we apply the converse just proven to the functor
  $H(X)=(\Z/2)^{(X)}\looparrowleft(\Z/2)^{(X)}$. For every $X$, we
  know from the first part of the proof that $H_0(F_0(X))$ is an
  amenable $(G\ltimes F_1(X))\ltimes H_1(F_0(X))$-set, since we
  assumed $F_1(X)\looparrowleft G\ltimes F_1(X)$ is extensively
  amenable. Therefore, the functor
  $X\mapsto H_0(F_0(X))\looparrowleft(F_1(X)\ltimes H_1(F_0(X)))$ is
  amenable, and yet again the second part of the proof allows us to
  deduce that $F_0(X)\looparrowleft G\ltimes F_1(X)$ is extensively
  amenable.
\end{proof}

\noindent A fundamental application of
Proposition~\ref{prop:F(X)amenable} is the following
\begin{corollary}\label{cor:twistedamen}
  Let $H$ be a subgroup of $G\ltimes F(X)$ for some extensively
  amenable $G$-set $X$. If $H\cap(G\times1)$ is amenable, then $H$ is
  amenable too.
\end{corollary}
\begin{proof}
  By Proposition~\ref{prop:F(X)amenable}, $F(X)$ is extensively
  amenable, so by Lemma~\ref{lem:extamen} the $H$-orbit
  $1\cdot H\subseteq X$ is an extensively amenable $H$-set, and is
  therefore amenable by Lemma~\ref{lem:extamen0}. The stabilizers in
  this action are conjugate to $H\cap(G\times1)$, which is amenable by
  assumption, so $H$ is amenable by
  Proposition~\ref{prop:stabilizers}.
\end{proof}

There is also a connection between extensive amenability and
laminability, see Definition~\ref{defn:laminable}: by
Corollary~\ref{cor:amentrans}, if $X\looparrowleft G$ is extensively
amenable then $F(X)\looparrowleft G\ltimes F(X)$ is laminable.

In the next section, we shall see a sufficient condition for an action
to be extensively amenable, and in Example~\ref{ex:iet} an application
to interval exchange transformations.

We finish this section by a very brief summary of the ``only if'' part
of a proof of Theorem~\ref{thm:polygrowth} due to
Kleiner~\cite{kleiner:gromov} and simplified by Tao; we include it
here because it combines amenability and the study of (now unbounded)
harmonic functions.

Let $G$ be a group of polynomial growth; we are to show that $G$ has a
nilpotent subgroup of finite index. We may of course assume that $G$
is infinite, and by induction on the growth degree it suffices to show
that $G$ has a finite-index subgroup mapping onto $\Z$. For that
purpose, it suffices to show that $G$ has an infinite image in some
virtually soluble group. By~\cite{shalom:linear}, every amenable
finitely generated subgroup of $\GL_n(\C)$ is virtually soluble, and
$G$ is amenable by Proposition~\ref{prop:subexp=>amenable}, so it
suffices to construct a representation $G\to\GL_n(\C)$ with infinite
image. The proof uses the following arguments:
\begin{lemma}\label{lem:kleiner:1}
  Let $G$ be a countably infinite amenable group. Then there exists an
  action of $G$ on a Hilbert space $\mathscr H$ with no fixed points.
\end{lemma}
\begin{proof}
  Consider $\mathscr H=\ell^2(\N\times G)$, the space of
  square-summable functions $(f_1,f_2,\dots)$ in $\ell^2(G)$. There is
  a natural, diagonal action of $G$ on $\mathscr H$ by
  right-translation. This action has a fixed point $0$, but we can
  construct an affine action without fixed point as follows.

  Let $(F_n)_{n\in\N}$ be a F\o lner sequence in $G$, and define
  $h=(\mathbb 1_{F_n}/\sqrt{\#F_n})_{n\in\N}$. Then
  $h\not\in\mathscr H$, but $h-h g\in\mathscr H$ for all $g\in G$,
  using the almost-invariance of $(F_n)$. We let $G$ act on
  $\mathscr H$ by $f\cdot g=f g+h-h g$, namely we move the fixed point
  to $h$.
\end{proof}

The main result, whose proof we omit, is the following control on the
growth of harmonic functions. It follows easily from Gromov's theorem,
but Kleiner gave a direct and elementary proof of it:
\begin{lemma}\label{lem:kleiner:3}
  Let $G$ be a group of polynomial growth, and let $\mu$ be a measure
  on $G$. Then for every $d\in\N$ the vector space of harmonic maps
  $u\colon G\to\C$ of growth degree at most $d$ (namely for which
  there is a constant $C$ with $|u(g)|\le C|g|^d$ for all $g\in G$) is
  finite-dimensional.\qed
\end{lemma}

The proof of Theorem~\ref{thm:polygrowth} is then finished: a group
$G$ of polynomial growth is amenable, so by Lemma~\ref{lem:kleiner:1}
it has an affine, fixed-point-free action on a Hilbert space
$\mathscr H$. Let $\mu$ be SRW on $G$, and define
\[E\colon\mathscr H\to\R_+,\qquad v\mapsto\frac{1}{2} \sum_{s \in
    S}\mu(s) \|v s-v\|^2.
\]
Since $\mathscr H$ has no fixed point, $E(v)>0$ for all
$v\in\mathscr H$. Let us assume that $E(v)$ attains its minimum ---
this can be achieved by considering a sequence of better and better
approximations to a minimum in an ultrapower of $\mathscr H$ --- and
call its minimum $h$. One directly sees from
$\partial E(v)/\partial v|_h=0$ that $h$ is $\mu$-harmonic, and it is
not constant. Then
$V\coloneqq\{\langle h|v\rangle\mid v\in\mathscr H\}$ is a vector
space of Lipschitz harmonic maps, so is finite-dimensional by
Lemma~\ref{lem:kleiner:3}, and $G$'s action on $V$ has infinite image
because $V$ has no non-zero $G$-fixed point.

\subsection{Recurrent actions}
We saw in Proposition~\ref{prop:subexp=>amenable} that actions on
subexponentially-growing spaces are amenable; and in
Theorem~\ref{thm:kesten} that random walks on graphs in which the
probability of return to the origin decays subexponentially give
amenable actions. We see here that stronger conditions --- quadratic
growth, recurrent random walks --- produce extensively amenable
actions.

Let $p_1\colon X\times X\to[0,1]$ be a random walk on a set $X$. It is
\emph{recurrent} at $x\in X$ if $\sum_{n\ge0}p_n(x,x)=\infty$, namely
if a random walk started at $x$ is expected to return infinitely often
to $x$, and equivalently if it is certain to return to $x$. It is
\emph{transient} if it is not recurrent.

We computed in Example~\ref{ex:rwZ} that the probability of return in
to the origin in $n$ steps of SRW on $\Z$ is $\propto n^{-1/2}$; so
the probability of return to the origin on $\Z^d$ is
$\propto n^{-d/2}$. It follows that SRW on $\Z^d$ is recurrent
precisely for $d\le2$.

\begin{lemma}\label{lem:asymprecurrent}
  The random walk $p$ is recurrent if and only if for every $x\in X$
  there exists a sequence of functions $(a_n)$ in $\ell^2(X)$ with
  $a_n(x)=1$ and $\|a_n-T a_n\|\to0$, for $T$ the associated random
  walk operator.
\end{lemma}
\begin{proof}
  For a function $\phi\in\ell^2(X)$, define its \emph{Dirichlet norm}
  as $D(\phi)=\|d\phi\|^2=\frac12\sum_{x,y\in
    X}(f(x)-f(y))^2p(x,y)$. The claim is equivalent to requiring the
  existence of functions $a_n\in\ell^2(X)$ with $a_n(x)=1$ and
  arbitrarily small Dirichlet norm.  If $X$ is finite, there is
  nothing to do, as the functions $a_n\equiv1$ have $D(a_n)=0$.

  Choose $x\in X$.  Assume first that $p$ is transient, so that
  $G(y)\coloneqq\sum_{n\ge0} p_n(x,y)$ is well-defined. Then for all
  $\phi\in\ell^2(X)$ we have
  \[\langle d\phi,d G\rangle=\langle\phi, d^*d G\rangle=\phi(x),\]
  and $|\langle d\phi,d G\rangle|^2\le D(\phi)D(g)$, so
  $D(\phi)\ge\phi(1)/D(g)$ is bounded away from $0$.

  Assume next that $p$ is recurrent. For every $n\in\N$, set
  $G_n(y)=\sum_{m=0}^n p_m(x,y)$ and $a_n(y)=G_n(y)/G_n(x)$. Since by
  assumption $G_n(x)\to\infty$, the functions $a_n(y)$ satisfy the
  requirement.
\end{proof}

For random walks with finite range, the following criterion due to
Nash-Williams is very useful.  Let $p$ be a transitive random walk on
a set $X$, and let $x\in X$ be a basepoint. A \emph{slow constriction}
of $X$ is a family $\{x\}=V_0\subset V_1\subset\cdots$ of finite
subsets of $X$, such that $\bigcup V_n=X$ and $p_1(V_m,V_n)=0$
whenever $|m-n|\ge2$ and $\sum_{n\ge0}p_1(V_n,V_{n+1})^{-1}=\infty$. A
\emph{refinement} of $p$ is the random walk on a set obtained by
subdividing arbitrarily each transition $p_1(x,y)$ by inserting
midpoints along it.
\begin{theorem}[Nash-Williams~\cite{nash-williams:electric}]\label{thm:nashwilliams}
  Let $p$ be a transitive random walk on a set $X$. Then $p$ is
  recurrent if and only if it has a refinement admitting a slow
  constriction.
\end{theorem}
The result applies to $\Z^d$ for $d\le2$: the sets $V_n$ may be chosen
as $\{-n,\dots,n\}^d$.  We only prove the ``only if'' direction, which
is the important direction for us.
\begin{proof}[First proof of Theorem~\ref{thm:nashwilliams}, ``only if'' direction]
  Given a constriction $(V_n)$, set $c_n=1/p_1(V_n,V_{n+1})$, and
  define an associated random walk $q$ on $\N$ by
  $q_1(n,n+1)=c_n/(c_n+c_{n-1})$ and
  $q_1(n,n-1)=c_{n-1}/(c_n+c_{n-1})$. It is easy to check
  $\sum_n q_n(1,1)=\infty$ if the constriction is slow.
\end{proof}

We shall give another proof of the ``only if'' direction, based on
Lemma~\ref{lem:asymprecurrent}; we begin by a simple
\begin{lemma}
  Let $\sum_i v_i$ be a positive, divergent series. Then there exist
  $\lambda_{i,n}\ge0$ such that $\sum_i\lambda_{i,n}v_i=1$ for all $n$
  and $\sum_i\lambda_{i,n}^2v_i\searrow0$ as $n\to\infty$.
\end{lemma}
\begin{proof}
  Let $\alpha_n=1+1/n$ be a decreasing sequence converging to
  $1$. Group the terms in $\sum v_i$ into blocks $w_1+w_2+\cdots$ such
  that $w_i\ge1$ for all $i$. Set
  \[\lambda_{i,n}=\frac{\alpha_n-1}{w_k\alpha_n^{k-1}}\text{ if $v_i$ belongs to the block }w_k.\qedhere\]
\end{proof}

\begin{proof}[Second proof of Theorem~\ref{thm:nashwilliams}, ``only if'' direction]
  Let $(V_i)_{i\ge1}$ be a slow constriction of $X$ with basepoint
  $x$, and set $v_i\coloneqq1/p_1(V_i,V_{i+1})$. Apply the lemma to
  the divergent series $\sum v_i$, and define maps
  $a_n\colon X\to[0,1]$ by
  \[a_n(y)=1-\sum_i\lambda_{i,n} v_i\mathbb 1_{y\not\in V_i}.
  \]
  Then $a_n$ has finite support so in particular belongs to
  $\ell^2(X)$; and $a_n(x)=1$ because $x\in V_i$ for all $i$; and
  $\|a_n-a_n g\|_2\to0$ for all $g\in G$ because
  $\sum\lambda_{i,n}^2 v_i\to0$.
\end{proof}

The main result of this section is the following. We will prove it in
two different manners, and in fact in this manner recover the ``if''
direction of Theorem~\ref{thm:nashwilliams}:
\begin{theorem}\label{thm:recurrent=>extamen}
  If $X$ is a $G$-set with a non-degenerate recurrent random walk,
  then $X$ is extensively amenable.
\end{theorem}

We begin with some preparation for the proof. Let $\mu$ be a
symmetric, non-degenerate measure on a group $G$, and let $X$ be a
$G$-set. For a basepoint $x\in X$ and a trajectory
$x,x g_1, x g_1 g_2,\dots$ of the random walk on $X$, the
corresponding length-$n$ \emph{inverted orbit} is the random subset
\[O_n=\{x,x g_n, x g_{n-1} g_n,\dots, x g_1\cdots g_n\}.
\]
If $X$ is transitive, then $\#O_n$ depends only mildly on the choice
of $x$.

\begin{proposition}\label{prop:invorbitextamenable}
  Let $X$ be a transitive $G$-set and let $\mu$ be a symmetric,
  non-degenerate probability measure on $G$. Then $X$ is extensively
  amenable if and only if
  \begin{equation}\label{eq:expinvorbit}
    \lim_{n\to\infty}\frac{-1}{n}\log\mathbb E(2^{-\#O_n})=0.
  \end{equation}
\end{proposition}
\begin{proof}
  Thanks to Proposition~\ref{prop:F(X)amenable}, it is enough to prove
  that~\eqref{eq:expinvorbit} is equivalent to the amenability of the
  $G\ltimes(\Z/2)^{(X)}$-set $(\Z/2)^{(X)}$. Choose a basepoint
  $x\in X$, and consider on $G\ltimes(\Z/2)^{(X)}$ the probability
  distribution
  $\nu\coloneqq\frac12(1+\delta_x)*\mu*\frac12(1+\delta_x)$, called
  the ``switch-walk-switch'' measure: in the action on $(\Z/2)^{(X)}$,
  it amounts to randomizing the current copy of $\Z/2$, moving to
  another position in $X$, and randomizing the new copy of $\Z/2$. By
  Kesten's Theorem~\ref{thm:kesten}, amenability of the action on
  $(\Z/2)^{(X)}$ is equivalent to subexponential decay of return
  probabilities of a random walk $(f_0=1,f_1,f_1 f_2,\dots)$ on
  $(\Z/2)^{(X)}$, namely to
  $\lim_{n\to\infty}\frac{-1}{n}\log\mathbb P(f_1\cdots f_n=1)=0$.
  Now the support of $f_1\cdots f_n$ is contained in $O_n$: writing
  each $f_i=\delta_x^{\epsilon_i}g_i\delta_x^{\zeta_i}$ with
  $g_i\in G$, we get
  $f_1\cdots f_n=g_1\cdots g_n\delta_{x g_1\cdots
    g_n}^{\epsilon_1}\delta_{x g_2\cdots
    g_n}^{\zeta_1+\epsilon_2}\cdots\delta_x^{\zeta_n}$; and $f_n$
  randomizes every copy of $\Z/2$ indexed by $O_n$, so
  $\mathbb P(f_n=1)=\mathbb E(2^{-\#O_n})$.
\end{proof}

\begin{lemma}\label{lem:sublinIO}
  Let $p$ be a transitive random walk on a $G$-set $X$ driven by a
  symmetric probability measure $\mu$. Then $X$ is recurrent if and
  only if $\lim\frac1n\mathbb E(\#O_n)=0$.
\end{lemma}
\begin{proof}
  Choose a basepoint $x\in X$ for the random walk $(x=x_0,x_1,\dots)$,
  and define the random variable $\Theta=\min\{n\ge1\mid x_n=x\}$. Then
  \begin{align*}
    \mathbb E(\#O_{n+1}-\#O_n) &= \mathbb P(x g_{n+1}\not\in O_n)\\
    &= \mathbb P(x g_{n+1}\not\in\{x,x g_n,x g_{n-1}g_n,\dots,x g_1\cdots g_n\})\\
    &= \mathbb P(\{x g_{n+1},x g_{n+1}g_n^{-1},x g_{n+1}g_n^{-1}g_{n-1}^{-1},\cdots,x g_{n+1}g_n^{-1}\cdots g_1^{-1}\}\not\ni x)\\
    &= \mathbb P(\Theta>n+1),
  \end{align*}
  because the random walk with increments
  $g_{n+1},g_n^{-1},\dots,g_1^{-1}$ has the same law as
  $\mu^n$. Therefore, $\mathbb E(\#O_n)/n\to\mathbb P(\Theta=\infty)$,
  which vanishes if and only if $X$ is recurrent.
\end{proof}

\begin{proof}[First proof of Theorem~\ref{thm:recurrent=>extamen}]
  We may assume, by Lemma~\ref{lem:extamen}, than $X$ is
  transitive. Let $p$ be a non-degenerate, transitive, recurrent
  random walk on $X$. By Lemma~\ref{lem:sublinIO}, we have
  $\frac1n\mathbb E(\#O_n)\to0$, so by convexity
  \[\frac{-1}{n}\log\mathbb E(2^{-\#O_n})\le\frac1n\mathbb E(\#O_n)\log2\to0,\]
  so $X$ is extensively amenable by
  Proposition~\ref{prop:invorbitextamenable}.
\end{proof}

\begin{proof}[Second proof of Theorem~\ref{thm:recurrent=>extamen}]
  Let $x\in X$ be arbitrary.  We start, using
  Lemma~\ref{lem:asymprecurrent}, with a sequence of functions $(a_n)$
  in $\ell^2(X)$ satisfying $a_n(x)=1$ and $\lim\|a_n-a_n g\|=0$ for
  all $g\in G$. (This is also the outcome of the second proof of
  Theorem~\ref{thm:nashwilliams}). We construct then maps
  $b_n\colon\mathfrak P_f(X)\to[0,1]$ by
  \[b_n(E)=\prod_{y\in E}a_n(y).
  \]
  They are finitely supported, and therefore may be viewed in
  $\ell^2(\mathfrak P_f(X))$. It remains to check that they are almost
  invariant under the action of $\Z/2\wr_X G$. Assuming that $X$ is
  transitive, this last group is generated by
  $\delta_x\colon X\to\Z/2$ and $G$. We have $b_n\delta_x=b_n$,
  because $b_n(E)=b_n(E\triangle\{x\})$.

  The spaces $\ell^2(\mathfrak P_f(X))$ and $\bigotimes_X\ell^2(C^2)$
  are isometric; the isometry is the obvious one mapping $\delta_E$ to
  $\bigotimes_{x\in X}\delta_{x\in E}$, if we take
  $\{\delta_{\text{false}},\delta_{\text{true}}\}$ as basis of
  $\ell^2(\C^2)$. We compute
  \[\|b_n\|^2=\langle b_n,b_n\rangle=\prod_{y\in X}(1^2+a_n(y)^2),\]
  and for $g\in G$ we similarly have
  $\langle b_n,b_n g^{-1}\rangle=\prod_{y\in X}(1+a_n(y)a_n(y g))$, so
  \[\bigg(\underbrace{\frac{\langle b_n,b_n\rangle}{\langle b_n,b_n g^{-1}\rangle}}_A\bigg)^2=\prod_{y\in X}\underbrace{\frac{(1+a_n(y)^2)(1+a_n(y g)^2)}{(1+a_n(y)a_n(y g))^2}}_B.\]
  Taking logarithms, and using the approximation $\log(t)\le t-1$,
  \[0\le2\log(A)\le\sum_{y\in X}log(B)\le\sum_{y\in X}\frac{(a_n(y)-a_n(y g))^2}{(1+a_n(y)a_n(y g))^2}\le\|a_n-a_n g^{-1}\|^2\to0,
  \]
  so $\langle b_n,b_n\rangle/\langle b_n,b_n g^{-1}\rangle\to1$ and
  therefore $\|b_n-b_n g\|\to0$.
\end{proof}

\begin{example}\label{ex:iet}
  An \emph{interval exchange} is a piecewise-translation self-map of the
  circle. More precisely, it is a right-continuous map
  $g\colon\R/\Z\righttoleftarrow$ such that $\sphericalangle(g)\coloneqq\{g(x)-x\mid
  x\in\R/\Z\}$ is finite.

  The rotation $x\mapsto x+\alpha$ is an extreme example of interval
  exchange.\footnote{The name ``interval exchange'' comes from opening
    up the circle into an interval $[0,1]$; the rotation on the circle
    may be viewed as an exchange of two intervals
    $[0,1-\alpha]\mapsto[\alpha,1],[1-\alpha,1]\mapsto[0,\alpha]$.}
  The interval exchange transformations naturally form a group $\IET$
  acting on $\R/\Z$; and every countable subgroup $G\le\IET$ can be
  made to act on the Cantor set by letting $\mathscr D$ be the union
  of the $G$-orbits of discontinuity points of $G$ (or of $0$ if all
  elements of $G$ are rotations) and replacing $\R/\Z$ by
  \[X\coloneqq(\R/\Z\setminus\mathfrak D)\cup(\mathfrak D\times\{+,-\}),
  \]
  namely by opening up the circle at every point of $\mathfrak D$;
  see~\cite{keane:iet}*{\S5}.
\end{example}

Little in known on the group $\IET$; in particular, it is not known whether
it contains non-abelian free groups, or whether it is amenable. We prove:
\begin{theorem}[\cite{juschenko-mattebon-monod-delasalle:extensive}*{Theorem~5.1}]\label{thm:iet}
  Let $\Lambda\le\R/\Z$ be a finitely generated subgroup with free rank at
  most $2$, namely $\dim(\Lambda\otimes\Q)\le2$. Then
  \[\IET(\Lambda)\coloneqq\{g\in\IET\mid
  \sphericalangle(g)\subseteq\Lambda\}
  \]
  is an amenable subgroup of $\IET$.
\end{theorem}
\begin{proof}
  We first prove that the action of $\IET(\Lambda)$ on $\R/\Z$ is
  extensively amenable. Choose a finite generating set for $\Lambda$; then
  the Cayley graph of $\Lambda$ is quasi-isometric to $\Z^d$ for $d\le2$,
  and in particular is recurrent. Let $G=\langle S\rangle$ be a finitely
  generated subgroup of $\IET(\Lambda)$. For $x\in\R/\Z$, the orbit $x G$
  injects into $\Lambda$ under the map $y\mapsto y-x$, and this map is
  Lipschitz with Lipschitz constant $\max_{s\in
    S}\max_{\lambda\in\sphericalangle(s)}\|\lambda\|$, so the Schreier graph of $x
  G$ is recurrent. Theorem~\ref{thm:recurrent=>extamen} implies that $x G$
  is an extensively amenable $G$-set, so $\R/\Z$ is an extensively amenable
  $\IET(\Lambda)$-set by Lemma~\ref{lem:extamen}.

  We wish to apply Corollary~\ref{cor:twistedamen} to $X=\R/\Z$ and
  $G=\IET(\Lambda)$ and $F(X)=X\looparrowleft\Sym(X)$. Given an
  interval exchange map $g\in\IET$, let $\tilde g$ be the unique
  left-continuous self-map of $\R/\Z$ that coincides with $g$ except
  at its discontinuity points, and let
  $\tau_g=g^{-1}\tilde g\in\Sym(\R/\Z)$ be the corresponding
  permutation of the discontinuity points of $g^{-1}$. We have a
  cocycle identity $\tau_{gh}=\tau_g^h\tau_h$, so the map
  \[\iota\colon\begin{cases}\IET &\to\IET\ltimes\Sym(\R/\Z)\\
    g &\mapsto (g,\tau_g)\end{cases}
  \]
  is an embedding. Observe that $\tau_g=1$ if and only if $g$ is
  continuous, namely is a rotation. Therefore,
  $\iota(\IET(\Lambda))\cap(\IET(\Lambda)\times1)=\Lambda$ consists of
  rotations, so is amenable. We deduce by Corollary~\ref{cor:twistedamen}
  that $\IET(\Lambda)$ is amenable.
\end{proof}

\subsection{Topological full groups}\label{ss:topfull}
We apply the results from the previous sections to exhibit a wide
variety of amenable groups.

We begin by a fundamental construction. Let $G$ be group acting on a
compact set $X$. The associated \emph{topological full group} is the
group $[[G,X]]$ of piecewise-$G$ homeomorphisms of $X$:
\[[[G,X]]=\{\phi\colon X\righttoleftarrow\mid\exists\nu\colon X\to G\text{ continuous with }\phi(x)=x\nu(x)\text{ for all }x\}.
\]
Note that $\nu$ takes finitely many values since it is a map from a
compact set to a discrete set. If we suppose $X$ discrete rather than
compact, then $[[G,X]]$ becomes the group of bijective $G$-wobbles of
$X$ that we saw in~\S\ref{ss:doubling}.  The connection is even more
direct: let $x\in X$ be such that its orbit $x G$ is dense in
$X$. Then $[[G,X]]$ acts faithfully on the orbit $x G$ by $G$-wobbles.

The natural setting for the definition of the topological full group
is that of \emph{groupoids of germs}. We recall the basic notions:
\begin{definition}\label{defn:groupoid}
  A \emph{groupoid} is a set $\mathfrak G$ with source and range maps
  $s,r\colon\mathfrak G\righttoleftarrow$, with an associative
  multiplication $\gamma_1\gamma_2$ defined whenever
  $r(\gamma_1)=s(\gamma_2)$, and with an everywhere-defined inverse
  satisfying $\gamma\gamma^{-1}=s(\gamma)=r(\gamma^{-1})$. Its
  \emph{set of units} is the subset $\mathfrak G_0$ of elements of the
  form $\gamma\gamma^{-1}$. The groupoid $\mathfrak G$ is called
  \emph{topological} if $\mathfrak G$ is a topological space and the
  multiplication and inverse maps are continuous. Note that for every
  $x\in\mathfrak G_0$ the subset
  $\mathfrak G_x\coloneqq\{\gamma\in\mathfrak G\mid
  s(\gamma)=r(\gamma)=x\}$ is a group, called the \emph{isotropy group}
  of $\mathfrak G$ at $x$.
\end{definition}

A fundamental example is given by a $G$-set $X$: the associated
groupoid is $X\times G$ as a set, with $s(x,g)=x$ and $r(x,g)=x g$ and
$(x,g)(x g,h)=(x,g h)$ and $(x,g)^{-1}=(x g,g^{-1})$. One writes this
groupoid as $X\rtimes G$ and calls in the \emph{action groupoid} of
$X\looparrowleft G$.

Another example is given by the groupoid of germs,
see~\S\ref{ss:fgfg}. Let $X\rtimes G$ be an action groupoid, and
declare $(x,g)\sim(y,h)$ when $x=y$ and there exists an open
neighbourhood of $x$ on which $g$ and $h$ agree. The set of
equivalence classes $\mathfrak G$ is called the \emph{groupoid of
  germs} of $X\rtimes G$.

\begin{definition}
  Let $\mathfrak G$ be a groupoid of germs, and let $\mathfrak G_0$ be
  its space of units. A \emph{bisection} is a subset $F$ of
  $\mathfrak G$ such that $s,r\colon F\to\mathfrak G_0$ are
  homeomorphisms. Note in particular that bisections are open and
  closed. Bisections may be composed and inverted, qua subsets of
  $\mathfrak G$.  The \emph{full group} $[[\mathfrak G]]$ of a
  groupoid $\mathfrak G$ is the group of its bisections.
\end{definition}
Note that the topological full group of the groupoid of germs of the
action of a group $G$ coincides with the earlier definition of
topological full group. It is more convenient to consider the full
group of a groupoid of germs, because it is defined only in terms of
local homeomorphisms, and not of the global action of a group.

\begin{theorem}[see~\cite{juschenko-nekrashevych-delasalle:extensions}*{Theorem~11}]\label{thm:extensionsfullgp}
  Let $X$ be a $G$-topological space, let $\mathfrak G$ denote the
  groupoid of germs of $X$, and let $\mathfrak H$ be a groupoid of
  germs of homeomorphisms of $X$. Assume that
  \begin{enumerate}
  \item $G$ is finitely generated;
  \item At every $x\in X$ the group of germs $\mathfrak G_x$ is
    amenable;
  \item For every $g\in G$, there are only finitely many $x\in X$ such
    that $(x,g)\not\in\mathfrak H$, and then for each of these $x$ the
    action of $G$ on $x G$ is extensively amenable;
  \item The topological full group $[[\mathfrak H]]$ is amenable.
  \end{enumerate}
  Then $G$ is amenable, and if $X$ is compact then $[[\mathfrak G]]$
  is amenable too.
\end{theorem}
\begin{proof}
  Let $P$ be the space of ``finitely-supported sections of
  $\mathfrak H\backslash\mathfrak G$'': the quotient
  $\mathfrak H\backslash\mathfrak G$ is the set of equivalence classes
  in $\mathfrak G$ under $\gamma\sim\delta\gamma$ for all
  $\gamma\in\mathfrak G,\delta\in\mathfrak H$, and
  \[P=\{\phi\colon X\to\mathfrak H\backslash\mathfrak G\text{ finitely
      supported }\mid s(\phi(x))\in x\mathfrak H\text{ for all }x\in X\}.
  \]
  There is a natural action of $G$ on $P$, by
  $(\phi g)(x)=\phi(x)\cdot(t(\phi(x)),g)$.

  We claim that $P$ is an amenable $G$-set. For this, note first that
  there are only finitely many $G$-orbits in $X$ at which at element
  of $P$ can possibly be non-trivial: let $S$ be a finite generating
  set for $G$; then for every $s\in S$ there is a finite subset
  $\Sigma_s\subseteq X$ at which $(x,s)\not\in\mathfrak H$, so if
  $(x,g)\not\in\mathfrak H$ for some $g=s_1\dots s_n$ then
  $(x s_1\dots s_{i-1},s_i)\not\in\mathfrak H$ for some $i$ and
  therefore $x\in\Sigma_{s_i}G$ for some $i$.

  The $G$-set $P$ is naturally the direct product, with diagonal
  action, of its restrictions to the finitely many $G$-orbits in $X$
  at which $P$ can possibly be non-trivial. We therefore restrict
  ourselves to a single $G$-orbit $Y\subseteq X$, and the
  corresponding image $P_Y=\{\phi\colon Y\to\mathfrak H\backslash G\}$
  of $P$.

  Let us choose, for every $y,z\in Y$, an element
  $f_{y,z}\in\mathfrak G$ with $s(f_{y,z})=y$ and $r(f_{y,z})=z$,
  taking $f_{z,y}=f_{y,z}^{-1}$ and $f_{y,y}=1$.  Choose also a
  basepoint $x\in Y$. We have a ``twisted'' embedding
  $\iota\colon G\to \mathfrak G_x\wr_Y G$ given by
  $g\mapsto((y\mapsto f_{x,y}(y,g)f_{y g,x}),g)$.  Note that $P_Y$ is
  isomorphic, qua $G$-set, to
  $\bigsqcup_Y\mathfrak H_x\backslash\mathfrak G_x$ with
  natural action of $\iota(G)$.

  Now since $\mathfrak G_x$ is amenable, we have a functor
  $\{\text{finite sets}\}\to\{\text{amenable groups}\}$ given by
  $E\mapsto\mathfrak G_x^{(E)}$; since $X$ and therefore $Y$ are
  extensively amenable, Proposition~\ref{prop:F(X)amenable} implies
  that $\bigsqcup_Y\mathfrak G_x$ is an amenable
  $\mathfrak G_x\wr_Y G$-set, and \emph{a fortiori} so is its quotient
  $P$.

  We next prove that the stabilizer $G_\phi$ of every $\phi\in P$ is
  amenable. Let $\{v_1,\dots,v_n\}$ be the support of $\phi$, and set
  $K=G_\phi\cap G_{v_1}\cap\cdots\cap G_{v_n}$. We have a natural
  homomorphism
  $K\to\mathfrak G_{v_1}\times\cdots\times\mathfrak G_{v_n}$ to an
  amenable group, whose kernel is contained in $[[\mathfrak H]]$; so
  $K$ is amenable. Iteratively applying
  Proposition~\ref{prop:stabilizers} proves that
  $G_\phi\cap G_{v_1}\cap\cdots\cap G_{v_i}$ is amenable for all
  $i=n,n-1,\dots,0$.

  We apply once more Proposition~\ref{prop:stabilizers} to deduce that
  $G$ is amenable. Finally, the full group $[[\mathfrak G]]$ is the
  union of groups generated by finite sets of bisections, to which the
  theorem applies, so $[[\mathfrak G]]$ itself is amenable.
\end{proof}

\begin{example}[\cite{juschenko-mattebon-monod-delasalle:extensive}*{Theorem~6.1}]\label{ex:herednotext}
  Consider the ``Frankenstein group'' $H(\mathbb A)$ from
  Theorem~\ref{thm:monod}. Then the action of $H(\mathbb A)$ on $\R$
  is hereditarily amenable, but is not extensively amenable.

  Indeed, consider first $H\le H(\mathbb A)$ and any $x\in\R$, and set
  $m\coloneqq\inf(x H)\in\R\cup\{\infty\}$, as at the end of the proof
  of Theorem~\ref{thm:monod}. Every element of $H''$ acts trivially in
  a neighbourhood of $m$. Consider a sequence $(x_n)$ in $\R$
  converging to $m$; then any cluster point of the sequence of
  measures $(\delta_{x_n})$ is an $H''$-invariant mean on $x H$. Since
  $H/H''$ is amenable, there is also an $H$-invariant mean on $x H$.

  On the other hand, since $H(\mathbb A)$ is not amenable there exists
  a non-amenable finitely generated subgroup $G\le H(\mathbb A)$, and
  Theorem~\ref{thm:extensionsfullgp} should \emph{not} apply to $G$
  with $\mathfrak H$ the groupoid of germs of the action of
  $\PSL_2(\R)$ on $\R\cup\{\infty\}$. However, the first condition is
  satisfied by assumption, the second one is satisfied because the
  group of germs at $x\in\R$ is at most
  $\operatorname{Affine}(\R)\times\operatorname{Affine}(\R)$, and the
  fourth one is satisfied because projective transformations are
  analytic, so their germs coincide with point stabilizers, namely
  with $\operatorname{Affine}(\R)$. Therefore, the third condition
  fails, so there exists $x\in\R$ such that the action of $G$ on $x G$
  is not extensively amenable.
\end{example}

We now specialize the results to $X$ a Cantor set, and more precisely
the Cantor set of paths in a specific kind of graph:
\begin{definition}[\cite{bratteli:diagrams}; see~\cite{durand:cant}]
  A \emph{Bratteli diagram} is a directed graph $\mathcal D=(V,E)$
  along with decompositions $V=\bigsqcup_{i\ge0} V_i$ and
  $E=\bigsqcup_{i\ge1} E_i$ in non-empty finite subsets, such that
  $e^-\in V_{i-1}$ and $e^+\in V_i$ for all $e\in E_i$. For $v\in V$
  we denote by $X_v$ the set of paths starting at $V_0$ and ending at
  $v$; by $X_n=\bigcup_{v\in V_n}X_v$ the set of paths of length $n$
  starting at $V_0$; and by $X$ the set of infinite paths starting at
  $V_0$.

  If for any $n\gg m$ there exists a path from every vertex in $V_m$
  to every vertex in $V_n$, the diagram is called \emph{simple}.
\end{definition}

For $e=(e_1,\dots,e_n)\in X_n$, we denote by $e X$ the set of paths
beginning with $e$; it is a basic open set for the topology on $X$,
which turns $X$ into a compact, totally disconnected space. If
$\mathcal D$ is simple then $X$ has no isolated points, so is a Cantor
set.

For two paths $e,f\in X_v$ for some $v\in V_n$ we define a
homeomorphism $T_{e,f}\colon e X\to f X$ by
\[T_{e,f}(e,e_{n+1},\dots)=(f,e_{n+1},\dots)\text{ for all }e_i\in E_i.
\]
Denote by $\mathfrak T$ the groupoid of germs of all homeomorphisms of
$T_{e,f}$. It coincides with the \emph{tail equivalence groupoid} of
$\mathcal D$:
\[\mathfrak T=\{(e,f)\in X\times X\mid e=(e_i)_{i\ge1},f=(f_i)_{i\ge1},\text{ and }e_i=f_i\text{ for all $i$ large enough}\},
\]
with the obvious groupoid structure $s(e,f)=e$, $r(e,f)=f$, and
$(e,f)\cdot(f,g)=(e,g)$. The topology on $\mathfrak T$ has as basic
open sets $\{\text{germs of }T_{e,f}\}$.

Let us describe the topological full group $[[\mathfrak T]]$. Every
$g\in[[\mathfrak T]]$ acts locally like $T_{e,f}$ for some $v\in X_n$
and some $e,f\in X_v$; since $X$ is compact, there exists a common
$n(g)\in\N$, assumed minimal, for all these local actions. Write
$[[\mathfrak T]]_n=\{g\in[[\mathfrak T]]\mid n(g)\le n\}$; then
$[[\mathfrak T]]_n$ is a group, and is in fact isomorphic to
$\prod_{v\in V_n}\Sym(X_v)\le\Sym(X_n)$, since every
$g\in[[\mathfrak T]]_n$ is uniquely determined by the rule
$(e,e_{n+1},\dots)^g=(e^g,e_{n+1},\dots)$. It follows that
$[[\mathfrak T]]=\bigcup_{n\ge0}[[\mathfrak T]]_n$ is a locally finite
group.

\begin{definition}[\cite{juschenko-nekrashevych-delasalle:extensions}]
  Consider a homeomorphism $a\colon X\righttoleftarrow$. For $v\in V_n$
  denote by $\alpha_a(v)$ the number of paths $e\in X_v$ such that
  $a\restrict{e X}$ does not coincide with a transformation of the form
  $T_{e,f}$ for some $f\in X_v$. The homeomorphism $a$ is called of
  \emph{bounded type} if $\|a\|\coloneqq\sup_{v\in V}\alpha_a(v)$ is finite
  and there are only finitely many points $x\in X$ at which the germ
  $(a,x)$ does not belong to $\mathfrak T$.
\end{definition}

It is easy to see that the set of bounded-type self-homeomorphisms of
$X$ forms a group. The following result produces a wide variety of
amenable groups:
\begin{theorem}[\cite{juschenko-nekrashevych-delasalle:extensions}*{Theorem~16}]
  Let $\mathcal D$ be a Bratteli diagram, and let $G$ be a group of
  homeomorphisms of bounded type of $X$. If the groupoid of germs of
  $G$ has amenable isotropy groups, then $G$ is amenable.
\end{theorem}
\begin{proof}
  We may assume without loss of generality that $G$ is finitely
  generated.  We apply Theorem~\ref{thm:extensionsfullgp} with
  $\mathfrak H=\mathfrak T$; since $[[\mathfrak T]]$ is locally
  finite, it is amenable. The only condition to check is that the
  action of $G$ on $X$ is extensively amenable; we prove that it is
  recurrent and apply Theorem~\ref{thm:recurrent=>extamen}.

  Consider therefore an orbit $x G$ of $G$, and a finite generating
  set $S$ of $G$. We will in fact prove that the simple random walk on
  $x G$ admits a slow constriction, and apply
  Theorem~\ref{thm:nashwilliams}.

  The Schreier graph of the orbit $x G\subset X$ is an $S$-labelled
  graph. In it, remove all edges $y\to y s$ such that the germ $(y,s)$
  does not belong to $\mathfrak T$. By assumption, only finitely many
  edges were removed, so the resulting graph has finitely many
  connected components; let $P\subseteq x G$ be a choice of one point
  per connected component. We have covered $x G$ by finitely many
  $\mathfrak T$-orbits. For $e=(e_i)_{i\ge1}\in P$ consider
  \[F_{n,e}=\{(a_1,a_2,\dots,a_n,e_{n+1},\dots)\in x G\mid a_1\in E_1,\dots,a_n\in E_n\},
  \]
  and set $F_n=\bigcup_{e\in P}F_{n,e}$. The $F_n$ are finite subsets of $x
  G$, and $x G=\bigcup F_n$. For $e\in P,s\in S$, there are at most
  $\alpha_s(e_n^+)$ paths $f\in F_{n,e}$ with $f s\not\in F_{n,e}$; so
  $\#(F_n S\setminus F_n)\le\#P\cdot\#S\cdot\max_{s\in S}\|s\|$ are
  bounded. Furthermore the $F_n S\setminus F_n$ may be assumed disjoint by
  passing to a subsequence.
\end{proof}

\begin{definition}[\cite{durand:cant}*{Definition~6.3.2}]
  A \emph{Bratteli-Vershik diagram} is a Bratteli diagram
  $\mathcal D=(V,E)$ together with a partial order $\le$ on $E$ such
  that $e,f$ are comparable if and only if $e^+=f^+$. For every
  $v\in V$ there is an induced linear order on $X_v$: if
  $e=(e_1,\dots,e_n),f=(f_1,\dots,f_n)\in X_v$ then $e\le f$ if and
  only if $e_i\le f_i,e_{i+1}=f_{i+1},\dots,e_n=f_n$ for some
  $i\in\{1,\dots,n\}$. We let $X^{\max}$ denote those
  $e=(e_1,\dots)\in X$ such that $(e_1,\dots,e_n)$ is maximal for all
  $n\in\N$, define $X^{\min}$ similarly, and say $\mathcal D$ is
  \emph{properly ordered} if $\#X^{\max}=\#X^{\min}=1$.

  The \emph{adic transformation} of a properly-ordered
  Bratteli-Vershik diagram $(\mathcal D,{\le})$ is the
  self-homeomorphism $a\colon X\righttoleftarrow$ defined as
  follows. If $e=(e_1,\dots)\in X$ is such that $(e_1,\dots,e_n)$ is
  not maximal in $X_{e_n^+}$ for some $n\in\N$, then
  $e^a\coloneqq(f_1,\dots,f_n,e_{n+1},\dots)$. Otherwise, $e$ is the
  unique maximal path in $X$, and $e^a$ is defined to be the unique
  minimal path in $X$.
\end{definition}
If $\mathcal D$ is simple, then $a$ is a minimal transformation of
$X$.  Bratteli-Vershik diagrams encode \emph{all} minimal
homeomorphisms of Cantor sets:
\begin{theorem}[\cite{herman-putnam-skau:bratteli}, see~\cite{durand:cant}*{Theorem~6.4.6}]\label{thm:hps}
  Every minimal homeomorphism of the Cantor set is topologically
  conjugate to the adic transformation of a properly ordered simple
  Bratteli-Vershik diagram.\qed
\end{theorem}
(The idea of the proof is to choose a decreasing sequence
$(C_n)_{n\ge0}$ of clopen sets, shrinking down to a base point
$\{x\}$, and to consider the associated ``Kakutani-Rokhlin tower'':
the largest collection of iterated images of $C_n$ under the
homeomorphism that are disjoint. These translates of $C_n$ make up the
$n$th level of the Bratteli-Vershik diagram.)

\begin{corollary}[\cite{juschenko-monod:cantor}]\label{cor:minimalZamen}
  Let $a$ be a minimal homeomorphism of a Cantor set $X$. Then the
  topological full group $[[\langle a\rangle,X]]$ is amenable.
\end{corollary}
\begin{proof}
  Using Theorem~\ref{thm:hps}, we may assume $a$ is the adic
  transformation of a Bratteli-Vershik diagram.  It follows directly
  that $\alpha_a(v)=1$ for every $v\in V$, and that the germs of $a$
  belong to $\mathfrak T$ for all points $x\in X\setminus
  X^{\max}$. No power of $a$ has fixed points so their germs are all
  trivial.
\end{proof}

Here are some typical examples of minimal $\Z$-actions on a Cantor set, to
which Corollary~\ref{cor:minimalZamen} applies to produce amenable groups:
\begin{example}\label{ex:minimalZ}
  Consider an irrational $\alpha\in(0,1)$, and the transformation $x\mapsto
  x+\alpha$ on $\R/\Z$. It is minimal, since $\Z+\Z\alpha$ is dense in
  $\R$. We can replace $\R/\Z$ by a Cantor set as follows: set
  \[X_\alpha\coloneqq(\R\setminus\Z\alpha\sqcup(\Z\alpha\times\{+,-\}))/\Z,\]
  namely replace every point $x\in\Z\alpha\subset\R/\Z$ by a pair
  $x^\pm$. Give $X_\alpha$ the cyclic order induced from the circle and
  $x^-<x^+$, and its associated topology. Then $X_\alpha$ is a Cantor set,
  and $x\mapsto x+\alpha$ is a minimal transformation of $X_\alpha$; see
  Example~\ref{ex:iet}.

  As another example, consider the substitution $a\mapsto ab,b\mapsto a$ on
  $\{a,b\}^*$ and let $x\in\{a,b\}^\Z$ denote a fixed point of the
  substitution; for example, with `$\underline a$' denoting the position of
  the $0$th letter, $x=\lim(\underline a b,a b\underline a a b,a b
  a\underline a b a b a,\dots)$. Set $X=\overline{x\Z}$. Then the action of
  $\Z$ by shift on $X$ is minimal.

  In fact, this example coincides with the first one if one takes
  $\alpha=(\sqrt5-1)/2$ the golden ratio and $x=0^+$, decomposes
  $X_\alpha=[0^+,\alpha^-]\cup[\alpha^+,1^-]$, defines $\pi\colon
  X_\alpha\to\{a,b\}$ by $\pi(x)=a$ if $x\in[0^+,\alpha^-]$ and $\pi(x)=b$
  if $x\in[\alpha^+,1^-]$, and puts $X_\alpha$ in bijection with $X$ via
  the map $x\mapsto(n\mapsto\pi(x+n))$.

  The encoding of this example as a Bratteli diagram $\mathcal D$ is
  as follows:
  \[\begin{tikzpicture}
      \node (a1) at (0,0) {$\bullet$};
      \node (a2) at (0,1) {$\bullet$};
      \node (a3) at (0,2) {$\bullet$};
      \node (r) at (2,-0.6) [label=below:{$V_0$}] {$\bullet$};
      \node (b1) at (4,0) [label=right:{$V_1$}] {$\bullet$};
      \node (b2) at (4,1) [label=right:{$V_2$}] {$\bullet$};
      \node (b3) at (4,2) [label=right:{$V_3$}] {$\bullet$};
      \node at (0,2.5) {$\vdots$};
      \node at (4,2.5) {$\vdots$};
      \draw[->] (r) -- node [below] {$a$} (a1);
      \draw[->] (r) -- node [below] {$b$} (b1);
      \foreach\i/\j in {1/2,2/3} {
        \draw[->] (a\i) -- node [left] {$a$} (a\j);
        \draw[->] (a\i) -- node [below,pos=0.35] {$b$} (b\j);
        \draw[->] (b\i) -- node [below,pos=0.35] {$a$} (a\j);
      }
    \end{tikzpicture}
  \]
  where now a point $x\in X_\alpha$ is encoded by the path in
  $\mathcal D$ with labels $(\pi(\tilde x/\alpha^n))_{n\ge1}$ for the
  unique representative $\tilde x$ of $x$ in $[0,1]$.
\end{example}

We next quote some results from~\cite{nekrashevych:simple} to exhibit
some properties of the topological full groups $[[G,X]]$ constructed
above.
\begin{definition}
  Let $\mathfrak G$ be a groupoid. A \emph{multisection} of degree $d$
  is a collection $M$ of $d^2$ non-empty, disjoint bisections
  $\{F_{i,j}\}_{i,j=1,\dots,d}$ of $\mathfrak G$ such that
  $F_{i,j}\subseteq\mathfrak G_0$ and $F_{i,j}F_{j,k}=F_{i,k}$ for all
  $i,j,k\in\{1,\dots,d\}$.

  For $\pi\in\Sym(d)$, we denote by $M_\pi$ the element of
  $[[\mathfrak G]]$ that maps $x$ to $x F_{i,i^\pi}$ if $x\in F_{i,i}$
  and fixes $\mathfrak G_0\setminus\bigcup_{i=1}^d F_{i,i}$, and by
  $\Alt(M)$ the subgroup $\{M_\pi\mid \pi\in\Alt(d)\}$ of
  $[[\mathfrak G]]$. Finally, we denote by $\Alt(\mathfrak G)$ the
  subgroup of $[[\mathfrak G]]$ generated by $\Alt(M)$ for all
  multisections $M$ of $\mathfrak G$.
\end{definition}

\begin{proposition}[\cite{nekrashevych:simple}*{Theorem~4.1}]\label{prop:simple}
  Let $\mathfrak G$ be a minimal groupoid of germs. Then every
  non-trivial subgroup of $[[\mathfrak G]]$ normalized by
  $\Alt(\mathfrak G)$ contains $\Alt(\mathfrak G)$. In particular,
  $\Alt(\mathfrak G)$ is simple and is contained in every non-trivial
  normal subgroup of $[[\mathfrak G]]$.\qed
\end{proposition}
(Note that the minimality assumption is always necessary: if
$\mathfrak G$ does not act minimally, then let $Y\neq\mathfrak G_0$ be
a closure of an orbit; then there is a natural quotient map
$[[\mathfrak G]]\to[[\mathfrak G\restrict Y]]$, proving that
$[[\mathfrak G]]$ is not simple.)

We call a groupoid $\mathfrak G$ \emph{compactly generated} if there
exists a compact subset $S$ of $\mathfrak G$ that generates it. This
is for example the case if $\mathfrak G$ is the action groupoid of a
finitely generated group $G$ acting on a compact set (in which case
one bisection per generator of $G$ suffices to generate
$\mathfrak G$).

Let $\mathfrak G$ be a compactly generated groupoid, say by
$S\subseteq\mathfrak G$. We call $\mathfrak G$ \emph{expansive} if
there exists a finite cover $\mathscr S$ of $S$ by bisections such
that $\bigcup_{n\ge0}\mathscr S^n$ generates the topology on
$\mathfrak G$; so in particular for every $x\neq y\in\mathfrak G_0$
there exists a bisection $F\in\mathscr S^n$ with $x\in s(F)\not\ni y$.
\begin{proposition}[\cite{nekrashevych:simple}*{Theorem~5.6}]\label{prop:fg}
  If $\mathfrak G$ is compactly generated and expansive then
  $\Alt(\mathfrak G)$ is finitely generated.\qed
\end{proposition}

\begin{example}
  Consider the $\Z$-action from Example~\ref{ex:minimalZ} for $\alpha$
  the golden ratio. We claim that the group $G=[[\Z,X]]'$ is infinite,
  amenable, finitely generated and simple.

  Amenability of $G$ was proven in
  Corollary~\ref{cor:minimalZamen}. Let $\mathfrak G$ be the groupoid
  of the action of $\Z=\langle a\rangle$ on $X$. It is minimal, so
  $\Alt(\mathfrak G)$ is simple by Proposition~\ref{prop:simple}; and
  it is easy to check $\Alt(\mathfrak G)=[[\mathfrak G]]'$. The
  groupoid $\mathfrak G$ is compactly generated, say by
  $S=X\cup\{(x,x a^{\pm1})\mid x\in X\}$. Finally $X\subset\{0,1\}^\Z$
  is a subshift, so $\mathfrak G$ is expansive: the cover of $S$ by
  $X$ by
  $\{\{x\in X\mid x_0=0\}\cup\{x\in X\mid x_0=1\}\cup\{(x,x a)\mid
  x\in X\}\cup\{(x,x a^{-1})\mid x\in X\}\}$ generates the topology on
  $X$ and therefore on $\mathfrak G$.
\end{example}

Finally, we end with examples of topological full groups of
non-minimal $\Z$-actions and of minimal $\Z^2$-actions which are
\emph{not} amenable, showing that Corollary~\ref{cor:minimalZamen}
does not generalize without extra conditions:
\begin{example}[Geodesic flow]
  Consider a free group $F_k$, and the space $X$ of geodesic maps
  $a\colon\Z\to F_k$ into the Cayley graph of $F_k$, namely of
  bi-infinite geodesic rays. The $\Z$-action is by shifting:
  $\sigma(a)=(i\mapsto a_{i+1})$. The space $X$ is a Cantor set, and
  may be identified with
  $\{a\in\{x_1^\pm,\dots,x_k^\pm\}^\Z\mid a_i a_{i+1}\neq1\text{ for
    all }i\in\Z\}$.  For $a\in X$ and $j\in\{1,\dots,k\}$, define
  \[a\cdot x_j=\begin{cases}\sigma(a) & \text{ if }a_0=x_j,\\
      \sigma^{-1}(a) & \text{ if }a_{-1}=x_j^{-1},\\
      a & \text{ otherwise}.
    \end{cases}
  \]
  This defines a piecewise-$\Z$ action of $F_k$ on $X$, which is
  easily seen to be faithful: for $w\in F_k$ a non-trivial reduced
  word, extend $w$ arbitrarily but non-periodically to a bi-infinite
  geodesic $a$ containing $w$ at positions $\{0,\dots,|w|-1\}$; then
  $a\cdot w=\sigma^{|w|}(a)\neq a$.
\end{example}

We may modify the example above by letting $C_2*C_2*C_2$ rather than
$F_k$ act on the space of geodesics of its Cayley graph, and then
embed that system into a minimal $\Z^2$-action, as follows:
\begin{example}[\cite{elek-monod:fullgroup}]
  Consider the space $X$ of proper colourings of the edges of the
  standard two-dimensional grid by $\Alphabet=\{A,B,C,D,E,F\}$. There
  is a natural action of $\Z^2$ on $X$ by translations.

  To each $a\in\Alphabet$ corresponds a continuous involution
  $a\colon X\righttoleftarrow$, defined as follows. For $\sigma\in X$,
  if there is an edge between $(0,0)$ and one of its neighbours $v$
  with colour $a$, then $\sigma\cdot a\coloneqq\sigma\cdot v$; otherwise
  $\sigma\cdot a\coloneqq\sigma$. These involutions clearly belong to
  $[[\Z^2,X]]$.

  We shall exhibit a minimal non-empty closed $\Z^2$-invariant subset
  $Y$ of $X$ on which $\Z^2$ acts freely and
  $H\coloneqq\langle A,B,C\mid A^2,B^2,C^2\rangle$ acts faithfully as
  subgroup of $[[\Z^2,X]]$; since $H$ contains free subgroups, we will
  have proved that $[[\Z^2,Y]]$ may contain free subgroups (and
  therefore be non-amenable) for minimal, free $\Z^2$-spaces $Y$.

  We create a specific colouring of the grid, namely an element
  $\sigma\in X$, as follows: first, colour every horizontal line of
  the grid alternately with $E$ and $F$. Enumerate
  $H=\{w_0,w_1,\dots\}$. For all $x\in\N$, write $x=2^ix'$ with $x'$
  odd, and colour the vertical lines $\{x\}\times\R$ and
  $\{-x\}\times\R$ by the infinite word $(w_i D)^\infty$. Set
  $Y=\overline{\sigma\Z^2}$.

  Every finite patch of $\sigma\restrict S$ repeats infinitely, and
  moreover there exists $n(S)$ such that every ball of radius $n(S)$
  in the grid contains a copy of $\sigma\restrict S$. It follows
  (see~\cite{gottschalk:periodic}) that $Y$ is minimal, that $\Z^2$
  acts freely on $Y$ because $\sigma$ is aperiodic, and that every
  $\tau\in Y$ also uniformly contains copies of every patch.

  Consider now $w\neq1\in\langle A,B,C\rangle$, and let $\tau$ be a
  translate of $\sigma$ in which $w D$ reads vertically at the
  origin. Then $\tau w$ reads $D$ vertically at the origin, so
  $\tau w\neq\tau$, and therefore $w$ acts non-trivially.
\end{example}

%%%%%%%%%%%%%%%%%%%%%%%%%%%%%%%%%%%%%%%%%%%%%%%%%%%%%%%%%%%%%%%%
\newpage\section{Cellular automata and amenable algebras}\label{ss:linear}
Von Neumann defined\footnote{It seems that von Neumann never published
  his work on cellular automata --- see~\cite{burks:cellularautomata}
  for history of the subject.}  \emph{cellular automata} as creatures
built out of infinitely many finite-state devices arranged on the
nodes of $\Z^2$ or $\Z^3$, each device being capable of interaction
with its immediate neighbours. Algebraically, we consider the natural
generalization to creatures living on the vertices of a Cayley
graph. We shall see that some fundamental properties of the automaton
are characterized by amenability of the underlying graph.

\begin{definition}\label{def:ca}
  Let $G$ be a group.  A finite \emph{cellular automaton} on $G$ is a
  $G$-equivariant continuous map
  $\Theta\colon \Alphabet^G\righttoleftarrow$, where $\Alphabet$, the
  \emph{state set}, is a finite set, and $G$ acts on $\Alphabet^G$ by
  left-translation: $(x g)(h)=x(g h)$ for $x\in \Alphabet^G$ and
  $g,h\in G$. Elements of $\Alphabet^G$ are called
  \emph{configurations}.

  A \emph{linear cellular automaton} is defined similarly, except that
  $\Alphabet$ is rather required to be a finite-dimensional vector space, and
  $\Theta$ is required to be linear.
\end{definition}
Note that usually $G$ is infinite; much of the theory holds trivially
if $G$ is finite. The map $\Theta$ computes the 1-step evolution of
the automaton; its continuity implies that the evolution of a site
depends only on a finite neighbourhood, and its $G$-equivariance
implies that all sites evolve with the same rule.

\begin{lemma}[Lyndon-Curtis-Hedlund]\label{lem:lyndoncurtishedlund}
  A map $\Theta\colon \Alphabet^G\righttoleftarrow$ is a cellular automaton if
  and only if there exists a finite subset $S\Subset G$ and a map
  $\theta\colon \Alphabet^S\to \Alphabet$ such that
  \[\Theta(x)(g)=\theta(s\mapsto x(g s))\]
  for all $x\in \Alphabet^G$. The minimal such $S$ is called the
  \emph{memory set} of $\Theta$.
\end{lemma}
\begin{proof}
  Such a map $\Theta$ is continuous in the product topology if and
  only if $\Theta(x)(1)$ depends only on the restriction of $x$
  to $S$ for some finite $S$.
\end{proof}

A classical example of cellular automaton is Conway's \emph{Game of
  Life}. It is defined by $G=\Z^2$ and
$\Alphabet=\{\textsf{alive},\textsf{dead}\}$, and by the following local rule
$\theta$ as in Lemma~\ref{lem:lyndoncurtishedlund}:
$S=\{-1,0,1\}\times\{-1,0,1\}$, and $\theta(x)$ depends only on
$x(0,0)$ and on the number of alive cells among its eight
neighbours:
\[\theta(x)(0,0)=\begin{cases}\textsf{alive} & \text{ if $x(0,0)$ is alive and two or three of its neighbours are alive},\\
    \textsf{alive}& \text{ if $x(0,0)$ is dead and exactly three of its neighbours are alive},\\
    \textsf{dead}& \text{ in all other cases, from loneliness or overpopulation}.
  \end{cases}
\]

\newcommand\lifeblock[3]{\begin{scope}
    \path[name path=border,ragged border] (0,0) rectangle (#1/2,#2/2);
    \draw[use path=border];
    \clip[use path=border];
    \draw[step=0.5,black,thick,xshift=0.25cm,yshift=0.25cm] (-1,-1) grid (#1/2+1,#2/2+1);
    \foreach \x/\y in {#3} {\fill (\x/2+0.25,\y/2+0.25) rectangle (\x/2+0.75,\y/2+0.75);}
  \end{scope}
}

For example, here is the evolution of a piece of the plane; we
represent \textsf{alive} in black and \textsf{dead} in white:
\[\tikz[baseline=15mm]{\lifeblock{6}{6}{1/1,2/1,3/1,3/2,2/3}}\rightarrow
  \tikz[baseline=15mm]{\lifeblock{6}{6}{1/2,2/1,3/1,3/2,2/0}}\rightarrow
  \tikz[baseline=15mm]{\lifeblock{6}{6}{1/1,3/2,3/1,2/0,3/0}}.
\]
Note that the last configuration is the first one, transformed by
$(x,y)\mapsto(1-y,-x)$, so the pattern moves by a sliding reflection
along the $x+y=0$ direction.

Some properties have been singled out in attempts to understand the
global, long-term behaviour of cellular automata: a cellular automaton
$\Theta$ may
\begin{description}
\item[have ``Gardens of Eden'' (GOE)] if the map $\Theta$ is not
  surjective, the biblical metaphor expressing the notion of paradise
  lost forever. Note that $\Theta(\Alphabet^G)$ is compact, hence closed in
  $\Alphabet^G$, so if $\Theta$ is not surjective then there exists a finite
  subset $F\Subset G$ such that the projection of $\Theta(\Alphabet^G)$ to
  $\Alphabet^F$ is not onto;
\item[have ``Mutually Erasable Patterns'' (MEP)] if $\Theta$ fails in
  a strong way to be injective: there are configurations $x\neq y$
  which nevertheless agree at all but finitely many places, and such
  that $\Theta(x)=\Theta(y)$. The opposite is sometimes called
  \emph{pre-injectivity};
\item[preserve the Bernoulli measure;] open sets of the form
  $\mathcal O_{g,q}=\{x\in \Alphabet^G\mid x(g)=q\}$ are declared to
  have measure $\beta(\mathcal O_{g,q})=1/\#\Alphabet$, and one may
  ask whether $\beta(M)=\beta(\Theta^{-1}(M))$ for every measurable
  $M\subseteq \Alphabet^G$.
\end{description}

For example, it is clear that the Game of Life has Mutually Erasable
Patterns, because of the ``loneliness'' clause:
\[\tikz[baseline=10mm]{\lifeblock{4}{4}{1/1}}\rightarrow
  \tikz[baseline=10mm]{\lifeblock{4}{4}{}}\righttoleftarrow
\]
but it is less clear that there are also Gardens of Eden (there are
some; the smallest known one is specified by $\#F=92$ cells).

Before addressing the question of relating the GOE and MEP properties,
we introduce one more tool: \emph{entropy}. Assume that the group $G$
is amenable, and let $(F_n)$ be a F\o lner net in $G$, which exists by
Lemma~\ref{lem:folnerlim} and Theorem~\ref{thm:folneramen}. For
subsets $X\subseteq \Alphabet^G$ and $S\subseteq G$, we let
$X\restrict S$ denote the projection of $X$ to $\Alphabet^S$. We set
\begin{equation}\label{eq:caentropy}
  h(X)=\liminf_n\frac{\log(\#X\restrict F_n)}{\#F_n}.
\end{equation}

If $X$ is $G$-invariant, then the liminf in~\eqref{eq:caentropy} is a
limit and is independent of the choice of F\o lner net. This follows
from the following more general statement (independence of the F\o
lner net follows from interleaving two F\o lner nets), which we quote
without proof:
\begin{lemma}[Ornstein-Weiss, see~\cite{gromov:topinv1}*{\S1.3.1} and~\cite{krieger:ow}]
  Let $h\colon\mathfrak P_f(G)\to\R$ be subadditive:
  $h(A\cup B)\le h(A)+h(B)$, and $G$-invariant: $h(A g)=h(A)$. Then
  the limit $\lim_{n\to\infty}h(F_n)/\#F_n$ exists for every F\o lner
  net $(F_n)_{n\in\mathscr N}$.\qed
\end{lemma}

\noindent The following is the ``Second Principle of Thermodynamics'':
\begin{lemma}\label{lem:entropydecreases}
  For every cellular automaton $\Theta$ and every $G$-invariant
  $X\subseteq \Alphabet^G$ we have $h(\Theta(X))\le h(X)$.
\end{lemma}
\begin{proof}
  Let $S\Subset G$ be a memory set for $G$. For every finite
  $F\Subset G$, consider $E\subset G$ such that $E S\subseteq F$; then
  $\Theta(x)\restrict E$ depends only on $x\restrict F$. Therefore,
  $\#(\Theta(X)\restrict E)\le\#(X\restrict F)$, so
  $\#(\Theta(X)\restrict F)\le\#(X\restrict
  F)\#\Alphabet^{\#F-\#E}$. Take now $F=F_n S$ and $E=F_n$ for a net
  of F\o lner sets, and apply the definition
  from~\eqref{eq:caentropy}.
\end{proof}

Finally, given a measure $\nu$ on $\Alphabet^G$, we may define a
\emph{measured entropy} as follows: for $S\subseteq G$ and
$y\in \Alphabet^S$, denote by $\mathcal O_y$ the open set
$\{x\in \Alphabet^G\mid x\restrict S=y\restrict S\}$; and for
$X\subseteq \Alphabet^G$ set
\[h_\nu(X)=\liminf\frac{-\sum_{y\in X\restrict F_n}\nu(\mathcal O_y)\log\nu(\mathcal O_y)}{\#F_n}.
\]
Note that $\beta(\mathcal O_y)=1/\#\Alphabet^{\#S}$ if
$y\in \Alphabet^S$, so the measured entropy coincides
with~\eqref{eq:caentropy} if $\nu=\beta$.

We are ready to state the main result, called the ``Gardens of Eden
theorem''. It was first proven for $G=\Z^d$ by
Moore~\cite{moore:ca}*{the $(1)\Rightarrow(2)$ direction},
Myhill~\cite{myhill:ca}*{the $(2)\Rightarrow(1)$ direction}, and
Hedlund~\cite{hedlund:endomorphisms}*{the $(1)\Leftrightarrow(3)$
  equivalence}:
\begin{theorem}[\cites{ceccherini-m-s:ca,meyerovitch:finiteentropy}]\label{thm:ceccherini+}
  Let $G$ be an amenable group, and let $\Theta$ be a cellular
  automaton. Then the following are equivalent:
  \begin{enumerate}
  \item $\Theta$ has Gardens of Eden;
  \item $\Theta$ has Mutually Erasable Patterns;
  \item $\Theta$ does not preserve Bernoulli measure $\beta$;
  \item $h(\Theta(\Alphabet^G))<\log\#\Alphabet$.
  \end{enumerate}
\end{theorem}

\begin{remark}
  The same theorem holds for linear cellular automata (except that I
  do not know an analogue of Bernoulli measure), with the entropy
  replaced in the last statement by \emph{mean dimension}:
  \[\operatorname{mdim}(X)=\liminf_n\frac{\dim(\#X\restrict F_n)}{\#F_n}.\]
\end{remark}

\begin{proof}
  Throughout the proof, we let $S$ denote the memory set of $\Theta$.

  $(1)\Rightarrow(4)$ If there exists a GOE, then there exists
  $F\Subset G$ with
  $\Theta(\Alphabet^G)\restrict F\neq \Alphabet^F$, so
  \[h(\Theta(\Alphabet^F))\le\frac{\log\#\Theta(\Alphabet^G)\restrict F}{\#F}<\log\#\Alphabet.\]

  $(4)\Rightarrow(1)$ If $h(\Theta(\Alphabet^G))<\log\#\Alphabet$,
  then there exists $F\Subset G$ with
  $\Theta(\Alphabet^G)\restrict F\neq \Alphabet^F$, and a GOE exists
  in $\Alphabet^F\setminus\Theta(\Alphabet^G)\restrict F$.

  $(2)\Rightarrow(4)$ If $y\neq z$ are MEP, which differ on $F$ and
  agree elsewhere, set $E=F S$ and let $T\subset G$ be maximal such
  that $E t_1\cap E t_2=\emptyset$ for all $t_1\neq t_2\in T$; note
  that $T$ intersects every translate of $E^{-1}E$. Define
  \[Z=\{x\in \Alphabet^G\mid x\restrict E t\neq y\restrict E t\text{ for all }t\in
    T\},
  \]
  and compute
  $h(\Theta(\Alphabet^G))=h(\Theta(Z))\le h(Z)<\log\#\Alphabet$; the
  first equality follows since given in $x\in\Theta(\Alphabet^G)$, say
  $x=\Theta(w)$, one may replace in $w$ every occurrence of
  $y\restrict E t$ by $z\restrict E t$ so as to obtain a
  $y\restrict E t$-free configuration, which therefore belongs to $Z$,
  and has the same image as $x$ under $\Theta$; the second inequality
  follows from Lemma~\ref{lem:entropydecreases}; and the last
  inequality because there are forbidden patterns $y\restrict E t$ in
  $Z$, with ``density'' at least $1/\#(E^{-1}E)$.

  $(4)\Rightarrow(2)$ If $h(\Theta(\Alphabet^G))<\log\#\Alphabet$,
  there exists $F_n$ with
  $\log\#(\Theta(\Alphabet^G)\restrict F_n S)/\#F_n<\log\#A$, because
  $\#F_n S$ may be made arbitrarily close to $\#F_n$ for $n$ large
  enough. Therefore, by the pigeonhole principle, there exist
  $y\neq z\in \Alphabet^G$ with
  $y\restrict (G\setminus F_n)=z\restrict (G\setminus F_n)$ and
  $\Theta(y)=\Theta(z)$.

  $(1)\Rightarrow(3)$ This is always true: if $\Theta$ has GOE, then
  there exists a non-empty open set $\mathcal U$ in
  $\Alphabet^G\setminus\Theta(\Alphabet^G)$; then $\beta(\mathcal U)\neq0$
  while $\beta(\Theta^{-1}(\mathcal U))=0$.

  $(3)\Rightarrow(1)$ Define
  \[K=\{\nu\text{ probability measure on
    }\Alphabet^G\mid\beta=\Theta_*\nu\}.
  \]
  Note that $K$ is convex and compact, no admits a $G$-fixed point
  because $G$ is amenable. Consider $\nu\in K^G$. Then
  $\phi\colon(\Alphabet^G,\nu)\to(\Alphabet^G,\beta)$ is a factor map
  because $\Theta$ is onto, so
  $h_\nu(\Alphabet^G)\ge h_\beta(\Alphabet^G)$. However, $\beta$ is
  the unique measure of maximal entropy\footnote{One says that the
    $G$-action is \emph{intrinsically ergodic}.}, so $\nu=\beta$ and
  therefore $\beta=\Theta_*\beta$.
\end{proof}

It turns out that Theorem~\ref{thm:ceccherini+} is essentially
optimal, and yields characterizations of amenable groups:
\begin{theorem}[\cites{bartholdi:moore,bartholdi:myhill}]\label{thm:mmconverse}
  Let $G$ be a non-amenable group. Then there exist
  \begin{enumerate}
  \item cellular automata (ad lib linear) that admit Mutually Erasable
    Patterns but no Gardens of Eden;
  \item cellular automata (ad lib linear) that admit Gardens of Eden
    but no Mutually Erasable Patterns;
  \item cellular automata that do not preserve Bernoulli measure but
    have no Gardens of Eden.
  \end{enumerate}
\end{theorem}

In fact, we shall prove Theorem~\ref{thm:mmconverse} for finite
fields, answering at the same time the classical and linear
questions. Let $\Theta$ be a linear cellular automaton; then
$\Alphabet=\Bbbk^n$ for some field $\Bbbk$ and some integer $n$, and
there exists an $n\times n$ matrix $\mathbf M$ over $\Bbbk G$ such
that $\Theta(x)=x \mathbf M$ for all $x\in \Alphabet^G$. Conversely,
every such matrix defines a linear cellular automaton.

The ring $\Bbbk G$ admits an anti-involution $*$, defined on its basis
$G$ by $g^*=g^{-1}$ and extended by linearity. This involution extends
to an anti-involution on square matrices by
$(\mathbf M^*)_{i,j}=(\mathbf M_{j,i})^*$, and $\mathbf M^*$ is called
the \emph{adjoint} of $\mathbf M$.

We put on $\Alphabet$ the natural scalar product
$\langle x,y\rangle=\sum_{i=1}^n x_i y_i$. Consider the vector space
$\Alphabet G=\bigoplus_{g\in G}\Alphabet$. Then $\Alphabet^G$ may be
naturally identified with the dual of $\Alphabet G$, under the
non-degenerate pairing
$\langle x,y\rangle=\sum_{g\in G}\langle x(g),y(g)\rangle$ for
$x\in \Alphabet G$ and $y\in \Alphabet^G$.
\begin{exercise}[*]
  Prove that $\mathbf M^*$ is the adjoint with respect to this pairing; namely
  $\langle x \mathbf M,y\rangle=\langle x,y \mathbf M^*\rangle$ for all
  $x\in \Alphabet G,y\in \Alphabet^G$.
\end{exercise}

We put a topology on $\Alphabet^G$ by declaring that, for every finite
$S\Subset G$ and every vector space $V\le\Alphabet^S$, the subset
$\{x\in\Alphabet^G\mid x\restrict S\in V\}$ is closed. With this
topology, $\Alphabet^G$ is compact (but not Hausdorff). Nevertheless,
\begin{lemma}
  If $\Theta$ is a cellular automaton then $\Theta(\Alphabet^G)$ is
  closed.
\end{lemma}
\begin{proof}
  Let $S$ be the memory of $\Theta$.  Consider $y$ in the closure of
  $\Theta(\Alphabet^G)$. Then for every $F\Subset G$ the affine space
  $L_F=\{x\in\Alphabet^{F S}\mid \Theta(x)\restrict F=y\restrict F\}$
  is finite-dimensional and non-empty, and if $F\subseteq F'$ then
  $L_{F'}\restrict{F S}\subseteq L_F$; so
  $\{L_{F'}\restrict{F S}\mid F'\supseteq F\}$ is a nested sequence of
  non-empty affine spaces, and in particular stabilizes at a non-empty
  affine space $J_F$. We still have restriction maps $J_{F'}\to J_F$
  for all $F\subseteq F'$, which are easily seen to be
  surjective. Then $\varprojlim_{F\Subset G}J_F$ is non-empty and
  contains all preimages of $y$.
\end{proof}

The following proposition extends to the infinite-dimensional setting
the classical statement that the image of a matrix is the orthogonal
of the nullspace of its transpose:
\begin{proposition}[\cite{tointon:harmonic}]\label{prop:tointon}
  Let $\mathbf M$ be an $n\times n$ matrix over $\Bbbk G$, let
  $\mathbf M^*$ be its adjoint, and set $\Alphabet=\Bbbk^n$. Then
  \[\ker(\mathbf M)\cap\Alphabet G=\operatorname{image}(\mathbf M^*)^\perp=\{x\in\Alphabet G\mid \langle x,\Alphabet^G\mathbf M^*\rangle=0\}.
  \]
  Equivalently, right-multiplication by $\mathbf M$ is injective on
  $\Alphabet G$ if and only if right-multiplication by $\mathbf M^*$
  is surjective on $\Alphabet^G$.
\end{proposition}
\begin{proof}
  Assume first that right-multiplication by $\mathbf M$ is not
  injective, and consider a non-trivial element $c\in \Alphabet G$
  with $c \mathbf M=0$. We claim that for every
  $y\in(\Alphabet^G)\mathbf M^*$ we have $\langle c,y\rangle=0$. Say
  $y= z \mathbf M^*$; then the claim follows from the computation
  \[\langle c,y\rangle=\langle c,z \mathbf M^*\rangle=\langle c \mathbf M,z\rangle=\langle0,z\rangle=0.\]
  Since $\langle{-},{-}\rangle$ is non-degenerate, this implies that
  $y$ cannot range over all of $\Alphabet^G$, so right-multiplication
  by $\mathbf M^*$ is not surjective.

  Conversely, suppose that right-multiplication by $\mathbf M$ is not
  surjective. Since $\Alphabet^G \mathbf M$ is closed, there exists an
  open set in its complement; so there exists a finite subset
  $S\Subset G$ and a proper subspace $V\lneqq\Alphabet^S$ such that,
  for every $c\in \Alphabet^G \mathbf M$, its projection
  $c\restrict S$ belongs to $V$. Since $\Alphabet^S$ is
  finite-dimensional, there exists a linear form $y$ on $\Alphabet^S$
  that vanishes on $V$. Note that $y$, qua element of
  $(\Alphabet^S)^*$, is canonically identified with an element of
  $\Alphabet^S$, and therefore with an element of $\Alphabet G$. We
  claim $y \mathbf M^*=0$, proving that right-multiplication by
  $\mathbf M^*$ is not injective. This follows from the following
  computation: consider an arbitrary $c\in \Alphabet^G$. Then
  \[\langle y \mathbf M^*,c\rangle=\langle y,c \mathbf M\rangle=0.\]
  Since $\langle{-},{-}\rangle$ is non-degenerate and $c\in V^G$ is
  arbitrary, this forces $y \mathbf M^*=0$.
\end{proof}

Before embarking in the main step of the proof of
Theorem~\ref{thm:mmconverse}, we give a simple example of a cellular
automaton that is pre-injective but not surjective:
\begin{example}[Muller, see~\cite{machi-m:ca}*{page 55}]
  Consider the free product of cyclic groups
  $G=\langle a,b,c|a^2,b^2,c^2\rangle$. Fix a field $\Bbbk$, and set
  $\Alphabet\coloneqq\Bbbk^2$. Define the linear cellular automaton
  $\Theta\colon \Alphabet^G\righttoleftarrow$ by
  \[\Theta(x)=x\cdot\begin{pmatrix}a+b & 0\\ b+c & 0\end{pmatrix}.\]
  It is obvious that $\Theta$ is not surjective: its image is
  $(\Bbbk\times0)^G$. To show that it is pre-injective, consider $x$ a
  non-zero configuration with finite support, and let $F\Subset G$
  denote its support. Let $f\in F$ be an element of maximal length;
  then at least two among $f a,f b,f c$ will be reached precisely once
  as products of the form $F\cdot\{a,b,c\}$. Write
  $x(f)=(\alpha,\beta)\neq(0,0)$; then at least two among the
  equations
  \[\Theta(x)(f a)=\alpha,\qquad\Theta(x)(f b)=\alpha+\beta,\qquad\Theta(x)(f c)=\beta
  \]
  hold, and this is enough to force $\Theta(x)\neq0$.
\end{example}

In the general case of a non-amenable group $G=\langle S\rangle$, we
may not claim that there exist two elements reached exactly once from
an arbitrary finite set $F$ under right $S$-multiplication; but we
shall see that there exists ``many'' elements reached ``not too many''
times, in the sense that there exists $f\in F$ with
$\sum_{s\in S}1/\#\{t\in S\mid f s\in F t\}>1$; and this will suffice
to construct a pre-injective, non-surjective cellular automaton. We
begin by a combinatorial lemma:
\begin{lemma}\label{lem:overlaps}
  Let $n$ be an integer. Then there exists a set $Y$ and a family of
  subsets $X_1,\dots,X_n$ of $Y$ such that, for all
  $I\subseteq\{1,\dots,n\}$ and all $i\in I$, we have
  \begin{equation}\label{eq:overlaps:1}
    \#\Big(X_i\setminus\bigcup_{j\in I\setminus\{i\}}X_j\Big)\ge\frac{\#Y}{(1+\log n)\#I}.
  \end{equation}
  Furthermore, if $n\ge2$ then we may require
  $X_1\cup\cdots\cup X_n\neq Y$.
\end{lemma}
\begin{proof}
  We denote by $\Sym(n)$ the symmetric group on $n$ letters.  Define
  \[Y\coloneqq\frac{\{1,\dots,n\}\times\Sym(n)}{(i,\sigma)\sim(j,\sigma)\text{ if $i$ and $j$ belong to the same cycle of }\sigma};
  \]
  in other words, $Y$ is the set of cycles of elements of
  $\Sym(n)$. Let $X_i$ be the natural image of $\{i\}\times\Sym(n)$ in
  the quotient $Y$.

  First, there are $(i-1)!$ cycles of length $i$ in $\Sym(i)$, given by
  all cyclic orderings of $\{1,\dots,i\}$; so there are
  $\binom n i(i-1)!$ cycles of length $i$ in $\Sym(n)$, and they can be
  completed in $(n-i)!$ ways to a permutation of $\Sym(n)$; so
  \begin{equation}\label{eq:pfoverlaps:1}
    \#Y=\sum_{i=1}^n\binom n i(i-1)!(n-i)!=\sum_{i=1}^n\frac{n!}i\le(1+\log n)n!
  \end{equation}
  since $1+1/2+\dots+1/n\le 1+\log n$ for all $n$.

  Next, consider $I\subseteq\{1,\dots,n\}$ and $i\in I$, and set
  $X_{i,I}\coloneqq X_i\setminus\bigcup_{j\in
    I\setminus\{i\}}X_j$. Then
  $X_{i,I}=\big\{(i,\sigma):(i,\sigma)\nsim(j,\sigma)\text{ for all
  }j\in I\setminus\{i\}\big\}$. Summing over all possibilities for the
  length-$(j+1)$ cycle $(i,t_1,\dots,t_j)$ of $\sigma$ intersecting
  $I$ in $\{i\}$, we get
  \begin{equation}\label{eq:pfoverlaps:2}
    \begin{split}
      \#X_{i,I}&=\sum_{j=0}^{n-\#I}\binom{n-\#I}j j!(n-j-1)!\\
      &=\sum_{k\coloneqq n-j=\#I}^n(n-\#I)!(\#I-1)!\binom{k-1}{k-\#I}\\
      &=(n-\#I)!(\#I-1)!\binom{n}{n-\#I}=\frac{n!}{\#I}.
    \end{split}
  \end{equation}
  Combining~\eqref{eq:pfoverlaps:1} and~\eqref{eq:pfoverlaps:2}, we get
  \[\#X_{i,I}=\frac{n!}{\#I}=\frac{(1+\log n)n!}{(1+\log n)\#I}\ge\frac{\#Y}{(1+\log n)\#I}.\]

  Finally, if $n\ge2$ then~\eqref{eq:overlaps:1} may be improved to
  $\#Y\le (0.9+\log n)n!$; for even larger $n$ one could get to
  $\#Y\le(0.57721\dots+\log n)n!$. Since clearly
  $\#Y/(0.9+\log n)\ge(\#Y+1)/(1+\log n)$, one may simply replace $Y$
  by $Y\sqcup\{\cdot\}$.
\end{proof}

\begin{proposition}\label{prop:extmat}
  Let $\Bbbk$ be a field, and let $G$ be a non-amenable group. Then
  there exists a finite extension $\mathbb K$ of $\Bbbk$ and an
  $n\times(n-1)$ matrix $\mathbf M$ over $\mathbb K G$ such that
  multiplication by $\mathbf M$ is an injective map
  $(\mathbb K G)^n\to(\mathbb K G)^{n-1}$.
\end{proposition}
\begin{proof}
  Since $G$ is non-amenable, there exists by
  Theorem~\ref{thm:folneramen} a finite subset $S_0\subset G$ and
  $\epsilon>0$ with $\#(F S_0)\ge(1+\epsilon)\#F$ for all finite
  $F\subset G$. We then have $\#(F S_0^k)\ge(1+\epsilon)^k\#F$ for all
  $k\in\N$. Let $k$ be large enough so that
  $(1+\epsilon)^k>1+k\log\#S_0$, and set $S\coloneqq S_0^k$ and
  $n\coloneqq\#S$. We will seek $\mathbf M$ supported in $\mathbb K S$. We
  have
  \begin{equation}\label{eq:pfextmat:folner}
    \begin{split}
      \#(F S)&\ge(1+\epsilon)^k\#F>(1+k\log\#S_0)\#F\\
      &\ge(1+\log n)\#F\text{ for all finite }F\subset G.
    \end{split}
  \end{equation}

  Apply Lemma~\ref{lem:overlaps} to this $n$, and identify
  $\{1,\dots,n\}$ with $S$ to obtain a set $Y$ and subsets $X_s$ for
  all $s\in S$. We have $\bigcup_{s\in S}X_s\subsetneqq Y$ and
  \[\#\bigg(X_s\setminus\bigcup_{t\in
      T\setminus\{s\}}X_t\bigg)\ge\frac{\#Y}{(1+\log n)\#T}\text{ for
      all }s\in T\subseteq S.
  \]

  We shall specify soon how large the finite extension $\mathbb K$ of
  $\Bbbk$ should be. Under that future assumption, set
  $\Alphabet\coloneqq \mathbb K Y$. For each $s\in S$, we shall construct a
  linear map $\alpha_s\colon \Alphabet\to\mathbb K X_s\subset \Alphabet$; for this, we
  introduce the following notation: for $T\ni s$ denote by
  $\alpha_{s,T}\colon \Alphabet\to\mathbb K X_{s,T}$ the composition of
  $\alpha_s$ with the co\"ordinate projection $\Alphabet\to\mathbb K
  X_{s,T}$. We wish to impose the condition that, whenever
  $\{T_s: s\in S\}$ is a family of subsets of $S$ with
  $\sum_{s\in S}\#X_{s,T_s}\ge\#Y$, we have
  \begin{equation}\label{eq:pfextmat:inj}
    \bigcap_{s\in S}\ker(\alpha_{s,T_s})=0.
  \end{equation}
  As a first step, we treat each $\alpha_s$ as a $\#X_s\times\#Y$
  matrix with variables as co\"efficients, by considering only its
  rows indexed by $X_s\subset Y$; and we treat each $\alpha_{s,T}$ as
  a $\#X_{s,T}\times\#Y$ submatrix of $\alpha_s$. The space of all
  $(\alpha_s)_{s\in S}$ therefore consists of
  $N\coloneqq\#Y\sum_{s\in S}\#X_s$ variables, so is an affine space
  of dimension $N$.

  Equations~\eqref{eq:pfextmat:inj} amounts to the condition, on these
  variables, that all matrices obtained by stacking vertically a
  collection of $\alpha_{s,T_s}$'s have full rank as soon as
  $\sum_{s\in S}\#X_{s,T_s}\ge\#Y$. The complement of these conditions
  is an algebraic subvariety of $\mathbb K^N$, given by a finite union
  of hypersurfaces of the form `$\det(\cdots)=0$'. Crucially, the
  equations of these hypersurfaces are defined over $\Z$, and in
  particular are independent of the field $\mathbb K$. Therefore, as
  soon as $\mathbb K$ is large enough, there exist points that belong
  to none of these hypersurfaces; and any such point gives a solution
  to~\eqref{eq:pfextmat:inj}.

  Define now the matrix $\mathbf M$ with co\"efficients in $\mathbb K G$ by
  \begin{equation}\label{eq:pfextmat:theta}
    \mathbf M=\sum_{s\in S}\alpha_s s.
  \end{equation}
  It maps $\mathbb K^n G$ to $\mathbb K^{n-1}G$ as required, since
  $\bigcup_{s\in S}X_s\subsetneqq Y$. To show that $\mathbf M$ is injective,
  consider $x\in\mathbb K^n G$ non-trivial, and let
  $\emptyset\neq F\Subset G$ denote its support. Define
  $\rho\colon F S\to(0,1]$ by
  $\rho(g)\coloneqq1/\#\{s\in S:g \in F s\}$. Now
  \[\sum_{f\in F}\Big(\sum_{s\in S}\rho(f s)\Big)=\sum_{g\in F S}\sum_{s\in S:g\in F s}\rho(g)=\sum_{g\in F S}1=\#(F S),
  \]
  so there exists $f\in F$ with
  $\sum_{s\in S}\rho(f s)\ge\#(F S)/\#F\ge1+\log n$
  by~\eqref{eq:pfextmat:folner}. For every $s\in S$, set
  $T_s\coloneqq\{t\in S:f s\in F t\}$, so $\#T_s=1/\rho(f s)$. We obtain
  \begin{align*}
    \sum_{s\in S}\#X_{s,T_s}&\ge\sum_{s\in S}\frac{\#Y}{(1+\log n)\#T_s}\text{ by Lemma~\ref{lem:overlaps}}\\
                            &=\sum_{s\in S}\frac{\#Y\rho(f s)}{1+\log n}\ge\#Y,
  \end{align*}
  so by~\eqref{eq:pfextmat:inj} the map
  $\Alphabet\ni a\mapsto(\alpha_{s,T_s}(a))_{s\in S}$ is injective. Set
  $y\coloneqq x \mathbf M$. Since by assumption $x(f)\neq0$, we get
  $(\alpha_{s,T_s}(x(f)))_{s\in S}\neq0$, namely there exists $s\in S$
  with $\alpha_{s,T_s}(x(f))\neq0$. Now
  $y(f s)\restrict X_{s,T_s}=\alpha_{s,T_s}(x(f))$
  by~\eqref{eq:pfextmat:theta}, so $y\neq0$ and we have proven that $\mathbf M$
  is injective.
\end{proof}

\begin{proof}[Proof of Theorem~\ref{thm:mmconverse}]
  We start by (2). Apply Proposition~\ref{prop:extmat} to
  $\Bbbk=\mathbb F_2$, and let $\mathbb K=\mathbb F_{2^q}$ and
  $\mathbf M$ be the $n\times(n-1)$ resulting matrix over
  $\mathbb K G$. Set $\Alphabet=\mathbb K^n$, and extend $\mathbf M$
  to an $n\times n$ matrix by adding a column on $0$'s to its
  right. Then $\Theta\colon \Alphabet^G\righttoleftarrow$ given by
  $\Theta(x)=x \mathbf M$ is a $G$-equivariant endomorphism of
  $\Alphabet^G$, is pre-injective because $\mathbf M$ is injective on
  $\Alphabet G$, and is not surjective because no configuration in its
  image has a non-trivial last co\"ordinate.

  Right-multiplication by $\mathbf M^*$ on $\Alphabet^G$ is surjective and not
  pre-injective by Proposition~\ref{prop:tointon}, so this
  answers~(1).

  Finally, let $y\in \Alphabet^S$, for some $S\Subset G$, be such that
  $\mathcal O_y$ is a Garden of Eden for $\mathbf M$. Then
  $\mathcal O_y \mathbf M^*=0$, so $\mathbf M^*$ does not preserve
  Bernoulli measure, answering~(3).
\end{proof}

\subsection{Goldie rings}
We saw in the last section that linear cellular automata are closely
related to group rings. We give now a characterization of amenability
of groups in terms of ring theory. We recommend~\cite{passman:gr} as a
reference for group rings.

\begin{definition}
  Let $R$ be a ring. It is \emph{semiprime} if $a R a\neq0$ whenever
  $a\in R\setminus\{0\}$. An element $a\in R$ is \emph{regular} if
  $x a y\neq 0$ whenever $x,y\in R\setminus\{0\}$, and the ring $R$ is
  a \emph{domain} if $x y\neq 0$ whenever $x,y\in
  R\setminus\{0\}$. The \emph{right annihilator} of $a\in R$ is
  $\{x\in R\mid a x=0\}$ and is a right ideal in $R$.

  The ring $R$ is \emph{Goldie} if (1) there is no infinite ascending
  chain of right annihilators in $R$ and (2) there is no infinite
  direct sum of nonzero right ideals in $R$.
\end{definition}
Clearly $R$ is a domain if and only if all its non-zero elements are
regular; annihilators of regular elements are trivial; and all domains
are semiprime.

These definitions may be difficult to digest, but they have strong
consequences for the structure of $R$,
see~\cite{coutinho-mcconnell:quest} and Goldie's theorem below. In
terms of their ideal structure, the simplest rings are \emph{skew
  fields}, in which all non-zero elements are invertible. Next best
are \emph{Artinian rings}, which do not admit infinite descending
chains of ideals. Finitely generated modules over Artinian rings have
a well-defined notion of dimension, namely the maximal length of a
composition series.

Ore studied in~\cite{ore:condition} when a ring $R$ may be imbedded in
a ring in which all regular elements of $R$ become invertible. Let us
denote by $R^*$ the set of regular elements in $R$. A naive attempt is
to consider expressions of the form $a s^{-1}$ with $a,s\in R$ and $s$
regular; then to multiply them one must rewrite
$a s^{-1}b t^{-1}=a b'(s')^{-1}t^{-1}=(a b')(t s')^{-1}$, and to add
them one must rewrite $a s^{-1}+b t^{-1}=(a t'+b s')(s t')^{-1}$. In
all cases, it is sufficient that $R$ satisfy the following property,
called \emph{Ore's condition}:
\[\text{for all $a,s\in R$ with $s$ regular there exist $b,t\in R$ with $t$ regular and }s b=a t,
\]
namely every pair of elements $a,s$ admits a common ``right multiple''
$a t=s b$. The ring
\[R(R^*)^{-1}\coloneqq\{a s^{-1}\mid a\in R,s\in R^*\}/\langle a
  s^{-1}=a t(s t)^{-1}\text{ for all }a\in R,s,t\in R^*\rangle
\]
is called $R$'s \emph{classical ring of fractions}. It naturally
contains $R$ as the subring $\{a1^{-1}\}$. If $R$ is a domain, then
$R(R^*)^{-1}$ is a skew field.

\begin{theorem}[Goldie~\cite{goldie:max}]\label{thm:goldie}
  Let $R$ be a semiprime Goldie ring. Then $R$ satisifes Ore's
  condition, and its classical ring of fractions is Artinian.\qed
\end{theorem}

Let $R\subseteq S$ be a subring of a ring. The ring $S$ is called
\emph{flat} over $R$ if for every exact sequence
$0\to A\to B\to C\to0$ of $R$-modules the corresponding sequence
$0\to A\otimes_R S\to B\otimes_R S\to C\otimes_R S\to0$ of $S$-modules
is exact.
\begin{exercise}[**]\label{ex:flat}
  For $R$ a domain, show that $S\coloneqq R(R^*)^{-1}$ is flat.

  \emph{Hint:} there is an equational criterion for flatness: $S$ is
  flat if and only if every $R$-linear relation $\sum r_i x_i=0$, with
  $r_i\in R$ and $x_i\in S$, ``follows from linear relations in $R$'',
  in the following sense: the equation in matrix form
  $\mathbf r^T\mathbf x=0$, with $\mathbf r\in R^n$ and
  $\mathbf x\in S^n$, implies equations $\mathbf r^T\mathbf B=0$ and
  $\mathbf x=\mathbf B\mathbf y$ for some $n\times m$ matrix
  $\mathbf B$ over $R$ and some $\mathbf y\in S^m$;
  see~\cite{lam:lmr}*{4.24(2)}.

  Using Ore's condition, apply this criterion by expressing in a
  $R$-linear relation $\sum r_i x_i=0$ every $x_i=a_i s^{-1}$ for
  $a_i\in R$ and a common denominator $s\in R^*$.
\end{exercise}

Let now $G$ be a group, let $\Bbbk$ be a field, and consider the group
ring $\Bbbk G$. It is the $\Bbbk$-vector space with basis $G$, and
multiplication extended multilinearly from the multiplication in
$G$. Is is well understood when the group ring $\Bbbk G$ is
semiprime:
\begin{theorem}[Passman, see~\cite{passman:gr}*{Theorems~2.12 and~2.13}]
  If $\Bbbk$ has characteristic $0$, then $\Bbbk G$ is semiprime for
  all $G$. If $\Bbbk$ has characteristic $p>0$, then $\Bbbk G$ is
  semiprime if and only if $G$ has no finite normal subgroup of order
  divisible by $p$.\qed
\end{theorem}

\begin{exercise}[*]
  If $G$ is non-amenable, then it has a non-amenable quotient
  $\overline G$ whose group ring $\Bbbk\overline G$ is semiprime for
  all $\Bbbk$.
\end{exercise}

\begin{theorem}[Tamari~\cite{tamari:folner}, Kielak~\cite{bartholdi:myhill}, Kropholler]
  Let $\Bbbk$ be a field and let $G$ be group such that $\Bbbk G$ is
  Goldie and semiprime. Then $G$ is amenable.

  Furthermore, if $\Bbbk G$ is a domain\footnote{Conjecturally
    (see~\cite{kaplansky:problemrings}
    and~\cite{kaplansky:problemrings2}*{Problem~6}), $\Bbbk G$ is a
    domain if and only if $G$ is torsion-free.}, then $\Bbbk G$
  satisfies Ore's condition if and only if $G$ is amenable.
\end{theorem}

\begin{proof}
  Assume first that $G$ is amenable and that $\Bbbk G$ is a domain,
  and let $a,s\in\Bbbk G$ be given. Let $S\Subset G$ contain the
  supports of $a$ and $s$. Since $G$ is amenable, there exists
  $F\Subset G$ with $\#(F S)<2\#F$, by F\o lner's
  Theorem~\ref{thm:folneramen}. Consider $b,t\in\Bbbk G$ as unknowns
  in $\Bbbk F$. The equation $s b= a t$ which they must satisfy is
  linear in their co\"efficients, and there are more variables
  ($2\#F$) than constraints ($\#(F S)$), so there exists a non-trivial
  solution, in which $t\neq0$ if $s\neq0$; so Ore's condition is
  satisfied.

  Assume next that $G$ is not amenable. By
  Proposition~\ref{prop:extmat}, there exists a finite field extension
  $\mathbb K$ of $\Bbbk$ and an $n\times(n-1)$ matrix $\mathbf M$ over
  $\mathbb K G$ such that multiplication by $\mathbf M$ is an injective map
  $(\mathbb K G)^n\to(\mathbb K G)^{n-1}$. Restricting scalars, namely
  writing $\mathbb K=\Bbbk^d$ qua $\Bbbk$-vector space, we obtain an
  exact sequence of free $\Bbbk G$-modules
  \begin{equation}\label{eq:pfore}
    0 \longrightarrow (\Bbbk G)^{d n} \longrightarrow (\Bbbk G)^{d(n-1)}.
  \end{equation}

  Suppose now for contradiction that $\Bbbk G$ is a semiprime Goldie
  ring, and let $S$ be its classical ring of fractions, which exists
  and is Artinian by Theorem~\ref{thm:goldie}.  By
  Exercise~\ref{ex:flat}, the ring $S$ is flat over $\Bbbk$, so
  tensoring~\eqref{eq:pfore} with $S$ we obtain an exact sequence
  \[0 \longrightarrow S^{d n} \longrightarrow S^{d(n-1)}
  \]
  which is impossible for reasons of composition length.
\end{proof}

\subsection{Amenable Banach algebras}
We concentrated, in this text, on amenability of groups. The topic of
amenability of associative algebras has been developed in various
directions; although the different definitions are in general
inequivalent, we stress here the connections between amenability of a
group (or a set) and that of an associated algebra (or module).

Let $\mathscr A$ be a Banach algebra, and let $V$ be a Banach
bimodule: a Banach space $V$ endowed with commuting actions
$V\widehat\otimes\mathscr A\to V$ and
$\mathscr A\widehat\otimes V\to V$. Recall that a \emph{derivation} is
a map $\delta\colon\mathscr A\to V$ satisfying
$\delta(ab)=a\delta(b)+\delta(a)b$, and a derivation $\delta$ is
\emph{inner} if it is of the form $\delta(a)=a v-v a$ for some
$v\in V$. The dual $V^*$ of a Banach bimodule is again a Banach
bimodule, for the adjoint actions
$(g\cdot\phi\cdot h)(x)=\phi(h^{-1}x g^{-1})$.

\begin{definition}
  The Banach $\mathscr A$-module $V$ is \emph{amenable} if all bounded
  derivations of $\mathscr A$ into $V$ are inner.  More pedantically:
  the Hochschild cohomology group $H^1(\mathscr A,V)$ is trivial.

  The algebra $\mathscr A$ itself is called amenable if all
  $H^1(\mathscr A,V^*)=0$ for all Banach bimodules $V$.
\end{definition}

\begin{exercise}[**, see~Johnson~\cite{johnson:cohomology}*{Proposition~5.1}]
  Prove that the tensor product of amenable Banach algebras is amenable.
\end{exercise}

This definition seems quite distinct from everything we have seen in
the context of groups and $G$-sets; yet it applies to the Banach
algebra $\ell^1(G)$ introduced in~\eqref{eq:convolution}. For a set
$X$, denote by $\ell^\infty(X)^*_0$ those functionals
$\Phi\colon\ell^\infty(X)\to\C$ such that $\Phi(\mathbb1_X)=0$.
\begin{theorem}[\cite{johnson:cohomology}*{Theorem~2.5}]
  Let $G$ be a group. Then the following are equivalent:
  \begin{enumerate}
  \item $G$ is amenable;
  \item $\ell^1(G)$ is amenable;
  \item the Banach $\ell^1(G)$-module $\ell^\infty(G)^*_0$ is
    amenable.
  \end{enumerate}
\end{theorem}
\begin{proof}
  We begin by remarking that the bimodule structure on $V$ can be
  modified into a right module structure: let $\overline V$ be $V$ qua
  Banach space, with actions
  $g\cdot \overline v\cdot h=\overline{h^{-1}v h}$ for $g,h\in G$; in
  other words, the left action becomes trivial while the right action
  is by conjugation. A derivation $\delta\colon\ell^1(G)\to V$ gives
  rise to a ``crossed homomorphism''
  $\eta\colon\ell^1(G)\to\overline V$, defined by
  $\eta(g)=\overline{g^{-1}\delta(g)}$. It satisfies
  $\eta(g h)=\eta(g)h+\eta(h)$. Inner derivations give rise to crossed
  homomorphisms of the form $\eta(g)=v-v g$ for some
  $v\in\overline V$. For the rest of the proof, we replace $V$ by
  $\overline V$.

  $(1)\Rightarrow(2)$ Let $m\colon\ell^\infty(G)\to\C$ be a mean on
  $G$. Given a Banach module $V$ and a crossed homomorphism
  $\eta\colon\ell^1(G)\to V^*$, define $v\in V^*$ by
  \[v(f)=m(g\mapsto\eta(g)(f))\text{ for all }f\in V.
  \]
  Compute then, for $h\in G$,
  \begin{align*}
    (v h)(f) &= v(f h^{-1}) = m\big(g\mapsto\eta(g)(f h^{-1})\big) = m\big(g\mapsto(\eta(g)h)(f)\big)\\
             &= m\big(g\mapsto(\eta(g h)-\eta(h))(f)\big) = (v-\eta(h))(f),
  \end{align*}
   so $\eta(h)=v-v h$.

  $(2)\Rightarrow(3)$ is obvious.

  $(3)\Rightarrow(1)$ More generally, if $X$ is a $G$-set and
  $\ell^\infty(X)^*_0$ is amenable then $X$ is amenable: choose
  $\Phi\in\ell^\infty(X)^*$ with $\Phi(\mathbb1_X)=1$, and set
  $\eta(g)\coloneqq\Phi-\Phi g$. Then
  $\eta\colon \ell^1(G)\to\ell^\infty(X)^*_0$ is a crossed
  homomorphism, so since $\ell^\infty(X)^*_0$ is amenable there exists
  $\Psi\in\ell^\infty(X)^*_0$ with $\Psi-\Psi g=\Phi-\Phi g$, namely
  $(\Phi-\Psi)g=\Phi-\Psi$. Then $\Phi-\Psi\colon\ell^\infty(X)\to\C$
  is a $G$-invariant functional on $X$.

  Furthermore, using~\eqref{eq:stonerep}, $\Phi-\Psi$ may be viewed as
  a measure on the Stone-\v Cech compactification $\beta X$; its
  normalized absolute value is a positive measure, and therefore a
  $G$-invariant mean on $X$.
\end{proof}

As a corollary, we may deduce that $\ell^1(G)$ is amenable if and only
if its augmentation ideal has \emph{approximate identities}; though we
prefer to give a direct proof. Recall that an approximate identity in
a Banach algebra $\mathscr A$ is a bounded net $(e_n)$ in $\mathscr A$
with $e_n a\to a$ for all $a\in\mathscr A$, and that the
\emph{augmentation ideal} $\varpi(\ell^1G)$ is
$\{f\in\ell^1(G)\mid\sum_{g\in G}f(g)=0\}$.
\begin{lemma}\label{lem:diagonalidentities}
  Let $\mathscr A$ be a Banach algebra with approximate identities,
  and let $f_1,\dots,f_N\in\mathscr A$ and $\epsilon>0$ be given. Then
  there exists $e\in\mathscr A$ with $\|f_i-e f_i\|<\epsilon$ for all
  $i=1,\dots,N$.
\end{lemma}
\begin{proof}
  Let $K=\sup\|e_n\|$ be a bound on the norms of approximate
  identities in $\mathscr A$. For $N=0$ there is nothing to do. If
  $N\ge1$, find by induction $e'\in\mathscr A$ satisfying
  $\|f_i-e' f_i\|<\epsilon/(1+K)$ for all $i<N$, and let $e''\in\mathscr A$
  satisfy $\|(f_N-e' f_N)-e''(f_N-e' f_N)\|<\epsilon$. Set
  $e\coloneqq e'+e''-e''e'$, and check.
\end{proof}

\begin{theorem}\label{thm:amen=approxid}
  Let $G$ be a group. Then $G$ is amenable if and only if $\varpi(\ell^1G)$
  has approximate identities.
\end{theorem}
\begin{proof}
  ($\Rightarrow$) Given $f\in\varpi(\ell^1G)$ and $\epsilon>0$, let
  $S\Subset G$ be such that
  $\sum_{g\in G\setminus S}|f(g)|<\epsilon/2$. Since $G$ is amenable,
  there exists $h\in\ell^1(G)$ with $h\ge0$ and $\|h\|=1$ and
  $\|h-h s\|<\epsilon/2$ for all $s\in S$; so $\|h f\|<\epsilon$. Set
  $e\coloneqq1-h$; then $\|e\|\le2$, and $\|f-e f\|=\|h f\|<\epsilon$.

  ($\Leftarrow$) Let $S=\{s_1,\dots,s_n\}\Subset G$ and $\epsilon>0$
  be given, and apply Lemma~\ref{lem:diagonalidentities} with
  $f_i=1-s_i$ to obtain $e\in\mathscr A$ satisfying
  $\|1-s-e(1-s)\|<\epsilon$ for all $s\in S$; set $g\coloneqq1-e$ to
  rewrite this as $\|g-g s\|<\epsilon$. Finally set
  $h(x)=|g(x)|/\|g\|$ for all $x\in G$; we have obtained $h\ge0$ and
  $\|h\|=1$ and $\|h-h s\|<\epsilon$, so $G$ is amenable by
  Theorem~\ref{thm:folneramen}(2).
\end{proof}

\noindent We recall without proof Cohen's factorization theorem:
\begin{lemma}[Cohen~\cite{cohen:factorization}]
  Let $\mathscr A$ be a Banach algebra with approximate identities,
  and consider $z\in\mathscr A$. Then for every $\epsilon>0$ there
  exists $x,y\in\mathscr A$ with $z=x y$ and $\|z-y\|<\epsilon$.\qed
\end{lemma}
For instance, it follows that if $G$ is an amenable group then
$\varpi(\ell^1G)^2=\varpi(\ell^1G)$.  Amenability, and the Liouville
property, are tightly related to the ideal structure of
$\ell^1(G)$. The following is in fact a reformulation of
Theorem~\ref{thm:kv}.
\begin{theorem}[Willis~\cite{willis:probability}]\label{thm:willis}
  Let $G$ be a group and let $X$ be a $G$-set. For a probability
  measure $\mu$ on $G$, let
  \[\ell^1_\mu(X)\coloneqq\overline{\{f-f\mu\mid f\in\ell^1(X)\}}
  \]
  denote the closed submodule of $\ell^1(X)$ generated by $1-\mu$, and
  write $\varpi(\ell^1X)=\{f\in\ell^1(X)\mid\sum_{g\in G}f(g)=0\}$. Then
  $(X,\mu)$ is Liouville if and only if $\ell^1_\mu(X)=\varpi(\ell^1X)$.

  In particular, $G$ is amenable if and only if
  $\{\ell^1_\nu(G)\mid \mu\in\mathscr P(G)\}$ has a unique maximal element,
  which is $\varpi(\ell^1G)$.
\end{theorem}
\begin{proof}
  Assume first that $(X,\mu)$ is Liouville, and consider an
  arbitrary $f\in\varpi(\ell^1X)$. By Proposition~\ref{prop:liouville}, we
  have $\|f\mu^n\|\to0$, so $f-f\mu^n\to f$, and
  $f-f\mu^n=f(1+\mu+\cdots+\mu^{n-1})(1-\mu)\in\ell^1_\mu(X)$, so
  $f\in\ell^1_\mu(X)$.

  Conversely, if $\mu$ is such that $\ell^1_\mu(X)=\varpi(\ell^1X)$, then
  given $f\in\varpi(\ell^1X)$ we may for every $\epsilon>0$ find
  $g\in\ell^1(X)$ with $\|f-g(1-\mu)\|<\epsilon$; then
  $\|f\cdot\frac1n\sum_{i=0}^{n-1}\mu^i\|\approx\|g(1-\mu^n)/n\|\to0$,
  so $f\mu^n\to0$. By Proposition~\ref{prop:liouville}, the random
  walk $(X,\mu)$ is Liouville.

  By Theorem~\ref{thm:kv}, $G$ is amenable if and only if there exists
  a Liouville measure on $G$.

  It remains to prove that if $\ell^1_\mu(X)$ is the unique maximal
  element in $\{\ell^1_\nu(X)\mid \nu\in\mathscr P(G)\}$ then
  $\ell^1_\mu(X)=\varpi(\ell^1X)$. For this, $f$ belong to
  $\varpi(\ell^1X)$ and write $f=g+i h$ with $g,h$ real. Furthermore,
  write $g=g^+-g^-$ and $h=h^+-h^-$ for positive $g^\pm,h^\pm$, and
  set $c=\sum_{x\in X}g^+(x)=\sum_{x\in X}g^-(x)$ and
  $d=\sum_{x\in X}h^+(x)=\sum_{x\in X}h^-(x)$. Then
  \[f=c(1-g^+/c) + (-c)(1-g^-/c)+(id)(1-h^+/d)+(-id)/(1-h^-/d),
  \]
  and each term belongs to some $\ell^1_\nu(X)$ and therefore to
  $\ell^1_\mu(X)$ because $\ell^1_\mu(X)$ is maximal; so
  $f\in \ell^1_\mu(X)$.
\end{proof}

\begin{exercise}[**, see~Johnson~\cite{johnson:cohomology}*{Proposition~5.1}]
  Let $\mathscr A$ be an amenable algebra, and let
  $J\triangleleft\mathscr A$ be a closed ideal. Prove that if $J$ and
  $\mathscr A/J$ are amenable, then $\mathscr A$ is
  amenable. Conversely, if $\mathscr A$ is amenable then
  $\mathscr A/J$ is amenable, and if $J$ has approximate identities
  then it is amenable.
\end{exercise}

\subsection{Amenable algebras}
We now turn to the group algebra $\Bbbk G$ for a field $\Bbbk$. Note
that we do not make any assumption on the field, which could be
finite.
\begin{definition}\label{defn:aa}
  Let $\mathscr A$ be an associative algebra, and let $V$ be an
  $\mathscr A$-module. We call $V$ \emph{amenable} if for every
  finite-dimensional subspace $S\le\mathscr A$ and every $\epsilon>0$
  there exists a finite-dimensional subspace $F\le V$ with
  \[\dim(F S)<(1+\epsilon)\dim(F).\]

  The algebra $\mathscr A$ itself is called \emph{amenable} if all
  non-zero $\mathscr A$-modules are amenable\footnote{Some people
    defined amenability of algebras --- erroneously, in my opinion ---
    as mere amenability of the regular right module.}.
\end{definition}

We note in passing that if $\mathscr A$ is finitely generated, then
the `$S$' in Definition~\ref{defn:aa} may be fixed once and for all to
be a generating subspace of $\mathscr A$.

\begin{theorem}[\cite{bartholdi:aa1}]\label{thm:aa}
  Let $G$ be a group and let $X$ be a $G$-set. Then $\Bbbk X$ is an
  amenable $\Bbbk G$-module if and only if $X$ is amenable.
\end{theorem}
\begin{proof}[Proof, after~\cite{gromov:linear}*{\S3.6}]
  ($\Rightarrow$) Consider the set $\mathcal O(X)$ of orders on $X$;
  it is a closed subspace of $\{0,1\}^{X\times X}$, so is compact. It
  is also the inverse limit of $\mathcal O(F)$ over all $F\Subset X$.

  Let $\Pi$ denote the group of all bijections of $X$. There exists a
  unique $\Pi$-invariant probability measure on $\mathcal O(X)$, which
  may be defined as the inverse limit of the uniform probability
  measures on $\mathcal O(F)$ over $F\Subset X$.  For an order
  ${\le}\in\mathcal O(X)$, consider
  \[\Phi^\le\colon\begin{cases}
      \{\text{finite-dim'l subspaces of }\Bbbk X\} &\to \{\text{finite subsets of }X\}\\
      W &\mapsto \big\{\min^\le(\supp(w))\mid w\in W\setminus\{0\}\big\},
    \end{cases}
  \]
  and let $m_W^\le\coloneqq\mathbb1_{\Phi^\le(W)}$ be the
  corresponding characteristic function in $\ell^1(X)$. We clearly
  have
  \begin{equation}\label{eq:mW}
    \|m_W^\le\|=\dim W,\qquad W_1\le W_2\Rightarrow m_{W_1}^\le m_{W_2}^\le\text{ pointwise}.
  \end{equation}
  Define then $m_W\coloneqq\int_{\mathcal O(X)}m_W^\le d\lambda(\le)$,
  and observe that~\eqref{eq:mW} still holds for $m_W$.  Furthermore,
  the map $W\mapsto m_W$ is $\Pi$-equivariant, so in particular is
  $G$-equivariant; and~\eqref{eq:mW} further implies
  $\|m_{W_2}-m_{W_1}\|=\dim W_2-\dim W_1$ whenever $W_1\le W_2$.

  Now given $S\Subset G$ finite and $\epsilon>0$, there exists
  $W\le\Bbbk X$ with $\dim(W+W s)<(1+\epsilon)\dim W$ for all
  $s\in S$, because $\Bbbk X$ is amenable. Thus
  $\|m_{W+W s}-m_W\|<\epsilon\dim W$, and similarly
  $\|m_{W+W s}-m_{W s}\|<\epsilon\dim W$, so
  \[\|m_W-m_{W s}\|=\|m_W-m_W s\|<2\epsilon\|m_W\|,
  \]
  and $G$ is amenable by Theorem~\ref{thm:folneramen}(2).

  ($\Leftarrow$) Let a finite-dimensional subspace $S$ of $\Bbbk G$
  and $\epsilon>0$ be given. There is a finite subset $T\Subset G$
  with $S\le\Bbbk T$, so because $X$ is amenable there is $F\Subset X$
  with $\#(F T)<(1+\epsilon)\#F$. Set $E\coloneqq\Bbbk F$; then
  \[\dim(E S)\le\dim((\Bbbk F)(\Bbbk T))\le\#(F
    T)<(1+\epsilon)\#F=(1+\epsilon)\dim E.\qedhere\]
\end{proof}

Note that, although $G_G$ is amenable if and only if
$\Bbbk G_{\Bbbk G}$ is amenable, the growth of almost-invariant
subsets and subspaces may behave quite differently. In
Example~\ref{ex:folnerll} we saw F\o lner sets $F_n$ for the
``lamplighter group'' $G$, and we may convince ourselves that they are
optimal, so $G$'s F\o lner function, see~\eqref{eq:fol}, satisfies
$\Fol(n)=n2^n$. On the other hand,
\[W_n=\Bbbk\bigg\{\sum_{\supp(f)\subseteq[-n,n]}(f,m)\mid
  m\in[-n,n]\bigg\}
\]
are subspaces of $\Bbbk G$ of dimension $2n+1$ with
$\dim(W_n+W_n s)/\dim W_n=\#(F_n\cup F_n s)/\#F_n$, so the ``linear
F\o lner function'' of $G$ grows linearly.

The following is an analogue, for linear spaces, of the space $\ell^1$
of summable functions on a set.  Let $V$ be a vector space. Consider
the free $\Z$-module with basis
$\{[A]\mid A\le V\text{ a finite-dimensional subspace}\}$, and let
$\ell^1(V,\Z)$ be its quotient under the relations
$[A]+[B]=[A\cap B]+[A+B]$ for all $A,B\le V$. Note that every
$x\in \ell^1(V,Z)$ may be represented as
$x=\sum_i[X_i^+]-\sum_j[X_j^-]$. Define a metric on $\ell^1(V,\Z)$ by
\[d(x,y)=\|x-y\|,\quad\|x\|=\inf\{\sum_i\dim(X_i)+\sum_j\dim(X_j^-)\mid x=\sum_i[X_i^+]-\sum_j[X_j^-]\}.\]
\begin{lemma}\label{lem:algl1}
  Let $\mathscr A$ be an algebra generated by a set $B$ of invertible
  elements, and let $V$ be an $\mathscr A$-module. Then $V$ is
  amenable if and only if for every $S\Subset B$ and every
  $\epsilon>0$ there exists $f\in\ell^1(V,\N)$ with
  $\|f-f s\|<\epsilon\|f\|$ for all $s\in S$.
\end{lemma}
\begin{proof}
  If $V$ is amenable, then for every $S\Subset B$ and every
  $\epsilon>0$ there exists $F\le V$ finite-dimensional with
  $\dim(F+F S)<(1+\epsilon)\dim F$; so in particular
  $\dim(F+F s)<(1+\epsilon)\dim F$ for all $s\in S$; since
  $\dim(F s)=\dim F$ because $s$ in invertible, we get
  $\dim(F\cap F s)>(1-\epsilon)\dim F$ so $f\coloneqq[F]$ satisfies
  $\|f-f s\|<2\epsilon\|f\|$.

  Conversely, given $S\Subset B$ and $f\in\ell^1(V,\N)$ with
  $\|f-f s\|<\epsilon\|f\|$ for all $s\in S$, we have
  $\sum_{s\in S}\|f-f s\|<\epsilon\#S\|f\|$. There is a unique
  expression $f=[X_0]+\cdots+[X_n]$ with $X_0\le \cdots\le X_n\le V$;
  so there exists $i\in\{0,\dots,n\}$ with
  $\sum_{s\in S}\|[X_i]-[X_i]s\|<\epsilon\#S\|[X_i\|$, and therefore
  $\sum_{s\in S}\dim(X_i+X_i s)<(1+\epsilon\#S)\dim X_i$, so
  $\dim(X_i+X_i S)<(1+\epsilon\#S)\dim X_i$. We are done since
  $S\Subset B$ was arbitrary and $B$ generates $\mathscr A$.
\end{proof}

\begin{corollary}
  Let $\mathscr A$ be a group ring. Then $\mathscr A$ is amenable if
  and only if the regular right module
  $\mathscr A\looparrowleft\mathscr A$ is amenable.
\end{corollary}
\begin{proof}
  Consider $\mathscr A=\Bbbk G$ a group ring.  If $\mathscr A$ is
  amenable, then obviously the regular module
  $\mathscr A_{\mathscr A}$ is amenable.

  Conversely, if $\mathscr A_{\mathscr A}$ is amenable, then $G_G$ is
  amenable by Theorem~\ref{thm:aa}. Let $V$ be a non-zero
  $\mathscr A$-module, and consider $v\in V\setminus\{0\}$. By
  Theorem~\ref{thm:folneramen}(5) for every $S\Subset G$ and every
  $\epsilon>0$ there exists a subset $F\Subset G$ with
  $\#(F\triangle F s)<\epsilon\#F$. Consider
  $x\coloneqq_{f\in F}[v f]\in\ell^1(V,\N)$, and note
  $\|x-x s\|<\epsilon\|x\|$. Thus $\mathscr A$ is amenable by
  Lemma~\ref{lem:algl1}.
\end{proof}

\begin{problem}[Gromov]
  Let $G$ be a group. If the $\R G$-module
  \[\mathcal C_0(G)=\{f\colon G\to\R\mid\inf_{F\Subset G}\sup(f\restrict G\setminus F)=0\}\]
  is amenable, does it follow that $G$ is amenable?
\end{problem}

%%%%%%%%%%%%%%%%%%%%%%%%%%%%%%%%%%%%%%%%%%%%%%%%%%%%%%%%%%%%%%%%
\newpage\section{Further work and open problems}

For lack of space, some important and interesting topics have been
omitted from this text. Here are a few of the most significant ones,
with very brief descriptions.

\subsection{Boundary theory}\label{ss:boundary}
Furstenberg initiated a deep theory of ``boundaries'' for random
walks. Given a random walk on a set $X$, say driven by a measure $\mu$
on a group $G$, a boundary is a measure space $(Y,\nu)$ with a
measurable map from the orbit space $(X^\N,\mu^\N)\to Y$ that
quotients through asymptotic equivalence, namely if $(x_0,x_1,\dots)$
and $(x'_0,x'_1,\dots)$ differ in only finitely many positions then
their images are the same in $Y$.

There is a universal such space, written $\partial(X,\mu)$ and called
the \emph{Poisson boundary} of $(X,\mu)$, such that all boundaries are
quotients of $\partial(X,\mu)$. This space, as a measure space, may be
characterized by the identity
\[L^1(\partial(X,\mu),\nu)=\ell^1(X)/\ell^1_\mu(X),
\]
see Theorem~\ref{thm:willis}. The Poisson boundary is reduced to a point if
and only if $(X,\mu)$ is Liouville.

In fact, it is better to view $\partial(X,\mu)$ as a measure space
with a family of measures $\nu_x$, one for each $x\in X$, satisfying
$\nu_{x g}=\nu_x g$ for all $g\in G$. One then has a ``Poisson
formula'' for harmonic functions on $X$: if $f\in\ell^\infty(X)$ is
harmonic, then there exists an integrable function $\hat f$ on
$\partial(X,\mu)$ such that
\[f(x)=\int_{\partial(X,\mu)}\hat f(\xi)d\nu_x(\xi).
\]
There is another construction of $\partial(X,\mu)$ based on
$\ell^\infty(X)$ rather than $\ell^1(X)$: the subspace
$h^\infty(X)\le\ell^\infty(X)$ of harmonic functions is a commutative
Banach algebra, under the product
\[(f_1\cdot f_2)(x)=\lim_{n\to\infty}\sum_{g\in G}f_1(x g)f_2(x g)\mu^n(g).
\]
The spectrum of $h^\infty(X)$, namely the set of algebra homomorphisms
$h^\infty(X)\to\C$, is naturally a measure space and is isomorphic to
$\partial(X,\mu)$. The function $\hat f$ is the Gelfand transform of
$f$, given by $\hat f(\xi)=\xi(f)$.

The Poisson boundary is naturally defined as a measure space, and is
directly connected to the space of bounded harmonic functions; but
other notions of boundary have been considered, for example the space
of \emph{positive} harmonic functions, leading to the \emph{Martin
  boundary} which is a well-defined topological space; for a natural
measure, it becomes measure-isomorphic to the Poisson boundary.

Glasner considers in~\cite{glasner:proximal} ``strongly amenable''
groups: they are groups all of whose proximal actions on a compact
space has a fixed point; see the comments at the end
of~\S\ref{ss:measures}.  Recall that an action of $G$ on a compact
Hausdorff space $X$ is \emph{proximal} if for every $x,y\in X$ there
exists a net $(g_n)$ of elements of $G$ such that
$\lim_n x g_n=\lim_n y g_n$.

For details, we refer to the original
articles~\cite{furstenberg:boundary,furstenberg:poisson}, the
classical~\cite{kaimanovich-v:entropy}, and the
survey~\cite{erschler:hyderabad}.

\subsection{Consequences}
Little has been said about the uses of amenability. On the one hand,
it plays a major role in the study of Lie groups and their lattices;
for example, Margulis's ``normal subgroup theorem'' states that a
normal subgroup of a lattice in a higher-rank semisimple Lie group is
either finite or finite-index~\cite{margulis:subgroups}. Ruling out
finite-index subgroups, the strategy is to show that such a group is
amenable and has property~(T).

Witte-Morris uses amenability, and Poincar\'e's recurrence theorem, to
prove in~\cite{wittemorris:orderable} that all finitely generated
amenable groups that act on the real line have homomorphisms onto
$\Z$.

Benjamini and Schramm consider in~\cite{benjamini-schramm:percolation}
\emph{percolation} on Cayley graphs. One fixes $p\in(0,1)$ and a
finitely generated group $G=\langle S\rangle$; call $\mathscr G$ the
corresponding Cayley graph. Then every vertex $v\in\mathscr G$ is made
independently at random ``open'' with probability $p$ (and ``closed''
with probability $1-p$). ``Open clusters'' are connected components of
the subgraph of $\mathfrak G$ spanned by open vertices. We define
\emph{critical probabilities}
\begin{align*}
  p_c&=\sup\{p\in(0,1)\mid \text{the open cluster containing $1$ is almost surely finite}\},\\
  p_u&=\inf\{p\in(0,1)\mid \text{there is almost surely a single infinite open cluster}\}.
\end{align*}
They conjecture that $p_c<1$ for all $G$ which are not virtually
cyclic; this is known for all groups of polynomial or exponential
growth, and for all groups containing subgroups of the form
$A\times B$ with $A,B$ infinite, finitely generated groups.

They also conjecture that $p_c<p_u$ holds precisely when $G$ is not
amenable; see~\cite{haggstrom-jonasson:uniqueness} for a survey of
known results.

\subsection{Ergodic theory}
One of the standard tools of ergodic theory is the ``Rokhlin-Kakutani
lemma'': let $T\colon X\righttoleftarrow$ be an invertible,
measure-preserving transformation of a measure space $(X,\mu)$ that is
\emph{aperiodic} in the sense that almost all points have infinite
orbits. Then for every $n\in\N$ and every $\epsilon>0$ there exists a
measurable subset $E\subseteq X$ such that $E,T(E),\dots,T^{n-1}(E)$ are
all disjoint with $\mu(E\sqcup\cdots\sqcup T^{n-1}(E))>1-\epsilon$.

It may be understood as the following statement. Given $S,T\colon
X\righttoleftarrow$, define their distance as $d(S,T)=\mu(\{x\in X\mid
S(x)\neq T(x)\})$. Then for every $n\in\N,\epsilon>0$ there exists $S$ of
period $n$ with $d(S,T)<\epsilon$. In other words, $\Z$ may be approximated
arbitrarily closely by $\Z/n$.

The Rokhlin lemma is essential in reducing ergodic theory problems to
combinatorial ones. For example, it serves to prove that two Bernoulli
shifts (the shift on $\Alphabet^\Z$ for a given probability measure on
$\Alphabet$) are isomorphic if and only if they have the same entropy.

Ornstein and Weiss generalize in~\cite{ornstein-weiss:entropyiso} the
Rokhlin lemma to some amenable groups; see
also~\cite{weiss:rokhlin}. Let $G$ be a group; we say that a subset
$F\Subset G$ \emph{tiles} $G$ if $G$ is a disjoint union of translates
of $F$; namely, if there exists a subset $C\subseteq G$ with
$G=\bigsqcup_{c\in C}F c$. They prove:
\begin{theorem}
  Let $G$ be amenable, and let $F\Subset G$ be a finite subset. Then $F$
  tiles $G$ if and only if for every free measure-preserving action of $G$
  on a probability space $(X,\mu)$ and every $\epsilon>0$ there is a
  measurable subset $E\subseteq X$ such that $\{E f\mid f\in F\}$ are all
  disjoint and $\mu(E F)>1-\epsilon$.
\end{theorem}

In~\cite{weiss:monotileable}, Weiss calls $G$ \emph{monotileable} if
it admits arbitrarily large tiles. He proves that amenable, residually
finite are monotileable; more precisely, in F\o lner's definition of
amenability it may be assumed that the F\o lner sets tile $G$. For
example, $\Z$ is tiled by sets of the form $\{-n,\dots,n\}$ which form
an exhausting sequence of F\o lner sets are are also transversals for
the subgroups $(2n+1)\Z$.

Let us denote by $MG$ the class of monotileable groups; then $MG$
contains all residually amenable groups, and is closed under taking
extensions, quotients, subgroups and directed
unions~\cite{chou:elementary}*{\S4}.

It is at the present (2017) unknown whether \emph{every} group is
monotileable, and whether $AG\subseteq MG$. It is also unknown
whether, if a group $G$ belongs to $MG\cap AG$, then $G$ may be tiled
by F\o lner sets.

\subsection{Numerical invariants}
Recall that the \emph{entropy} of a probability measure $\mu$ on a
countable set $X$ is defined as
\[H(\mu)=-\sum_{x\in X}\mu(x)\log\mu(x),\text{ where as usual }0\log(0)=0.
\]

The Liouville property can, in some favourable cases, be detected by a
single numerical invariant, its \emph{entropy} or its
\emph{drift}. Given a random walk $p$ on a set $X$, starting at
$x\in X$, its \emph{entropy growth} is the function
$h(n)\coloneqq H(p_n(x,{-}))$ computing the entropy of distribution of
the random walker after $n$ steps. If furthermore $X$ is a metric
space, the \emph{drift growth} of $p$ is the function
$\ell(n)\coloneqq\sum_{y\in X}p_n(x,y)d(x,y)$ estimating the expected
distance from the random walker to the origin after $n$ steps.

A celebrated criterion by Derriennic~(\cite{derriennic:thergodique};
see also~\cite{kaimanovich-v:entropy}) shows that, if $H(\mu)<\infty$,
then $(X,\mu)$ is Liouville if and only if $h$ is sublinear. Moreover,
the volume, entropy and drift growth are related by the inequality
\[\lim_{n\to\infty}\frac{\log v(n)}n\lim_{n\to\infty}\frac{\ell(n)}n\le \lim_{n\to\infty}\frac{h(n)}n.
\]
Finer estimates relate these functions $\log v,\ell,h$, in particular
if all are sublinear; additionally, the probability of return
$p(n)=-\log p_n(x,x)$ and the \emph{$\ell^p$ distortion}
\[d_p(n)=\sup_{\Phi\colon G\to\ell^p\text{
      $1$-Lipschitz}}\inf\{\|\Phi(g)-\Phi(h)\|\mid d(g,h)\ge n\}
\]
are all related by various inequalities;
see~\cite{gournay:liouville,naor-peres:embeddings,peres-zheng:decay}.

\subsection{Sofic groups}
The class of sofic groups is a common extension of amenable and
residually-finite groups. We refer
to~\cite{weiss:sofic,gromov:endomorphisms} for its introduction. The
definition may be seen as a variant of F\o lner's criterion:
\begin{definition}
  Let $G$ be a group. It is \emph{sofic} if for every finite subset
  $S\Subset G$ and every $\epsilon>0$ there exists a finite set $F$
  and a mapping $\pi\colon S\to\Sym(F)$ such that
  \begin{align*}
    \text{ if $s,t,s t\in S$ then }&\#\{f\in F\mid f\pi(s)\pi(t)\neq f\pi(s t)\}<\epsilon\#F,\\
    \text{ if $s\neq t\in S$ then }&\#\{f\in F\mid f\pi(s)=f\pi(t)\}<\epsilon\#F.
  \end{align*}
\end{definition}
Two cases are clear: if $G$ is residually finite, then for every
$S\Subset G$ there exists a homomorphism $\rho\colon G\to F$ to a
finite group that is injective on $S$; define then $f\pi(s)=f\rho(s)$
for all $s\in S,f\in F$, showing that $G$ is sofic. If on the other
hand $G$ is amenable, then for every $S\Subset G$ and every
$\epsilon>0$ there exists $F\Subset G$ with
$\#(F S\setminus F)<\epsilon\#F$; define then $f\pi(s)=f s$ if
$f s\in F$, and extend the partial map
$\pi(s)\colon F\dashrightarrow F$ arbitrarily into a permutation,
showing that $G$ is sofic.

Remarkably, there is at the present time (2017) no known example of
a non-sofic gruop.

\subsection{It this group amenable?}
We list here some examples of groups for which it is not known whether
they are amenable or not. These problems are probably very hard.

\begin{problem}[Geoghegan]\label{problem:thompson}
  Is Thompson's group $F$ amenable?

  Recall that $F$ is the group of piecewise-linear homeomorphisms of
  $[0,1]$, with slopes in $2^\Z$ and breakpoints in $\Z[\tfrac12]$;
  see~\cite{cannon-f-p:thompson} and Example~\ref{ex:F}.
\end{problem}
There have been numerous attempts at answering
Problem~\ref{problem:thompson}, too many to cite them all; a promising
direction appears in~\cite{moore:hindman}. Kaimanovich proves
in~\cite{kaimanovich:liouville} that, for every finitely supported
measure $\mu$ on $F$, the orbit $(\tfrac12 F,\mu)$ is not Liouville;
however Juschenko and Zhang prove in~\cite{juschenko-zhang:liouville}
that $\tfrac12 F$ \emph{is} laminable.

There is a group that is related to $F$, and acts on the circle
$[0,1]/(0\sim1)$: it satisfies the same definition as $F$, namely the
group $T$ of piecewise-linear self-homeomorphisms with slopes in
$2^\Z$ and breakpoints in $\Z[\tfrac12]/\Z$. Its amenable subgroup
$\Z[\tfrac12]/\Z$ acts transitively on the orbit $0T$, so
$0T\looparrowleft T$ is laminable by Corollary~\ref{cor:amentrans}.

\begin{problem}[Nekrashevych]
  Are all contracting self-similar groups amenable?

  Recall that a self-similar group is a group $G$ generated invertible
  transducers; it acts on $\Alphabet^\N$, and may be given by a map
  $\phi\colon G\to G\wr_\Alphabet\Sym(\Alphabet)$, as
  in~\eqref{eq:treerec}. It is \emph{contracting} if there is a proper
  metric on $G$ and constants $\lambda<1,C$ such that whenever
  $\phi(g)=\pair<g_1,\dots,g_{\#\Alphabet}>\pi$ we have
  $\|g_i\|<\lambda\|g\|+C$. See~\cite{nekrashevych:ssg}.
\end{problem}

\begin{problem}[Folklore, often attributed to Katok]
  Is the group of interval exchange transformations amenable? Does it
  contain non-abelian free subgroups?

  A partial, positive result appears in Example~\ref{ex:iet}. It would
  suffice, following the strategy in that example
  (see~\cite{juschenko-mattebon-monod-delasalle:extensive}*{Proposition~5.3}),
  to prove that the group of $\Z^d$-wobbles $W(\Z^d)$ acts extensively
  amenably on $\Z^d$ for all $d\in\N$; at present (2017), this is
  known only for $d\le2$, see Theorem~\ref{thm:iet}.
\end{problem}

\end{document}